\pdfoutput=1
\documentclass[12pt,oneside,reqno]{amsart}
\usepackage{lmodern}
\usepackage[T1]{fontenc}
\usepackage[english,activeacute]{babel}
\usepackage[applemac]{inputenc}
\usepackage{cite}
\usepackage{geometry}      		
\usepackage[normalem]{ulem}		
\usepackage{graphicx}			
\usepackage{amsmath}
\usepackage{amsfonts}
\usepackage{amssymb}
\usepackage{amsthm}
\usepackage{bbm}
\usepackage{MnSymbol}	
\usepackage{verbatim}
\usepackage{mathrsfs}
\usepackage[foot]{amsaddr}
\usepackage{color}
\usepackage{pgf}
\usepackage{tikz}

\usetikzlibrary{arrows,shapes,decorations,automata,backgrounds,petri,positioning}

\tikzset{main node/.style={circle,fill=blue!20,draw,minimum size=0.8cm,inner sep=0pt},}

\usepackage{layout}

\newtheorem{thm}{Theorem}[section]
\newtheorem{cor}[thm]{Corollary}
\newtheorem{lem}[thm]{Lemma}
\newtheorem{prop}[thm]{Proposition}
\theoremstyle{definition}
\newtheorem{defn}[thm]{Definition}
\theoremstyle{remark}

\newtheorem{ex}[thm]{Example}

\newcommand{\opn}[1]{\operatorname{#1}}

\newcommand{\mbb}[1]{\mathbb{#1}}
\newcommand{\mc}[1]{\mathcal{#1}}

\newcommand{\bs}[1]{\boldsymbol{#1}}

\def\>{\rangle}
\def\<{\langle}
\def\RR{\rrangle}

\def\im{\opn{Im}}
\def\tr{\opn{Tr}}

\def\ker{\opn{ker}}

\def\spn{\opn{span}}

\def\0{\bs{0}}
\def\1{\mathbbm{1}}
\def\N{\mbb{N}}
\def\Z{\mbb{Z}}
\def\R{\mbb{R}}
\def\C{\mbb{C}}
\def\T{\mbb{T}}
\def\D{\mbb{D}}

\def\HH{\mathscr{H}}
\def\DD{\mathscr{D}}
\def\RR{\mathscr{R}}
\def\SS{\mathscr{S}}
\def\BB{\mathscr{B}}

  \def\XXint#1#2#3{{\setbox0=\hbox{$#1{#2#3}{\int}$}
      \vcenter{\hbox{$#2#3$}}\kern-.47\wd0}}

\begin{document}


\title[A generalization of Schur functions]
{A generalization of Schur functions: \\ applications to Nevanlinna functions, \\ orthogonal polynomials, random walks \\ and unitary and open quantum walks}
\author{\small F.A. Gr\"unbaum$^1$, L. Vel\'azquez$^2$}
\address{\scriptsize 
$^1$Department of Mathematics, University of California, Berkeley, CA 94720, USA}
\address{\scriptsize 
$^2$Departamento de Matem\'{a}tica Aplicada $\&$ IUMA, Universidad de Zaragoza, Mar\'{\i}a de Luna 3, 50018 Zaragoza, Spain.  
\footnotesize
\it E-mail address: \rm \texttt{velazque@unizar.es}
\it (corresponding author)}



\date{}					

\begin{abstract}

Recent work on return properties of quantum walks (the quantum analogue of random walks) has identified their generating functions of first returns as Schur functions. This is connected with a representation of Schur functions in terms of the operators governing the evolution of quantum walks, i.e. the unitary operators on Hilbert spaces. 

In this paper we propose a generalization of Schur functions by extending the above operator representation to arbitrary closed operators on Banach spaces. Such generalized `Schur functions' meet the formal structure of first return generating functions, thus we call them FR-functions. We derive some general properties of FR-functions, among them a simple relation with an operator version of Stieltjes functions which generalizes the renewal equation already known for random and quantum walks. We also prove that FR-functions satisfy splitting properties which extend useful factorizations of Schur functions. 

When specialized to self-adjoint operators on Hilbert spaces, we show that FR-functions become Nevanlinna functions. This allows us to obtain properties of Nevanlinna functions which, as far as we know, seem to be new. The FR-function structure leads to a new operator representation of Nevanlinna functions in terms of self-adjoint operators, whose spectral measures provide also new integral representations of such functions. This allows us to characterize each Nevanlinna function by a measure on the real line, which we refers to as `the measure of the Nevanlina function'. In contrast to standard operator and integral representations of Nevanlinna functions, these new ones are exact analogues of those already known for Schur functions. The above results are also the source of a very simple `Schur algorithm' for Nevanlinna functions based on interpolations at points on the real line, which we refer to as the `Schur algorithm on the real line'.  
 
The paper is completed with several applications of FR-functions to orthogonal polynomials and random and quantum walks which illustrate their wide interest: an analogue for orthogonal polynomials on the real line of the Khrushchev formula for orthogonal polynomials on the unit circle, and the use of FR-functions to study recurrence in random walks, quantum walks and open quantum walks. These applications provide numerous explicit examples of FR-functions, clarifying the meaning of these functions --as first return generating functions-- and their splittings --which become recurrence splitting rules.  They also show that these new tools, despite being extensions of very classical ones, play an important role in the study of physical problems of a highly topical nature. 


\end{abstract}

\maketitle

\tableofcontents

\section{Introduction}
\label{sec:INTRO}

The analytic functions mapping the open unit disk into its closure are known as Schur functions. These functions, as well as their matrix and operator valued versions, are central objects in harmonic analysis, but their interest goes beyond this area, covering a wide variety of applications such as linear system theory, electrical engineering, signal processing, geophysics, stochastic processes, operator theory, interpolation problems, orthogonal polynomials on the unit circle (OPUC) or quantum walks (QW).

Among the main and most fruitful features of these functions is their characterization by a sequence of complex parameters --the Schur parameters-- arising from the so called Schur algorithm \cite{Schur} (see \cite{DGK-S} for the matrix valued case), which generates a sequence of Schur functions --the Schur iterates-- starting from the original one. The standard Schur algorithm, based on the evaluation of the iterates at the origin, can be generalized to deal with evaluation points arbitrarily chosen in the open unit disk, and then it is known under the name of the Nevanlinna-Pick algorithm \cite{Pick,Nevanlinna} (see \cite{DGK-N} for the matrix valued case).  

Schur functions are also characterized by an integral representation which follows from the Riesz-Herglotz representation of Carath\'eodory functions \cite{Herglotz,Riesz} (see \cite[Sect. I.4]{Brodskii} for the operator valued case), i.e. the analytic functions mapping the open unit disk into the closed right half-plane. This yields a one-to-one correspondence between Schur functions and normalized measures on the unit circle, linking such functions to unitary operators via spectral measures. This leads to a representation of Schur functions in terms of unitary operators and orthogonal projections \cite{GVWW,BGVW,CGVWW}, which is connected to a known realization as transfer/characteristic functions related to unitary colligations \cite{NF,Brodskii2}. 

A further feature which adds value to Schur functions is the existence of factorization properties related to certain compositions of linear systems or unitary operators. The importance of these factorizations can be hardly summarized in a single paragraph. As an example let us mention their key role in the invariant subspace problem of operator theory \cite{NF,Brodskii}, in Wiener-Hopf methods and realizability in linear systems \cite{BGK,Kaashoek}, in a recent OPUC revolution known as Khrushchev theory \cite{Khrushchev,Khrushchev2,Simon-OPUC} (see \cite{CGVWW} for the matrix-valued case) or in the study of splitting rules for QW recurrence, see Section~\ref{sec:QW}. 

Some of these results can be generalized to the so called Nevanlinna functions, the analogue of the Schur functions when the open unit disk is replaced by the open upper half-plane. For instance, the linear fractional transformations between these two domains have suggested Nevanlinna analogues of the Schur algorithm which use interpolation points in the open upper half-plane, see \cite[Chapter 3]{Akhiezer} and \cite{ADL,ADLV,De}. Nevertheless, these algorithms are not as simple as that one for Schur functions. 

This is in contrast with the simplicity of the standard Schur algorithm for Nevanlinna functions of Stieltjes type, which uses a point in the boundary of the upper half-plane, the point at infinity, as interpolation point, see \cite[Chapter 3]{Akhiezer} and \cite{ADL2}. This suggests to explore the possibility of taking the rest of such a boundary, i.e. the real line, as a natural place for the interpolation points of a simple Schur algorithm for Nevanlinna functions. Among other things, this would have the benefit of leading to a simple Nevanlinna version of the Nevanlinna-Pick algorithm when chosing a different interpolation point in the real line at each step.

Nevanlinna functions also have integral and operator representations, but in terms of measures on the real line and self-adjoint operators instead of measures on the unit circle and unitary operators, see \cite[Chapter III]{ST}, \cite[Sect. I.4]{Brodskii}, \cite{Shmulyan,NK,GT,BHST,BS,GKMT,KL} and references therein. However, in contrast to the case of Schur functions, the integral representations either are restricted to certain type of Nevanlinna functions or require additional parameters, thus they do not yield a one-to-one correspondence between Nevanlinna functions and measures. Besides, these integral and operator representations are qualitatively different from those of Schur functions. 

Mappings between the open unit disk and the open upper half-plane may be useful for surmising results for Nevanlinna functions inspired by the case of Schur functions, but this is not the end of the story. For instance, these transformations do not suggest the simplicity of taking the point at infinity as interpolation point in a Schur algorithm for Nevanlinna functions (concerning the differences beween the interpolation problems for Schur and Nevanlinna functions, see \cite{De2}). Besides, such mappings have not been used to understand the Nevanlinna version of factorizations already known for Schur functions. This also holds for the Nevanlinna version of applications already developed for Schur functions. By analogy with that case one should expect, for instance, applications of Nevanlinna functions to orthogonal polynomials on the real line (OPRL), as well as to symmetrizable random walks (RW) --i.e., RW whose stochastic matrix can be made self-adjoint--, such as those which are both irreducible and reversible \cite[Chapter 6]{Stroock}. 

\smallskip 

The previous discussion leads to the following natural questions:

\begin{itemize}

\item Is there any simple Schur algorithm for Nevanlinna functions based on interpolations at points on the real line?  

\item Are there integral and operator representations of Nevanlinna functions which resemble those of Schur functions? Do they establish a one-to-one correspondence between the set of all Nevanlinna functions and some set of measures? 

\item Do Nevanlinna functions satisfy any splitting properties that could be viewed as an analogue of known factorizations of Schur functions?

\item Do the previous ideas shed light on applications of Nevanlinna functions which parallel those of Schur functions? In particular, do they have any interesting consequence for the study of OPRL or RW? 

\end{itemize}

Among the main results of this paper is the discovery of new properties of Nevanlinna functions which answer affirmatively the above questions. We do not intend to state that these new findings on Nevanlinna functions are, for any purpose, better than similar ones already existing in the literature. However these new results are specially useful for the applications to OPRL and RW described in Sections~\ref{sec:OP} and \ref{sec:RW}. 

Answering the first question, we will see that there exists a very natural and surprisingly simple Schur algorithm for Nevanlinna functions which uses interpolation points lying in the real line instead of the interior of the upper half-plane. This shows once more that the strategy of translating ideas from Schur to Nevanlinna functions by using M\"obius transformations does not exhaust all the possibilities, and does not even yield necessarily the simplest answers. A single step of the alluded Schur algorithm looks like
$$
 f(z) \to \frac{1}{z} \frac{f(z)-f(0)-f'(0)z}{f(z)-f(0)}
$$
for a Nevanlinna function $f$ analytic at the origin. A real translation gives the version of this algorithm for Nevanlinna functions analytic at other points on the real line, while the matrix valued case only requires a slight modification which is similar to the one required for matrix valued Schur functions, see \eqref{eq:Nalg-op} and Definition~\ref{def:SAR}. Also, the analyticity condition can be weakened to the existence of derivatives along normal directions to the real line.     

We will also obtain new integral and operator representations of Nevanlinna functions which are in perfect analogy with those already known for Schur functions, see Theorem~\ref{thm:N=FR}. In particular, we will find an integral representation without additional parameters which characterizes each Nevanlinna function by a measure on the real line --much in the same way as in the case of Schur functions and measures on the unit circle-- which we call `the measure of the Nevanlinna function', see Theorem~\ref{thm:N=FR}.{\it(ii)} and Definition~\ref{def:N-mu}. Explicitly, the measure $\mu$ of a Nevanlinna function $f$ is defined by the following identity for the Stieltjes function of $\mu$,
\begin{equation} \label{eq:int-rep}
 \int\frac{d\mu(t)}{1-zt} = (1-zf(z))^{-1}.
\end{equation} 
This is the exact analogue of the known integral representation of Schur functions via Carath\'eodory functions, except for the fact that now the measure is not necessarily normalized by $\mu(\R)=1$, but $\mu(\R)\le1$. Actually, we will show that the integral representation \eqref{eq:int-rep} establishes a one-to-one correspondence between Nevanlinna functions and measures on the real line such that $\mu(\R)\le1$. All these results remain true for matrix valued Nevanlinna functions and measures, where 1 stands for the identity matrix of appropriate size wherever necessary.

As for the new operator representations of Nevanlinna functions, they are given in terms of self-adjoint operators and orthogonal projections in a way almost identical to the representation of Schur functions by unitary operators given in \cite{GVWW,BGVW,CGVWW}, see Theorem~\ref{thm:N=FR}.{\it(iv)}. Namely, we will prove that every Nevanlinna function can be expressed as
$$
 f(z) = -dz^{-1} + PT(1-zQT)^{-1}P,
 \qquad
 \begin{aligned}
 & T \text{ self-adjoint operator}, \quad d\ge0,
 \\
 & P \text{ rank 1 orthogonal projection}, \quad Q=1-P,
\end{aligned}
$$
and viceversa. In the matrix valued case $d$ becomes a non-negative definite matrix and $P$ is simply finite rank. The only quirk of the Nevanlinna case is the eventual presence of additional terms proportional to $-z^{-1}$, the only ones that cannot be reproduced by the operator representation. This difference with respect to Schur functions is related to that one in the normalization condition of the measure $\mu$ previously noticed. Actually, $d=0$ exactly when $\mu(\R)=1$ because in this case $\mu$ becomes a spectral measure for the self-adjoint operator $T$, see Corollary~\ref{cor:N=FR}. This is one of the interesting features of the integral representation \eqref{eq:int-rep} since such spectral measures codify dynamical properties of random systems described by the self-adjoint operator $T$ \cite{Last,RS,KMcG,DS,MW} (see \cite{GVWW,BGVW,CGMV,CGMV1,ASW} for the unitary case).     

Regarding the splitting properties, it can be seen that certain factorizations of a unitary operator yield similar factorizations of Schur functions generated by such an operator \cite{NF,Brodskii2,BGK,CGVWW}. Analogously, we will show that a decomposition of a Nevanlinna function into a sum of other ones follows from similar decompositions of the self-adjoint operator giving the corresponding operator representation, see Theorem~\ref{thm:split}.{\it(i)} particularized to self-adjoint operators. This can be considered as a Nevanlinna analogue of the alluded factorizations of Schur functions.   

The parallelism between these new results for Nevanlinna functions and those already known for Schur functions opens the possibility of exporting to the Nevanlinna case the applications of Schur functions previously described. For instance, as it is shown in \cite{CGVWW}, Khrushchev formula \cite{Khrushchev,Khrushchev2} (see also \cite[Chapters 4 and 9]{Simon-OPUC}) --the cornerstone of OPUC Khrushchev theory-- can be understood as a factorization of Schur functions generated by certain factorizations of the so called CMV matrices \cite{CMV,Watkins} (see also \cite[Chapter 4]{Simon-OPUC}), unitary analogue of Jacobi matrices. Analogously, we will show that the decomposition properties of Nevanlinna functions uncovered in the present work lead to a Khrushchev type formula for OPRL --see Theorem~\ref{thm:K-R}-- which should be the starting point of a Khrushchev theory for OPRL. More precisely, this OPRL Khrushchev formula will be a consequence of certain decompositions of the Jacobi matrix encoding the OPRL three term recurrence relation.  

Besides, we will prove that the success of Schur functions in encrypting the recurrence --i.e., the return properties-- of QW \cite{GVWW,BGVW,CGMV2} has a Nevanlinna counterpart: the recurrence of a symmetrizable RW is best codified by Nevanlinna functions associated with the related self-adjoint operator, see Section~\ref{sec:RW} when applied to such a RW. Also, it will be shown that the role of the factorization properties of Schur functions as splitting rules for QW recurrence is paralleled by the splitting rules for RW recurrence generated by the decomposition properties of Nevanlinna functions, see Theorem~\ref{thm:split-RW}.{\it(i)} when applied to a symmetrizable RW.

Actually, we will do more than this since we will prove that both, Schur and Nevanlinna functions, are just a particular case of a much more general class of functions linked to arbitrary closed operators and bounded projections on Banach spaces, see Definition~\ref{def:FR}. The starting point for this generalization is the extension to arbitrary operators of the representation of Schur functions in terms of unitaries uncovered in \cite{GVWW,BGVW,CGVWW}, giving rise to the definition of what we will call first return functions --the name refers to the fact that these functions have the formal structure of a generating function of first returns--, in short FR-functions. More precisely, we will define an FR-function $f$ by the operator representation
\begin{equation} \label{eq:fr}
 f(z) := PT(1-zQT)^{-1}P,
 \qquad
 \begin{aligned}
 & T \text{ closed operator},
 \\
 & P \text{ bounded projection}, \quad Q=1-P.
 \end{aligned}
\end{equation}

On the one hand, FR-functions arise as an abstract generalization of the notion of generating function of first returns, whose origin goes back to G.~P\'olya's celebrated paper on RW recurrence \cite{Polya} (see also \cite{Feller,Stroock}). Thus their study has a direct impact in the analysis of recurrence in different dynamical systems, see Sections~\ref{sec:RW}, \ref{sec:QW} and \ref{sec:OQW}. On the other hand, FR-functions provide generalizations of transfer/characteristic functions \cite{NF,Brodskii,Brodskii2,BGK,Kaashoek} to unbounded situations --see Proposition~\ref{prop:falt} and subsequent comments--,  a fact which is behind new representations of Nevanlinna functions in terms of self-adjoint operators and measures on the real line, see Theorem~\ref{thm:N=FR} and Corollary~\ref{cor:N=FR}.
 
The interesting fact is that some properties already known in much more restricted settings are satisfied by FR-functions in this general abstract context of operators on Banach spaces, even in the absence of an inner product. For instance, we will see in Theorem~\ref{thm:dom2} that \eqref{eq:int-rep} generalizes to every FR-function \eqref{eq:fr} as 
$$
 s(z) := P(1-zT)^{-1}P = (1-zf(z))^{-1},
$$
where $s$ is an operator version of the Stieltjes function of a measure, which formally plays the role of a generating function of returns. This relation between generating functions of returns and first returns also generalizes the well known renewal equation for RW, first derived by G.~P\'olya in \cite{Polya} (see also \cite[Chapter XIII]{Feller} and \cite[Chapter 4]{Stroock}), as well as its QW version which has been recently obtained \cite{GVWW,BGVW,CGMV2}.

Another example of the properties which hold for general FR-functions of operators on Banach spaces are the factorizations and decompositions mentioned above in the unitary and self-adjoint case respectively. We will prove that these splitting properties have a mere algebraic nature and are simply a consequence of similar splittings of the underlying operator, regardless of any symmetry that the operator may have, see Theorem~\ref{thm:split}. 

The extension of these results to the general setting of operators on Banach spaces is not a mere exercise of abstraction, but it is motivated by the requirements of important applications. For instance, the stochastic matrix of a RW is connected to a self-adjoint operator only in some situations --for instance, when the RW is irreducible and reversible \cite[Chapter 6]{Stroock}--, however every stochastic matrix defines an operator in a Banach space. Therefore, the general results on FR-functions apply to every RW and not only to symmetrizable ones. In particular, this allows us to prove that the splitting rules --including not only decomposition, but also factorization rules-- for RW recurrence constitute a completeley general feature of RW, see Theorem~\ref{thm:split-RW}.   

Even more interesting is the possibility of applying the results on FR-functions to more burning issues. The emergence of quantum information as one of the most promising scientific developments for the practical implementation of quantum technologies has promted the research on quantum systems in interaction with the macroscopic world to better reproduce realistic situations. Among these kind of systems are the so called quantum channels, which have been recently used to build a version of QW in interaction with the enviroment under the name of open QW (OQW) \cite{APSS} (see also \cite{APS,SP,SP2,SP3,SaPa,LS,LS2,CGL,CP,CP2,SKA,SKKA}). The evolution of such systems is governed by an operator on a Banach space rather than on a Hilbert space. Hence, every result on general FR-functions may be applied to the study of such open quantum systems. We will show that, as in the case of standard QW, the recurrence properties of OQW --or more generally, iterated quantum channels-- are codified by FR-functions of the corresponding evolution operator. This not only makes possible to perform calculations which are hard to tackle with other methods, but also allows us to generalize to OQW results obtained for RW or QW recurrence, such as establishing a renewal equation or developing splitting techniques for OQW recurrence, see Section~\ref{sec:OQW}. 

Summarizing, this work is an attempt to extend tools and techniques of harmonic analysis and operator theory already used with great success for the study of unitary operators, OPUC and QW, to other contexts such as self-adjoint operators, OPRL, RW and OQW. This requires the generalization of the notion of Schur function and its properties to arbitrary operators on Banach spaces. As a byproduct, the specialization of these ideas to self-adjoint operators on Hilbert spaces leads to results on Nevanlinna functions which, to the best of our knowledge, seem not to have been noticed before.

The content of the paper is organized as follows: the rest of the introduction is devoted to a summary on Schur functions and unitary operators --highlighting the results that we wish to generalize-- and also to some results on Schur complements for operators on Banach spaces which will be of interest for the analysis of FR-functions. The whole paper revolves around the concept of FR-function as a generalization of the notion of Schur function. FR-functions for arbitrary operators on Banach spaces are introduced in Section~\ref{sec:GF}, which also includes some results on FR-functions, the main one being the generalization of the renewal equation to this abstract setting. The special features of the self-adjoint case are considered in Section~\ref{sec:s-a}, which links FR-functions of self-adjoint operators to Nevanlinna functions. This relation is examined in more detail in Section~\ref{sec:rep-N}, which provides different equivalent characterizations of Nevanlinna functions, among them integral and operator representations which are in perfect analogy with those of Schur functions. The FR-function approach leads to a Schur algorithm for Nevanlinna functions based on interpolations at points on the real line, which we call the Schur algorithn on the real line. This algorithm and its different generalizations are discussed in Section~\ref{sec:SA}. The factorization and decomposition formulas for FR-functions are derived in Section~\ref{sec:KF} from similar splittings of the underlying operators. The application of such decomposition formulas to Nevanlinna functions associated with Jacobi matrices leads to a Khrushchev formula for OPRL, which is the aim of Section~\ref{sec:OP}. Sections~\ref{sec:RW}, \ref{sec:QW} and \ref{sec:OQW} are devoted to the applications of FR-functions to the study of recurrence in RW, QW and OQW respectively. The examples presented in these sections also provide many explicit examples of FR-functions and illustrate some of their features, such as the splitting properties mentioned above. Finally, Appendices~A and B prove some technical results on Nevanlinna functions which are needed for the development of the paper.

\subsection{The unitary case - Schur functions}
\label{ssec:u}
\quad

Let us summarize some known connections between Schur functions and unitary operators whose origin lies in the study of QW, to better understand the kind of results we wish to generalize. Quick references to scalar and matrix valued Schur functions which stress the aspects that we need can be found in \cite[Chapter 1]{Simon-OPUC} and \cite[Sect. 1 and 3]{DPS} respectively. The connection between Schur functions and unitary operators via QW recurrence appears in \cite{GVWW,BGVW,CGVWW}. For good references to QW, a notion first introduced in \cite{ADZ}, see \cite{Ambainis,Childs,Kempe,Kendon}. Also, Section~\ref{sec:QW} provides a summary of QW recurrence which may be useful to follow some of the ideas presented in this section. 

Given any unitary operator $U$ on a Hilbert space $\HH$, we can consider the QW driven by $U$. If $P$ is the orthogonal projection onto a closed subspace $\HH_0$ and $Q=1-P$ is the projection onto $\HH_0^\bot$, the generating function of the {\bf first time return} amplitudes to $\HH_0$ is (up to multiplication by $z$) the function $f$ with values in operators on $\HH_0$ given by 
\begin{equation} \label{eq:f-U}
 f(z) = \sum_{n\ge0} z^n PU(QU)^nP = PU(1-zQU)^{-1}P = P(U^\dag-zQ)^{-1}P. 
\end{equation}
This function is obviously analytic on the unit disk 
$$
 \D:=\{z\in\C : |z|<1\}.
$$
Note that this region coincides with that one bounded by the unit circle 
$$
 \T:=\{z\in\C : |z|=1\}
$$ 
where the spectrum of a unitary lives. 

We will refer to the {\bf first return generating function} $f$ in short as the {\bf FR-function} of the subspace $\HH_0$ with respect to the unitary $U$.

We point out that other generating functions appear in connection with return properties of QW. The generating function of the return amplitudes to $\HH_0$ is given by
\begin{equation} \label{eq:s-U}
 s(z) = \sum_{n\ge0} z^nPU^nP = P(1-zU)^{-1}P,
\end{equation}
which is again analytic on $\D$. The spectral decomposition $U=\int\!t\,dE(t)$ gives to this generating function a special meaning since one has
\begin{equation} \label{eq:s-mu}
 s(z) = \int \frac{d\mu(t)}{1-zt},
 \qquad \mu = PEP.
\end{equation}
That is, $s$ is the Stieltjes function of a measure $\mu$ suported on $\T$, namely, the spectral measure of $\HH_0$ with respect to $U$. As in the case of $f$, the spectral measure $\mu$ and the Stieltjes function $s$ take values in operators on $\HH_0$. In particular, 
$$
 s(0) = \mu(\T) = 1_0 := \text{identity on } \HH_0.
$$
Both generating functions, $f$ and $s$, are related by the {\bf quantum renewal equation} \cite{GVWW,BGVW,CGMV2},  
\begin{equation} \label{eq:st-sc}
 s(z)^{-1}=1_0-zf(z), 
 \qquad z\in\D,
\end{equation}
a quantum version of the renewal equation for classical random walks \cite[Chapter XIII]{Feller} (see also \cite[Chapter 4]{Stroock}). This allows us to connect $f$ to the Carath\'eodory function $F$ of $\mu$, defined by
$$
 F(z) := \int \frac{t+z}{t-z}\,d\mu(t) = 2s^\dag(z)-1_0,
 \qquad
 s^\dag(z) := s(\overline{z})^\dag,
$$
leading to the identification of 
\begin{equation} \label{eq:S-C}
 f^\dag(z) = z^{-1}(F(z)-1_0)(F(z)+1_0)^{-1}
\end{equation} 
as the {\bf Schur function} of $\mu$ \cite[Chapter 1]{Simon-OPUC} (see \cite[Sect. 3]{DPS} for the matrix-valued case). As a consequence, besides its analyticity, the FR-function $f$ has a remarkable contractivity property in $\D$, 
\begin{equation} \label{eq:Schur}
 \|f(z)\|\le1 \; \text{ for } \; z\in\D.
\end{equation}
This means that every FR-function related to a unitary operator is a Schur function.

The integral representation \eqref{eq:s-mu} of Stieltjes functions combined with \eqref{eq:st-sc} provides an integral representation of Schur functions, thus of FR-functions for unitary operators, 
\begin{equation} \label{eq:f-mu}
 f(z) = z^{-1} \left(1_0-\left(\int \frac{d\mu(t)}{1-zt}\right)^{-1}\right)
 = \left(\int\frac{t\,d\mu(t)}{1-zt}\right) 
 \left(\int\frac{d\mu(t)}{1-zt}\right)^{-1}.
\end{equation}

In the case of a one-dimensional subspace $\HH_0$, the FR-function is an analytic map $f\colon\D\to\overline\D$. According to the maximum modulus principle, there are two possibilities: 
\begin{itemize}
 \item $f$ is {\it degenerate}: 
 $|f(z_0)|=1$ for some $z_0\in\D$, then $f$ is a unimodular constant on $\D$. 
 \item $f$ is {\it non-degenerate}:
 $|f|<1$ on $\D$, i.e. $f(\D)\subset\D$. 
\end{itemize}
In this scalar valued case $f$ is characterized by a finite or infinite number of Schur parameters $\alpha_n$ via the Schur algorithm \cite{Schur} (see also \cite[Chapter 1]{Simon-OPUC}), 
\begin{equation} \label{eq:Salg}
\begin{aligned}
 & f_0(z)=f(z),
 \\
 & f_{n+1}(z) = z^{-1} M_{\alpha_n}(f_n(z)),
 \qquad \alpha_n=f_n(0),
 \qquad n\ge0,
\end{aligned}
\end{equation}
where $M_\alpha$ is the M\"obius transformation 
\begin{equation} \label{eq:Mobius}
 M_\alpha(z) = \frac{z-\alpha}{1-\overline\alpha z}, 
 \qquad |\alpha|<1. 
\end{equation}
This algorithm generates new Schur functions $f_n$, the iterates of $f$, and it terminates if some iterate $f_N$ is degenerate, i.e. if $|\alpha_N|=1$. Thus, the Schur parameters lie on $\D$, except for the last one, if any, which lies on $\T$.

A canonical unitary operator can be associated with any scalar valued Schur function $f$, namely, the operator in $\ell^2$ defined by the CMV representation \cite{CMV,Watkins} (see also \cite[Chapter 4]{Simon-OPUC}) of the multiplication operator $h(z)\mapsto zh(z)$ in $L^2_\mu$, where $\mu$ is the measure on $\T$ associated with $f$. Geronimus theorem \cite{Geronimus} (see also \cite[Chapter 3]{Simon-OPUC}) implies that the Schur parameters of $f$ coincide with the Verblunsky coefficients which parametrize the related CMV matrix. Thus, the Schur algorithm can be viewed as a procedure to extract from a Schur function the parameters giving the corresponding canonical unitary operator.

The Schur algorithm also works for higher dimensional subspaces $\HH_0$, which leads to matrix valued Schur functions. Then, the Schur parameters $\alpha_n$ are square matrices and the scalar M\"obius transformation must be changed by its matrix version \cite{DGK-S} (see also \cite[Sect. 1 and 3]{DPS})
\begin{equation} \label{eq:Mobius-op}  
 M_\alpha(T) 
 = \rho_{\alpha^\dag}^{-1} (T-\alpha)(1-\alpha^\dag T)^{-1} \rho_\alpha
 = \rho_{\alpha^\dag} (1-T\alpha^\dag)^{-1}(T-\alpha) \rho_\alpha^{-1},
 \quad
 \begin{aligned}
 	& \|\alpha\|<1,
	\\[-3pt]
 	& \rho_\alpha = (1-\alpha^\dag\alpha)^{1/2}.
 \end{aligned}
\end{equation}
The equivalence between the above two expressions follows from the identity $\rho_{\alpha^\dag}^2\alpha=\alpha\rho_\alpha^2$, which implies that $\alpha\rho_\alpha^{-2}=\rho_{\alpha^\dag}^{-2}\alpha$.
As in the scalar case, there are two possibilities:
\begin{itemize}
 \item $f$ is {\it degenerate}: 
 $\|f(z_0)\|=1$ for some $z_0\in\D$, then $\|f\|=1$ on $\D$ (but now $f$ is not necessarily constant). 
 \item $f$ is {\it non-degenerate}:
 $\|f\|<1$ on $\D$. 
\end{itemize}
Thus, $\|\alpha_n\|<1$ unless one runs into a degenerate iterate $f_N$, a situation which terminates the algorithm because $\|\alpha_N\|=1$.      

Another feature of FR-functions is a remarkable factorization property which lies behind Khrushchev formula for OPUC: suppose that a unitary $U$ on a Hilbert space with an orthogonal decomposition $\HH=\HH_-\oplus\HH_0\oplus\HH_+$ has a factorization into unitaries which only overlap on $\HH_0$, i.e.
$$
 U = (U_L \oplus 1_+) (1_- \oplus U_R),
 \qquad 
 \begin{aligned}
 	& U_L = \text{unitary on } \HH_L:=\HH_-\oplus\HH_0,
	\\
	& U_R = \text{unitary on } \HH_R:=\HH_0\oplus\HH_+,
	\\
	& 1_\pm = \text{identity on } \HH_\pm.
 \end{aligned}
$$
Then, the FR-function $f$ of $\HH_0$ with respect to $U$ also factorizes as  \cite{CGVWW}
$$
 f = f_L f_R,
$$
where $f_{L,R}$ are the FR-functions of $\HH_0$ with respect to the `left/right' unitaries $U_{L,R}$. 

The above factorization determines the FR-function of $\HH_0$ in the QW driven by $U$ in terms of those in the left/right QW driven by $U_{L,R}$. That is, any overlapping factorization of a QW into smaller ones allows us to recover the return properties of the overlapping subspace for the larger QW from those in the smaller QWs. 

Finally, FR-functions related to unitary operators are not just examples of Schur functions, rather both concepts coincide exactly. In other words, every Schur function with values in operators on a Hilbert space is the FR-function of a closed subspace with respect to a unitary operator. This follows from the integral representation of Schur functions in terms of measures on the unit circle --via Carath\'eodory or Stieltjes functions-- and Naimark's dilation theorem \cite{Naimark}. Let $f$ be a Schur function with values in operators on a Hilbert space $\HH_0$. The relation \eqref{eq:S-C} between Schur and Carath\'eodory functions --or equivalently the relation \eqref{eq:f-mu}-- associates to $f$ a measure $\mu$ on $\T$ with values in operators on $\HH_0$. Naimark's dilation theorem implies that $\mu$ comes from the projection of an spectral measure $E$, defined in a larger Hilbert space $\HH\supset\HH_0$, i.e. $\mu=PEP$ with $P$ the orthogonal projection of $\HH$ onto $\HH_0$. Then, the unitary operator $U=\int\!t\,dE(t)$ associates the spectral measure $\mu$ to the subspace $\HH_0$ and, hence, $f^\dag(z)$ becomes the FR-function of $\HH_0$ with respect to $U$. In other words, $f$ is the FR-function of $\HH_0$ with respect to the unitary $U^\dag$.    

\medskip

We intend to show that an analogue of these operator valued FR-functions also appears in other areas where, to the best of our knowledge, they have not been considered so far. Indeed, we will see that FR-functions and their properties generalize to arbitrary operators on Banach spaces. We will exploit our obsevation of the role and properties of these generalized FR-functions for OPRL, RW, QW and OQW. Their analogy with the FR-functions of the unitary case --i.e., the Schur functions-- rests on the following similarities:  

\begin{itemize}
\item[(A)] {\bf Definition:} formally identical to that of the unitary case and similarly related to Stieltjes functions by a renewal equation. Recognizable as a generating function of first time returns. 
\item[(B)] {\bf Properties:} domain of analyticity similarly related to the spectrum of the underlying operator. Analogous transformation properties in the self-adjoint and unitary cases.
\item[(C)] {\bf ``Schur algorithm'' (self-adjoint and unitary cases):} existence of an algorithm which, starting from a FR-function, yields iteratively a set of  FR-functions of the same kind and characterizes the original one by means of certain parameters which also provide a connection with a related canonical operator. 
\item[(D)] {\bf Splittings:} existence of splitting rules for FR-functions, consequence of similar ``overlapping'' splittings for the underlying operator. Interpretation as rules to split return properties when splitting a system into overlapping smaller subsystems.   
\end{itemize}

\subsection{Block operators and Schur complements}
\label{ssec:SC}
\quad

The generalization of FR-functions mentioned above requires dealing with operators related to resolvents, as in the unitary case. Here we summarize the results that we will use along the paper regarding operator inverses.

As required by the applications to RW and open QW, we will develop most of the results in the abstract context of operators on Banach spaces, a generalization which requires no more effort than the case of Hilbert spaces.  

The results previously summarized for the unitary case show the prominent role that projections will play in this generalization. A projection on a Banach space $\BB$ is a linear operator $P\colon\BB\to\BB$ such that $P^2=1$ is the identity on $\BB$. Any such a projection is specified by a decomposition $\BB=\BB_0\oplus\BB_1$ into the direct sum of two subspaces, $\BB_0=\RR(P)$ and $\BB_1=\ker P$. We say that $P$ is the projection onto $\BB_0=\RR(P)$ along $\BB_1=\ker P$, which is a bounded operator exactly when these subspaces are closed. Bounded projections --specified by pairs of complementary closed subspaces-- are the only ones that we will consider along the paper.  

Let $T \colon \DD(T) \to \BB$ be a linear operator on a Banach space $\BB$ with domain $\mathscr{D}(T)$ and range $\RR(T)$. Since we will get away from unitarity, we will assume that the domain is an arbitrary subspace, not necessarily the whole space or even a closed one, to include the case of unbounded operators. A bounded projection $P$ onto $\BB_0$ along $\BB_1$, together with the complementary projection $Q=1-P$, yield a block representation of $T$ whenever 
\begin{equation} \label{eq:bc}
 \DD(T)=(\BB_0\cap\DD(T))\oplus(\BB_1\cap\DD(T)), 
\end{equation}
i.e. as long as $P\DD(T)=\BB_0\cap\DD(T)$ and $Q\DD(T)=\BB_1\cap\DD(T)$. Then, 
\begin{equation} \label{eq:T-block}
 T = \begin{pmatrix} A & B \\ C & D \end{pmatrix},
 \qquad
 \begin{aligned}
 	& A=PTP \colon \BB_0\cap\DD(T) \to \BB_0, 
 	& \quad & B=PTQ \colon \BB_1\cap\DD(T) \to \BB_0,
 	\\
 	& C=QTP \colon \BB_0\cap\DD(T) \to \BB_1,
 	& & D=QTQ \colon \BB_1\cap\DD(T) \to \BB_1,
\end{aligned}
\end{equation}
makes sense acting on arbitrary vectors $v\in\DD(T)$ in block form
$$
 v = \begin{pmatrix} Pv \\ Qv \end{pmatrix}.  
$$
We will refer to \eqref{eq:T-block} as the block representation of the operator $T$ generated by the bounded projection $P$.

Condition \eqref{eq:bc} implies that 
\begin{equation} \label{eq:bc2}
 P\DD(T)\subset\DD(T).
\end{equation}
Conversely, \eqref{eq:bc2} gives $Q\DD(T)=(1-P)\DD(T)\subset\DD(T)$, thus $\DD(T)=P\DD(T)\oplus Q\DD(T)$ which becomes \eqref{eq:bc}. This means that \eqref{eq:bc} is equivalent to \eqref{eq:bc2}, a condition which therefore characterizes the projections $P$ which generate block representations for $T$. Condition \eqref{eq:bc2} holds for instance if $\BB_0\subset\DD(T)$, a particularly simple situation generating block representations which will be of interest later on.

Since we will use the block representation of $T \colon \DD(T)\to\BB$ to obtain a similar one for $T^{-1} \colon \RR(T)\to\BB$, besides \eqref{eq:bc2} we also need to suppose that $P\RR(T)\subset\RR(T)$. To avoid unnecessary complications we will assume that $\RR(T)=\BB$, a condition that will cover all our needs. 

A key tool for the analysis of inverses is the notion of Schur complement for block operators on Banach spaces. The Schur complement of the block $D$ with respect to the block operator $T$ given in \eqref{eq:T-block} is the operator
\begin{equation} \label{eq:SC1}
 T/D := A-BD^{-1}C = PTP-PTQ(QTQ)^{-1}QTP,
\end{equation}
which defines an operator on $\BB_0$ with domain $\DD(T/D)=\BB_0\cap\DD(T)$ whenever $D$ is invertible with $\RR(C)\subset\RR(D)$. Again, we will make the stronger assumption that $\RR(D)=\BB_1$, which will be enough to deal with all the situations appearing in the paper.

The following two results, usually given for finite dimension, hold for any Banach space under the previous assumptions. We present their proof in such a general case to take care of the subtleties of infinite dimension.

\begin{lem} \label{lem:SC}
Let $T \colon \DD(T)\to\BB$ be a linear operator on a Banach space $\BB$ with block representation \eqref{eq:T-block} generated by a projection $P$ onto $\BB_0$ along $\BB_1$ satisfying \eqref{eq:bc2}. If $T$ and $D$ have inverses everywhere defined on $\BB$ and $\BB_1$ respectively, then $T^{-1} \colon \HH\to\HH$ has a block representation generated by $P$ given as follows 
$$
 T^{-1} = \begin{pmatrix} (T/D)^{-1} & * \\ * & * \end{pmatrix}, 
$$
i.e., 
\begin{equation} \label{eq:SC}
 PT^{-1}P = (A-BD^{-1}C)^{-1} = (PTP-PTQ(QTQ)^{-1}QTP)^{-1}.
\end{equation}
\end{lem}

\begin{proof}
Rewriting the Schur complement as 
$$ 
\begin{aligned}
 T/D & = PT-PTQ-PTQ(QTQ)^{-1}QTP = PT-PTQ(QTQ)^{-1}(QTQ+QTP)
 \\ 
 & = PT-PTQ(QTQ)^{-1}QT,
 \\
 T/D & = TP-QTP-PTQ(QTQ)^{-1}QTP = TP-(QTQ+PTQ)(QTQ)^{-1}QTP
 \\ 
 & = TP-TQ(QTQ)^{-1}QTP,
\end{aligned}
$$
gives 
$$
\begin{aligned}
 & (T/D)PT^{-1}P = (T/D)T^{-1}P = PTT^{-1}P 
 = \text{ identity on } \BB_0,
 \\
 & PT^{-1}P (T/D) = PT^{-1}(T/D) = PT^{-1}TP 
 = \text{ identity on } \BB_0\cap\DD(T),
\end{aligned}
$$
since $TT^{-1}$ is the identity on $\BB$ and $T^{-1}T$ is the identity on $\DD(T)$. This proves that $PT^{-1}P \colon \BB_0\to\BB_0$ and $T/D \colon \BB_0\cap\DD(T)\to\BB_0$ are inverses of each other.
\end{proof} 

The particular case of \eqref{eq:T-block} in which $PTQ=0$ corresponds to a triangular block operator
\begin{equation} \label{eq:triang}
 T = \begin{pmatrix} A & 0 \\ C & D \end{pmatrix},
\end{equation}
which will be of especial interest for us. 

\begin{lem} \label{lem:triang}
Let $T \colon \DD(T)\to\BB$ be a linear operator on a Banach space $\BB$ with block representation \eqref{eq:triang} generated by a projection $P$ onto $\BB_0$ along $\BB_1$ satisfying \eqref{eq:bc2}. Then,
\begin{center}
 \parbox{131pt}{\begin{center} 
 				$T$ has an inverse \break
 				everywhere defined on $\BB$ 
				\end{center}}
 $\quad\Longrightarrow\quad$  
 \parbox{239pt}{\begin{center} 
 				$A$ has an inverse everywhere defined on $\BB_0$,\break 
				$D$ has an inverse everywhere defined on $\BB_1$,						\end{center}}
\end{center}
whenever any of the following conditions is satisfied:
\begin{center}
{\it (i)} $C=0$, \qquad
{\it (ii)} $A$ is invertible, \qquad
{\it (iii)} $\RR(D)=\BB_1$, \qquad
{\it (iv)} $\dim\HH_0<\infty$. 
\end{center}
Conversely, if $A$ and $D$ have inverses everywhere defined on $\BB_0$ and $\BB_1$ respectively, then $T$ has an inverse everywhere defined on $\BB$ with block representation generated by $P$ given by
\begin{equation} \label{eq:triang-inv}
 T^{-1} = \begin{pmatrix} A^{-1} & 0 \\ -D^{-1}CA^{-1} & D^{-1} \end{pmatrix}.
\end{equation}
\end{lem}

\begin{proof}
The result under condition {\it (i)} is obvious since $T=A\oplus D$ in such a case. 
 
Assume only the existence of $T^{-1}$ everywhere defined on $\BB$. 

Using $PTQ=0$, we find that $(PTP)(PT^{-1}P) = PTT^{-1}P$ is the identity on $\BB_0$, hence $A=PTP$ has range $\BB_0$ but it is not necessarily invertible. Also, $(QT^{-1}Q)(QTQ) = QT^{-1}TQ$ is the identity on $\BB_1\cap\DD(T)$, so $D=QTQ$ has an inverse but not necessarily everywhere defined on $\BB_1$. Hence, {\it (ii)} implies that $A$ has an inverse everywhere defined on $\BB_0$, while {\it (iii)} implies that $D$  has an inverse everywhere defined on $\BB_1$.

Furthermore, $PTQ=0$ also gives
$$
 (PTP)(PT^{-1}Q) = PTT^{-1}Q = 0, \qquad (PT^{-1}Q) (QTQ) = PT^{-1}TQ = 0,
$$
thus any of the conditions {\it (ii)} or {\it (iii)} implies that $PT^{-1}Q=0$. If this is the case, then $(PT^{-1}P)(PTP) = PT^{-1}TP$ is the identity on $\BB_0\cap\DD(T)$, while $(QTQ)(QT^{-1}Q) = QTT^{-1}Q$ is the identity on $\BB_1$, thus $A=PTP$ is invertible and $D=QTQ$ has range $\BB_1$. 

Combining the previous results we conclude that any of the conditions {\it (ii)} or {\it (iii)} imply that $A$ and $D$ have inverses everywhere defined on $\BB_0$ and $\BB_1$ respectively. This also holds for condition {\it (iv)} because $\RR(A)=\BB_0$ is equivalent to $A$ being invertible if $\dim\BB_0<\infty$.

Suppose now that that $A$ and $D$ have inverses everywhere defined on $\BB_0$ and $\BB_1$. Then, the block operator given by \eqref{eq:triang-inv} has domain $\BB$ and it is straightforward to check that it is the inverse of $T$.     
\end{proof}

The right implication of the previous lemma and the lower triangular block representation of $T^{-1}$ may not hold unless an additional condition is added as to the existence of such an inverse with $\DD(T^{-1})=\BB$. As a counter example, take $T$ as the forward shift in $\ell^2(\Z)$ and $\BB_0=\{(x_n)\in\ell^2(\Z):x_n=0,n\ge0\}$. Then, $\BB_0$ generates a lower triangular block representation for $T$, but an upper triangular one for $T^{-1}$, the backward shift,
$$
 T = 
 \text{\tiny $\left( 
 \begin{array}{cccc|cccc}
 	\cdots & \cdots & \cdots & \cdots & \cdots & \cdots & \cdots & \cdots  
	\\
    \cdots & 0 & 0 & 0 & 0 & 0 & 0 & \cdots
	\\
 	\cdots & 1 & 0 & 0 & 0 & 0 & 0 & \cdots
	\\  
 	\cdots & 0 & 1 & 0 & 0 & 0 & 0 & \cdots
	\\ \hline
	\cdots & 0 & 0 & 1 & 0 & 0 & 0 & \cdots
	\\ 
	\cdots & 0 & 0 & 0 & 1 & 0 & 0 & \cdots
	\\ 
	\cdots & 0 & 0 & 0 & 0 & 1 & 0 & \cdots
	\\
	\cdots & \cdots & \cdots & \cdots & \cdots & \cdots & \cdots & \cdots
 \end{array}
 \right)$},
 \qquad\quad
 T^{-1} = 
 \text{\tiny $\left( 
 \begin{array}{cccc|cccc}
 	\cdots & \cdots & \cdots & \cdots & \cdots & \cdots & \cdots & \cdots  
	\\
    \cdots & 0 & 1 & 0 & 0 & 0 & 0 & \cdots
	\\
 	\cdots & 0 & 0 & 1 & 0 & 0 & 0 & \cdots
	\\  
 	\cdots & 0 & 0 & 0 & 1 & 0 & 0 & \cdots
	\\ \hline
	\cdots & 0 & 0 & 0 & 0 & 1 & 0 & \cdots
	\\ 
	\cdots & 0 & 0 & 0 & 0 & 0 & 1 & \cdots
	\\ 
	\cdots & 0 & 0 & 0 & 0 & 0 & 0 & \cdots
	\\
	\cdots & \cdots & \cdots & \cdots & \cdots & \cdots & \cdots & \cdots
 \end{array}
 \right)$}.
$$
Besides, $\RR(A)=\BB_0$ but $A$ is not invertible, while $D$ has an inverse but $\DD(D^{-1})=\RR(D)=\{(x_n)\in\ell^2(\Z):x_n=0, n\le0\}$ is not the whole subspace $\BB_1=\{(x_n)\in\ell^2(\Z):x_n=0, n\le-1\}$.

\section{FR-functions for arbitrary operators}
\label{sec:GF}

In this section we will extend the notion of FR-functions to arbitrary operators on Banach spaces and study some of their general properties. Concerning the technical details of operator theory arising in the discussion below, we refer to \cite{Kato}.

Although we will work in general with operators which are not necessarily bounded, some of the operators involved in the discussions will be bounded and everywhere defined. Thus, for convenience, given a Banach space $\BB$ we introduce the notation 
$$
 \frak{B}(\BB) := \{T\colon\BB\to\BB \text{ bounded linear}\}
 = \{T\colon\BB\to\BB \text{ closed linear}\},
$$
where $T\colon\BB\to\BB$ indicates that $T$ has domain $\DD(T)=\BB$ and the second identity is due to the closed graph theorem.
   
Let $T \colon \DD(T) \to \BB$ be a linear operator on a Banach space $\BB$ with domain $\DD(T)$ and range $\RR(T)$. We will denote
$$
\begin{aligned}
 & \varpi(T)  
 := \{z\in\C : \exists (z-T)^{-1}\colon\BB\to\BB\} 
 = \{z\in\C : \ker(z-T)=\{0\}, \; \RR(z-T)=\BB\},
 \\
 & \varrho(T)  
 := \{z\in\C : \exists (z-T)^{-1}\in\frak{B}(\BB)\}.
\end{aligned}
$$
The set $\varrho(T)$, obviously included in $\varpi(T)$, is known as the resolvent set of $T$ and it is non-empty only for a closed operator $T$. When $T$ is closed, the resolvent operator $(z-T)^{-1}$ is not only bounded and everywhere defined on $\BB$ for $z\in\varrho(T)$, but it is also analytic in $z$ in such a domain. Besides, the closed graph theorem implies that
\begin{equation} \label{eq:closed}
 T \text{ closed} \quad\Rightarrow\quad \varrho(T)=\varpi(T).
\end{equation}

Closely related to the resolvent operator are the notions of Stieltjes function and FR-function. The expressions \eqref{eq:f-U} and \eqref{eq:s-U} of FR-functions and Stieltjes functions for the unitary case suggest a generalization to arbitrary operators which also involves projections. The validity of the results of this section for arbitrary Banach spaces --even in the absence of an inner product-- reflects the fact that they are indeed true for arbitrary bounded --not necessarily orthogonal-- projections, i.e. for arbitrary pairs of complementary closed subspaces. 

Following these ideas, we will introduce the Stieltjes function and FR-function of a {\bf bounded projection}, which we will define for an arbitrary {\bf closed operator}. Since we assume no symmetry for the operator, no spectral measure is available in this case, so we should define such functions resorting to operator identities similar to \eqref{eq:f-U} and \eqref{eq:s-U}.

\begin{defn} 
Let $T\colon\DD(T)\to\HH$ be a closed operator on a Banach space $\BB$ and $P$ a bounded projection of $\BB$ onto a closed subspace $\BB_0$. The {\bf Stieltjes function} of the projection $P$ with respect to the operator $T$ is the function $s$ with values in operators on $\BB_0$ given by
$$
 s(z) := P(1-zT)^{-1}P,
 \qquad z^{-1}\in\varrho(T).
$$
When $\BB$ is a Hilbert space and $P$ is the orthogonal projection onto $\BB_0$ we will also refer to $s$ as the Stieltjes function of the subspace $\BB_0$ with respect to the operator $T$.
\end{defn}        

The existence of $s(z)\colon\BB_0\to\BB_0$ for every bounded projection only requires the existence of $(1-zT)^{-1}\colon\BB\to\BB$, i.e. $z^{-1}\in\varpi(T)$. However, \eqref{eq:closed} implies that $\varpi(T)=\varrho(T)$, thus the domain of definition $z^{-1}\in\varrho(T)$ is the largest one which guarantees the existence of the Stieltjes function for every bounded projection $P$ as an operator everywhere defined on $\BB_0=\RR(P)$. Actually, these arguments and the properties of the resolvent operator show that $s(z)$ is an analytic function with values in $\frak{B}(\BB_0)$ for $z^{-1}\in\varrho(T)$. 

The origin may also lie in the domain of analyticity of $s$ --for instance, if $T\in\frak{B}(\BB)$-- but is in general a special point which requires particular care. For example, a naive evaluation suggests that $s(0)$ is always the identity on $\BB_0$. However, by definition, $1-zT$ has domain $\DD(T)$ for any value of $z$, hence we should identify this operator with the identity on $\DD(T)$ for $z=0$. Therefore, the origin lies outside of the set of values of $z$ for which $(1-zT)^{-1}$ is everywhere defined on $\BB$ unless $\DD(T)=\BB$, i.e. unless $T\in\frak{B}(\BB)$. Hence, $s(0)$ is the identity on $\BB_0\cap\DD(T)$, which belongs to $\frak{B}(\BB_0)$ only when $\BB_0\subset\DD(T)$.   

The Stieltjes function given by the above definition becomes the usual Stieltjes function for the spectral measure $\mu$ of the subspace $\BB_0$ when $T$ has a spectral decomposition, i.e. if $\BB$ is a Hilbert space and $T$ is a normal operator on $\BB$,
\begin{equation} \label{eq:normal}
 T = \int\!t\,dE(t) 
 \quad\Rightarrow\quad \mu=PEP 
 \quad\Rightarrow\quad s(z) = \int\frac{d\mu(t)}{1-zt}.
\end{equation}
Note that $\mu$ is normalized so that $\mu(\R)=1_0$ is the identity on $\BB_0$ because $E(\R)=1$ is the identity on $\BB$. 

The key point of this section is the generalization to arbitrary closed operators of the Schur functions appearing in the unitary case. As in the case of Stieltjes functions, such a generalization calls for a purely operator definition, suggested in this case by the identification between Schur functions and FR-functions related to unitary operators. This leads to the introduction of the FR-function of a bounded projection with respect to a closed operator, which constitutes the generalization of Schur functions beyond the unitary case.

\begin{defn} \label{def:FR}
Let $T\colon\DD(T)\to\BB$ be a closed operator on a Banach space $\BB$ and $P$ a bounded projection of $\BB$ onto a closed subspace $\BB_0$. The {\bf first return function (FR-function)} of the projection $P$ with respect to the operator $T$ is the function $f$ with values in operators on $\BB_0$ given by   
\begin{equation} \label{eq:f}
 f(z) := PT(1-zQT)^{-1}P, 
 \qquad Q=1-P,
 \qquad z^{-1}\in\varpi(QT).
\end{equation}
When $\BB$ is a Hilbert space and $P$ is the orthogonal projection onto $\BB_0$ we will also refer to $f$ as the FR-function of the subspace $\BB_0$ with respect to the operator $T$.
\end{defn}

The condition $z^{-1}\in\varpi(QT)$ means that $(1-zQT)^{-1}$ exists as an operator with domain $\BB$ and range $\DD(T)$, so that $T(1-zQT)^{-1}$ is also everywhere defined on $\BB$. This ensures that, for every projection $P$, its FR-function gives at $z$ an operator $f(z)$ everywhere defined on $\BB_0=\RR(P)$.  

The origin requires once again special care. The operator $QT$ has domain $\DD(T)$, so for $z=0$ we should consider $1-zQT$ as the identity on $\DD(T)$. Thus, $f(0)=PTP$, which fails to be everywhere defined on $\BB_0$ unless $\BB_0\subset\DD(T)$. An FR-function $f$ may also satisfy $f(0)\in\frak{B}(\BB_0)$, as it is the case of $T\in\frak{B}(\HH)$ where $f$ becomes analytic at the origin. 

In contrast to the Stieltjes case, the condition $z^{-1}\in\varpi(QT)$ is not necessarily equivalent to $z^{-1}\in\varrho(QT)$ because the class of closed operators on $\HH$ is preserved by the product by an operator of $\frak{B}(\BB)$ on the right but not on the left, i.e.
$$
 ST \text{ closed} \quad\nLeftarrow\quad
 T \text{ closed}, \quad S\in\frak{B}(\BB) 
 \quad\Rightarrow\quad TS \text{ closed}.
$$ 
This means that $QT$ need not be closed even if $T$ is closed (this is for instance the case if a sequence $v_n\in\DD(T)$ converging to $v\in\HH\setminus\DD(T)$ satisfies $Tv_n\in\BB_0$ because then $(v_n,QTv_n)=(v_n,0)$ converges to $(v,0)$ which does not belong to the graph of $QT$). Hence, $z^{-1}\in\varpi(QT)$ implies the existence of $f(z)\colon\BB_0\to\BB_0$ for every projection $P$, but a priori does not guarantee that $f(z)\in\frak{B}(\BB_0)$. 

The following proposition shows that the stronger condition $z^{-1}\in\varrho(QT)$ selects an open subset of $z^{-1}\in\varpi(QT)$ in which $f(z)\in\frak{B}(\BB_0)$. 

\begin{prop} \label{prop:rhoQT}
Let $T\colon\DD(T)\to\BB$ be a closed operator on a Banach space $\BB$, $P$ a bounded projection of $\BB$ onto a closed subspace $\BB_0$ and $Q=1-P$. Then, the FR-function $f(z)$ of $P$ with respect to $T$ is an analytic function with values in $\frak{B}(\BB_0)$ for $z^{-1}\in\varrho(QT)$. 
\end{prop}

\begin{proof}
The condition $z^{-1}\in\varrho(QT)$ means that $(1-zQT)^{-1}\in\frak{B}(\BB)$ so, as a product of a closed operator by an operator of $\frak{B}(\BB)$ on the right, $T(1-zQT)^{-1}\colon\BB\to\BB$ is also closed, hence it is bounded by the closed graph theorem. We conclude that $T(1-zQT)^{-1}\in\frak{B}(\BB)$ and $f(z)\in\frak{B}(\BB_0)$ for $z^{-1}\in\varrho(QT)$, where both operator valued functions are also analytic in $z$ due to the properties of the resolvent operator. 
\end{proof}

The domain of definition of $f$, determined by the properties of $QT\colon\DD(T)\to\BB$, can be rewritten sometimes in terms of the properties of the block $QTQ \colon \BB_1\cap\DD(T)\to\BB_1$. The proposition below provides one of such situations which will be of interest later on. 

\begin{prop} \label{prop:QTQ}
Let $T\colon\DD(T)\to\BB$ be a closed operator on a Banach space $\BB$, $P$ a bounded projection of $\BB$ onto a closed subspace $\BB_0$ and $Q=1-P$. Then,
$$
 \BB_0\subset\DD(T) 
 \quad\Rightarrow\quad 
 \varpi(QT)\setminus\{0\} = \varpi(QTQ)\setminus\{0\},
 \quad
 \varrho(QT)\setminus\{0\} = \varrho(QTQ)\setminus\{0\}.
$$
\end{prop}

\begin{proof}
The proposition can be restated by saying that the condition $\BB_0\subset\DD(T)$ ensures that $1-zQT$ and $Q-zQTQ$ have simultaneously (bounded) inverses everywhere defined on $\BB$ and $\BB_1$ respectively. According to the comments of Section~\ref{ssec:SC}, a block representation of $T$ is available whenever $\BB_0\subset\DD(T)$, which also leads to a block representation of $1-zQT$ generated by $P$, 
\begin{equation} \label{eq:(1-zQT)-block}
 1-zQT = \begin{pmatrix} P & 0 \\ -zQTP & Q-zQTQ \end{pmatrix}.
\end{equation}
From Lemma~\ref{lem:triang}, we know that $1-zQT$ has an inverse everywhere defined iff the block $Q-zQTQ$ does likewise, and the inverse of $1-zQT$ has the block representation
$$
 (1-zQT)^{-1} = 
 \begin{pmatrix}
 	P & 0
 	\\
 	z(Q-zQTQ)^{-1}QTP & (Q-zQTQ)^{-1}
 \end{pmatrix}.
$$
Therefore, $(1-zQTQ)^{-1}$ bounded implies that $(Q-zQTQ)^{-1}=Q(1-zQT)^{-1}Q$ is also bounded. On the other hand, the operator $TP$ is closed --as a product of a closed operator by an operator of $\frak{B}(\BB)$ on the right-- and has domain $\BB$ --because $\BB_0\subset\DD(T)$-- so it is bounded by the closed graph theorem. Hence, $(Q-zQTQ)^{-1}$ bounded implies that $(Q-zQTQ)^{-1}QTP$, and thus $(1-zQT)^{-1}$, are also bounded.
\end{proof}

As a consequence of this result, the conditions $z^{-1}\in\varpi(QT)$ in \eqref{eq:f} and $z^{-1}\in\varrho(QT)$ in Proposition~\ref{prop:rhoQT} can be substituted respectively by $z^{-1}\in\varpi(QTQ)$ and $z^{-1}\in\varrho(QTQ)$ if $\BB_0\subset\DD(T)$. This has some advantages when dealing with a self-adjoint operator $T$ since such a symmetry is usually inherited by $QTQ$, as it is the case for instance when $T\in\frak{B}(\BB)$. 

As in the unitary case, Stieltjes functions and FR-functions are intimately related also for general operators. This relation is established in the following theorem, which constitutes the main result of this section. The theorem also uses the close link between Stieltjes functions and FR-functions to provide information about the domain of the latter ones. More precisely, the theorem below shows on the one hand that the renewal equation \eqref{eq:st-sc} holds beyond the unitary case, and on the other hand that the common part of the domain of definition for Stieltjes and FR-functions is characterized in terms of Stieltjes functions theirselves.

\begin{thm}[\bf renewal equation for FR-functions] \label{thm:dom2}
Let $T\colon\DD(T)\to\BB$ be a closed operator on a Banach space $\BB$, $P$ a bounded projection of $\BB$ onto a closed subspace $\BB_0$ and $Q=1-P$. Denote by $s$ and $f$, respectively, the Stieltjes function and FR-function of $P$ with respect to $T$. Then, if $z^{-1}\in\varrho(T)$,
\begin{align}
 z^{-1}\in\varpi(QT) \;\Leftrightarrow\; \exists s(z)^{-1}\colon\HH_0\to\HH_0,
 \label{eq:piQT}
 \\
 z^{-1}\in\varrho(QT) \;\Leftrightarrow\; \exists s(z)^{-1}\in\frak{B}(\HH_0).
 \label{eq:rhoQT}
\end{align}
Besides, $f$ and $s$ are related by
\begin{equation} \label{eq:sf}
 s(z)^{-1} = 1_0-zf(z), 
 \qquad 1_0 = \text{identity on } \BB_0,
 \qquad z^{-1}\in\varrho(T)\cap\varpi(QT),
\end{equation}
which we will call the {\bf generalized renewal equation} for arbitrary closed operators. 
\end{thm}

\begin{proof}
In what follows, $z\in\C$ is chosen so that $z^{-1}\in\varrho(T)$. To prove \eqref{eq:piQT} we must show that the existence of $(1-zQT)^{-1}$ everywhere defined on $\BB$ is equivalent to the existence of $s(z)^{-1}$ everywhere defined on $\BB_0$. Then, \eqref{eq:rhoQT} means that $(1-zQT)^{-1}$ and $s(z)^{-1}$ are simultaneously bounded.

\noindent \fbox{$\Leftarrow$} \,
Assume the existence of $s(z)^{-1}\colon\BB_0\to\BB_0$. To prove that $z^{-1}\in\varpi(QT)$ we need to show that the equation
\begin{equation} \label{eq:uv1}
 (1-zQT)u = v, \qquad v\in\BB,
\end{equation}
has a unique solution $u\in\DD(T)$ for each $v\in\BB$. Bearing in mind that $z^{-1}\in\varrho(T)$, the above equation may be rewritten as
\begin{equation} \label{eq:uv2}
 (1-zT)u + zPTu = v 
 \;\Leftrightarrow\;
 u + z(1-zT)^{-1}PTu = (1-zT)^{-1}v.
\end{equation}
Applying the operator $PT$ to both sides of the last equation we get
$$
 (P+zPT(1-zT)^{-1}P)PTu = PT(1-zT)^{-1}v,
$$
which, using that 
\begin{equation} \label{eq:1-zT}
 1+zT(1-zT)^{-1}=(1-zT)^{-1}, 
\end{equation}
reads as
$$
 s(z)PTu = PT(1-zT)^{-1}v.
$$
Since $s(z)$ has an inverse everywhere defined on $\BB_0$, the above equation implies that
\begin{equation} \label{eq:PTu}
 PTu = s(z)^{-1}PT(1-zT)^{-1}v.
\end{equation}
Inserting this into \eqref{eq:uv2} yields
$$
 u = (1-zT)^{-1} (1-zs(z)^{-1}PT(1-zT)^{-1}) v,
$$
a vector in $\DD(T)$ which is the only possible solution of \eqref{eq:uv1}. Actually, using \eqref{eq:1-zT} again, it is straightforward to check that this vector satisfies \eqref{eq:PTu} and \eqref{eq:uv1}, so it is really the solution that we were looking for. This proves the left implication of \eqref{eq:piQT}.

The above arguments give an expression for $(1-zQT)^{-1}$ whenever $s(z)^{-1}$ is everywhere defined on $\BB_0$, namely   
$$
\begin{aligned}
 (1-zQT)^{-1} & = (1-zT)^{-1} (1-zs(z)^{-1}PT(1-zT)^{-1})
 \\
 & = (1-zT)^{-1}(1+s(z)^{-1}P-s(z)^{-1}P(1-zT)^{-1}),
\end{aligned}
$$
where in the second equality we have used once again \eqref{eq:1-zT}. Bearing in mind that $z^{-1}\in\varrho(T)$, this expression shows that $(1-zQT)^{-1}$ is bounded whenever $s(z)^{-1}$ is bounded. This gives the left implication of \eqref{eq:rhoQT}.

\noindent \fbox{$\Rightarrow$} \,
The right implication of \eqref{eq:rhoQT} follows from Proposition~\ref{prop:rhoQT}, assuming the right implication of \eqref{eq:piQT} and the renewal equation \eqref{eq:sf}, so we only need to prove the two latter ones. 

Suppose now that $z^{-1}\in\varpi(QT)$, so that $f(z)$ and $s(z)$ exist because $z^{-1}\in\varrho(T)$ by hypothesis, and let us prove that $s(z)$ has an inverse everywhere defined on $\BB_0$. Indeed, we will show that 
\begin{equation} \label{eq:s-f}
 s(z)^{-1}=1_0-zf(z), 
\end{equation}
which will finally finish the proof of the theorem. We start with the identities
\begin{equation} \label{eq:res-eq}
\begin{aligned}
 (1-zT)^{-1}-(1-zQT)^{-1} 
 & = z(1-zT)^{-1}PT(1-zQT)^{-1} 
 = z(1-zQT)^{-1}PT(1-zT)^{-1} 
 \\
 & = (1-zQT)^{-1}P(1-zT)^{-1} - (1-zQT)^{-1}P,
\end{aligned}
\end{equation}
where in the last step we have used \eqref{eq:1-zT}. Multiplying \eqref{eq:res-eq} by the projection $P$ on the right, we obtain from the equality between the first and the last term that
\begin{equation} \label{eq:(1-zT)^{-1}P}
 (1-zT)^{-1}P = (1-zQT)^{-1}s(z).
\end{equation}
On the other hand, multiplying \eqref{eq:res-eq} by $P$ both on the left and the right, the first two equalities yield
\begin{equation} \label{eq:s-P}
 s(z)-P = zs(z)f(z) = zP(1-zQT)^{-1}PT(1-zT)^{-1}P.
\end{equation}
Here we have used that $P(1-zQT)^{-1}=P$, which follows from the identity obtained by substituting $T$ by $QT$ in \eqref{eq:1-zT}. Finally, inserting \eqref{eq:(1-zT)^{-1}P} into the last term of \eqref{eq:s-P} gives
$$
 s(z)-P = zs(z)f(z) = zf(z)s(z),
$$
which is equivalent to \eqref{eq:s-f} since $s(z)$ and $f(z)$ must be understood as operators on $\BB_0$.
\end{proof}

The generalized renewal equation can be also expressed as
\begin{equation} \label{eq:f-s}
 f(z) = z^{-1}(1_0-s(z)^{-1}), \qquad z^{-1}\in\varrho(T)\cap\varpi(QT),
\end{equation}
which gives the FR-function in terms of the Stieltjes function. Inserting the definition of the Stieltjes function into \eqref{eq:f-s} yields an alternative explicit expression of the FR-function in terms of the operator $T$,
$$
 f(z) = z^{-1} (s(z)-1_0) s(z)^{-1} = (PT(1-zT)^{-1}P)(P(1-zT)^{-1}P)^{-1},
 \kern15pt z^{-1}\in\varrho(T)\cap\varpi(QT).
$$ 

As a direct consequence of Proposition~\ref{prop:rhoQT} and Theorem~\ref{thm:dom2}, we get the following result concerning the domain of analyticity of FR-functions. 

\begin{cor} \label{cor:dom3}
Let $T\colon\DD(T)\to\BB$ be a closed operator on a Banach space $\BB$ and $P$ a bounded projection of $\BB$ onto a closed subspace $\BB_0$. Denote by $s$ and $f$, respectively, the Stieltjes function and FR-function of $P$ with respect to $T$. Then, $f(z)$ is an analytic function with values in $\frak{B}(\BB_0)$ for $z$ in  
$$
 \{z\in\C : z^{-1}\in\varrho(T), \; \exists s(z)^{-1}\in\frak{B}(\BB_0)\}.
$$ 
\end{cor}

If the range of a projection is included in the domain of an operator, the corresponding FR-function admits alternative expressions, each one of them with its own interest. 

\begin{prop} \label{prop:falt}
Let $T\colon\DD(T)\to\BB$ be a closed operator on a Banach space $\BB$, $P$ a bounded projection of $\BB$ onto a closed subspace $\BB_0$ and $Q=1-P$. If $\BB_0\subset\DD(T)$, the FR-function $f$ of $P$ with respect to $T$ can be expressed as
\begin{itemize}
\item[{\it (i)}] 
$f(z) = P(1-zTQ)^{-1}TP$, \; $z^{-1}\in\varpi(QTQ)$,
\medskip
\item[{\it (ii)}] 
$f(z) = PTP + zPTQ(Q-zQTQ)^{-1}QTP$, \; $z^{-1}\in\varpi(QTQ)$,
\medskip
\item[{\it (iii)}] 
$f(z) = PTP + zPTQ(1-zQTQ)^{-1}QTP$, \; $z^{-1}\in\varpi(QTQ)$.
\end{itemize}
\end{prop}

\begin{proof}
\quad

\noindent{\it (i)}
As in the proof of Proposition~\ref{prop:QTQ}, we can see that $z^{-1}\in\varpi(TQ)$ is equivalent to $z^{-1}\in\varpi(QTQ)$ whenever $\BB_0\subset\DD(T)$. If $v\in\DD(T)$ and $z^{-1}\in\varpi(QTQ)$, we can define $w=(1-zQT)^{-1}v\in\DD(T)$, which satisfies $QTw=z^{-1}(w-v)\in\DD(T)$ so that $TQTw$ makes sense. Therefore,
$$
 (1-zTQ)(T(1-zQT)^{-1}-(1-zTQ)^{-1}T)v = (1-zTQ)Tw-T(1-zQT)w = 0.
$$
Since $\ker(1-zQT)=\{0\}$ for $z^{-1}\in\varpi(QT)$, i.e. for $z^{-1}\in\varpi(QTQ)$, we conclude that
$$
 T(1-zQT)^{-1}v=(1-zTQ)^{-1}Tv, \qquad \forall v\in\DD(T).
$$
Hence $f(z)=PT(1-zQT)^{-1}P=P(1-zTQ)^{-1}TP$ because $\BB_0\subset\DD(T)$.

\noindent{\it (ii)}
The condition $\BB_0\subset\DD(T)$ guarantees that $\BB_0$ generates the block representations \eqref{eq:T-block} and \eqref{eq:(1-zQT)-block} for $T$ and $1-zQT$ respectively. Applying Lemma~\ref{lem:triang} to \eqref{eq:(1-zQT)-block} and combining the result with \eqref{eq:T-block} gives
$$
 T(1-zQT)^{-1} = 
 \begin{pmatrix}
 	PTP-zPTQ(Q-zQTQ)^{-1}QTP & * \\ * & *
 \end{pmatrix},
 \qquad z^{-1}\in\varpi(QTQ),
$$
which proves the result.

\noindent{\it (iii)}
It follows from the previous result and the block diagonal representation
$$
 1-zQTQ =
 \begin{pmatrix}
 	P & 0 \\ 0 & Q-zQTQ
 \end{pmatrix},
$$
which yields the identity $Q(1-zQTQ)^{-1}Q=Q(Q-zQTQ)^{-1}Q$.   
\end{proof}
 
Proposition~\ref{prop:falt}.{\it(ii)} relates the FR-function $f$ to the so called transfer/characteristic function \cite{NF,Brodskii,Brodskii2,BGK,Kaashoek} associated with the colligation operator \eqref{eq:T-block} when $T\in\frak{B}(\BB)$ --so that $A,B,C,D$ are bounded and everywhere defined--, as is the case for instance when $T$ is unitary. However, since unbounded operators cannot be everywhere defined, the `transfer/characteristic function representation' of general FR-functions --with respect to not necessarily bounded operators-- is more limited than the original notion (Definition~\ref{def:FR}) or the alternative expression provided by the renewal equation (Theorem~\ref{thm:dom2}), for which $\DD(T)$ imposes no constraint on $\BB_0$ (they hold even if $\BB_0\cap\DD(T)=\{0\}$!). In other words, FR-functions constitute a generalization of transfer/characteristic functions to the unbounded case. This will be crucial, for instance, when identifying the whole class of functions which are expressible as FR-functions associated with --bounded or unbounded-- self-adjoint operators on Hilbert spaces. This is at the root of the wide validity of the integral and operator representations of Nevanlinna functions given in Theorem~\ref{thm:N=FR} and Corollary~\ref{cor:N=FR}. 

Sometimes the following modified version of Stieltjes functions will be useful.

\begin{defn}
 
Let $T\colon\DD(T)\to\BB$ be a closed operator on a Banach space $\BB$ and $P$ a bounded projection of $\BB$ onto a closed subspace $\BB_0$. The {\bf modified Stieltjes function (m-function)} of the projection $P$ with respect to the operator $T$ is the function $m$ with values in operators on $\BB_0$ given by
\begin{equation} \label{eq:ms}
 m(z) := P(T-z)^{-1}P,
 \qquad z\in\varrho(T).
\end{equation}
When $\BB$ is a Hilbert space and $P$ is the orthogonal projection onto $\BB_0$ we will also refer to $m$ as the m-function of the subspace $\BB_0$ with respect to the operator $T$.
\end{defn}
The function $m$ is analytic on $\varrho(T)$ with values in $\frak{B}(\BB_0)$, connected to the corresponding Stieltjes function $s$ by
$$
  m(z) = -z^{-1}s(z^{-1}), \qquad z\in\varrho(T).
$$
This relation allows us to rewrite the previous results using m-functions instead of Stieltjes ones. For instance, the renewal equation \eqref{eq:f-s} reads in terms of m-functions as
\begin{equation} \label{eq:f-ms}
 f(z) = z^{-1}1_0 +  m(z^{-1})^{-1}, 
 \qquad z^{-1}\in\varrho(T)\cap\varpi(QT).
\end{equation} 

\subsection{The bounded case}
\label{ssec:GB}
\quad

It is worth to specialize the above discussion to the case of bounded everywhere defined operators, where many of the subtleties in the previous relations and results disappear. For instance, if $T\in\frak{B}(\BB)$ every subspace lies on $\DD(T)=\BB$. Also $QT\in\frak{B}(\BB)$ and $QTQ\in\frak{B}(\BB_0)$, thus both operators are closed. Hence, $\varrho(QT)=\varpi(QT)$ and $\varrho(QTQ)=\varpi(QTQ)$ by the closed graph theorem, while $\varrho(QT)\setminus\{0\}=\varrho(QTQ)\setminus\{0\}$ by Proposition~\ref{prop:QTQ}. Since the spectrum of $T\in\frak{B}(\BB)$ is bounded, it is natural to consider the extended resolvent set $\overline\varrho(T):=\varrho(T)\cup\{\infty\}$, so that the condition $z^{-1}\in\overline\varrho(T)$ defines an open set which contains the origin. 

Bearing in mind the results of the previous sections in the light of the above comments, we find that FR-functions have the following list of equivalent expressions which are valid for any subspace. 

\begin{prop} \label{prop:B-falt}
Let $T\in\frak{B}(\BB)$, $P$ a bounded projection of $\BB$ onto a closed subspace $\BB_0$ and $Q=1-P$. Denote by $s$ and $f$, respectively, the Stieltjes function and FR-function of $P$ with respect to $T$. Then, $s(z)$ and $f(z)$ are analytic functions with values in $\frak{B}(\BB_0)$ for $z^{-1}\in\overline\varrho(T)$ and $z^{-1}\in\overline\varrho(QTQ)$ respectively. Besides one has,
\begin{itemize}
\item[{\it (i)}] 
$f(z)=PT(1-zQT)^{-1}P$, \; $z^{-1}\in\overline\varrho(QTQ)$,
\medskip
\item[{\it (ii)}] 
$f(z)=P(1-zTQ)^{-1}TP$, \; $z^{-1}\in\overline\varrho(QTQ)$,
\medskip
\item[{\it (iii)}] 
$f(z)=PTP+zPTQ(Q-zQTQ)^{-1}QTP$, \; $z^{-1}\in\overline\varrho(QTQ)$,
\medskip
\item[{\it (iv)}] 
$f(z)=PTP+zPTQ(1-zQTQ)^{-1}QTP$, \; $z^{-1}\in\overline\varrho(QTQ)$,
\medskip
\item[{\it (v)}] 
$zf(z)=1_0-s(z)^{-1}$, \; 
$z^{-1}\in\overline\varrho(T)\cap\overline\varrho(QTQ)$,
\medskip
\item[{\it (vi)}]
$f(z)=(PT(1-zT)^{-1}P)(P(1-zT)^{-1}P)^{-1}$, \; 
$z^{-1}\in\overline\varrho(T)\cap\overline\varrho(QTQ)$,
\medskip
\item[{\it (vii)}] 
$\displaystyle f(z)=\sum_{n\ge0}z^nPT(QT)^nP$, \; $|z|<\|QTQ\|^{-1}$.
\end{itemize}
\end{prop}

\begin{proof}
Taking into account Theorem~\ref{thm:dom2}, Proposition~\ref{prop:falt}, the comments just before this theorem and the identities $s(0)=1_0$, $f(0)=PTP$, it only remains to prove the analyticity of $s$ and $f$ at the origin as well as the expression {\it (vii)}. The power expansion
$$
 (1-zT)^{-1} = \sum_{n\ge0} z^nT^n, \qquad \|zT\|<1,
$$
shows that $s(z)$ is analytic for $|z|<\|T\|^{-1}$. Similar arguments for $QTQ$ and the power expansion of $(1-zQTQ)^{-1}$ around the origin, combined with the expression {\it (iv)}, prove that $f(z)$ is analytic for $|z|<\|QTQ\|^{-1}$ with an expansion given by {\it (vii)}.
\end{proof}

When $\frak{B}$ is a Hilbert space and $P$ is an orthogonal projection we have that $\|P\|=\|Q\|=1$. Hence, if $f$ is the FR-function of a closed subspace with respect to a bounded operator on a Hilbert space, the power expansion of Proposition~\ref{prop:B-falt}.{\it(vii)} holds for $|z|<\|T\|^{-1}$ regardless of the subspace.

The power expansion of FR-functions given in Proposition~\ref{prop:B-falt}.{\it (vii)} identifies them as `true generating functions' in the bounded case. Actually, this expansion gives to $f$ the meaning of a generating function of `first returns': suppose that each time step the state $v\in\BB$ of a system evolves according to $v \to Tv$. If the evolution starts at a state in $\BB_0$, the projection $P$ conditions on the event `return to $\BB_0$', while the projection $Q$ conditions on the event `no return to $\BB_0$'. Thus, the operator coefficient $PT(QT)^nP$ reflects the fact that, starting at $\BB_0$, the system has not returned to $\BB_0$ during the first $n$ steps, but has returned to $\BB_0$ in the last step, i.e. the system has returned to $\BB_0$ for the first time in the last step. The power expansion of the corresponding Stieltjes function,
$$
 s(z) = \sum_{n\ge0} z^n PT^nP, \qquad |z|<\|T\|^{-1},
$$
also uncovers its interpretation as generating function of `returns': the operator coefficient $PT^nP$ is associated with the return to $\BB_0$ in the $n$-th step, but without excluding any return at previous steps.
 
Proposition~\ref{prop:B-falt}.{\it (vii)} implies that $f(0)=PTP$ and $f^{(n)}(0)=n!PTQ(QTQ)^{n-1}QTP$ for $T\in\frak{B}(\BB)$. Properly interpreted, this result makes sense also for certain unbounded situations, see Lemma~\ref{lem:der} in Appendix~A.

\section{The self-adjoint case - Nevanlinna functions}
\label{sec:s-a}

FR-functions related to self-adjoint operators have a especial interest because, together with the unitaries, they are the operators that appear more frequently in applications. In this section we will discuss the FR-function of a self-adjoint bounded projection $P$ with respect to a self-adjoint operator $T\colon\DD(T)\to\HH$ on a Hilbert space $\HH$. The self-adjointness of the projection means that it is orthogonal, i.e. $\ker P=\RR(P)^\bot$. In this case we talk about the FR-function of the closed subspace $\HH_0:=\RR(P)$ --since it determines the subspace $\HH_1:=\ker P=\HH_0^\bot$ along which the projection takes place-- and likewise for Stieltjes functions and m-functions. This is the kind of functions that we will consider in this section, as well as in Sections~\ref{sec:rep-N} and \ref{sec:SA}. 

Continuing in the spirit of the previous section, we will deal with general self-adjoint operators, bounded or not. Although including unbounded situations will make the discussions more complicated, this is a price that we will pay to arrive later on at operator and integral representations valid for every Nevanlinna function, see Section~\ref{sec:rep-N}. For the technical results on unbounded self-adjoint operators that we will use below we refer to \cite{Kato,Schmudgen}. 

As in the unitary case, self-adjoint operators generate FR-functions with special properties. Remember that the FR-functions for unitary operators are Schur functions. This section provides a similar identification in the self-adjoint case: the FR-functions for self-adjoint operators are operator valued {\bf Nevanlinna functions}, i.e. analytic functions $f\colon\C\setminus\R\to\frak{B}(\HH_0)$ satisfying
\begin{equation} \label{eq:N}
 f(z)^\text{\rm \dag}=f(\overline{z}), 
 \qquad \frac{\im f(z)}{\im z}\ge0, 
 \qquad z\in\C\setminus\R.
\end{equation}
In other words, they are analytic functions in the upper and lower half-planes 
$$
 \C_+:=\{z\in\C : \im z>0\}, \qquad\qquad \C_-:=\{z\in\C : \im z<0\},
$$
such that 
\begin{equation} \label{eq:N1}
 \im f(z)\ge0 \text{ for } z\in\C_+, 
 \qquad\qquad
 \im f(z)\le0 \text{ for } z\in\C_-. 
\end{equation}
These inequalities, which hold simultaneously under the symmetry $f(z)^\dag=f(\overline{z})$, are the self-adjoint version of the contractivity property \eqref{eq:Schur} of Schur functions in the unitary case. 

The subtleties arising in the unbounded case --absent for unitaries-- make it difficult to know if the above connection between FR-functions for self-adjoint operators and Nevanlinna functions holds in its full extent, so we will prove it in two general cases:
\begin{itemize}
\item For arbitrary {\bf finite-dimensional subspaces}, not necessarily included in the domain of the operator. 
\item For closed {\bf subspaces included in the domain of the operator}, but with no restriction on their dimension.
\end{itemize} 

Since we will deal with adjoints of operators which are not necessarily bounded, we must be very careful with the manipulation of adjoints. First, the adjoint of $T\colon\DD(T)\to\HH$ exists iff $\DD(T)$ is dense in $\HH$, a requirement which is implicit when stating that $T$ is self-adjoint. Besides, if we only know that $T$ and $S$ are operators densely defined on $\HH$, all that we can say about $(S+T)^\dag$ and $(ST)^\dag$ is that, whenever $S+T$ and $ST$ are also densely defined on $\HH$,
$$ 
 S^\dag+T^\dag \subset (S+T)^\dag,
 \qquad \qquad
 T^\dag S^\dag \subset (ST)^\dag. 
$$ 
In spite of the above, to guarantee the equalities in the above relations it is not necessary for both operators to be bounded and everywhere defined, it is enough to impose this requirement for only one of the operators,
$$
 S\in\frak{B}(\HH) \quad\Rightarrow\quad  
 S^\dag+T^\dag = (S+T)^\dag,
 \quad
 T^\dag S^\dag = (ST)^\dag. 
$$ 
These remarks will be important in what follows.   

\begin{thm}[\bf FR-functions of self-adjoint operators] \label{thm:FR=N}
Let $f$ be the FR-function of a closed subspace $\HH_0\subset\HH$ with respect to a self-adjoint operator $T\colon\DD(T)\to \HH$ on a Hilbert space $\HH$. Then, $f$ is a Nevanlinna function in any of the following cases:
\begin{itemize}
\item[{\it (i)}] $\dim\HH_0<\infty$.
\item[{\it (ii)}] $\HH_0\subset\DD(T)$.
\end{itemize} 
\end{thm}
 
\begin{proof}
\quad

\noindent{\it (i)}
Let $s$ be the Stieltjes function of $\HH_0$ with respect to $T$. The spectrum of $T$ is real, thus $\C\setminus\R\subset\varrho(T)$ and $s(z)$ exists as an operator everywhere defined on $\HH_0$ for every non-real $z$. Also, $s(z)v=0$ for some $v\in\HH_0$ and $z\in\C\setminus\R$ implies that
$$
 0 = \<s(z)v|v\> = \<(1-zT)^{-1}v|v\> = \<w|(1-zT)w\> = \|w\|^2-z\<w|Tw\>,
 \qquad w=(1-zT)^{-1}v.
$$
Since $\<w|Tw\>$ is real and $z$ is not, we find that $w=0$, hence $v=0$. Therefore, if $z\in\C\setminus\R$, we conclude that $\ker s(z)=\{0\}$, which is equivalent to stating that $s(z)^{-1}\in\frak{B}(\HH_0)$ because $\dim\HH_0<\infty$. Then, Corollary~\ref{cor:dom3} implies that $f(z)\in\frak{B}(\HH_0)$ and is analytic for $z\in\C\setminus\R$. 

Let $z\in\C\setminus\R$. Taking adjoints we get $s(z)^\dag=s(\overline{z})$. Applying Theorem~\ref{thm:dom2} we find that 
\begin{equation} \label{eq:fs-sa}
 f(z)=z^{-1}(1_0-s(z)^{-1}), \qquad z\in\C\setminus\R,
\end{equation} 
hence $f(z)^\dag=f(\overline{z})$. 

To finish the proof, it is convenient to rewrite \eqref{eq:fs-sa} using the m-function \eqref{eq:ms},
$$
 g(z) := f(z^{-1}) = z1_0 +  m(z)^{-1}, \qquad z\in\C\setminus\R.
$$
Then, $\im f(z)/\im z\ge0$ is equivalent to $\im g(z)/\im z\le0$. As in the case of $s$, the adjoint of $ m$ is given by $ m(z)^\dag= m(\overline{z})$. Therefore, if $z\in\C\setminus\R$, 
$$
\begin{aligned}
 \im  m(z) & = \frac{1}{2i} P((T-z)^{-1}-(T-\overline{z})^{-1})P 
 =  (\im z) P(T-\overline{z})^{-1}(T-z)^{-1}P,
 \\
 \im g(z) & = \im z + \frac{1}{2i}( m(z)^{-1}-( m(z)^\dag)^{-1})
 = \im z - ( m(z)^\dag)^{-1} (\im m(z)) \, m(z)^{-1}
 \\ 
 & = (\im z) \, 
 (1_0 -( m(z)^\dag)^{-1} P(T-\overline{z})^{-1}(T-z)^{-1}P \, m(z)^{-1})
 \\
 & = -(\im z) \,
 ( m(z)^\dag)^{-1} P(T-\overline{z})^{-1}Q(T-z)^{-1}P \, m(z)^{-1},
 \qquad Q=1-P,
\end{aligned}
$$ 
which gives
$$
 \frac{\im g(z)}{\im z} = -R(z)^\dag R(z), 
 \qquad R(z):=Q(T-z)^{-1} m(z)^{-1}. 
$$
This proves that $\im g(z)/\im z\le0$ for $z\in\C\setminus\R$.

\noindent{\it (ii)} 
Denote by $P$ and $Q$ the orthogonal projections onto $\HH_0$ and $\HH_0^\bot$ respectively. The block operator $QTQ$ --which controls the domain of definition of $f$-- has domain $\HH_0^\bot\cap\DD(T)$, which is dense in $\HH_0^\bot$ because $\HH_0\subset\DD(T)$ implies that $\DD(T)=\HH_0\oplus(\HH_0\cap\DD(T))$, while $\DD(T)$ is dense in $\HH$ because $T$ is self-adjoint. Thus the adjoint of $QTQ$ makes sense. Indeed $(QTQ)^\dag$ is an extension of $QTQ$, i.e. $QTQ$ is symmetric. 

Let us see that the symmetric block operator $QTQ$ is indeed self-adjoint. For this purpose we can consider $QTQ$ equivalently as an operator on $\HH$. Then $(QTQ)^\dag=(TQ)^\dag Q$ because $Q\in\frak{B}(\HH)$. As for $TQ=T-TP$, remember from the proof of Proposition~\ref{prop:QTQ} that $TP\in\frak{B}(\HH)$ whenever $\HH_0\subset\DD(T)$ because then $TP$ is closed and everywhere defined on $\HH$. Hence $(TQ)^\dag = T-(TP)^\dag$. Note that $(TP)^\dag\in\frak{B}(\HH)$ is an extension of $PT$ to the whole space $\HH$. Despite this, we can substitute $(TP)^\dag$ by $PT$ in $T-(TP)^\dag$ because this operator has the same domain $\DD(T)$ as $PT$. We conclude that $(TQ)^\dag=T-PT=QT$, so that $(QTQ)^\dag=(TQ)^\dag Q=QTQ$. 
 
As a self-ajoint operator, $QTQ$ has real spectrum, thus $\C\setminus\R\subset\varrho(QTQ)$. Then, Proposition~\ref{prop:QTQ} implies that $\C\setminus\R\subset\varrho(QT)$, so that from Proposition~\ref{prop:rhoQT} we know that $f(z)\in\frak{B}(\HH_0)$ and is analytic for $z\in\C\setminus\R$. Since Theorem~\ref{thm:dom2} gives \eqref{eq:fs-sa}, we can follow the same steps as in the case {\it (i)} to prove \eqref{eq:N}. 
\end{proof}

As a consequence of the previous theorem, if a subspace has finite dimension or lies in the domain of the operator, the versions \eqref{eq:f-s} and \eqref{eq:f-ms} of the corresponding renewal equation read in the self-adjoint case as
\begin{align} \label{eq:f-s-sa} 
 & \kern70pt f(z) = z^{-1}(1_0-s(z)^{-1}) = z^{-1}1_0 +  m(z^{-1})^{-1},
 \qquad z\in\C\setminus\R,
 \\ \nonumber
 & s(z)=P(1-zT)^{-1}P=\int\frac{d\mu(t)}{1-zt},
 \qquad
  m(z)=P(T-z)^{-1}P=\int\frac{d\mu(t)}{t-z},
 \qquad
 \mu(\R)=1_0.
\end{align}
Here $\mu$ is the spectral measure assigned to the subspace $\HH_0$ by the self-adjoint operator $T$.

As in the case of Schur functions, a dichotomy concerning Nevanlinna functions classifies them into two types. If $f$ is a scalar Nevanlinna function, i.e. an analytic map $f\colon\C\setminus\R\to\C$ such that $\overline{f(z)}=f(\overline{z})$ and $f(\C_\pm)\subset\overline\C_\pm$, general principles of complex analysis lead to the following options: 

\begin{itemize}
\item $f$ is {\it degenerate}: $\im f(z_0)=0$ for some $z_0\in\C\setminus\R$, then $f$ is a real constant on $\C\setminus\R$. 
\item $f$ is {\it non-degenerate}: $\im f\ne0$ on $\C\setminus\R$, then $f(\C_\pm)\subset\C_\pm$.
\end{itemize}

\noindent On the other hand, the operator valued Nevanlinna functions $f\colon\C\setminus\R\to\frak{B}(\HH_0)$ are obviously characterized by the fact that $\<v|f(z)v\>$ is a scalar Nevanlinna function for every $v\in\HH_0$. This allows us to study many properties of the operator valued case by reducing them to the scalar situation. 

This is the case for instance for the analysis of the strict positivity of $\im f$. To state that $\im f(z_0)/\im z_0$ is not positive definite at some $z_0\in\C\setminus\R$ is equivalent to the existence of $v\in\HH_0$ such that $\im f(z_0)v=0$. Then, $\im\<v|f(z_0)v\>=\<v|\im f(z_0)v\>=0$, so that $\<v|f(z)v\>$ is a degenerate scalar Nevanlinna function, i.e. a real constant. As a consequence $\<v|\im f(z)v\>=\im\<v|f(z)v\>=0$ for every $z\in\C\setminus\R$, which implies that $\im f(z)v=0$ at every non-real $z$ because $\im f(z)/\im z\ge0$. Therefore, $\ker\im f(z)$ is independent of $z\in\C\setminus\R$, leaving only two possibilities:  

\begin{itemize}
\item $f$ is {\it degenerate}: $\ker\im f$ is non-trivial on $\C\setminus\R$, then $\im f/\im z\ge0$ but nowhere positive definite on $\C\setminus\R$. 
\item $f$ is {\it non-degenerate}: $\ker\im f$ is trivial on $\C\setminus\R$, then $\im f/\im z>0$ on $\C\setminus\R$.
\end{itemize} 

Degenerate Nevanlinna functions have several equivalent characterizations. Some of them, of interest for the Nevanlinna version of the Schur algorithm discussed below, have to do with the limit value and the derivative of Nevanlinna functions at a point in the real line. When this point is the origin, the existence of such limits and derivatives is guaranteed for FR-functions of bounded operators, but may hold even for an unbounded self-adjoint operator $T$ if understood as limits along the imaginary axis, i.e. normal to the real line: if $\HH_0\subset\DD(T)$, the corresponding FR-function $f$ has a well defined normal limit at the origin 
$$
 f(0):=\lim_{y\to0}f(iy)=PTP\in\frak{B}(\HH_0), 
$$
while the more restrictive condition $\HH_0\subset\DD(T^2)$ ensures that $f$ has also a normal derivative at the origin 
$$
 f'(0):=\lim_{y\to0}\frac{f(iy)-f(0)}{iy}=PTQTP\in\frak{B}(\HH_0). 
$$
These limits must be understood in the weak sense if $\dim\HH_0=\infty$, see Lemma~\ref{lem:der} in Appendix~A. 

In the case of a Nevanlinna FR-function $f$, the above expressions show that $f(0)$ is self-adjoint, while $\displaystyle f'(0)$ is not only self-adjoint, but also non-negative definite. Actually, these properties of normal limits and normal derivatives hold for every Nevanlinna function at every point in the real line where such limit and derivative exist.

\begin{prop} \label{prop:der-p}
Given an arbitrary Nevanlinna function $f \colon \C\setminus\R \to \frak{B}(\HH_0)$, consider the following weak limits for $x\in\R$,
$$
\begin{aligned}
 & f(x):=\lim_{y\to0}f(x+iy), 
 & & \text{(normal limit of $f$ at $x$)},
 \\
 & f'(x):=\lim_{y\to0}\frac{f(x+iy)-f(x)}{iy} 
 & \qquad & \text{(normal derivative of $f$ at $x$)}.
\end{aligned}
$$
Then, whenever these limits exist as everywhere defined operators on $\HH_0$, they are self-adjoint operators of $\frak{B}(\HH_0)$ and $\displaystyle f'(x)\ge0$.
\end{prop}

\begin{proof}
Suppose that $f(x)$ is everywhere defined on $\HH_0$ for some $x\in\R$. Since $f(x+iy)^\dag=f(x-iy)$ for $y\in\R\setminus\{0\}$, we get 
$$
 \<u|f(x)v\>=\lim_{y\to0}\<u|f(x+iy)v\>=\lim_{y\to0}\<f(x-iy)u|v\>=\<f(x)u|v\>,
 \qquad u,v\in\HH_0.
$$
This means that $f(0)$ is self-adjoint, hence closed, thus it is also bounded because $f(0)$ has closed domain $\HH_0$ by hypothesis. 

If, besides, $\displaystyle f'(x)$ is everywhere defined on $\HH_0$, then 
$$
\begin{aligned}
 \<v|f'(x)v\>
 & =\lim_{y\to0}\frac{\<v|(f(x+iy)-f(x))v\>-\<v|(f(x-iy)-f(x))v\>}{2iy}
 \\
 & =\lim_{y\to0}\frac{\<v|\im f(x+iy)v\>}{y} \ge 0,
 \qquad v\in\HH_0.
\end{aligned}
$$
This implies that $\displaystyle f'(x)$ is self-adjoint and non-negative definite, hence also bounded since we assume that it has closed domain $\HH_0$. 
\end{proof}

The following proposition summarizes different properties that distinguish between degenerate and non-degenerate Nevanlinna functions, among them those given in terms of their limits and derivatives at a point in the real line. The proof is given in Appendix~B. 

\begin{prop} \label{prop:deg-NFR}
If $f\colon\C\setminus\R\to\frak{B}(\HH_0)$ is a Nevanlinna function and $x\in\R$, then, for every $z\in\C\setminus\R$,
$$
\begin{aligned}
 \ker\im f(z) 
 & = \{v\in\HH_0:fv\text{ is constant on }\C\setminus\R\}
 \\[2pt]
 & = \ker(f(z)-f(x)), 
 & & \kern-70pt \text{if } \exists 
 f(x):=\lim_{y\to0}f(x+iy),
 \\[-2pt]
 & = \ker f'(x), 
 & & \kern-70pt \text{if } \exists 
 f'(x):=\lim_{y\to0}(f(x+iy)-f(x))/iy,
\end{aligned}  
$$
where $f(x)$ and $\displaystyle f'(x)$ are assumed to exist as weak limits everywhere defined on $\HH_0$.
As a consequence, any of the following conditions characterizes the non-degenerate Nevanlinna functions:
\begin{itemize}
 \item[{\it (i)}] $\im f(z_0)$ is invertible for some $z_0\in\C\setminus\R$.
 \item[{\it (ii)}] $\im f(z)$ is invertible for every $z\in\C\setminus\R$.
 \item[{\it (iii)}] $fv$ is not constant on $\C\setminus\R$ 
 for any $v\in\HH_0\setminus\{0\}$.
 \item[{\it (iv)}] $f(z_0)-f(x)$ is invertible for some $z_0\in\C\setminus\R$
 \; (if $f(x)$ exists).  
 \item[{\it (v)}] $f(z)-f(x)$ is invertible for every $z\in\C\setminus\R$
 \; (if $f(x)$ exists).
 \item[{\it (vi)}] $\displaystyle f'(x)$ is invertible
 \; (if $\displaystyle f'(x)$ exists).
\end{itemize}
\end{prop} 

Taking into account the previous remark, Propositions \ref{prop:der-p} and \ref{prop:deg-NFR} imply that for every non-degenerate Nevanlinna function $f$ and $x\in\R$, $\displaystyle f'(x)$ is positive definite when it exists as an everywhere defined operator. 

The above properties of non-degenerate Nevanlinna functions will be key to guarantee the proper running of the Schur algorithm on the real line discussed in Section~\ref{sec:SA}: analogously to the case of Schur functions, such a Schur algorithm for Nevanlinna functions will not stop unless a degenerate Nevanlinna iterate arises.

\section{Representations of Nevanlinna functions}
\label{sec:rep-N}
 
As in the unitary case, the FR-functions for self-adjoint operators turn out to be more than mere examples of Nevanlinna functions. We will see in this section that the FR-functions of finite-dimensional subspaces with respect to self-adjoint operators provide all the matrix Nevanlinna functions, up to additive terms which are non-negative multiples of $-z^{-1}$. This gives a new operator representation for matrix Nevanlinna functions, analogue of the operator representation for matrix Schur functions as FR-functions related to unitaries. This result also yields a new integral representation of matrix Nevanlinna functions, providing a self-adjoint version of the integral representation \eqref{eq:f-mu} of matrix Schur functions. This establishes a connection between matrix Nevanlinna functions and matrix measures on the real line which, in contrast to the standard one, is one-to-one, and is in perfect analogy with the usual correspondence between matrix Schur functions and matrix measures on the unit circle via Carath\'eodory functions.

Before proving the above results, let us recall the details of the standard integral representations of matrix Nevanlinna functions see \cite[Chapter III]{ST}, \cite[Sect. I.4]{Brodskii} and \cite{Shmulyan,GT,GKMT}. A general representation formula states that every such a function $f$ can be expressed as
\begin{equation} \label{eq:N-mu}
 f(z) = a + bz + \int \frac{1+zt}{t-z} \, d\nu(t),
\end{equation}
with $a$ a self-adjoint matrix, $b$ a non-negative definite matrix and $\nu$ a finite positive matrix measure on $\R$. This integral representation establishes a one-to-one correspondence between matrix Nevanlinna functions and matrix measures on $\R$ only up to the matrix coefficients $a$, $b$.

A particular situation will be of especial interest. It can be proved that the condition 
\begin{equation} \label{eq:St-char}
 \sup_{y>0} y\|f(iy)\| < \infty
\end{equation}
is equivalent to stating that the matrix Nevanlinna function $f$ is the m-function of a finite (not necessarily normalized) matrix measure $\mu$ on $\R$, i.e.
\begin{equation} \label{eq:St-int}
 f(z) = \int \frac{d\mu(t)}{t-z}.
\end{equation}
Actually, by dominated convergence, such an m-function satisfies 
\begin{equation} \label{eq:St-mu}
 \mu(\R)=\lim_{y\to\infty}-iyf(iy)=\lim_{y\to\infty}y\im f(iy). 
\end{equation}
In the case of an m-function $f(z)=P(T-z)^{-1}P$ associated with a self-adjoint operator $T=\int\!t\,dE(t)$, the related measure $\mu$ is the spectral measure $PEP$ of the subspace where $P$ projects, thus $\mu$ is normalized by $\mu(\R)=1_0$ and 
\begin{equation} \label{eq:St-norm}
 \lim_{y\to0}-iyf(iy)=1_0.
\end{equation} 

The above results will be the key to proving the theorem below. 

To connect with the language in the paper, note that any Nevanlinna function with values in operators on a finite-dimensional Hilbert space $\HH_0$ can be identified with a matrix valued Nevanlinna function by choosing an orthonormal basis in $\HH_0$. Hence, the above results about integral representations can be obviously translated to the operator valued setting if $\dim\HH_0<\infty$.  

\begin{thm}[\bf characterizations of matrix Nevanlinna functions] \label{thm:N=FR}
If $\HH_0$ is a finite-dimensional Hilbert space, the following statements are equivalent:
\begin{itemize}
\item[{\it (i)}] $f$ is a Nevanlinna function with values in $\frak{B}(\HH_0)$.
\item[{\it (ii)}] {\bf Integral representation I:} There exists a positive measure $\mu$ on $\R$ with values in operators on $\HH_0$ such that $\mu(\R)\le1_0$ and, for $z\in\C\setminus\R$,
$$
 \kern40pt
 f(z) = z^{-1} \left(1_0-\left(\int \frac{d\mu(t)}{1-zt}\right)^{-1}\right)
 = z^{-1}1_0 +  m(z^{-1})^{-1},
 \quad  m(z) = \int \frac{d\mu(t)}{t-z}. 
$$
\item[{\it (iii)}] {\bf Integral representation II:} There exist a positive measure $\hat\mu$ on $\R$ with values in operators on $\HH_0$ and a self-adjoint operator $c$ on $\HH_0$ such that $\hat\mu(\R)=1_0$, $c\le1_0$ and, for $z\in\C\setminus\R$,
$$
 \kern40pt
 f(z) = z^{-1} \left(c-\left(\int \frac{d\hat\mu(t)}{1-zt}\right)^{-1}\right)
 = cz^{-1} +  \hat{m}(z^{-1})^{-1},
 \quad  \hat{m}(z) = \int \frac{d\hat\mu(t)}{t-z}. 
$$
\item[{\it (iv)}] {\bf Operator representation:} $f$ is the FR-function of a closed subspace with respect to a self-adjoint operator, up to the addition of a non-negative multiple of $-z^{-1}$, i.e. there exist a self-adjoint operator $T$ on a Hilbert space $\HH\supset\HH_0$ and a self-adjoint operator $d$ on $\HH_0$ such that $d\ge0$ and, for $z\in\C\setminus\R$,
$$
 \kern40pt f(z) = -dz^{-1} + PT(1-zQT)^{-1}P,
 \quad
 \begin{aligned}
 	& P = \text{orthogonal projection of } \HH \text{ onto } \HH_0,
	\\[-1pt]
 	& Q = 1-P.
 \end{aligned}
$$
\end{itemize}
The correspondence between Nevanlinna functions and measures on $\R$ established by the integral representation {\it (ii)} is one-to-one. The operators $c$ and $d$ in {\it (iii)} and {\it (iv)} are given by $d=\lim_{y\to0}-iyf(iy)=\lim_{y\to0}y\im f(iy)$ and $c=1_0-d$. The correspondence between Nevanlinna functions and pairs $(\hat\mu,c)$ established by the integral representation {\it (iii)} is one-to-one.
\end{thm}

\begin{proof}
Bearing in mind that $-z^{-1}$ is a Nevanlinna function, Theorem~\ref{thm:FR=N}.{\it(i)} gives the implication {\it (iv) $\Rightarrow$ (i)}. Also, using the one-to-one correspondence $\mu\mapsto m(z)$ between finite operator valued measures on $\R$ and their m-functions, {\it (ii)} leads to a one-to-one correspondence $\mu\mapsto f(z)=z^{-1}1_0+ m(z^{-1})^{-1}$ between Nevanlinna functions $f$ and measures $\mu$ such that $\mu(\R)\le1_0$. Due to a similar reason, asuming that the self-adjoint operator $c$ in {\it (iii)} is fixed by the Nevanlinna function $f$, {\it (iii)} yields a one-to-one correspondence between such functions and pairs $(\mu,c)$ with $\mu(\R)=1_0$ and $c\le1_0$. To prove the theorem we will show that {\it (i) $\Rightarrow$ (ii) $\Rightarrow$ (iv) $\Rightarrow$ (iii) $\Rightarrow$ (i)}, obtaining also the values of the operators $c$ and $d$ in terms of the Nevanlinna function $f$. 

\medskip

\noindent \fbox{\it (i) $\Rightarrow$ (ii)} 

\smallskip

All that we must prove is that, if $f$ is a Nevanlinna function, then $ m(z)=(f(z^{-1})-z1_0)^{-1}$ is the m-function of a matrix measure $\mu$ on $\R$ such that $\mu(\R)\le1_0$. This is proved if we show that $ m$ is a Nevanlinna function satisfying $y\| m(iy)\|\le1$ for $y>0$ because in this case $\sup_{y>0} y\| m(iy)\|<\infty$ and
$$
 \<v|y\im m(iy)v\> = \im\<v|y m(iy)v\> 
 \le |\<v|y m(iy)v\>| \le y\|s(iy)\| \|v\|^2 \le \|v\|^2, 
 \quad v\in\HH_0, \quad y>0, 
$$
so that $y\im s(iy)\le1_0$ for $y>0$ and the measure $\mu$ related to $ m$ satisfies
$$
 \mu(\R) = \lim_{y\to\infty}y\im m(iy) \le 1_0.
$$ 

Note that $g(z)=f(z)-z^{-1}1_0$ is a non-degenerate Nevanlinna function because $f(z)$ and $-z^{-1}$ are both Nevanlinna functions, $-z^{-1}$ being non-degenerate. Therefore, $\ker g$ is trivial on $\C\setminus\R$, which means that $g$ is invertible there. Bearing in mind that $\dim\HH_0<\infty$, we conclude that $g(z)^{-1}\in\frak{B}(\HH_0)$ is analytic for $z\in\C\setminus\R$. From the identity $\im g^{-1}=-(g^{-1})^\dag (\im g) g^{-1}$ we find that $\im g(z)^{-1}/\im z<0$. Thus, $\im g(z^{-1})^{-1}/\im z>0$, proving that $ m(z)=g(z^{-1})^{-1}$ is a Nevanlinna function again. 

Concerning the bound for $y\| m(iy)\|$, let us write for convenience $y m(iy)=-h(y)^{-1}$ with $h(y)=i1_0-y^{-1}f(-iy^{-1})$. Assume that $y>0$. Then $\im f(-iy^{-1})\le0$ so that $\im h(y)\ge1_0$. Hence,
$$
 |\<v|h(y)v\>| \ge |\im\<v|h(y)v\>| = \<v|\im h(y)v\> \ge \|v\|^2,
 \qquad v\in\HH_0.
$$
Rewriting this inequality in terms of $u=h(y)v$ --a vector of $\HH_0$ which is as arbitrary as $v$-- yields
$$
 \|y m(iy)u\| \|u\| \ge |\<y m(iy)u|u\>| \ge \|y m(iy)u\|^2,
 \qquad u\in\HH_0.
$$
Therefore, $\|y m(iy)u\|\le\|u\|$ for $u\in\HH_0$, which means that $y\| m(iy)\|\le1$. 

\medskip

\noindent \fbox{\it (ii) $\Rightarrow$ (iv)} 

\smallskip

Suppose that $f(z)=z^{-1}1_0+ m(z^{-1})^{-1}$ for every non-real $z$, with $ m$ the m-function of a finite measure $\mu$ on $\R$ such that $\mu(\R)\le1$. Let us see first that $\mu(\R)$ is invertible. This follows from the fact that $ m$ is an invertible Nevanlinna function, which implies that
$$
 0 < \frac{\im m(z)}{\im z} = \int \frac{d\mu(t)}{|t-z|^2} 
 \le \int \frac{d\mu(t)}{(\im z)^2} = \frac{\mu(\R)}{(\im z)^2},
 \qquad z\in\C\setminus\R.
$$

Therefore, we can define the measure $\mu_0=\mu(\R)^{-1/2}\,\mu\;\mu(\R)^{-1/2}$, which is normalized by $\mu_0(\R)=1_0$. Naimark's dilation theorem guarantees that $\mu_0=PEP$, where $E$ is a spectral measure on $\R$ with values in operators on a Hilbert space $\HH\supset\HH_0$ and $P$ is the orthogonal projection of $\HH$ onto $\HH_0$. Then, the self-adjoint operator $T=\int\!t\,dE(t)$ generates the m-function $ m_0(z)=P(T-z)^{-1}P=\int\!d\mu_0(t)/(t-z)$ for $\HH_0$, related to the Nevanlinna FR-function $f_0(z)=PT(1-zQT)^{-1}P$ by $f_0(z)=z^{-1}+ m_0(z^{-1})^{-1}$, according to \eqref{eq:f-s-sa}. Since $ m(z)=\mu(\R)^{1/2}\, m_0(z)\,\mu(\R)^{1/2}$, we find that
$$
 f(z) = z^{-1}1_0 + \mu(\R)^{-1/2}\, m_0(z^{-1})^{-1}\mu(\R)^{-1/2}
 = z^{-1}(1_0-\mu(\R)^{-1}) + \mu(\R)^{-1/2}f_0(z)\,\mu(\R)^{-1/2}.
$$

Consider the self-adjoint operator $\tilde{T}=\hat{T}_0T\hat{T}_0$, where $\hat{T}_0=T_0 \oplus 1_0^\bot$ is the sum of $T_0=\mu(\R)^{-1/2}$ on $\HH_0$ and the identity $1_0^\bot$ on $\HH_0^\bot$. Applying Theorems~\ref{thm:split2}.{\it(ii)} and \ref{thm:split2}.{\it(iii)} we find that the FR-function of $\HH_0$ with respect to $\tilde{T}$ is $\tilde{f}(z)=T_0f_0(z)T_0=\mu(\R)^{-1/2}f_0(z)\mu(\R)^{-1/2}$. Therefore,
$$
 f(z) = -dz^{-1} + P\tilde{T}(1-zQ\tilde{T})^{-1}P,
 \qquad d=\mu(\R)^{-1}-1_0,
$$ 
where $d=\mu(\R)^{-1/2}(1_0-\mu(\R))\mu(\R)^{-1/2}\ge0$ because $\mu(\R)\le1_0$.

\medskip

\noindent \fbox{\it (iv) $\Rightarrow$ (iii)} 

\smallskip

Let $f(z) = -dz^{-1} + PT(1-zQT)^{-1}P$ for $z\in\C\setminus\R$, with $T$ a self-adjoint operator on $\HH\supset\HH_0$, $P$ the orthogonal projection onto $\HH_0$, $Q=1-P$ and $d\ge0$. From \eqref{eq:f-s-sa}, the FR-function $PT(1-zQT)^{-1}P$ can be expressed as $PT(1-zQT)^{-1}P=z^{-1}+ \hat{m}(z^{-1})^{-1}$, where $\hat{m}$ is the m-function of the spectral measure $\hat\mu$ of $\HH_0$ with respect to $T$, which satisfies $\hat\mu(\R)=1_0$. Therefore,
$$
 f(z) = cz^{-1} +  \hat{m}(z^{-1})^{-1},
 \qquad c=(1_0-d)z^{-1}, 
$$
with $c\le1_0$ because $d\ge0$.

\medskip

\noindent \fbox{\it (iii) $\Rightarrow$ (i)} 

\smallskip

Assume that $f(z)=cz^{-1}+ \hat{m}(z^{-1})^{-1}$ for every non-real $z$, with $c\le1_0$ and $\hat{m}$ the m-function of a measure $\hat\mu$ on $\R$ such that $\hat\mu(\R)=1_0$. Since $\hat{m}$ is a Nevanlinna function and $c$ is self-adjoint, $f$ is analytic on $\C\setminus\R$ and $f(z)^\dag=f(\overline{z})$. To show that $\im f(z)/\im z\ge0$ we will use again Naimark's dilation theorem to get a spectral measure $E$ on $\R$ with values in operators on a Hilbert space $\HH\supset\HH_0$ such that $\hat\mu=PEP$, where $P$ is the orthogonal projection of $\HH$ onto $\HH_0$. Then, the self-adjoint operator $T=\int\!t\,dE(t)$ yields the operator representation $\hat{m}(z)=P(T-z)^{-1}P$. Using this representation we find that
$$
 \im \hat{m}(z) = \frac{1}{2i} P((T-z)^{-1}-(T-\overline{z})^{-1})P 
 =  (\im z) P(T-\overline{z})^{-1}(T-z)^{-1}P.
$$
Hence,
$$
\begin{aligned}
 \im f(z^{-1}) 
 & = c\im z - (\hat{m}(z)^\dag)^{-1} (\im \hat{m}(z)) \, \hat{m}(z)^{-1}
 \\ 
 & = (\im z) \, 
 (c - 
 (\hat{m}(z)^\dag)^{-1} P(T-\overline{z})^{-1}(T-z)^{-1}P \, 
 \hat{m}(z)^{-1})
 \\
 & = -(\im z) \,
 (\hat{m}(z)^\dag)^{-1} 
 P(T-\overline{z})^{-1}(1-PcP)(T-z)^{-1}P 
 \, \hat{m}(z)^{-1}
 \\
 & = -(\im z) \,
 (\hat{m}(z)^\dag)^{-1} 
 P(T-\overline{z})^{-1}(P(1_0-c)P+Q)(T-z)^{-1}P 
 \, \hat{m}(z)^{-1}.
\end{aligned}
$$ 
This expression shows that $\im f(z^{-1})/\im z\le0$, which is equivalent to $\im f(z)/\im z\ge0$, proving that $f$ is a Nevanlinna function. 

Concerning the values of the operators $c$ and $d$, note that the proof of {\it (iv) $\Rightarrow$ (iii)} has shown that, if $d$ satisfies {\it(iv)}, then $c=1_0-d$ satisfies {\it(iii)}. Besides, if $c$ satisfies {\it(iii)}, using \eqref{eq:St-mu} for the m-function $\hat{m}$, we find for the function $f$ given by {\it(iii)} that we have
$$
 \lim_{y\to0}-iyf(iy) 
 = -c + \left(\lim_{y\to0}iy^{-1} \hat{m}(-iy^{-1})\right)^{-1}
 = -c+\hat\mu(\R) = 1_0-c.
$$
This finishes the proof of the theorem.     
\end{proof}

The degree of freedom of the integral representation {\it(ii)} in the normalization of the measure is translated into the degree of freedom in the coefficient of $z^{-1}$ for the integral representation {\it(iii)}. The impossibility to guarantee the simultaneous normalization of both terms has to do with the fact that FR-functions of self-adjoint operators give all the Nevanlinna functions but only up to $-z^{-1}$ terms, as {\it(iv)} states.

The above theorem uncovers a surprising result about the set of Nevanlinna functions: all its elements can be recovered from those of the proper subset of m-functions by performing the simple transformation $m(z) \mapsto z^{-1}1_0+ m(z^{-1})^{-1}$. This result is the reason behind the following definition.  

\begin{defn} \label{def:N-mu}
To each Nevanlinna function $f$ with values in operators on a finite-dimensional Hilbert space $\HH_0$ we associate a measure $\mu$ on $\R$ such that $\mu(\R)\le1_0$ through the integral representation I in Theorem~\ref{thm:N=FR}.{\it(ii)}, i.e. $\mu$ is the only of such measures satisfying
$$
 \int \frac{d\mu(t)}{1-zt} = (1_0-zf(z))^{-1},
 \qquad 
 z\in\C\setminus\R.
$$
We say that $\mu$ is the {\bf measure of the Nevanlinna function} $f$, or that $f$ is the {\bf Nevanlinna function of the measure} $\mu$. 
\end{defn}

The measure $\mu$ of a Nevanlinna function, as defined previously, is different from the measure given by the integral representation II in Theorem~\ref{thm:N=FR}.{\it(iii)} unless $\mu(\R)=1_0$. It is also different from the measure appearing in the usual integral representation \eqref{eq:N-mu}, and from the measure of an m-function given by \eqref{eq:St-int}. In contrast to these other measures, according to Theorem~\ref{thm:N=FR}, that one given in Definition~\ref{def:N-mu} establishes a one-to-one correspondence between the set of all matrix Nevanlinna functions and a subset of matrix measures on $\R$, namely, the subset of measures $\mu$ such that $\mu(\R)\le1_0$. Definition~\ref{def:N-mu} is the exact analogue of the relation \eqref{eq:st-sc} characterizing the measure of a Schur function.  

It is worth to remark that we would not have reached the almost equivalence --only up to $-z^{-1}$ terms-- between Nevanlinna functions and FR-functions for self-adjoint operators if the latter ones were defined by the transfer/characteristic function representation given by Proposition~\ref{prop:falt}.{\it (ii)}. The reason for this is that the expressions of FR-functions given by this proposition are limited to subspaces included in the domain of the operator. As Lemma~\ref{lem:der} in Appendix~A shows, this requires for the Nevanlinna functions to have a well defined normal limit at the origin, something violated by some Nevanlinna functions different from $-z^{-1}$, such as for instance $\ln(z)$. 

As a direct consequence of Theorem~\ref{thm:N=FR}, we get the following characterization of the Nevanlinna FR-functions, that is, the Nevanlinna functions for which there is an operator representation such as in {\it(iv)} of the previous theorem with $d=0$. This is equivalent to stating that $c=1_0$, which means that $\hat\mu=\mu$, i.e. the integral representations I and II of such theorem coincide. In other words, Nevanlinna FR-functions are those Nevanlinna functions whose measure $\mu$ --as given by Definition~\ref{def:N-mu}-- satisfies $\mu(\R)=1_0$.

\begin{cor}[\bf characterizations of matrix Nevanlinna FR-functions] \label{cor:N=FR}
If $\HH_0$ is a finite-dimensional Hilbert space, the following statements are equivalent:
\begin{itemize}
\item[{\it (i)}] $f$ is a Nevanlinna function with values in $\frak{B}(\HH_0)$ such that $\lim_{y\to0}y\im f(iy)=0$.
\item[{\it (ii)}] $f$ has an integral representation \eqref{eq:N-mu} whose measure $\nu$ has no mass point at the origin.
\item[{\it (iii)}] {\bf Integral representation:} There exists a positive measure $\mu$ on $\R$ with values in operators on $\HH_0$ such that $\mu(\R)=1_0$ and, for $z\in\C\setminus\R$,
$$
\begin{aligned}
 \kern40pt
 f(z) & = z^{-1} \left(1_0-\left(\int \frac{d\mu(t)}{1-zt}\right)^{-1}\right)
 = \left(\int\frac{t\,d\mu(t)}{1-zt}\right) 
 \left(\int\frac{d\mu(t)}{1-zt}\right)^{-1}
 \\
 & = z^{-1}1_0 +  m(z^{-1})^{-1},
 \qquad  m(z) = \int \frac{d\mu(t)}{t-z}. 
\end{aligned}
$$
\item[{\it (iv)}] {\bf Operator representation:} $f$ is strictly the FR-function of a closed subspace with respect to a self-adjoint operator, i.e. there exists a self-adjoint operator $T$ on a Hilbert space $\HH\supset\HH_0$ such that, for $z\in\C\setminus\R$,
$$
 \kern40pt f(z) = PT(1-zQT)^{-1}P,
 \quad
 \begin{aligned}
 	& P = \text{orthogonal projection of } \HH \text{ onto } \HH_0,
	\\[-1pt]
 	& Q = 1-P.
 \end{aligned}
$$
\end{itemize}
The correspondence between Nevanlinna FR-functions and measures on $\R$ established by the representation {\it (iii)} is one-to-one.
\end{cor}
    
\begin{proof}
In view of Theorem~\ref{thm:N=FR}, it only remains to see the equivalence {\it (i) $\Leftrightarrow$ (ii)}. This follows from the fact that a matrix valued function $f$ given by the integral representation \eqref{eq:N-mu} satisfies
$$
 \lim_{y\to0} -iyf(iy) = \lim_{y\to0} \int \frac{1+iyt}{1+iy^{-1}t} \, d\nu(t) 
 = \nu(\{0\}).
$$
The last equality is a consequence of the dominated convergence theorem, bearing in mind that 
$$
 \lim_{y\to0} \frac{1+iyt}{1+iy^{-1}t} = 
 \begin{cases}
 	0, & t\in\R\setminus\{0\}, \\ 1, & t=0,
 \end{cases}
 \qquad 
 \left|\frac{1+iyt}{1+iy^{-1}t}\right|^2 = \frac{1+y^2t^2}{1+y^{-2}t^2} 
 \le 1+y^4,
 \qquad
 y\in\R\setminus\{0\}. 
$$  
\end{proof} 

In the case of a Nevanlinna FR-function $f$, the measure $\mu$ given by Definition~\ref{def:N-mu} coincides with the spectral measure of the subspace related to an operator representation of $f$ as FR-function.

Some of the examples in Section~\ref{ssec:RW-ex} provide operator representations of Nevanlinna functions because there we calculate FR-functions related to stochastic matrices, which in some instances are symmetrizable. We complete this discussion with some examples of the integral representations I and II of Nevanlinna functions.

\subsection{Examples of measures of Nevanlinna functions}
\label{ssec:measure-ex}

To illustrate the previous results, let us uncover the measures of some simple scalar Nevanlinna functions. For this purpose we must identify the measure $\mu$ of the m-function $m$ related to the Nevanlinna function $f$ in question by the integral representation I, i.e.
$$
 \int \frac{d\mu(t)}{t-z} = m(z) = \frac{1}{f(z^{-1})-z}.
$$

\begin{ex} \label{ex:1p} 

Constant Nevanlinna functions.

Every constant Nevanlinna function is real. The m-function associated with a function $f(z)=c\in\R$ is given by
$$
 m(z) = \frac{1}{c-z} = \int \frac{\delta(t-c)\,dt}{t-z}.
$$ 
Therefore the measure of a constant Nevanlinna function $f(z)=c$ is a single Dirac delta $d\mu(t)=\delta(t-c)\,dt$ at the constant $c$. Since $\mu(\R)=1$, constant Nevanlinna functions are FR-functions. 

\end{ex}

\begin{ex} \label{ex:2p}

Nevanlinna polynomials of degree one.

These Nevanlinna functions are those of the form $f(z)=\lambda z+c$ with $\lambda>0$ and $c\in\R$. The corresponding m-function is given by
$$
 m(z) = \frac{z}{\lambda+cz-z^2}.
$$
The polynomial $z^2-cz-\lambda=(z-a)(z-b)$ has roots $a$, $b$ of different sign because $\lambda=-ab>0$. Hence,
$$
 m(z) = \frac{m_a}{a-z} + \frac{m_b}{b-z} 
 = \int \frac{(m_a\delta(t-a)+m_b\delta(t-b))\,dt}{t-z},
 \qquad
 \begin{aligned}
 	& m_a+m_b=1,
	\\
	& m_ab+m_ba=0,
 \end{aligned}
$$
which shows that the measures $d\mu(t)=(m_a\delta(t-a)+m_b\delta(t-b))\,dt$, constituted by two Dirac deltas at the roots of $z^2-cz-\lambda$, are the measures of these Nevanlinna polynomials. As in the previous example, these polynomials are Nevanlinna FR-functions becasue $\mu(\R)=m_a+m_b=1$.

As a particular case, $d\mu(t)=\frac{1}{2}(\delta(t-1)+\delta(t+1))\,dt$ is the measure of $f(z)=z$.

\end{ex}

\begin{ex} \label{ex:1/z}

The Nevanlinna function 
$$
 f(z) = -\lambda z^{-1}, \qquad \lambda>0,
$$
yields the m-function
$$
 m(z) = -\frac{1}{(\lambda+1)z} 
 = \frac{1}{(\lambda+1)} \int \frac{\delta(t)\,dt}{t-z},
$$
thus $d\mu(t)=(\lambda+1)^{-1}\delta(t)\,dt$ is the corresponding measure. In contrast to the previous examples, this Nevanlinna function is not an FR-function for any $\lambda>0$ because $\mu(\R)=(\lambda+1)^{-1}<1$. To obtain the normalized measure $\hat\mu$ of the alternative integral representation II note that $c=1-\lim_{y\to0}-iyf(iy)=1-\lambda$, hence
$$
 \hat{m}(z) = \frac{1}{f(z^{-1})-cz} = -\frac{1}{z} 
 = \int \frac{\delta(t)\,dt}{t-z},
$$
so that $d\hat\mu(t)=\delta(t)\,dt$. This corresponds to the integral representation II given by
$$
 f(z) = (1-\lambda)z^{-1} + \frac{1}{\hat{m}(z^{-1})}.
$$

\end{ex}

\begin{ex} \label{ex:3p}

The Nevanlinna function 
$$
 f(z)=\frac{z}{\lambda-z^2}, \quad \lambda>0,
$$
has an associated m-function with the form
$$
 m(z) = \frac{\lambda z^2-1}{2z-\lambda z^3} = \frac{z^2-a^2/2}{z(a^2-z^2)},
 \qquad
 a=\sqrt{\frac{2}{\lambda}},
$$
which can be expressed as
$$
 m(z) = \frac{1}{4} \left(\frac{1}{a-z}-\frac{1}{a+z}\right) - \frac{1}{2z}
 = \int \frac{
 (\frac{1}{4}\delta(t-a)+\frac{1}{4}\delta(t+a)+\frac{1}{2}\delta(t))\,dt}
 {t-z}. 
$$
This identifies 
$d\mu(t)=(\frac{1}{4}\delta(t-a)+\frac{1}{4}\delta(t+a)+\frac{1}{2}\delta(t))\,dt$ as the measure of $f$, which is an FR-function because $\mu(\R)=1$.

As for the Nevanlinna function $-1/f(z)=z-\lambda z^{-1}$, its m-function  $m(z) = 1/(z^{-1}-(\lambda+1)z)$, rewritten as
$$
 m(z) = \frac{1}{2(\lambda+1)} \left(\frac{1}{a-z}-\frac{1}{a+z}\right)
 = \frac{1}{2(\lambda+1)} \int 
 \frac{(\delta(t-a)+\delta(t-a))\,dt}{t-z},
 \qquad
 a=\frac{1}{\sqrt{\lambda+1}}, 
$$
shows that the corresponding measure is
$d\mu(t) = \frac{1}{2(\lambda+1)} (\delta(t-a)+\delta(t-a))\,dt$. In constrast to $f$, the Nevanlinna function $-1/f$ is not an FR-function because $\mu(\R)=(\lambda+1)^{-1}<1$. The integral representation II follows from the value of $c=1-\lim_{y\to0}-iy(-1/f(iy))=1-\lambda$ and the expression of the m-function
$$
 \hat{m}(z) = \frac{1}{-1/f(z^{-1})-cz} = \frac{1}{z^{-1}-z}
 = \frac{1}{2} \left(\frac{1}{1-z}-\frac{1}{1+z}\right)
 = \frac{1}{2} \int \frac{(\delta(t-1)+\delta(t+1))\,dt}{t-z}.
$$
This identifies $d\hat\mu(t)=\frac{1}{2}(\delta(t-1)+\delta(t+1))\,dt$ as the measure of the alternative integral representation II for $-1/f$,
$$
 -\frac{1}{f(z)} = (1-\lambda)z^{-1} + \frac{1}{\hat{m}(z^{-1})}.
$$

\end{ex}

\begin{ex} \label{ex:sqrt}

As a final example, let us consider a non-rational Nevanlinna function to show how to proceed in such a general case. Let $f(z)=\sqrt{z}$ be given by the branch which is analytical on $\C\setminus(-\infty,0]$, i.e. $f(z)=\sqrt{|z|}\,e^{i\theta/2}$ with $z=|z|\,e^{i\theta}$, $\theta\in(-\pi,\pi)$. The Stieltjes inversion formula \cite{ST,GT} for the corresponding m-function
$$
 m(z) = \frac{1}{\sqrt{z^{-1}}-z} 
 = \frac{|z|^{1/2}\,e^{i\theta/2}}{1-|z|^{3/2}\,e^{i3\theta/2}}
$$
allows us to recover the measure $\mu$ of $f$. The absolutely continuous part $w(t)\,dt$ is given by the weight 
$$
 w(t) = \frac{1}{\pi} \lim_{y\downarrow0} \im m(t+iy)
 = \frac{1}{\pi} \frac{\sqrt{|t|}}{1+|t|^3},
 \qquad t\in(-\infty,0], 
$$
while the singular part is concentrated on the points $t\in\R$ such that $\lim_{y\downarrow0}m(t+iy)=\infty$, with the mass points being those such that $\mu(\{t\})=\lim_{y\downarrow}-iy(m(\{t+iy\})\ne0$. Therefore, the sigular part is in this case a Dirac delta at 1 with mass 2/3, so that the measure of $\sqrt{z}$ is given finally by
$$
 d\mu(t) = w(t) \, dt + \frac{2}{3} \delta(t-1)\,dt.
$$
The Nevanlinna function $\sqrt{z}$ is an FR-function because $\lim_{y\to0}y\sqrt{iy}=0$, which means that $\mu(\R)=1$ as can be easily verified.
 
\end{ex}

It is worth remarking that none of the first two examples --Nevanlinna polynomials of degree not bigger than one-- is characterized by a measure neither by the general representation \eqref{eq:N-mu} nor by the m-function representation \eqref{eq:St-int}. Besides, the correspondence between constant Nevanlinna functions and single Dirac deltas supported on the real line is in perfect analogy with the case of constant Schur functions, whose associated measures are the single Dirac deltas supported on the unit circle.

\section{Schur algorithm on the real line}
\label{sec:SA}

We will see that in the self-adjoint case there exists a natural analogue of the Schur algorithm for the Nevanlinna FR-function of a subspace. Schur algorithms for Nevanlinna functions using interpolation points outside of the real line, including the point at infinity, have been already proposed in the literature, see \cite[Chapter 3]{Akhiezer} and \cite{ADL,ADL2,ADLV,De}. Based on our FR-function approach, we will present here a very simple and natural version of the Schur algorithm for Nevanlinna functions which interpolates them at points on the real line. We will refer to it as the `Schur algorithm on the real line'. As we will see in Sections~\ref{sec:OP} and \ref{sec:RW}, this Schur algorithm turns out to be especially useful for applications to OPRL and RW. 

The Schur algorithm for Schur functions can be seen as a way to extract iteratively the parameters of a related canonical unitary, namely, the Verblunsky coefficients (Schur parameters) of the CMV matrix associated with the related measure on $\T$. CMV matrices are the unitary analogue of Jacobi matrices \cite{CMV,Watkins} (see also \cite[Chapter 4]{Simon-OPUC}), so a Schur algorithm for FR-functions in the self-adjoint case should extract from them the Jacobi parameters of the related measure on $\R$. 

Let $T$ be a self-adjoint operator on a Hilbert space $\HH$. For simplicity, consider first the related Nevanlinna FR-function $f$ of a one-dimensional subspace $\HH_0=\spn\{v\}$. Its spectral measure $\mu$ on $\R$ is given in terms of the spectral decomposition $T=\int\!t\,dE(t)$ by
$$
 \mu = \<v|Ev\>,
$$
while the integral representation of Theorem~\ref{thm:N=FR}.{\it (ii)} reads as
\begin{equation} \label{eq:f-s-sc}
 f(z) = \frac{1}{z} 
 \left( 1-\frac{1}{\displaystyle \int \frac{d\mu(t)}{1-zt}} \right)
 = \frac{\displaystyle\int\frac{t\,d\mu(t)}{1-zt}}
 {\displaystyle\int\frac{d\mu(t)}{1-zt}},
 \qquad z\in\C\setminus\R.
\end{equation}
Since $f$ is completely determined by $\mu$, we obtain the same FR-function for any other operator with a subspace having $\mu$ as a spectral measure. An example of this is the self-adjoint multiplication operator in $L^2_\mu$ given by 
\begin{equation} \label{eq:m-op}
 T_\mu \colon \kern-7pt \mathop{L^2_\mu \xrightarrow{\kern12pt} L^2_\mu}
 \limits_{\textstyle \kern7pt h(t) \mapsto th(t)}
\end{equation} 
which assigns to the function $h(t)=1$ the spectral measure $\mu$. Therefore, $f$ becomes the FR-function of $\spn\{1\}$ with respect to $T_\mu$.

Since $T_\mu$ has simple spectrum, it can be represented as a Jacobi matrix, the canonical object we need to build up a Schur algorithm. Nevertheless, in the case of $\mu$ with unbounded support the Jacobi representation may require a basis excluding the function $h(t)=1$, which can make it difficult to recover the Stieltjes function $s$ of $\mu$ --hence the FR-function $f$-- from the Jacobi matrix. 

To introduce a useful Jacobi matrix in this set up we need an assumption, non-trivial only in the unbounded case: we will suppose that $v$ lies in the domain of every positive power of $T$. Then, the equality $\int t^n d\mu(t)=\<v|T^nv\>$ shows that $\mu$ has finite moments, so one can consider the corresponding OPRL $p_n$. The related three term recurrence relation is encoded by a Jacobi matrix 
\begin{equation} \label{eq:J}
 \mc{J} = 
 \begin{pmatrix}
 b_0 & a_0
 \\
 a_0 & b_1 & a_1
 \\
 & a_1 & b_2 & a_2
 \\
 & & \kern-15pt \ddots & \kern-15pt \ddots & \kern-6pt \ddots
 \end{pmatrix},
 \qquad a_n>0, \qquad b_n\in\R.
\end{equation}
If $\mu$ is finitely supported, there will be a finite number of orthonormal polynomials $p_n$, so the Jacobi matrix $\mc{J}$ will be finite. Otherwise, it will be an infinite Jacobi matrix. 

We will identify the Jacobi matrix $\mc{J}$ with the operator $v \mapsto \mc{J}v$ on $\ell^2$ with maximal domain $\DD(\mc{J})=\{v\in\ell^2:\mc{J}v\in\ell^2\}$. To avoid getting away from self-adjointness, let us suppose that $\mc{J}$ is self-adjoint, again a non-trivial assumption only in the unbounded case. In this situation $\mc{J}$ is unitarily equivalent to $T_\mu$ because then the polynomials are dense in $L^2_\mu$ \cite{Akhiezer}. This unitary equivalence amounts to the identification of $p_n\in L^2_\mu$ with the canonical vector $e_n=(\delta_{k,n})_{k\ge0}\in\ell^2$. As a consequence, the Nevanlinna FR-function $f$ of $\spn\{p_0(t)\}=\spn\{1\}$ with respect to $T_\mu$ coincides with that of $\spn\{e_0\}$ with respect to $\mc{J}$. Therefore we get $f$ by substituting
$$
 T \to \mc{J}, 
 \kern7pt 
 P \to \begin{pmatrix} 1 \\ & 0 \\ & & 0 \\ & & & \ddots \end{pmatrix},
 \kern7pt
 \begin{aligned}
 	& PTP \to b_0,
 	\\
	& PTQ \to 
	\begin{pmatrix} 
		a_0 & \kern-3pt 0 & \kern-3pt 0 & \kern-3pt \cdots
	\end{pmatrix},
	\\
	& QTP \to 
	\begin{pmatrix} 
		a_0 & \kern-3pt 0 & \kern-3pt 0 & \kern-3pt \cdots 
	\end{pmatrix}^T,
 \end{aligned}
 \kern7pt
 QTQ \to \mc{J}^{(1)} = 
 \begin{pmatrix}
 b_1 & a_1
 \\
 a_1 & b_2 & a_2
 \\
 & a_2 & b_3 & a_3
 \\
 & & \kern-15pt \ddots & \kern-15pt \ddots & \kern-6pt \ddots
 \end{pmatrix}, 
$$
in the expression given by Proposition~\ref{prop:falt}.{\it (ii)}. This yields 
\begin{equation} \label{eq:f-S1}
 f(z) = b_0 + a_0^2 zs_1(z),
\end{equation}
where $s_1$ is the Stieltjes function for the orthogonality measure of the Jacobi matrix $\mc{J}^{(1)}$ obtained from $\mc{J}$ by coefficient stripping.   

The renewal equation \eqref{eq:f-s-sc} implies that the Stieltjes function 
$s_1$ is related to the corresponding FR-function $f_1$ (of the subspace spanned by the first canonical vector with respect to $\mc{J}^{(1)}$) by  
$s_1(z)=(1-zf_1(z))^{-1}$. Inserting this into \eqref{eq:f-S1} yields
\begin{equation} \label{eq:f-f1}
 f(z) = b_0  + \frac{a_0^2z}{1-zf_1(z)}.
\end{equation}
Equivalently,
\begin{equation} \label{eq:f1-f}
 f_1(z) = \frac{1}{z} \left(1+\frac{a_0^2z}{b_0-f(z)}\right)
 = \frac{1}{z} \frac{f(z)-b_0-a_0^2z}{f(z)-b_0},
\end{equation}
which is the version of the Schur transformation for FR-functions related to self-adjoint operators. Since we are dealing with a subspace included in the domain of $\mc{J}$, the stripped Jacobi matrix $\mc{J}^{(1)}=Q\mc{J}Q$ is also self-adjoint according to the proof of Theorem~\ref{thm:FR=N}.{\it(ii)}. Thus $f_1$ is also a Nevanlinna function and the above Schur transformation can be iterated.

It is important to notice that the Schur algorithm for Schur functions does not resort to any external object such as a CMV matrix, rather the Schur parameters come from the evaluation at the origin of the Schur functions generated by the algorithm itself. As for the self-adjoint case, being considered now, how do the Jacobi parameters $b_0$, $a_0$ come from the original FR-function $f$?
   
Suppose for the moment that $\mc{J}$ is bounded so that $f$ is analytic around the origin, and the same holds for $f_1$ because $\mc{J}^{(1)}$ is bounded too. Then, the parameters $b_0$ and $a_0$ have a remarkable meaning which follows from \eqref{eq:f-f1},
$$
 f(z) = b_0 + a_0^2z + O(z^2) 
 \kern7pt\Rightarrow\kern7pt
 b_0=f(0), \quad a_0^2=f'(0).
$$
This result remains almost unchanged in the unbounded case, despite the lack of analyticity at the origin: the evaluations at the origin must be substituted by normal limits along the imaginary axis, $z=iy \to 0$. More precisely, as a consequence of Lemma~\ref{lem:der} in Appendix~A, the assumption $v\in\DD(T^2)$ implies that 
\begin{equation} \label{eq:der}
 f(0):=\lim_{y\to0}f(iy)=P\mc{J}P=b_0,
 \qquad
 f'(0):=\lim_{y\to0}\frac{f(iy)-f(0)}{iy}=(Q\mc{J}P)^\dag(Q\mc{J}P)=a_0^2.
\end{equation}
Actually, since we are assuming that $v\in\DD(T^n)$ for every $n\in\N$, Lemma~\ref{lem:der} implies that $f$ has normal derivatives of all orders at the origin. 

Iterating the process, we find from \eqref{eq:f1-f} a `Schur algorithm' for Nevanlinna FR-functions related to self-adjoint operators, which is given by
\begin{equation} \label{eq:Nalg}
\begin{aligned}
 & f_0(z)=f(z),
 \\
 & f_{n+1}(z) = \frac{1}{z} \frac{f_n(z)-b_n-a_n^2z}{f_n(z)-b_n},
 \qquad b_n=f_n(0), \qquad a_n^2=f_n'(0),
 \qquad n\ge0.
\end{aligned}
\end{equation}
As follows from Proposition~\ref{prop:deg-NFR}, the algorithm stops whenever a degenerate Nevanlinna iterate $f_n$ arises since then $f_n(z)=f_n(0)$ for every $z\in\C\setminus\R$, otherwise $f_n(z)-f_n(0)$ does not vanish anywhere on $\C\setminus\R$ and the algorithm continues. The terminating situation corresponds to a finite Jacobi matrix related to a finitely supported measure. 

The above result identifies the `Schur parameters' of this Schur algorithm as the first two Taylor coefficients of the iterates, which also coincide with the coefficients of the corresponding Jacobi matrix. This constitutes the self-adjoint version of Geronimus theorem \cite{Geronimus} (see also \cite[Chapter 3]{Simon-OPUC}) relating Schur parameters and Verblunsky coefficients in the unitary case.

If $b_0,a_0,b_1,a_1,\dots$ are the Schur parameters of $f$, then its $n$-th iterate $f_n$ is again a Nevanlinna FR-function with Schur parameters $b_n,a_n,b_{n+1},a_{n+1},\dots$. This is in perfect analogy with the unitary case. While the unitary Schur algorithm \eqref{eq:Salg} removes from a Schur function the first term of its power expansion around the origin, the self-adjoint version \eqref{eq:Nalg} needs to remove the first and second terms of a similar expansion for Nevanlinna functions. Notice that in both cases we substract from the power expansion of the FR-function information corresponding to two real parameters.  

The Schur algorithm \eqref{eq:Nalg} is easily generalized to the case of a higher finite-dimensional subspace $\HH_0$. A choice of an orthonormal basis in $\HH_0$ allows us to identify its spectral measure as a matrix valued one. This leads to matrix valued OPRL \cite{DPS} when $\HH_0\subset\DD(T^n)$ for every $n\ge1$. The corresponding recurrence relation leads to a block Jacobi matrix which takes the place of \eqref{eq:J}. It has the same form as \eqref{eq:J}, but now the coefficients are square matrices $b_n=b_n^\dag$ and $a_n>0$ of size $\dim\HH_0$. Assuming $\mc{J}$ self-adjoint --as an operator on $\ell^2$ with maximal domain--, similar arguments to those of the one-dimensional case show that \eqref{eq:f-S1} and \eqref{eq:f-f1} read now as follows
$$
 f(z) = b_0 + za_0s_1(z)a_0 = b_0 + za_0(1_0-zf_1(z))^{-1}a_0,
 \quad 1_0=\text{identity matrix of size } \dim\HH_0.
$$
Here $s_1$ is the Stieltjes function of the matrix measure associated with the block Jacobi matrix $\mc{J}^{(1)}$ with blocks $b_n$, $a_n$, $n\ge1$. Also, $f_1$ is the FR-function with respect to $\mc{J}^{(1)}$ for the subspace of $\ell^2$ on which the first block column of $\mc{J}^{(1)}$ acts. We finally find the following form for the Schur algorithm in the case of matrix valued Nevanlinna FR-functions associated with self-adjoint operators,
\begin{equation} \label{eq:Nalg-op}
\begin{aligned}
 & f_0(z)=f(z),
 \\
 & 
 \begin{aligned}
 	f_{n+1}(z) 
 	& = z^{-1} 
	a_n^{-1} (f_n(z)-b_n-a_n^2z) (f_n(z)-b_n)^{-1} a_n,
	& \kern10pt & b_n=f_n(0),
	\\
	& = z^{-1} 
	a_n (f_n(z)-b_n)^{-1} (f_n(z)-b_n-a_n^2z) a_n^{-1},
	& & a_n^2=f_n'(0),
 \end{aligned}
 \kern20pt n\ge0.
\end{aligned}
\end{equation}
In the unbounded case $f_n(0)$ and $f_n'(0)$ must be understood as normal limits along the imaginary axis analogous to \eqref{eq:der}, as follows from Lemma~\ref{lem:der} in Appendix~A. 

The above algorithm does not stop unless an iterate $f_n$ is degenerate because Proposition~\ref{prop:deg-NFR} shows that non-degeneracy is equivalent to the invertibility of $f_n(z)-f_n(0)$, a property which holds simultaneously for every $z\in\C\setminus\R$.  

As in the one-dimensional case, the $n$-th iterate $f_n$ is again a Nevanlinna FR-function whose matrix Schur parameters $b_n,a_n,b_{n+1},a_{n+1},\dots$ are obtained by stripping those of $f$.

When compared with the Schur algorithm \eqref{eq:Nalg} for scalar Nevanlinna functions, the only additional ingredients of the matrix valued version \eqref{eq:Nalg-op} are the left and right factors $a_n^{\pm1}$ which obviously cancel in the scalar case. This in in perfect analogy with the Schur algorithm for Schur functions, whose matrix valued version involves the matrix M\"obius transformations \eqref{eq:Mobius-op}. The factors $a_n^{\pm1}$ are the Nevanlinna analogues of the left and right factors of the type $\rho_\alpha^{\pm1}$ appearing in such matrix M\"obius transformations, which cancel in the the scalar case leading to the Schur algorithm defined by the scalar M\"obius transformations \eqref{eq:Mobius}.   

\medskip   

The FR-functions introduced in this paper are not the only Nevanlinna functions that can be considered in the self-adjoint case. For instance, the m-function $ m$ is also a Nevanlinna function. Indeed, m-functions for measures on $\R$ are sometimes considered as an analogue of Schur functions for measures on $\T$. That is, m-functions fit with the requirement (B) in the introduction concerning the similarity with the properties of Schur functions. Regarding the requirement (C) about the existence of a Schur algorithm, using the relation \eqref{eq:f-s-sa} between Stieltjes functions and FR-functions it can be easily seen that the Schur algorithm \eqref{eq:Nalg} is equivalent to 
\begin{equation} \label{eq:Stalg}
\begin{aligned}
 &  m_0(z)= m(z),
 \\
 &  m_{n+1}(z) = 
 \frac{1}{a_n^2} \left(b_n-z-\frac{1}{ m_n(z)}\right),
 \qquad
 \begin{aligned}
 	& b_n=\lim_{y\to\infty}(iy+1/m_n(iy)), 
	\\
 	& a_n^2=\lim_{y\to\infty}iy(iy+1/m_n(iy)-b_n),
 \end{aligned}
 \qquad n\ge0.
\end{aligned}
\end{equation}
This algorithm requires the existence of the limits giving $b_n$ and $a_n^2$ at every step, a condition which is not guaranteed for the m-function $m$ of an arbitray normalized measure on $\R$, as the examples in the next subsection show. When such conditions hold, the algorithm yields a terminating sequence of m-functions if $a_n=0$ for some $n$, and an infinite sequence otherwise.

The algorithm \eqref{eq:Stalg} is considered sometimes as the Schur algorithm for Nevanlinna functions of Stieltjes form, see \cite[Chapter 3]{Akhiezer} and \cite{De2,ADL2}. However, \eqref{eq:Nalg} seems a more natural version of the standard Schur algorithm \eqref{eq:Salg}. Besides, \eqref{eq:Nalg} provides a transparent meaning for the `Schur parameters' of Nevanlinna functions --quite close to that of the Schur parameters for Schur functions-- which is obscured in \eqref{eq:Stalg}. Rewriting \eqref{eq:Nalg} as \eqref{eq:Stalg} is as unnatural as rewriting the Schur algorithm for Schur functions in terms of Stieltjes or Carath\'eodory functions.

The Schur algorithms \eqref{eq:Nalg} and \eqref{eq:Stalg} are equivalent once the relation between FR-functions and m-functions is taken into account. Nevertheless, if we think of them simply as Schur algorithms for Nevanlinna functions, they have a different range of applicability. The usual Schur algorithm \eqref{eq:Stalg} applies to Nevanlinna functions with an asymptotic behaviour at infinity controlled by \eqref{eq:St-norm} and such that the limits involved in \eqref{eq:Stalg} exist. Indeed, this convergence requirement is needed also for each iterate, which amounts to asking for the analyticity at infinity of the Nevanlinna function restricted to the imaginary axis \cite[Chapter 3]{Akhiezer}. 

On the other hand, the Schur algorithm \eqref{eq:Nalg} requires of the Nevanlinna function and its iterates certain regularity at the origin, where they must have a well defined normal limit and derivative. However, as we will see below, the applicability of \eqref{eq:Nalg} at every step can be easily determined from the beginning, because this new algorithm works as long as the original Nevanlinna function has normal derivatives of all orders at the origin. Therefore, \eqref{eq:Nalg} applies for instance to every Nevanlinna function which is analytic at the origin. All these comments remain true even in the matrix valued case, as the next theorem shows.

For convenience, we introduce the following definition.
\begin{defn}
We denote by $\frak{N}_r(x)$ the set of {\bf matrix valued Nevanlinna functions $f$ with $r$-th normal derivative at $x\in\R$}, which means that the limits below do exist
$$
\begin{aligned}
 & f^{(0)}(x)=f(x) := \lim_{y\to0} f(x+iy),
 \\
 & f^{(k)}(x) := \lim_{y\to0} \frac{f^{(k)}(x+iy)-f^{(k)}(x)}{iy}, 
 \qquad k=1,2,\dots,r.
\end{aligned}
$$
\end{defn}

With this notation, we can state the alluded result about the Schur algorithm \eqref{eq:Nalg} and its matrix version \eqref{eq:Nalg-op}. 

\begin{thm}
The transformation 
$$
\begin{aligned}
 f(z) \; \mapsto \;
 g(z) & = z^{-1} f'(0)^{-1/2} (f(z)-f(0)-f'(0)z) (f(z)-f(0))^{-1} f'(0)^{1/2},
 \\
 & = z^{-1} f'(0)^{1/2} (f(z)-f(0))^{-1} (f(z)-f(0)-f'(0)z) f'(0)^{-1/2},
 \\
 & f'(0)^{1/2} = \text{non-negative definite square root of } f'(0),
\end{aligned}
$$
satisfies the following properties:
\begin{itemize}
\item[{\it (i)}] $f\in\frak{N}_1(0) \text{ non-degenerate} 
\;\Rightarrow\; g$ is a Nevanlinna function.
\item[{\it (ii)}] $f\in\frak{N}_r(0) \text{ non-degenerate} 
\;\Rightarrow\; g\in\frak{N}_{r-2}(0), \quad r\ge2.$
\end{itemize}
As a consequence, the scalar and matrix Schur algorithms \eqref{eq:Nalg} and \eqref{eq:Nalg-op} are valid for every $f\in\frak{N}_\infty(0)$, and in this case generate finitely or infinitely many Nevanlinna iterates $f_n$ depending whether a degenerate iterate arises or not. 
\end{thm}

\begin{proof}
\quad  

\noindent{\it (i)} 
The transformation is well defined for every non-degenerate $f\in\frak{N}_1(0)$ because, by Proposition~\ref{prop:deg-NFR}, $f(z)-f(0)$ is invertible for every $z\in\C\setminus\R$, while $\displaystyle f'(0)$ is not only invertible but also positive definite due to Proposition~\ref{prop:der-p}. In this case, the equality between the two alternative expressions of $g$ follows from the fact that both of them can be rewritten as
$$
 g(z) = z^{-1}1_0 - f'(0)^{1/2} (f(z)-f(0))^{-1} f'(0)^{1/2}.
$$
Therefore, 
$$
 g(z)= z^{-1}1_0 +  m(z^{-1})^{-1},
 \qquad
  m(z) = f'(0)^{-1/2} (f(0)-f(z^{-1})) f'(0)^{-1/2}.
$$
Since $-f(z^{-1})$ is a Nevanlinna function and, from Proposition~\ref{prop:der-p}, $f(0)$ and $\displaystyle f'(0)$ are self-adjoint, we conclude that $ m$ is also a Nevanlinna function.
Besides, the limit
$$
 \lim_{y\to\infty} -iy m(iy) 
 = \lim_{y\to\infty} 
 f'(0)^{-1/2} \frac{f(-iy^{-1})-f(0)}{-iy^{-1}} f'(0)^{-1/2} 
 = 1_0,
$$
proves that $ m$ satisfies the condition \eqref{eq:St-char} which identifies it as an m-function, being the related measure $\mu$ such that $\mu(\R)=\lim_{y\to\infty}-iy m(iy)=1_0$. Therefore, Theorem~\ref{thm:N=FR} implies that $g$ is a Nevanlinna function.

\noindent{\it (ii)}
Let us write $\displaystyle g(z) = f'(0)^{-1/2} h_1(z) h_0(z)^{-1} f'(0)^{1/2}$ with
$$
 h_0(z)=\frac{f(z)-f(0)}{z},
 \qquad h_1(z)=\frac{\displaystyle f(z)-f(0)-f'(0)z}{z^2}, 
$$
which are analytic on $\C\setminus\R$. Actually, for every non-degenerate $f\in\frak{N}_2(0)$, the functions $h_0$ and $h_1$ have normal limits at the origin given by
$$
 h_0(0)=f'(0), \qquad h_1(0)=\frac{1}{2}f''(0).
$$
Since $h_0(0)$ is invertible, $g\in\frak{N}_{r-2}(0)$ whenever the normal derivatives $h_0^{(k)}(0)$ and $h_1^{(k)}(0)$ exist for $k=1,2,\dots,r-2$. This follows from the condition $f\in\frak{N}_r(0)$ because
$$
 h_0^{(k)}(0) = \frac{1}{k+1} f^{(k+1)}(0),
 \qquad
 h_1^{(k)}(0) = \frac{1}{(k+1)(k+2)} f^{(k+2)}(0).
$$
\end{proof}

There exists an additional advantage of the Schur algorithm \eqref{eq:Nalg}. Using a real shift $f(z) \to f(z+x)$, $x\in\R$, we see that \eqref{eq:Nalg} can be generalized to a Schur algorithm with a wider applicability, which requires the existence of normal derivatives at an arbitrarily chosen point $x\in\R$. Such an algorithm amounts to substituting the origin by $x\in\R$ in \eqref{eq:Nalg} and \eqref{eq:Nalg-op} for the scalar and matrix valued case respectively.

\begin{defn} \label{def:SAR}
The {\bf Schur algorithm on the real line} is defined for a given $x\in\R$ and every Nevanlinna function $f\in\frak{N}_\infty(x)$ by  
\begin{equation} \label{eq:Nalg-x}
\begin{aligned}
 & f_0(z)=f(z),
 \\
 & f_{n+1}(z) = \frac{1}{z-x} \frac{f_n(z)-f_n(x)-f_n'(x)(z-x)}{f_n(z)-f_n(x)},
 \qquad n\ge0,
\end{aligned}
\end{equation}
in the scalar case, while its matrix valued extension reads as follows
$$
\begin{aligned}
 & f_0(z)=f(z),
 \\
 & 
 \begin{aligned}
 	f_{n+1}(z) 
 	& = \frac{1}{z-x} \, f_n'(x)^{-1/2} (f_n(z)-f_n(x)-f_n'(x)(z-x)) 
	(f_n(z)-f_n(x))^{-1} f_n'(x)^{1/2},
	\\
	& = \frac{1}{z-x} \, f_n'(x)^{1/2} (f_n(z)-f_n(x))^{-1} 
	(f_n(z)-f_n(x)-f_n'(x)(z-x)) f_n'(x)^{-1/2},
 \end{aligned}
 \kern15pt n\ge0.
\end{aligned}
$$
\end{defn}

Since this Schur algorithm works for every $f\in\frak{N}_\infty(x)$, it is valid in particular for all the Nevanlinna functions with $x\in\R$ as an analyticity point. This is for instance the case of every FR-function with respect to a bounded self-adjoint operator $T$ when $|x|<\|T\|^{-1}$.

This Schur algorithm for Nevanlinna functions opens up many other questions in parallel with the theory of Schur functions, such as continued fraction expansions and Wall's polynomials for Nevanlinna functions, determinacy issues for the parametrization of Nevanlinna functions by Schur parameters or the relation of this algorithm with interpolation and moment problems on the real line. For instance, in contrast to the Schur algorithm interpolating at the point of infinity, the Schur algorithm on the real line offers the possibility of using a different interpolation point at each step --very much in line with the Nevanlinna-Pick algorithm--, a fact which links to interpolation problems on the real line.

Although these topics will be discussed elsewhere, for convenience we will briefly comment here on the uniqueness of the Nevanlinna function related to given Schur parameters. First, suppose for simplicity that the sequences of Schur parameters $b_n$, $a_n$ derived from the Schur algorithm \eqref{eq:Nalg} are bounded. There is a one-to-one correspondence between such bounded sequences, the bounded Jacobi matrices and the normalized measures with a bounded support on the real line \cite{Akhiezer}. Bearing in mind the correspondence between measures on the real line and Nevanlinna functions established in Theorem~\ref{thm:N=FR} and Corollary~\ref{cor:N=FR} we conclude that any sequence of bounded Schur parameters determines a unique Nevanlinna function --which is indeed an FR-function-- via the Schur algorithm \eqref{eq:Nalg}. This result generalizes obviously to the matrix Schur algorithm \eqref{eq:Nalg-op}, while the use of a real shift shows that it also holds for the corresponding shifted versions of the algorithm. In general, the uniqueness of the Nevanlinna function for not necessarily bounded sequences of Schur parameters $b_n$, $a_n$ holds as long as the corresponding Jacobi matrix is associated with a unique normalized measure on the real line, i.e. when the corresponding moment problem is determinate \cite{Akhiezer}. This situation is characterized by the fact that the Jacobi matrix, considered as an operator on $\ell^2$ with maximal domain, is self-adjoint.

\subsection{Examples of the Schur algorithm on the real line}
\label{ssec:SA-ex}

Let us briefly compare the above Schur algorithm on the real line with the standard one for m-functions using a few simple examples.

\begin{ex} \label{ex:SA1}

The Nevanlinna function $z$ goes well with the Schur algorithm \eqref{eq:Nalg} on the real line but not the standard one \eqref{eq:Stalg} because it is not an m-function according to the characterization \eqref{eq:St-char}. The opposite holds for the Nevanlinna function $-z^{-1}$, which has no normal limit at the origin. This leads to the following terminating sequences of iterates, 
$$
 f(z)=z \;\to\; f_1(z)=0,
 \qquad\quad
  m(z)=-z^{-1} \;\to\;  m_1(z)=0.
$$
However, $-z^{-1}$ is analytic outside of the origin, thus it goes well with the shifted version \eqref{eq:Nalg-x} of the new algorithm for any $x\ne0$, giving also a terminating sequence of iterates,
$$
 f(z)=-z^{-1} \;\to\; f_1(z)=1-x^{-2}.
$$

\end{ex}

\begin{ex} \label{ex:SA2}

The Nevanlinna function
$$
 \frac{z}{\lambda-z^2}, \qquad \lambda>0,
$$
goes well with both, the standard Schur algorithm \eqref{eq:Stalg} for m-functions and the Schur algorithm \eqref{eq:Nalg} on the real line. The corresponding sequences of iterates are given by
$$
\begin{aligned}
 & m(z)=\frac{z}{\lambda-z^2} \;\to\; m_1(z)=-z^{-1} 
 \;\to\; m_2(z)=0,
 \\[2pt]
 & f(z)=\frac{z}{\lambda-z^2} \;\to\; f_1(z)=\frac{z}{\lambda} 
 \;\to\; f_2(z)=0.
\end{aligned}
$$

However, none of these algorithms work for the Nevanlinna function 
$$
 \frac{z^2-\lambda}{z}, \qquad \lambda>0,
$$
which instead goes well with the shifted Schur algorithm \eqref{eq:Nalg-x} for any $x\ne0$, leading to the sequence of iterates
$$
 f(z)=\frac{z^2-\lambda}{z} 
 \;\to\; f_1(z)=-\frac{\lambda}{x^2z+x\lambda} 
 \;\to\; f_2(z)=-\frac{x}{x^2+\lambda}.
$$

\end{ex}

\begin{ex} \label{ex:SA3}

Consider for $n\in\N$ the Nevanlinna m-functions of the normalized measures $nt^{-(n+1)}dt$ on $[1,\infty)$, given explicitly by
$$
\begin{aligned}
 f(n;z) = \int_1^\infty \frac{n}{t^{n+1}} \frac{dt}{t-z}
 & = \frac{n}{z^{n+1}} \int_1^\infty
 \left(\frac{1}{t-z}-\frac{1}{t}-\frac{z}{t^2}-\frac{z^2}{t^3}-\frac{z^3}{t^4}-
 \cdots-\frac{z^n}{t^{n+1}}\right) dt
 \\
 & = - \frac{n}{z^{n+1}} 
 \left(\ln(1-z)+z+\frac{z^2}{2}+\frac{z^3}{3}+\cdots+\frac{z^n}{n}\right).
\end{aligned}
$$
Applying the standard Schur algorithm \eqref{eq:Stalg} to $f(n;z)$ we obtain the following results for the first values of $n$:
\medskip
\begin{center}
\renewcommand{\arraystretch}{1.5}
\begin{tabular}{|c|l|}
\hline
	$n=1$ & $b_0=\infty$
\\ 	\hline 
	$n=2$ & $b_0=2, \quad a_0=\infty$
\\ 	\hline 
	$n=3$ & $b_0=3/2, \quad a_0=\sqrt{3}/2, \quad b_1=\infty$
\\ 	\hline
	$n=4$ & $b_0=4/3, \quad a_0=\sqrt{2}/3, \quad b_1=14/3, \quad a_1=\infty$
\\ 	\hline
\end{tabular}
\end{center}
\medskip
In general, $f(2k+1;z)$ has well defined values of $b_j$ and $a_j$ for $j=0,1,\dots,k-1$ but $b_k=\infty$, while $f(2k+2;z)$ has also a well defined value of $b_k$ but $a_k=\infty$. Therefore, although $f(n;z)$ is a Nevanlinna m-function, the standard Schur algorithm \eqref{eq:Stalg} does not apply to it for any value of $n$. The reason for this is the lack of analiticity of $f(n;iy)$ at infinity, which is related to the fact that only the first $n$ moments of the related measure exist.  

On the contrary, since $f(n;z)$ is analytic for $z\in\C\setminus[1,\infty)$, the shifted Schur algorithm \eqref{eq:Nalg-x} on the real line is applicable for any $x<1$. For instance, using the power expansion $f(n;z)=n\sum_{k\ge0}z^k/(n+k+1)$ it is possible to see that the choice $x=0$ yields the following Schur parameters for $f(n;z)$,
$$
\begin{aligned}
 & b_0 = \frac{n}{n+1}, 
 & \quad &  a_0 = \frac{n}{n+2},
 \\[5pt] 
 & b_k = \frac{n^2+(2k+1)n+2k^2}{(n+2k)^2-1},
 & & a_k = \frac{k^2(n+k+1)^2}{(n+2k+1)^2((n+2k+1)^2-1)},
 \qquad k\ge1.
\end{aligned}
$$
Since $f(n;z)$ is analytic at the origin, it is an FR-function with an operator representation given by the Jacobi matrix built out of the above Schur parameters. The boundedness of these parameters with respect to $k$ implies the boundedness of the support of the corresponding orthogonality measure, which is not the original one supported on $[1,\infty)$ and defining $f(n;z)$ as an m-function, but that one attached to the Nevanlinna function $f(n;z)$ by Definition~\ref{def:N-mu}.

\end{ex}

\begin{ex} \label{ex:SA4}

In this example we will proceed in reverse order, obtaining the Nevanlinna function with a prescribed sequence of Schur parameters for the Schur algorithm \eqref{eq:Nalg}. In particular, we will determine the Nevanlinna function with constant Schur parameters $b_n=b$, $a_n=a$ for any $a,b\in\R$. Such a constant sequence is obviouly bounded, thus it determines a unique Nevanlinna function $f$. Since the first iterate $f_1$ has the same Schur parameters as $f$ we conclude that $f_1=f$, so that
$$
 f(z) = \frac{1}{z} \frac{f(z)-b-a^2z}{f(z)-b}.
$$
Solving this equation yields
\begin{equation} \label{eq:f-const}
 f(z) = \frac{1+bz-\sqrt{(1-bz)^2-4a^2z^2}}{2z},
\end{equation}
where the branch for the square root is that one which takes the value 1 at $z=0$ because $f$ must have a well defined normal limit $b$ at the origin. 

\end{ex} 

Coming back to the role of Nevanlinna m-functions as the self-adjoint version of Schur functions, as we pointed out, m-functions of measures on $\R$ have been already considered as an analogue of Schur functions of measures on $\T$. At this point it is worth mentioning the discussion in B.~Simon's monograph \cite[Appendix B.2]{Simon-OPUC} about the variety of OPUC analogs for the m-function of OPRL, which includes Schur functions among them. We go in the opposite direction, trying find the right OPRL analogue of Schur functions for OPUC. The results in the present paper intend to show that this OPRL analogue, although close to m-functions --the renewal equation links both of them-- is given instead by what we call FR-functions. 

Regarding the requirements stated in the Introduction for the generalization of Schur functions, even if one is pleased with the fact that m-functions satisfy the requirements (B) and (C), they do not meet requirement (A), unlike FR-functions. As we pointed out, in the bounded case the Stieltjes function represents a return generating function, while the corresponding first return generating function is an FR-function which becomes a Schur/Nevanlinna function in the unitary/self-adjoint case. This interpretation is key in understanding the interest of FR-functions for the study of RW and QW, as the applications discussed in Sections~\ref{sec:RW}, \ref{sec:QW} and \ref{sec:OQW} will show. The FR-functions for both, QW and RW, turn out to be Schur functions, but with values in operators on Hilbert and Banach spaces respectively. Besides, for a large class of RW --those which are irreducible and reversible-- the corresponding evolution operator is self-adjoint with respect to a suitable inner product \cite[Chapter 6]{Stroock} and hence the first return probabilities generate an FR-function which is simultaneously a Schur and a Nevanlinna function. None of these results hold for m-functions.

Above and beyond the previous remarks, the best argument in support for the FR-functions as the right generalization of Schur functions beyond unitarity will be given in the next section. There is a clear asymmetry in the fact that Schur functions are known to satisfy very useful factorization properties, but no Nevanlinna analogue exists so far. This asymmetry reappears in other contexts, for instance, when noticing that Khrushchev factorization formula for Schur functions related to OPUC has no OPRL version. This is not merely an aesthetic question since Schur factorizations have been key in the invariant subspace problem of operator theory \cite{NF,Brodskii} and realizabilty in linear system theory \cite{BGK,Kaashoek}, while Khrushchev formula is the origin of a recent revolution in OPUC theory with deep impact and far-reaching consequences \cite{Khrushchev,Khrushchev2,Simon-OPUC}, even beyond OPUC theory --for instance with applications in QW \cite{GVWW,BGVW}. 

We will see that splitting formulas similar to the alluded factorizations of Schur functions appear also for Nevanlinna FR-functions associated with self-adjoint operators. Actually, we will find that such kind of splitting formulas are a general feature of FR-functions for arbitrary operators. This opens the possibility of applying these splittings, not only to OPRL theory, but also to other areas like for instance the study of RW and open QW. The suitability of FR-functions --under the more restricted notion of transfer/characteristic functions-- for factorizations formulas was already recognized since the 50's in works on system theory, dilation theory and related areas, where they have been profusely exploited in the context of unitary colligations \cite{NF,Brodskii2}. A much more recent result very much influenced by the connection with QW --and indeed at the root of this paper-- has shown that Khrushchev factorization formula for OPUC can be generalized to arbitary unitary operators \cite{CGVWW}. In the next section we will see that FR-functions are not only suitable for factorizations, but also for other kinds of splittings which mimic similar splittings of the underlying operator. Moreover, these splittings rules require no symmetry at all for the related operator, which makes them a powerful tool for developing splitting techniques in different applications of FR-functions, see Sections~\ref{sec:OP}, \ref{sec:RW}, \ref{sec:QW} and \ref{sec:OQW}.

\section{Splitting rules for FR-functions}
\label{sec:KF}

The objective of this section is to prove that the FR-functions related to arbitrary bounded projections and arbitrary closed operators on Banach spaces satisfy a couple of splitting rules with respect to the sum and product of operators. Natural extensions of Khrushchev factorization formulas for unitaries found in \cite{CGVWW}, these splitting rules --given in Theorem~\ref{thm:split}-- are central for the applications discussed in Sections~\ref{sec:OP}, \ref{sec:RW}, \ref{sec:QW} and \ref{sec:OQW}. 

An essential ingredient of these rules is the assumption of an overlapping splitting of the underlying operator $T \colon \DD(T) \to \BB$ on a Banach space $\BB$, i.e. a splitting into operators $T_L \colon \DD(T_L) \to \BB_L$ and $T_R \colon \DD(T_R) \to \BB_R$ on subspaces related to a decomposition $\BB=\BB_-\oplus\BB_0\oplus\BB_+$ with $\BB_L=\BB_-\oplus\BB_0$ and $\BB_R=\BB_0\oplus\BB_+$. The splitting rules in question state that such an overlapping splitting of an operator induces a similar splitting on the FR-function of the projection of $\BB$ onto the overlapping subspace $\BB_0=\BB_L\cap\BB_R$ along the complementary one $\BB_-\oplus\BB_+$. We will deal with two situations corresponding to the following operator overlapping splittings
$$
 T = 
 \left(
 \begin{array}{c|c} 
 	\\[-5pt]
	\kern7pt T_L \kern9pt 
	\\[7pt] \hline 
	& 0_+ \kern-2pt
 \end{array}
 \right)
 +
 \left(
 \begin{array}{c|c} 
 	0_- \kern-2pt 
 	\\ \hline 
	\\[-5pt]
	& \kern9pt T_R \kern7pt
	\\ [5pt]
 \end{array}
 \right),
 \qquad
 T = 
 \left(
 \begin{array}{c|c} 
 	\\[-5pt]
	\kern7pt T_L \kern9pt 
	\\[7pt] \hline 
	& 1_+ \kern-2pt
 \end{array}
 \right)
 \left(
 \begin{array}{c|c} 
 	1_- \kern-2pt 
 	\\ \hline 
	\\[-5pt]
	& \kern9pt T_R \kern7pt
	\\ [5pt]
 \end{array}
 \right),
$$
where $0_\pm$ and $1_\pm$ stand for the null and identity operators on $\HH_\pm$ respectively. We will refer to $\HH_{L,R}$ and $T_{L,R}$ as the `left/right' subspaces and operators respectively. 

\begin{thm}[\bf splitting rules for FR-functions] \label{thm:split}
Let $\BB=\BB_-\oplus\BB_0\oplus\BB_+$ be a decomposition of a Banach space into a direct sum of closed subspaces and $T_L$, $T_R$ operators on $\BB_L=\BB_-\oplus\BB_0$, $\BB_R=\BB_0\oplus\BB_+$ respectively. Denote by $0_k$ and $1_k$ the null and identity operators on $\BB_k$ for any subindex $k$, while
$$
\begin{aligned}
 & P = \text{projection of $\BB$ onto $\BB_0$ along $\BB_-\oplus\BB_+$},
 & \quad & Q=1-P, 
 \\
 & P_L = \text{projection of $\BB_L$ onto $\BB_0$ along $\BB_-$},
 & & Q_L=1_L-P_L
 \\
 & P_R = \text{projection of $\BB_R$ onto $\BB_0$ along $\BB_+$},
 \qquad
 & & Q_R=1_R-P_R.
\end{aligned} 
$$
Then, if $\BB_0\subset\DD(T_{L,R})$ and $f_{L,R}$ is the FR-function of $P_{L,R}$ with respect to $T_{L,R}$, we have the following splitting rules:
\begin{itemize}
\vskip7pt
\item[{\it (i)}] {\bf Decomposition:}
$T = (T_L \oplus 0_+) + (0_- \oplus T_R)$ has domain $\DD(T)=\DD(T_L)+\DD(T_R)$ and the FR-function of $P$ with respect to $T$ is given by
$$
 f(z)=f_L(z)+f_R(z), \qquad z^{-1}\in\varpi(Q_LT_LQ_L)\cap\varpi(Q_RT_RQ_R).
$$
\vskip7pt
\item[{\it (ii)}] {\bf Factorization:}
$T = (T_L \oplus 1_+)(1_- \oplus T_R)$ has domain $\DD(T)=\DD(T_L)+\DD(T_R)$ and the FR-function of $P$ with respect to $T$ is given by
$$
 f(z)=f_L(z)f_R(z), \qquad z^{-1}\in\varpi(Q_LT_LQ_L)\cap\varpi(Q_RT_RQ_R).
$$
\end{itemize} 
\end{thm}

\begin{proof}
\quad

\noindent{\it (i)}
For convenience, let us write the overlapping splitting of $T$ as
$$
 T=\hat{T}_L+\hat{T}_R,
 \qquad 
 \hat{T}_L = T_L \oplus 0_+, 
 \qquad 
 \hat{T}_R = 0_- \oplus T_R,
$$
using the enlarged left/right operators $\hat{T}_{L,R}$. The condition $\BB_0\subset\DD(T_{L,R})$ implies that 
$$
 \DD(T_L)=(\DD(T_L)\cap\BB_-)\oplus\BB_0,
 \qquad
 \DD(T_R)=\BB_0\oplus(\DD(T_R)\cap\BB_+),
$$ 
hence
\begin{equation} \label{eq:DTLR}
 \DD(T)=\DD(\hat{T}_L)\cap\DD(\hat{T}_R)
 =(\DD(T_L)\cap\BB_-)\oplus\BB_0\oplus(\DD(T_R)\cap\BB_+)
 =\DD(T_L)+\DD(T_R).
\end{equation}

The condition $\BB_0\subset\DD(T_{L,R})$ ensures that $f_{L,R}$ admit expressions like the one in Proposition~\ref{prop:falt}.{\it(ii)}, i.e.
\begin{equation} \label{eq:fLfR0}
\begin{aligned}
 & f_L(z) = P_LT_LP_L + zP_LT_LQ_L(Q_L-zQ_LT_LQ_L)^{-1}Q_LT_LP_L, 
 \qquad z^{-1}\in\varpi(Q_LT_LQ_L),
 \\
 & f_R(z) = P_RT_RP_R + zP_RT_RQ_R(Q_R-zQ_RT_RQ_R)^{-1}Q_RT_RP_R, 
 \qquad z^{-1}\in\varpi(Q_RT_RQ_R).
\end{aligned}
\end{equation}
The FR-functions $f_{L,R}$ can be rewritten in terms of the enlarged operators as
\begin{equation} \label{eq:fLfR}
\begin{aligned}
 & f_L(z) 
 = P\hat{T}_LP + zP\hat{T}_LP_-(P_--zP_-\hat{T}_LP_-)^{-1}P_-\hat{T}_LP, 
 \qquad z^{-1}\in\varpi(P_-\hat{T}_LP_-),
 \\
 & f_R(z) 
 = P\hat{T}_RP + zP\hat{T}_RP_+(P_+-zP_+\hat{T}_RP_+)^{-1}P_+\hat{T}_RP, 
 \qquad z^{-1}\in\varpi(P_+\hat{T}_RP_+),
\end{aligned}
\end{equation} 
with $P_\pm$ the projection of $\BB$ onto $\BB_\pm$ along $\BB_{L,R}$. 

The above results imply that $\BB_0\subset\DD(T)$, which guarantees that the expression in Proposition~\ref{prop:falt}.{\it(ii)} is also valid for the FR-function $f$, i.e.
\begin{equation} \label{eq:fQ}
 f(z) = PTP + zPTQ(Q-zQTQ)^{-1}QTP, 
 \qquad
 Q=P_-+P_+.
\end{equation}
Besides, from \eqref{eq:DTLR} we know that $\DD(T)\cap(\BB_-\oplus\BB_+)=(\DD(T)\cap\BB_-)\oplus(\DD(T)\cap\BB_+)$, which implies that $\DD(Q-zQTQ)=(\DD(Q-zQTQ)\cap\BB_-)\oplus(\DD(Q-zQTQ)\cap\BB_+)$. This means that $Q-zQTQ$ has a block representation with respect to the direct sum $\BB_-\oplus\BB_+$, which is given by
$$
 Q-zQTQ =
 \begin{pmatrix} 
 	P_--zP_-\hat{T}_LP_- & 0
	\\
	0 & P_+-zP_+\hat{T}_RP_+ 
 \end{pmatrix},
$$
where we have taken into account that 
\begin{equation} \label{eq:rel-enl}
 P_+\hat{T}_L=\hat{T}_LP_+=P_-\hat{T}_R=\hat{T}_RP_-=0.
\end{equation}
Therefore, the domain of definition of $f(z)$ is 
$$
 z^{-1}\in\varpi(QTQ)
 =\varpi(P_-\hat{T}_LP_-)\cap\varpi(P_+\hat{T}_RP_+)
 =\varpi(Q_LT_LQ_L)\cap\varpi(Q_RT_RQ_R),
$$ 
and in this domain we can rewrite $f$ in terms of block representations as follows
$$
 f(z) = P\hat{T}_LP + P\hat{T}_RP
 + z\begin{pmatrix} P\hat{T}_LP_- & \kern-5pt P\hat{T_R}P_+ \end{pmatrix}
 \begin{pmatrix} 
 	(P_--zP_-\hat{T}_LP_-)^{-1} & \kern-12pt 0
	\\
	0 & \kern-12pt (P_+-zP_+\hat{T}_RP_+)^{-1} 
 \end{pmatrix}
 \begin{pmatrix} P_-\hat{T}_LP \\ P_+\hat{T}_RP \end{pmatrix},
$$
where we have used \eqref{eq:rel-enl} again. The decomposition $f(z)=f_L(z)+f_R(z)$ follows from the above block representation of $f$ and the expression of $f_{L,R}$ given in \eqref{eq:fLfR}.

\noindent{\it (ii)}
The splitting of $T$ is now an overlapping factorization 
$$
 T=\hat{T}_L\hat{T}_R,
 \qquad
 \hat{T}_L = T_L \oplus 1_+, \qquad \hat{T}_R = 1_- \oplus T_R,
$$
where the enlarged left/right operators $\hat{T}_{L,R}$ have been redefined. Concerning the domain of $T$, note that $\RR(\hat{T}_R)=\BB_-\oplus \RR(T_R)$ and $\DD(\hat{T}_L)=(\DD(T_L)\cap\BB_-)\oplus\BB_R$. Hence, bearing in mind that $\hat{T}_R$ acts as the identity on $\BB_-$ we get
$$
 \DD(T) = (\DD(T_L)\cap\BB_-) \oplus \DD(T_R) = \DD(T_L)+\DD(T_R).
$$

The assumption $\BB_0\subset\DD(T_{L,R})$ implies that the expressions \eqref{eq:fLfR} for $f_{L,R}$ also hold in this case with the redefined operators $\hat{T}_{L,R}$.

As a consequence of the previous results, $\BB_0\subset\DD(T)$, thus $f$ is given again by \eqref{eq:fQ} but with the new meaning of $T$. Analogously to the case {\it (i)}, $Q-zQTQ$ has a block representation with respect to the direct sum $\BB_-\oplus\BB_+$, which in this case looks like
$$
 Q-zQTQ = 
 \begin{pmatrix} 
 	P_--zP_-\hat{T}_LP_- & -zP_-\hat{T}_L\hat{T}_RP_+
	\\
	0 & P_+-zP_+\hat{T}_RP_+ 
 \end{pmatrix}.
$$
The analogue of Lemma~\ref{lem:triang} for upper triangular block operators provides the inverse of $Q-zQTQ$ for 
$$
 z^{-1}\in\varpi(Q_LT_LQ_L)\cap\varpi(Q_RT_RQ_R) 
 =\varpi(P_-\hat{T}_LP_-)\cap\varpi(P_+\hat{T}_RP_+)
 \subset\varpi(QTQ), 
$$
leading to the following expression of $f$ in such a domain,  
\begin{equation} \label{eq:f-fact}
\begin{gathered}
 f(z) = P\hat{T}_L\hat{T}_RP
 + z\begin{pmatrix} P\hat{T}_LP_- & P\hat{T}_L\hat{T}_RP_+ \end{pmatrix}
 \begin{pmatrix} 
 	S_- & zS_-P_-\hat{T}_L\hat{T}_RP_+S_+
	\\
	0 & S_+ 
 \end{pmatrix}
 \begin{pmatrix} P_-\hat{T}_L\hat{T}_RP \\ P_+\hat{T}_RP \end{pmatrix},
 \\
 S_-=(P_--zP_-\hat{T}_LP_-)^{-1}, 
 \qquad 
 S_+=(P_+-zP_+\hat{T}_RP_+)^{-1}.
\end{gathered}
\end{equation}
Here we have used the new relations 
$$
 P_+\hat{T}_L=\hat{T}_LP_+=P_+, \qquad P_-\hat{T}_R=\hat{T}_RP_-=P_-,
$$ 
satisfied by the redefined enlarged operators. These relations also imply that
$$
 \hat{T}_L\hat{T}_R = \hat{T}_L(P_-+P+P_+)\hat{T}_R 
 = \hat{T}_LP_- + \hat{T}_LP\hat{T}_R + P_+\hat{T}_R,
$$
so that
\begin{equation} \label{eq:rel-enl-2}
\begin{aligned}
 & P\hat{T}_L\hat{T}_RP = (P\hat{T}_LP)(P\hat{T}_RP),
 & \quad & P\hat{T}_L\hat{T}_RP_+ = (P\hat{T}_LP)(P\hat{T}_RP_+),
 \\
 & P_-\hat{T}_L\hat{T}_RP = (P_-\hat{T}_LP)(P\hat{T}_RP),
 & & P_-\hat{T}_L\hat{T}_RP_+ = (P_-\hat{T}_LP)(P\hat{T}_RP_+).  
\end{aligned}
\end{equation}
Inserting \eqref{eq:rel-enl-2} into the block representation \eqref{eq:f-fact} of $f$ and comparing with the expressions \eqref{eq:fLfR} of $f_{L,R}$ yields the factorization $f(z)=f_L(z)f_R(z)$.
\end{proof} 

Other useful but less surprising splitting rules for FR-functions appearing in the next theorem deal with the following situations,
$$
 \tilde{T} = T +
 \left(
 \begin{array}{c|c|c} 
 	0_- & & 
 	\\ \hline
	& T_0 &
	\\ \hline
	& & 0_+	 
 \end{array}
 \right),
 \qquad
 \tilde{T} =
 \left(
 \begin{array}{c|c|c} 
 	1_- & & 
 	\\ \hline
	& T_0 &
	\\ \hline
	& & 1_+	 
 \end{array}
 \right)
 T,
 \qquad
 \tilde{T} = T
 \left(
 \begin{array}{c|c|c} 
 	1_- & & 
 	\\ \hline
	& T_0 &
	\\ \hline
	& & 1_+	 
 \end{array}
 \right),
$$
where $T_0$ is an operator on the subspace $\BB_0$.

\begin{thm} \label{thm:split2}
Let $T$ be an operator on a Banach space $\BB$, $P$ a bounded projection of $\BB$ onto $\BB_0$ along $\BB_1$ and $Q=1-P$. Denote by $0_1$ and $1_1$ the null and identity operators on $\BB_1$. Then, if $T_0$ is an operator on $\BB_0$ and $f$ is the FR-function of $P$ with respect to $T$, we have the following rules: 
\begin{itemize}
\vskip7pt
\item[{\it (i)}] 
The FR-function of $P$ with respect to $\tilde{T} = T + (T_0 \oplus 0_1)$ is given by
$$
 \tilde{f}(z)=f(z)+T_0, \qquad z^{-1}\in\varpi(QT).
$$
\vskip7pt
\item[{\it (ii)}] 
The FR-function of $P$ with respect to $\tilde{T} = (T_0 \oplus 1_1)T$ is given by
$$
 \tilde{f}(z)=T_0f(z), \qquad z^{-1}\in\varpi(QT).
$$
\item[{\it (iii)}] 
The FR-function of $P$ with respect to $\tilde{T} = T(T_0 \oplus 1_1)$ is given by
$$
 \tilde{f}(z)=f(z)T_0, \qquad z^{-1}\in\varpi(QT).
$$
\end{itemize} 
\end{thm}

\begin{proof}
\quad

\noindent{\it (i)}
By definition,
\begin{equation} \label{eq:tf}
 \tilde{f}(z) = P\tilde{T}(1-zQ\tilde{T})^{-1}P,
 \qquad z^{-1}\in\varpi(Q\tilde{T}),
\end{equation}
where $\tilde{T}=T+\hat{T}_0$ in terms of the enlarged operator $\hat{T}_0 = T_0 \oplus 0_1$. Since $P\hat{T}_0=\hat{T}_0=\hat{T}_0P$ and $Q\hat{T}_0=0$ we find that $Q\tilde{T}=QT$ and
$$
 \tilde{f}(z) = PT(1-zQT)^{-1}P + P\hat{T}_0P(1-zQT)^{-1}P = f(z) + T_0,
 \qquad z^{-1}\in\varpi(QT),
$$
where we have used that $(1-zQT)^{-1}=1+zQT(1-zQT)^{-1}$ and $P\hat{T}_0P=T_0$ as an operator on $\BB_0$.

\noindent{\it (ii)}
The FR-function $\tilde{f}$ is given again by \eqref{eq:tf}, but with $\tilde{T}=\hat{T}_0T$ in terms of the redefined operator $\hat{T}_0=T_0\oplus1_1$. This redefined operator satisfies $P\hat{T}_0=T_0\oplus0_1=\hat{T}_0P$ and $Q\hat{T}_0=Q$, which lead to $Q\tilde{T}=QT$ and
$$
 \tilde{f}(z) = P\hat{T}_0PT(1-zQT)^{-1}P = T_0f(z),
 \qquad z^{-1}\in\varpi(QT),
$$
where we have used that $(1-zQT)^{-1}=1+zQT(1-zQT)^{-1}$ and $P\hat{T}_0P=T_0$ as an operator on $\BB_0$.

\noindent{\it (iii)}
Now the FR-function $\tilde{f}$ is given by \eqref{eq:tf} with $\tilde{T}=T\hat{T}_0$ and the same redefined operator $\hat{T}_0=T_0\oplus1_1$ as in {\it(ii)}. Using $\hat{T}_0Q=Q$ we get the identity
$$
 (1-zQT) \hat{T}_0 = \hat{T}_0 (1-zQT\hat{T}_0), \qquad z\in\C.
$$
Since $(1-zQT\hat{T}_0)(1-zQT\hat{T}_0)^{-1}=1$ for $z\in\varpi(QT)$, we obtain
$$
 (1-zQT) \hat{T}_0 (1-zQT\hat{T}_0)^{-1} = \hat{T_0}.
$$
When $z\in\varpi(QT)$, the operator $(1-zQT\hat{T}_0)^{-1}$ has domain $\BB$ and range equal to $\DD(T\hat{T}_0)=\{v\in\BB:\hat{T}_0v\in\DD(T)\}$. Hence, $\hat{T}_0 (1-zQT\hat{T}_0)^{-1}$ has domain $\BB$ and range $\DD(T)$, so that we can use the fact that $(1-zQT)^{-1}(1-zQT)$ is the identity in $\DD(T)$ to conclude that
$$
 \hat{T}_0 (1-zQT\hat{T}_0)^{-1} = (1-zQT)^{-1} \hat{T_0}, 
 \qquad z^{-1}\in\varpi(QT).
$$
Multiplying on the left by $PT$ and on the right by $P$, this leads to 
$$
 \tilde{f}(z) = PT(1-zQT)^{-1}P\hat{T}_0P = f(z)T_0,
 \qquad z^{-1}\in\varpi(QT),
$$ 
once we take into account that $P\hat{T}_0=\hat{T}_0P$ and $P\hat{T}_0P$ is $T_0$ as an operator on $\BB_0$.
\end{proof}

The previous splitting rules constitute a characteristic feature of FR-functions which is not shared by other relevant functions such as Stieltjes of m-functions. This is an indication of the fact that FR-functions do a better job at codifying important properties of those ``systems'' to which they are linked. This is the case for instance for the return properties of RW and QW, as Sections~\ref{sec:RW}, \ref{sec:QW} and \ref{sec:OQW} will show. The above splitting properties of FR-functions become recurrence splitting rules for such walks. The next section will show the importance of the above splitting rules also for the theory of orthogonal polynomials: the choice of a right OPRL analogue of Schur functions for OPUC allows us to obtain for the first time an OPRL version of Khrushchev factorization formula for OPUC.

\section{Applications to orthogonal polynomials on the real line: 
\break OPRL Khrushchev formula}
\label{sec:OP}

The aim of this section is to obtain an OPRL analogue of Khrushchev formula for OPUC, which should be the starting point of a Khrushchev theory for OPRL. The bulk of OPUC Khrushchev theory appeared in S.~Khrushchev's papers \cite{Khrushchev,Khrushchev2}, but the reader may also consult the monograph \cite[Chapters 4 and 9]{Simon-OPUC}.

Let us describe first the OPUC version of Khrushchev formula. Consider a probability measure $\mu$ on $\T$ --i.e. a positive measure normalized by $\mu(\T)=1$-- with Verblunsky coefficients $\alpha_0,\alpha_1,\alpha_2,\dots$. These coefficients not only determine the corresponding CMV matrix, but also the Schur function $f$ related to $\mu$ because its Schur parameters coincide with the Verblunsky coefficients due to Geronimus theorem \cite{Geronimus} (see also \cite[Chapter 3]{Simon-OPUC}). This establishes a one-to-one correspondence among probability measures on $\T$, Schur functions and sequences in $\D$ with the last term on $\T$ in the case of a terminating sequence.

Khrushchev formula answers the following question: if $\varphi_n$ is the orthonormal polynomial of degree $n$ in $L^2_\mu$, what is the Schur function of the probability measure $|\varphi_n|^2\,d\mu$? Khrushchev formula states that this Schur function factorizes as the product
\begin{equation} \label{eq:gnfn}
 g_nf_n,
 \qquad
 \begin{cases}
 	\; f_n = \text{Schur function with Schur parameters } 
 	\alpha_n,\alpha_{n+1},\alpha_{n+2},\dots,
 	\\
 	\; g_n = \text{Schur function with Schur parameters } 
 	-\overline\alpha_{n-1},-\overline\alpha_{n-2},\dots,-\overline\alpha_0,1.
 \end{cases}
\end{equation}
While $f_n$ are the Schur iterates of $f$, the rational Schur functions $g_n$ are called the inverse Schur iterates of $f$. It can be proved that the inverse Schur iterates are given by
$$
 g_n=\frac{\varphi_n}{\varphi_n^*}, 
 \qquad
 \varphi_n^*(z)=z^n\overline{\varphi_n(1/\overline{z})},
$$
so that the Schur function of $d\mu_n:=|\varphi_n|^2\,d\mu$ becomes 
$$
\frac{\varphi_n}{\varphi_n^*}f_n,
$$
where $\varphi_n^*$ is known as the reversed polynomial of $\varphi_n$. Writing OPUC Khrushchev formula as a formula by itself, it takes the form 
$$
 \frac{\displaystyle\int\frac{t\,d\mu_n(t)}{1-zt}}
 {\displaystyle\int\frac{d\mu_n(t)}{1-zt}} 
 = \frac{\varphi_n(z)}{\varphi_n^*(z)}f_n(z),
 \qquad z\in\D.
$$

In the FR-function approach, Khrushchev factorization formula for OPUC is a special instance of the factorization of FR-functions given in Theorem~\ref{thm:split}.{\it (ii)}, and comes from a natural overlapping factorization of CMV matrices into smaller ones \cite{CGVWW}. It is natural to expect an OPRL version of Khrushchev formula linked to a splitting of Jacobi matrices. Since self-adjointness is in general not preserved by products but only by sums, we expect a Khrushchev formula for OPRL based on the decomposition of FR-functions given in Theorem~\ref{thm:split}.{\it(i)}.
 
Let $\mc{J}$ be a Jacobi matrix with Jacobi coefficients $b_0,a_0,b_1,a_1,b_2,a_2,\dots$, and consider the related sequence of orthonormal polynomials $p_n$ starting at $p_0(t)=1$. Suppose that $\mc{J}$ is self-adjoint as an operator on $\ell^2$ with maximal domain. Then, the moment problem is determinate \cite{Akhiezer}, which means that $p_n$ are orthonormal with respect to a unique probability measure $\mu$ on $\R$. Also, the polynomials are dense in $L^2_\mu$ \cite{Akhiezer}, thus $\mc{J}$ represents the self-adjoint multiplication operator $T_\mu$ given in \eqref{eq:m-op} with respect to the basis $p_n$. The Jacobi matrix $\mc{J}$ determines the basis $p_n$, thus the measure $\mu$, which becomes the spectral measure of $p_0(t)=1$ with respect to $T_\mu$, equivalently the spectral measure of $e_0=(\delta_{k,0})_{k\ge0}\in\ell^2$ with respect to $\mc{J}$. Therefore, the FR-function $f$ of the subspace $\spn\{p_0\}$ with respect to $T_\mu$ is the Nevanlinna function of the measure $\mu$, in the sense of Definition~\ref{def:N-mu}. 

Since $p_n\in\DD(T_\mu)$, it makes sense to ask about the Nevanlinna function of the probability measure on $\R$ given by $p_n^2\,d\mu$. This is precisely the FR-function of the subspace $\spn\{p_n\}$ with respect to the multiplication operator $T_\mu$, i.e. the FR-function of the subspace $\spn\{e_n\}$, $e_n=(\delta_{k,n})_{k\ge0}\in\ell^2$, with respect to $\mc{J}$. 

The answer to the above question should give the OPRL version of OPUC Khrushchev formula. To obtain such a formula let us consider the following overlapping decomposition of the Jacobi matrix
$$
 \mc{J} \kern-2pt = \kern-1pt {\small
 \left(
 \begin{array}{ccccc|ccc}
 	b_0 & a_0 & & &
 	\\
 	a_0 & b_1 & a_1 & & 
 	\\
 	& \kern-15pt \ddots & \kern-10pt \ddots & \kern-10pt \ddots & 
 	\\
	& & \kern-10pt a_{n-2} & \kern-5pt b_{n-1} & \kern-7pt a_{n-1}
	\\
	& & & \kern-5pt a_{n-1} & \kern-7pt b_n
	\\ \hline 
	& & & & & & &
	\\ 
	& & & & & & &
	\\ 
	& & & & & & &
 \end{array}
 \right)
 \kern-1pt + \kern-1pt
 \left(
 \begin{array}{cccccc|ccccc}
 	& & & & & 
 	\\
 	& & & & & 
 	\\
    & & & & &
 	\\
	& & & & &
	\\ \hline
	& & & & & & b_n & \kern-3pt a_n
	\\ 
	& & & & & & a_n & \kern-3pt b_{n+1} & \kern-3pt a_{n+1}
	\\ 
	& & & & & & & \kern-3pt a_{n+1} & \kern-3pt b_{n+2} & \kern-3pt a_{n+2}
	\\ 
	& & & & & & & & \kern-30pt \ddots & \kern-30pt \ddots & \kern-10pt \ddots
 \end{array}
 \right)
 \kern-1pt - \kern-1pt
 \left(
 \begin{array}{cc|c|cc}
 	& & & &  
 	\\
 	& & & & 
 	\\
    & & & &
 	\\
	& & & &
	\\ \hline 
	& & \kern-1pt b_n \kern-1pt & & 
	\\ \hline
	& & & &  
	\\ 
	& & & &   
	\\ 
	& & & &  
 \end{array}
 \right)}.
$$
In a more compact notation
\begin{equation} \label{eq:Jdec}
 \mc{J} = (\mc{J}_{n+1} \oplus O) + (O_n \oplus \mc{J}^{(n)}) 
 - (O_n \oplus b_n \oplus O),  
\end{equation}
where $O$ and $O_n$ stand for the null $\infty\times\infty$ and $n\times n$ matrices respectively, $\mc{J}_n$ is the principal $n\times n$ submatrix of $\mc{J}$ and $\mc{J}^{(n)}$ is the Jacobi matrix obtained from $\mc{J}$ by deleting the first $n$ rows and columns. 

We can apply Theorem~\ref{thm:split2}.{\it (i)} to the operators on the Hilbert space $\HH=\ell^2$ involved in \eqref{eq:Jdec}, taking $\tilde{T}=\mc{J}$, $T=(\mc{J}_{n+1} \oplus O) + (O_n \oplus \mc{J}^{(n)})$ and $T_0$ being the operator on the overlapping subspace $\HH_0=\spn\{e_n\}$ given by $T_0e_n=-b_ne_n$. According to Theorem~\ref{thm:split2}.{\it (i)}, the Nevanlinna FR-function $\tilde{f}$ of $\spn\{e_n\}$ with respect to $\mc{J}$ is given by 
$$
 \tilde{f}(z) = h(z) - b_n, \qquad z\in\C\setminus\R,
$$ 
where $h$ is the FR-function of $\spn\{e_n\}$ with respect to $(\mc{J}_{n+1} \oplus O) + (O_n \oplus \mc{J}^{(n)})$. 

On the other hand, the FR-function $h$ comes about from Theorem~\ref{thm:split}.{\it(i)} applied to $T_L=\mc{J}_{n+1}$ acting on $\HH_L=\spn\{e_0,e_1,\dots,e_n\}$ and $T_R=\mc{J}^{(n)}$ acting on $\HH_R=\spn\{e_n,e_{n+1},\dots\}$. We find that $h$ is the sum of the FR-functions of $\HH_0=\spn\{e_n\}$ with respect to $\mc{J}_{n+1}$ and $\mc{J}^{(n)}$. 

Remember that the Nevanlinna function $f$ of the measure $\mu$ coincides with the FR-function of $\spn\{e_0\}$ with respect to $\mc{J}$. Obviously, the FR-function of $\spn\{e_n\}$ with respect to $\mc{J}^{(n)}$ follows from applying to $f$ the Schur algorithm \eqref{eq:Nalg} for a total of $n$ steps. Therefore, such an FR-function is the corresponding $n$-th iterate $f_n$ of $f$, a Nevanlinna FR-function whose Schur parameters are $b_n,a_n,b_{n+1},a_{n+1},\dots$ with respect to the Schur algorithm \eqref{eq:Nalg} at the origin. 

As for the FR-function of $\spn\{e_n\}$ with respect to $\mc{J}_{n+1}$, the renewal equation \eqref{eq:f-s-sa} allows us to express it as 
\begin{equation} \label{eq:FR-Jn}
 z^{-1} + \frac{1}{\<e_n|(\mc{J}_{n+1}-z^{-1}I_{n+1})^{-1}e_n\>}
 = z^{-1} + \frac{\det(\mc{J}_{n+1}-z^{-1}I_{n+1})}{\det(\mc{J}_n-z^{-1}I_n)},
 \qquad z\in\C\setminus\R,
\end{equation}
where $I_n$ stands for the $n\times n$ identity matrix. Using the expression of the orthonormal polynomials in terms of truncated Jacobi matrices
\begin{equation} \label{eq:p-J}
 p_n(z) = \kappa_n\det(zI_n-\mc{J}_n), 
 \qquad \kappa_n=\frac{1}{a_0a_1\cdots\,a_{n-1}},
\end{equation}
the FR-function \eqref{eq:FR-Jn} becomes
$$
 z^{-1} - a_n\frac{p_{n+1}(z^{-1})}{p_n(z^{-1})},
 \qquad z\in\C\setminus\R.
$$

Combining all the previous results we find that the FR-function of $\spn\{e_n\}$ with respect to $\mc{J}$ is given by
$$
\begin{aligned}
 \tilde{f}(z) 
 & = z^{-1} - a_n\frac{p_{n+1}(z^{-1})}{p_n(z^{-1})} + f_n(z) - b_n 
 = \frac{(z^{-1}-b_n)p_n(z^{-1})-a_np_{n+1}(z^{-1})}{p_n(z^{-1})} + f_n(z),
 \\
 & = a_{n-1}\frac{p_{n-1}(z^{-1})}{p_n(z^{-1})} + f_n(z),
 \qquad\quad z\in\C\setminus\R,
\end{aligned}
$$
that is,
\begin{equation} \label{eq:gn+fn}
 \tilde{f}(z) = g_n(z) + f_n(z),
 \qquad\quad 
 g_n(z) := \frac{\kappa_{n-1}}{\kappa_n} \frac{p_{n-1}(z^{-1})}{p_n(z^{-1})},
 \qquad\quad
 z\in\C\setminus\R.
\end{equation}
This can be considered as the OPRL version of Khrushchev formula, because $\tilde{f}$ is precisely the Nevanlinna function of the measure $p_n^2\,d\mu$, a statement which we enunciate separately below. Before doing that we will make a couple of remarks.

First, remember that behind the previous arguments there was the self-adjointness assumption for the Jacobi matrix. This is equivalent to the determinacy of the corresponding moment problem --i.e. the uniqueness of the  orthogonality measure for the polynomials generated by the three term recurrence relation associated with the Jacobi matrix-- and guarantees the uniqueness of the Nevanlina function whose Schur parameters --with respect to the Schur algorithm \eqref{eq:Nalg}-- are the Jacobi coefficients. This uniqueness holds simultaneously for a Nevanlinna function and its iterates, as follows by the connection among themselves due to the Schur algorithm on the real line.  

Furthermore, we could have arrived at OPRL Khrushchev formula alternatively by applying Theorem~\ref{thm:split}.{\it(i)} directly to the decomposition 
$$
\begin{aligned}
 \mc{J} & = {\small
 \left(
 \begin{array}{ccccc|ccc}
 	b_0 & a_0 & & &
 	\\
 	a_0 & b_1 & a_1 & & 
 	\\
 	& \kern-15pt \ddots & \kern-10pt \ddots & \kern-10pt \ddots & 
 	\\
	& & \kern-10pt a_{n-2} & \kern-5pt b_{n-1} & \kern-7pt a_{n-1}
	\\
	& & & \kern-5pt a_{n-1} & \kern-7pt 0
	\\ \hline 
	& & & & & & &
	\\ 
	& & & & & & &
	\\ 
	& & & & & & &
 \end{array}
 \right)
 +
 \left(
 \begin{array}{cccccc|ccccc}
 	& & & & & 
 	\\
 	& & & & & 
 	\\
    & & & & &
 	\\
	& & & & &
	\\ \hline
	& & & & & & b_n & \kern-3pt a_n
	\\ 
	& & & & & & a_n & \kern-3pt b_{n+1} & \kern-3pt a_{n+1}
	\\ 
	& & & & & & & \kern-3pt a_{n+1} & \kern-3pt b_{n+2} & \kern-3pt a_{n+2}
	\\ 
	& & & & & & & & \kern-30pt \ddots & \kern-30pt \ddots & \kern-10pt \ddots
 \end{array}
 \right)}
 \\
 & = (\hat{\mc{J}}_{n+1} \oplus O) + (O_n \oplus \mc{J}^{(n)}).
\end{aligned}
$$
Therefore, $g_n$ must be the Nevanlinna FR-function of $\spn\{e_n\}$ with respect to $\hat{\mc{J}}_{n+1}$. Reordering the basis of $\spn\{e_0,e_1,\dots,e_n\}$ as $\{e_n,\dots,e_1,e_0\}$ we conclude that $g_n$ is the Nevanlinna function characterized by the terminating sequence of Schur parameters given by $0,a_{n-1},b_{n-1},a_{n-2},b_{n-2},\dots,a_0,b_0$. Comparing with \eqref{eq:gnfn} we see that $g_n$ should be considered as the inverse Schur iterates of the Nevanlinna function $f$.

Although the validity of OPRL Khrushchev formula \eqref{eq:gn+fn} requires the self-adjointness of the underlying Jacobi matrix, the fact that $g_n$ is a Nevanlinna function is indeed true even if $\mc{J}$ is not self-adjoint since it refers to a property of the modified finite submatrix $\hat{\mc{J}}_{n+1}$, which is always self-adjoint. Therefore, the quotient $p_{n-1}(z^{-1})/p_n(z^{-1})$ is a Nevanlinna function for every sequence $p_n$ of orthonormal polynomials. This means that $p_{n-1}(z)/p_n(z)$ maps $\C_\pm$ into $\overline\C_\mp$, so that $-p_{n-1}(z)/p_n(z)$ is again a Nevanlinna function. 

The conclusion of the previous discussion is summarized in the following theorem, which is stated in general for matrix valued OPRL. Since we are now in the matrix valued case, attention should be paid to commutativity issues.   

\begin{thm}[\bf Khrushchev formula for matrix valued OPRL] \label{thm:K-R}
Let $f$ be the matrix Nevanlinna function of a positive matrix measure $\mu$ on $\R$ such that $\mu(\R)=1_0$ is the identity matrix. Suppose that there exists a sequence of matrix orthonormal polynomials $p_n(z)=\kappa_nz^n+\cdots$ with respect to $\<p|q\>=\int\!p^\text{\rm$\dag$}d\mu\,q$ and that the related block Jacobi matrix is self-adjoint as an operator on $\ell^2$ with maximal domain. Then, the Nevanlinna function of the matrix measure $d\mu_n:=p_n^\text{\rm$\dag$}\,d\mu\,p_n$ is given by the following sum of Nevanlinna functions
$$
 g_n+f_n, 
 \qquad
 \begin{cases}
 	\; g_n(z) = 
	p_n(z^{-1})^{-1} p_{n-1}(z^{-1}) \, \kappa_n^{-1} \kappa_{n-1},
 	\\[5pt]
 	\; f_n = 
	n\text{-th iterate of $f$ for the Schur algorithm \eqref{eq:Nalg-op}}.
 \end{cases}
$$ 
Equivalently,
$$
\begin{gathered}
 \left(\int\frac{t\,d\mu_n(t)}{1-zt}\right)
 \left(\int\frac{d\mu_n(t)}{1-zt}\right)^{-1} =
 g_n(z) + f_n(z),
 \qquad\quad z\in\C\setminus\R,
 \\[3pt]
 \begin{cases}
 	\; f_n = \text{Nevanlinna function with Schur parameters } 
 	b_n,a_n,b_{n+1},a_{n+1},b_{n+2},a_{n+2},\dots,
 	\\
 	\; g_n = \text{Nevanlinna function with Schur parameters } 
 	0,a_{n-1},b_{n-1},a_{n-2},b_{n-2},\dots,a_0,b_0,
 \end{cases}
\end{gathered}
$$
where $b_0,a_0,b_1,a_1,b_2,a_2,\dots$ are the corresponding Jacobi coefficients, i.e. the Schur parameters of $f$ with respect to \eqref{eq:Nalg-op}.    
\end{thm}  

\begin{proof}
The steps of the proof are similar to those of the scalar case discussed previously, but now the Jacobi matrix $\mc{J}$ is made up of $d\times d$-matrix blocks $b_n=b_n^\dag$, $a_n>0$ where $d$ is the size of the matrix measure.
These blocks provide the three term recurrence relation for the orthonormal polynomials \cite{DPS},
\begin{equation} \label{eq:OP}
 zp_n(z) = p_{n+1}(z)a_n + p_n(z)b_n + p_{n-1}(z)a_{n-1},
 \qquad n\ge0 \qquad (a_{-1}=0, \kern5pt p_0(z)=1_0),
\end{equation}
hence the leading matrix coefficients satisfy $\kappa_{n+1}a_n=\kappa_n$. 

When $\mc{J}$ is self-adjoint --as an operator on $\ell^2$ with maximal domain-- it represents, in the basis given by the columns of the orthonormal polynomials, the self-adjoint multiplication operator $T_\mu$ defined in \eqref{eq:m-op}, where $L^2_\mu$ is now the Hilbert space of square-summable $d$-vector valued functions with inner product $\<f|g\>=\int\!f^\dag\,d\mu\,g$. Hence, $\mc{J}$ is unitarily equivalent to $T_\mu$. This unitary equivalence identifies the columns $\{p_n^{(0)},p_n^{(1)},\dots,p_n^{(d-1)}\}\subset L^2_\mu$ of $p_n$ with the $n$-th $d$-block $B_n:=\{e_{nd},e_{nd+1},\dots,e_{nd+d-1}\}\subset\ell^2$ of canonical vectors $e_n=(\delta_{k,n})_{k\ge0}$. 

Since $p_n^\dag\,d\mu\,p_n$ is the spectral measure of the subspace spanned by the columns of $p_n$ with respect to $T_\mu$, it is also the spectral measure of $\spn{B_n}$ with respect to $\mc{J}$. Therefore, the Nevanlinna function of $p_n^\dag\,d\mu\,p_n$ is the FR-function of $\spn{B_n}$ with respect to $\mc{J}$. As in the scalar case, the rest of the proof consists in identifying such an FR-function by decomposing it using Theorems~\ref{thm:split}.{\it(i)} and \ref{thm:split2}.{\it(i)}. 

The steps for this identification are similar to those of the scalar case. The only quirk in the matrix valued situation comes from the identification of the FR-function $g_n$ of $\spn B_n$ with respect to the $(n+1)d\times(n+1)d$ principal submatrix $\mc{J}_{n+1}$, which now will not rely on determinant identities. Instead, we will use Schur complements to obtain the m-function 
$$
  m_n(z) = P_n(\mc{J}_{n+1}-zI_{(n+1)d})^{-1}P_n,
 \qquad
 \begin{aligned}
 & P_n = \text{orthogonal projection of $\ell^2$ onto } \spn B_n,
 \\
 & I_k = k \times k \text{ identity matrix},
 \end{aligned}
$$
giving $g_n$ via the renewal equation $g_n(z)=z^{-1}I_d+ m_n(z^{-1})^{-1}$. The application of Lemma~\ref{lem:SC} to the block representation
$$
 \mc{J}_{n+1}-zI_{n+1} = 
 \begin{pmatrix}
 	\mc{J}_n-zI_{nd} & v_n^\dag \\ v_n & b_n-zI_d
 \end{pmatrix},
 \quad v_n = 
 \begin{pmatrix} 
 	O_d & \kern-5pt O_d & \kern-5pt \cdots & \kern-5pt a_{n-1} 
 \end{pmatrix},
 \quad O_k = k \times k \text{ null matrix},
$$ 
leads to the recurrence
$$
  m_n(z)^{-1} = b_n-zI_d - a_{n-1}  m_{n-1}(z) a_{n-1}, 
 \qquad n\ge0 \qquad (a_{-1}=0). 
$$
Comparing this recurrence with \eqref{eq:OP} rewritten as 
$$
 -p_n(z)^{-1}p_{n+1}(z)a_n = b_n-zI_d + p_n(z)^{-1}p_{n-1}(z)a_{n-1},
 \qquad n\ge0 \qquad (a_{-1}=0, \kern5pt p_0(z)=I_d),
$$
yields the identity
$$
  m_n(z) = -a_n^{-1} p_{n+1}(z)^{-1} p_n(z),
 \qquad n\ge0.
$$
Therefore, using \eqref{eq:OP} again we obtain
$$
\begin{aligned}
 g_n(z)-b_n & = z^{-1}I_d-m_n(z^{-1})^{-1}-b_n
 = p_n(z^{-1})^{-1} (p_n(z^{-1})(z^{-1}I_d-b_n)-p_{n+1}(z^{-1})a_n)
 \\ 
 & = p_n(z^{-1})^{-1} p_{n-1}(z^{-1}) a_{n-1}
 = p_n(z^{-1})^{-1} p_{n-1}(z^{-1}) \kappa_n^{-1} \kappa_{n-1}.
\end{aligned}
$$   
\end{proof}

The use of the Khrushchev formula above to develop a Khrushchev theory for OPRL remains as a challenge. 

\medskip

In the next sections we will discuss several applications of FR-functions which are at the heart of their origin: the study of recurrence --i.e. return properties-- in RW and QW. While the context of FR-functions in Hilbert spaces will be enough for the analysis of recurrence in unitary QW, the general discussion of recurrence in RW as well as in open QW will require the setting of FR-functions in Banach spaces.

A common feature of all these applications is the fact that the evolution is governed by a contraction and the related projections have norm one. Hence, the FR-functions involved in these applications are analytical on the open unit disk. Besides, at least in the case of RW and unitary QW, the FR-functions have values in contractions --on a Banach and a Hilbert space respectively-- so they turn out to be Schur functions. As for RW, there is large class --which includes all those which are irreducible and reversible-- whose stochastic matrices are also self-adjoint with respect to suitable inner products \cite[Chapter 6]{Stroock}, thus the corresponding FR-functions are simultaneously Schur and Nevanlinna functions.

As we will see, the Hilbert FR-function approach to recurrence in unitary QW clearly differs from that of RW recurrence. However, the density operator formalism in which open QW are usually described establishes a close parallelism between the Banach FR-function approaches to recurrence in RW and open QW.  

When applied to stochastic matrices, the relation between operator valued FR-functions and Stieltjes functions given in Theorem~\ref{thm:dom2} becomes the generalization of the classical renewal equation for RW, first obtained in \cite{Polya} (see also \cite{Feller,Stroock}), to the recurrence of a subset of states. The unitary situation leads to the version of the renewal equation for QW, already uncovered in \cite{GVWW,BGVW} (see also \cite{CGMV2}). The fact that Theorem~\ref{thm:dom2} holds for arbitrary operators on Banach spaces implies that open QW also have a renewal equation.

In this context, the splitting rules for FR-functions become splitting rules for the study of recurrence properties in a walk: they tell us how to split a walk into overlapping smaller walks so that the return properties of the overlapping piece in the original walk are determined by its return properties in the smaller walks.



\section{Applications to random walks: RW recurrence}
\label{sec:RW}

There is a huge literature on RW. For good references which highlight the return aspects and the connection with self-adjoint operators see \cite{Feller,Stroock}. 

Consider a RW on a countable set $\SS$ of states. By this, we mean a Markov chain on $\SS$ whose evolution is given at any time by a stochastic matrix $\Pi=(\Pi_{i,j})_{i,j\in\SS}$, i.e. a matrix with non-negative entries whose rows sum up to 1,
$$
 \Pi_{i,j}\ge0 \quad \forall i,j\in\SS, 
 \qquad\qquad 
 \sum_{j\in\SS}\Pi_{i,j}=1 \quad \forall i\in\SS.
$$
Since multiplication preserves stochasticity, $\Pi^n$ is again stochastic and its entry $\Pi^n_{i,j}$ gives the $n$-step probability transition $i \to j$. 

In general, the stochastic matrix $\Pi$ defines an operator $v \mapsto v\Pi$ on the Banach space $\ell^1(\SS)$ rather than on a Hilbert space since
$$
 \|v\Pi\| = \sum_{j\in\SS} \left|\sum_{i\in\SS} v_i\Pi_{i,j}\right| 
 \le \sum_{i\in\SS} |v_i| \sum_{j\in\SS} \Pi_{i,j} = \|v\|,
 \qquad v\in\ell^1(\SS),
$$
where the norm is that of $\ell^1(\SS)$ and the commutation of sums is possible even in case of infinitely many terms because they are non-negative. The above identity also shows that $\Pi$ is bounded with $\|\Pi\|\le1$, i.e. $\Pi$ is a contraction on $\ell^1(\SS)$.  

With any --finite or infinite-- subset $\Omega\subset\SS$ we can associate the projection $P$ of $\ell^1(\SS)$ onto $\ell^1(\Omega)$ along $\ell^1(\SS\setminus\Omega)$, represented by the matrix
\begin{equation} \label{eq:P-RW}
 P_{i,j}=
 \begin{cases}
 	1 & i=j\in\Omega,
	\\
	0 & \text{otherwise},
 \end{cases}
 \qquad \|P\|=1,
\end{equation}
as well as the complementary projection $Q=1-P$, which also has norm $\|Q\|=1$.
In short, we will refer to $P$ as the projection of $\SS$ onto $\Omega$.

The FR-function $f$ of $P$ with respect to $\Pi$ is the function with values in operators on $\ell^1(\Omega)$ given by 
$$
 f(z) = P\Pi(1-zQ\Pi)^{-1}P = \sum_{n\ge0} z^nP\Pi(Q\Pi)^nP,
 \qquad z\in\D,
$$
where the validity of the above expressions for $z\in\D$ follows from $\|Q\Pi\|\le1$. 
We will also refer to $f$ as the FR-function of $\Omega$ with respect to $\Pi$. The interest of this FR-function lies in its relation with the return properties of the subset $\Omega$, in particular with the following recurrence concepts. 

\begin{defn} 
Given a subset $\Omega\subset\SS$, we define the following notions for a RW on the set of states $\SS$:

$\pi(i\to\Omega)=$ probability of returning to $\Omega$ when starting at the state $i\in\Omega$.

$\pi(i\xrightarrow{\Omega}j)=$ probability of landing on the state $j\in\Omega$ when returning to $\Omega$ starting at the \hspace*{67pt} state $i\in\Omega$.

$\pi_n(i\to\Omega)=$ probability of returning for the first time in $n$ steps to $\Omega$ when starting at \hspace*{74pt} the state $i\in\Omega$.

$\pi_n(i\xrightarrow{\Omega}j)=$ probability of landing on the state $j\in\Omega$ when returning for the first time \hspace*{74pt} in $n$ steps to $\Omega$ starting at the state $i\in\Omega$.

$\tau(i\to\Omega)=$ expected return time to $\Omega$ when starting at the state $i\in\Omega$. 

\noindent When $\Omega=\{i\}$ we write $\pi(i\to i)$, $\pi_n(i\to i)$ and $\tau(i\to i)$ for the corresponding quantities.
\end{defn}

These probabilistic notions are related with each other, with the stochastic matrix $\Pi$ and with the FR-function $f$ of $\Omega$ with respect to $\Pi$, as follows 
\begin{align}
 & \pi_n(i\xrightarrow{\Omega}j) = 
 \sum_{j_k\in\SS\setminus\Omega} 
 \Pi_{i,j_1} \Pi_{j_1,j_2} \cdots \, \Pi_{j_{n-1},j} =
 (\Pi(Q\Pi)^{n-1})_{i,j},
 \qquad 
 \pi_n(i\to\Omega) = \sum_{j\in\Omega} \pi_n(i\xrightarrow{\Omega}j),
 \notag
 \\
 & \kern40pt 
 f(z)_{i,j} = \sum_{n\ge1} z^{n-1} \pi_n(i\xrightarrow{\Omega}j),
 \qquad
 \sum_{j\in\Omega} f(z)_{i,j} = 
 \sum_{n\ge1} z^{n-1} \pi_n(i\to\Omega), 
 \qquad z\in\D, 
 \label{eq:f-RW}
 \\
 & \pi(i\xrightarrow{\Omega}j) = \sum_{n\ge1} \pi_n(i\xrightarrow{\Omega}j) 
 = f(1)_{i,j},
 \qquad
 \pi(i\to\Omega) = \sum_{n\ge1} \pi_n(i\to\Omega) = 
 \sum_{j\in\Omega} \pi(i\xrightarrow{\Omega}j) =
 \sum_{j\in\Omega} f(1)_{i,j},
 \notag
 \\
 & \tau(i\to\Omega) = 
 \begin{cases}
 	\infty 
	& \text{ if } \pi(i\to\Omega)<1,
	\\
 	\displaystyle \sum_{n\ge1} n\,\pi_n(i\to\Omega) 
 	= \lim_{x\uparrow1} \left(\sum_{j\in\Omega}xf(x)_{i,j}\right)' 
 	= 1 + \sum_{j\in\Omega} f'(1)_{i,j},
	& \text{ if } \pi(i\to\Omega)=1,
 \end{cases}
 \notag
 \\[5pt]
 & \kern70pt f(1)_{i,j} := \lim_{x\uparrow1} f(x)_{i,j},  
 \qquad 
 f'(1)_{i,j}:=\lim_{x\uparrow1} f'(x)_{i,j},
 \qquad x\in[0,1).
 \notag
\end{align}
The exchange of sums is legitimate by their absolute convergence. The commutation of limits and derivatives with eventual infinite sums follows from the fact that $f(x)_{i,j}$ and their derivatives are non-decreasing in $x$ for $x\in[0,1)$.

The relations \eqref{eq:f-RW} identify the FR-function $f$ as a true generating function of first time return probabilities to $\Omega$ (up to multiplication by $z$). The corresponding Stieltjes function 
$$
 s(z) = P(1-z\Pi)^{-1}P = \sum_{n\ge0} z^n P\Pi^nP,
 \qquad z\in\D,
$$ 
is instead the generating function of the return probabilities to $\Omega$ because $\Pi^n_{i,j}$ is the probability of the transition $i \to j$ in $n$ steps. The generalized renewal equation \eqref{eq:sf} becomes now
$$
 s(z)^{-1} = 1_0-zf(z), 
 \qquad 1_0 = \text{ identity of order the size of } \Omega,
 \qquad z\in\D,
$$
which is the extension of the standard renewal equation for the return to a state in a RW \cite{Polya,Feller,Stroock} to the case of the return to a subset of states.

While the expected time $\tau(i\to\Omega)$ is given essentially by the sum of the $i$-th row of the weak derivative $\displaystyle f'(1)$, the probability $\pi(i\to\Omega)$ is the sum of the $i$-th row of $f(1)$. Hence, $f(1)$ should be a substochastic matrix, i.e. a matrix with non-negative entries whose rows sum up to no more than 1. The absence of negative entries is obvious, but a bound for the sum of the rows requires a proof. As a byproduct, we will find that $f$ is a Schur function.

\begin{prop} \label{prop:schur-RW}
The FR-function $f$ of a subset $\Omega\in\SS$ with respect to a stochastic matrix $\Pi$ on $\SS$ satisfies 
$$
 \left|\sum_{j\in\Omega}f(z)_{i,j}\right| 
 \le \sum_{j\in\Omega}f(1)_{i,j} \le 1, 
 \qquad i\in\Omega, \qquad z\in\D.
$$
Therefore, $\|f(z)\|\le1$ for $z\in\D$, so that $f$ is a Schur function with values in operators on the Banach space $\ell^1(\Omega)$. 
\end{prop}

\begin{proof}
The inequality $|\sum_{j\in\Omega}f(z)_{i,j}| \le \sum_{j\in\Omega}f(1)_{i,j}$ for $z\in\D$ follows from 
$$
 |f(z)_{i,j}| \le \sum_{n\ge1} |z|^{n-1} \pi_n(i\xrightarrow{\Omega}j)
 \le \sum_{n\ge1} \pi_n(i\xrightarrow{\Omega}j) = f(1)_{i,j},
 \qquad |z|<1.
$$ 

As for the remaining inequality, 
$$
 \sum_{j\in\Omega}f(1)_{i,j}\le1
 \quad\Leftrightarrow\quad
 \pi_1(i\to\Omega)+\pi_2(i\to\Omega)+\cdots+\pi_n(i\to\Omega)\le1 
 \quad \forall n\in\N. 
$$
Let us first see by induction on $r=0,1,\dots,n-2$ that 
\begin{equation} \label{eq:sumpi}
 \pi_{n-r}(i\to\Omega)+\cdots+\pi_{n-1}(i\to\Omega)+\pi_n(i\to\Omega)
 \le \sum_{j_k\in\SS\setminus\Omega} 
 \Pi_{i,j_1} \Pi_{j_1,j_2} \cdots \, \Pi_{j_{n-r-2},j_{n-r-1}}. 
\end{equation}
The result follows for $r=0$ from the inequality $\sum_{j\in\Omega} \Pi_{j_{n-1},j} \le \sum_{j\in\SS} \Pi_{j_{n-1},j} = 1$, which gives $\pi_n(i\to\Omega)\le\sum_{j_k\in\SS\setminus\Omega} \Pi_{i,j_1} \Pi_{j_1,j_2} \cdots \, \Pi_{j_{n-2},j_{n-1}}$. Assuming \eqref{eq:sumpi} for some $r<n-2$ we obtain 
$$
\begin{aligned}
 \pi_{n-r-1}(i\to\Omega)+\pi_{n-r}(i\to\Omega)+\cdots+\pi_n(i\to\Omega) &
 \\
 & \kern-100pt \le \sum_{\substack{j_k\in\SS\setminus\Omega \\ j\in\Omega}} 
 \Pi_{i,j_1} \Pi_{j_1,j_2} \cdots \, \Pi_{j_{n-r-2},j}
 + \sum_{j_k\in\SS\setminus\Omega} 
 \Pi_{i,j_1} \Pi_{j_1,j_2} \cdots \, \Pi_{j_{n-r-2},j_{n-r-1}}
 \\
 & \kern-100pt = \sum_{\substack{j_k\in\SS\setminus\Omega \\ j\in\SS}} 
 \Pi_{i,j_1} \Pi_{j_1,j_2} \cdots \, \Pi_{j_{n-r-2},j}
 = \sum_{j_k\in\SS\setminus\Omega} 
 \Pi_{i,j_1} \Pi_{j_1,j_2} \cdots \, \Pi_{j_{n-r-3},j_{n-r-2}},
\end{aligned}
$$
where in the last equality we have used that $\sum_{j\in\SS} \Pi_{j_{n-r-2},j} = 1$. This proves \eqref{eq:sumpi} for $r+1$. 

Therefore, \eqref{eq:sumpi} holds for $r\le n-2$. In particular, this inequality reads for $r=n-2$ as follows
$$
 \pi_2(i\to\Omega)+\cdots+\pi_n(i\to\Omega)
 \le \sum_{j\in\SS\setminus\Omega} \Pi_{i,j}, 
$$
which, combined with $\pi_1(i\to\Omega)=\sum_{j\in\Omega}\Pi_{i,j}$, yields
$$
 \pi_1(i\to\Omega)+\pi_2(i\to\Omega)+\cdots+\pi_n(i\to\Omega)
 \le \sum_{j\in\SS} \Pi_{i,j} = 1.
$$

Since the operator norm of the operator on $\ell^1(\Omega)$ given by $v\mapsto vf(z)$ is $\|f(z)\|=\sup_{i\in\Omega}\sum_{j\in\Omega}|f(z)_{i,j}|$, the previous results prove that $\|f(z)\|\le1$ for $z\in\D$.
\end{proof} 

The previous proposition shows that every FR-function $f$ for a RW is a Schur function with values $f(z)$ in operators on a $\ell^1$ Banach space, i.e. taking for $f(z)$ the operator norm with respect to the $\ell^1$ norm. Nevertheless, in many situations they are also Nevanlinna functions. This is the case for instance when the RW is simultaneously irreducible and reversible because in this situation the stochastic matrix $\Pi$ defines a self-adjoint operator on certain Hilbert space \cite[Chapter 6]{Stroock}. A particular case of this are the birth-death processes, which will appear later on as an example illustrating the use of FR-function splitting techniques for RW.  

The splitting formulas for FR-functions lead to splitting rules for recurrence properties of RW. Such splitting rules follow from the splitting of a RW into overlapping ones, corresponding to one of the following overlapping splittings for the related stochastic matrix $\Pi$,
$$
 \Pi = 
 \left(
 \begin{array}{c|c} 
 	\\[-5pt]
	\kern7pt \Pi_L \kern9pt 
	\\[7pt] \hline
	&
	\\[-13pt] 
	& \kern1pt 0_+ \kern-3pt
 \end{array}
 \right)
 +
 \left(
 \begin{array}{c|c} 
 	0_- \kern-2pt 
 	\\ \hline 
	\\[-5pt]
	& \kern9pt \Pi_R \kern7pt
	\\ [5pt]
 \end{array}
 \right)
 -
 \left(
 \begin{array}{c|c|c} 
 	0_- \kern-2pt & & 
 	\\ \hline
	& &
	\\[-13pt]
	& \kern2pt \Pi_0 &
	\\ \hline
	& &
	\\[-13pt]
	& & \kern1pt 0_+ \kern-3pt 
 \end{array}
 \right),
 \quad
 \Pi = 
 \left(
 \begin{array}{c|c} 
 	\\[-5pt]
	\kern7pt \Pi_L \kern9pt 
	\\[7pt] \hline
	& 
	\\[-13pt] 
	& \kern1pt 1_+ \kern-3pt
 \end{array}
 \right)
 \left(
 \begin{array}{c|c} 
 	1_- \kern-2pt 
 	\\ \hline 
	\\[-5pt]
	& \kern9pt \Pi_R \kern7pt
	\\ [5pt]
 \end{array}
 \right).
$$
Following the terminology in Theorem~\ref{thm:split}, we refer below to these two cases as a decomposition and a factorization respectively.

The second situation, which deals with the factorization of a stochastic matrix $\Pi$ into a product of overlapping ones $\Pi_{L,R}$, needs no clarification since the product of stochastic matrices is again stochastic. In the first case however, the sum of overlapping matrices does not preserve the stochasticity. Actually, a sum of two overlapping stochastic matrices $\Pi_{L,R}$ has rows summing up to 1, except for the overlapping rows which sum up to 2. Hence, subtracting an additional stochastic matrix $\Pi_0$ on the overlapping subset yields a matrix $\Pi$ whose rows sum up to 1 all of them. This matrix $\Pi$ is stochastic as long as the subtraction of $\Pi_0$ does not lead to any negative matrix entry.   

\begin{thm}[\bf splitting rules for RW recurrence] \label{thm:split-RW}
Let $\SS=\Omega_-\cup\Omega\cup\Omega_+$ be a decomposition of a set of states into disjoint subsets, $P$ the projection \eqref{eq:P-RW} of $\SS$ onto $\Omega$ and $\Pi_L$, $\Pi_R$, $\Pi_0$ stochastic matrices on $\SS_L=\Omega_-\cup\Omega$, $\SS_R=\Omega\cup\Omega_+$ and $\Omega$ respectively. Then, using the subscript $L,R$ to distinguish the quantities related to the RW given by $\Pi_{L,R}$ and denoting by $0_\pm$ and $1_\pm$ the null and identity matrices on $\Omega_\pm$, we have the following splitting rules:
\begin{itemize}
\vskip7pt
\item[{\it (i)}] {\bf Decomposition:}
$\Pi = (\Pi_L \oplus 0_+) + (0_- \oplus \Pi_R) - (0_- \oplus \Pi_0 \oplus 0_+)$ is a stochastic matrix on $\SS$ whenever the block $P\Pi P$ has non-negative entries. Then, the corresponding return probabilities and expected return times to $\Omega$ are given by
$$
\begin{aligned}
 & \pi(i\to\Omega) = \pi_L(i\to\Omega) + \pi_R(i\to\Omega) - 1,
 \\
 & \tau(i\to\Omega) = \tau_L(i\to\Omega) + \tau_R(i\to\Omega) - 1,
\end{aligned}
\qquad i\in\Omega.
$$
As a consequence, for each $i\in\Omega$,
$$
 \pi(i\to\Omega)=1 \quad\Leftrightarrow\quad 
 \pi_L(i\to\Omega)=\pi_R(i\to\Omega)=1.
$$
\vskip7pt
\item[{\it (ii)}] {\bf Factorization:}
$\Pi = (\Pi_L \oplus 1_+)(1_- \oplus \Pi_R)$ is a stochastic matrix on $\SS$ and the corresponding return probabilities and expected return times to $\Omega$ are given by
$$
\begin{aligned}
 & \pi(i\to\Omega) = 
 \sum_{k\in\Omega} \pi_L(i\xrightarrow{\Omega}k) \, \pi_R(k\to\Omega),
 \\
 & \tau(i\to\Omega) = \tau_L(i\to\Omega) + 
 \sum_{k\in\Omega} \pi_L(i\xrightarrow{\Omega}k) \, \tau_R(k\to\Omega) - 1,
\end{aligned}
\qquad i\in\Omega, 
$$
where, in case of indetermination $\pi_L(i\xrightarrow{\Omega}k) \, \tau_R(k\to\Omega)=0.\infty$, we set
$$
 \pi_L(i\xrightarrow{\Omega}k) \, \tau_R(k\to\Omega) = 
 \begin{cases}
 	0 & \text{if } \pi_R(k\to\Omega)=1,
	\\
	\infty & \text{if } \pi_R(k\to\Omega)<1.
 \end{cases}
$$
As a consequence, for each $i\in\Omega$, 
$$
 \pi(i\to\Omega)=1 \quad\Leftrightarrow\quad 
 \pi_L(i\to\Omega)=\pi_R(k\to\Omega)=1
 \quad \forall k\in\Omega.
$$
\end{itemize}
\end{thm}

\begin{proof}
\quad

\noindent{\it (i)}
The rows of the matrix $\Pi$ sum up to 1 because $\Pi_L$, $\Pi_R$, $\Pi_0$ are stochastic and
$$
 \sum_{j\in\SS} \Pi_{i,j} = 
 \begin{cases}
 	\sum_{j\in\SS_L}(\Pi_L)_{i,j} & \; i\in\Omega_-,
	\\
	\sum_{j\in\SS_L}(\Pi_L)_{i,j} + \sum_{j\in\SS_R}(\Pi_R)_{i,j}
	- \sum_{j\in\Omega}(\Pi_0)_{i,j} & \; i\in\Omega, 
	\\
	\sum_{j\in\SS_R}(\Pi_R)_{i,j} & \; i\in\Omega_+.
 \end{cases}
$$
Hence, $\Pi$ is stochastic whenever the entries of the block $P\Pi P$ are non-negative because they are the only entries of $\Pi$ which differ from those of $\Pi_L$ or $\Pi_R$.  

The decomposition rules follow from Theorem~\ref{thm:split}.{\it(i)} applied to $\Pi_L$ and $\Pi_R$ as operators on the Banach spaces $\ell^1(\SS_L)$ and $\ell^1(\SS_R)$ respectively, combined with Theorem~\ref{thm:split2}.{\it(i)} applied to the perturbation $\Pi_0$. Such theorems imply that the FR-function of $P$ with respect to $\Pi$ is given by $f=f_L+f_R-\Pi_0$, where $f_{L,R}$ is the FR-function of the projection of $\SS_{L,R}$ onto $\Omega$ with respect to $\Pi_{L,R}$. Therefore,
$$
 \pi(i\to\Omega) = \sum_{j\in\Omega} f(1)_{i,j} 
 = \sum_{j\in\Omega} f_L(1)_{i,j} + \sum_{j\in\Omega} f_R(1)_{i,j}
 + \sum_{j\in\Omega} (\Pi_0)_{i,j}
 = \pi_L(i\to\Omega) + \pi_R(i\to\Omega) - 1.
$$ 
Hence, $\pi(i\to\Omega)=1$ iff $\pi_L(i\to\Omega) + \pi_R(i\to\Omega) = 2$, i.e. $\pi_L(i\to\Omega) = \pi_R(i\to\Omega) = 1$.

The expression for the expected return time holds trivially when $\pi(i\to\Omega)<1$ because this is equivalent to $\pi_L(i\to\Omega)<1$ or $\pi_R(i\to\Omega)<1$. In such a situation $\tau(i\to\Omega)=\infty=\tau_L(i\to\Omega)+\tau_R(i\to\Omega)-1$. 

Suppose now that $\pi(i\to\Omega)=1$. Then, 
$$
 \tau(i\to\Omega) = 1 + \sum_{j\in\Omega} f'(1)_{i,j} 
 = 1 + \sum_{j\in\Omega} f'_L(1)_{i,j} + \sum_{j\in\Omega} f'_R(1)_{i,j}
 = \tau_L(i\to\Omega) + \tau_R(i\to\Omega) - 1.
$$

\noindent{\it (ii)}
As a product of stochastic matrices, $\Pi$ is also stochastic. 

The factorization rules are a consequence of Theorem~\ref{thm:split}.{\it(ii)} for $\Pi_L$ and $\Pi_R$. This theorem states that the FR-function of $P$ with respect to $\Pi$ is given by $f=f_Lf_R$, in terms of the FR-functions $f_{L,R}$ introduced in {\it(i)}. Thus,
\begin{equation} \label{eq:pi-fact}
 \pi(i\to\Omega) = \sum_{j\in\Omega} f(1)_{i,j} 
 = \sum_{j,k\in\Omega} f_L(1)_{i,k} f_R(1)_{k,j}
 = \sum_{k\in\Omega} \pi_L({i\xrightarrow{\Omega}k}) \, \pi_R(k\to\Omega).
\end{equation} 

This implies that $\pi(i\to\Omega)\le\sum_{k\in\Omega}\pi_L({i\xrightarrow{\Omega}k})=\pi_L(i\to\Omega)$, so that $\pi(i\to\Omega)=1$ leads to $\pi_L(i\to\Omega)=1$. Also, $\pi_R(k\to\Omega)<1$ for some $k\in\Omega$ yields $\pi(i\to\Omega)<\sum_{k\in\Omega}\pi_L({i\xrightarrow{\Omega}k})=\pi_L(i\to\Omega)\le1$. Hence, $\pi(i\to\Omega)=1$ also implies that $\pi_R(k\to\Omega)=1$ for every $k\in\Omega$. Conversely, if $\pi_L(i\to\Omega)=\pi_R(k\to\Omega)=1$ for every $k\in\Omega$, then the relation \eqref{eq:pi-fact} becomes $\pi(i\to\Omega)=\sum_{k\in\Omega}\pi_L({i\xrightarrow{\Omega}k})=\pi_L(i\to\Omega)=1$. These arguments prove the equivalence between $\pi(i\to\Omega)=1$ and $\pi_L(i\to\Omega)=\pi_R(k\to\Omega)=1$ for every $k\in\Omega$.

Concerning the expected return time, $\tau(i\to\Omega)=\infty$ when $\pi(i\to\Omega)<1$. This situation corresponds to $\pi_L(i\to\Omega)<1$ or $\pi_R(k\to\Omega)<1$ for some $k\in\Omega$. In any of these cases $\tau_L(i\to\Omega)+\sum_{k\in\Omega}\pi_L(i\xrightarrow{\Omega}k)\,\tau(k\to\Omega)-1=\infty$ with the convention established in the statement {\it(ii)} of the theorem for the indetermination $\pi_L(i\xrightarrow{\Omega}k)\,\tau(k\to\Omega)=0.\infty$.

Assume now that $\pi(i\to\Omega)=1$. Then, using that $\sum_{k\in\Omega}f_L(1)_{i,k}=\pi_L(i\to\Omega)=1$ and $\sum_{j\in\Omega}f_R(1)_{k,j}=\pi_R(k\to\Omega)=1$ for all $k\in\Omega$, we get
$$
\begin{aligned}
 \tau(i\to\Omega) & = 1 + \sum_{j\in\Omega} f'(1)_{j,k} 
 = 1 + \sum_{j\in\Omega} 
 \left(\sum_{k\in\Omega}f'_L(1)_{i,k}f_R(1)_{k,j} + 
 \sum_{k\in\Omega}f_L(1)_{i,k}f'_R(1)_{k,j}\right)
 \\
 & = 1 + \sum_{k\in\Omega} f'_L(1)_{i,k} + 
 \sum_{j\in\Omega}\sum_{k\in\Omega}f_L(1)_{i,k}f'_R(1)_{k,j}
 \\
 & = \tau_L(i\to\Omega) + 
 \sum_{k\in\Omega}\pi_L(i\xrightarrow{\Omega}k)\,(\tau_R(k\to\Omega)-1)
 \\
 & = \tau_L(i\to\Omega) + 
 \sum_{k\in\Omega}\pi_L(i\xrightarrow{\Omega}k)\,\tau_R(k\to\Omega)-1.
\end{aligned}
$$
The only potential problem with the above manipulations arises in the case of an indetermination $\pi_L(i\xrightarrow{\Omega}k)\,\tau_R(k\to\Omega)=0.\infty$ for some $k\in\Omega$. Nevertheless, in this situation $f_L(1)_{i,k}=0$, which gives $\sum_{j\in\Omega}\sum_{k\in\Omega}\displaystyle f_L(1)_{i,k}f'_R(1)_{k,j}=0$ regardless of the value of $\sum_{j\in\Omega}\displaystyle f'_R(1)_{k,j}$. This is in agreement with the convention established for the indetermination in question.         
\end{proof}

Under the hypothesis of Theorem~\ref{thm:split-RW}, we will refer to the RW driven by $\Pi_{L,R}$ as the left/right subsystems in which the RW governed by $\Pi$ splits. We will also say that these left/right subsystems overlap on the subset $\Omega$.

In the case of the recurrence of a single state, the splitting rules for the factorization of Theorem~\ref{thm:split-RW}.{\it(ii)} are particularly simple. 

\begin{cor} \label{cor:split-RW}
With the same notation as in Theorem~\ref{thm:split-RW}, if $\Omega=\{i\}$ we have the following splitting rules for $\Pi = (\Pi_L \oplus 1_+)(1_- \oplus \Pi_R)$:
$$
\begin{aligned}
 & \pi(i\to i) = \pi_L(i\to i) \, \pi_R(i\to i),
 \\
 & \tau(i\to i) = \tau_L(i\to i) + \tau_R(i\to i) - 1.
\end{aligned} 
$$
As a consequence,  
$$
 \pi(i\to i)=1 \quad\Leftrightarrow\quad 
 \pi_L(i\to i)=\pi_R(i\to i)=1.
$$
\end{cor}

The previous results have a remarkable consequence: if a RW splits into subsystems which overlap at a subset $\Omega$, the return probabilities to $\Omega$ are independent of the details of those subsystems for which $\Omega$ has certain recurrence properties. The precise statements of these results are given by the corollary below, but for convenience we first introduce the following recurrence concepts concerning states and subsets of states.

\begin{defn}
Given $i\in\SS$ and $\Omega\subset\SS$, we define the following notions for a RW on the set of states $\SS$:

The state $i\in\SS$ is recurrent if $\pi(i\to i)=1$.

The state $i\in\Omega$ is $\Omega$-recurrent if $\pi(i\to\Omega)=1$.

The subset $\Omega\subset\SS$ is recurrent if all its states are $\Omega$-recurrent. 
\end{defn}

Using this terminology, Theorem~\ref{thm:split-RW} and Corollary~\ref{cor:split-RW} have the following immediate consequences.  

\begin{cor} \label{cor:indep-RW}
With the same notation as in Theorem~\ref{thm:split-RW}, we have the following results:
\begin{itemize}
\vskip7pt
\item[{\it (i)}] 
The return probability $\pi(i\to\Omega)$ for a RW whose stochastic matrix decomposes as $\Pi = (\Pi_L \oplus 0_+) + (0_- \oplus \Pi_R) - (0_- \oplus \Pi_0 \oplus 0_+)$ is independent of the left or right subsystem whenever the state $i$ is $\Omega$-recurrent for such a subsystem. More precisely,
$$
\begin{aligned}
 & \pi_L(i\to\Omega)=1 \quad\Rightarrow\quad \pi(i\to\Omega)=\pi_R(i\to\Omega),
 \\
 & \pi_R(i\to\Omega)=1 \quad\Rightarrow\quad \pi(i\to\Omega)=\pi_L(i\to\Omega).
\end{aligned}
$$
\item[{\it (ii)}] 
The return probabilities to $\Omega$ for a RW whose stochastic matrix factorizes as $\Pi = (\Pi_L \oplus 1_+)(1_- \oplus \Pi_R)$ are independent of the right subsystem whenever the subset $\Omega$ is recurrent for such a subsystem. More precisely,
$$
 \pi_R(i\to\Omega)=1 \quad \forall i\in\Omega \quad\Rightarrow\quad 
 \pi(i\to\Omega)=\pi_L(i\to\Omega) \quad \forall i\in\Omega.
$$

If $\Omega=\{i\}$, a similar independence with respect to the left subsystem holds, i.e.
$$
 \pi_L(i\to i)=1 \quad\Rightarrow\quad \pi(i\to i)=\pi_R(i\to i).
$$ 
Hence, the return probability of the overlapping state $i$ is independent of any of the left/right subsystems for which $i$ is recurrent.
\end{itemize}  
\end{cor}

\subsection{Examples of RW recurrence}
\label{ssec:RW-ex}

We will illustrate the splitting rules for RW recurrence with a few examples. As previously, in what follows $O$ and $I$ stand for the infinite null and identity matrices, while their $n \times n$ versions are denoted by $O_n$ and $I_n$ respectively. When the size of these matrices is undetermined we will denote them by $O$ and $I$. 

\begin{ex} \label{ex:RW1}
The first example is a simple RW on a finite set of states where we will illustrate both the decomposition and factorization rules. The RW in question is given by the following diagram of one-step probability transitions,
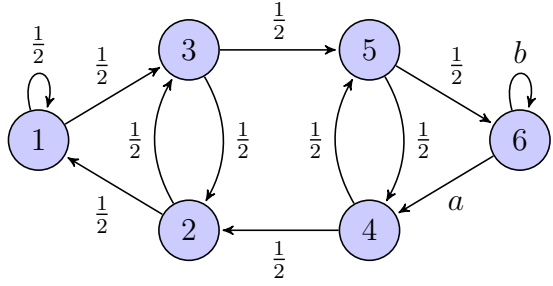
\begin{center}
\begin{tikzpicture}[->,>=stealth',shorten >=1pt,auto,node distance=2.4cm,
                    semithick]
  \tikzstyle{every state}=[fill=black,draw=none,text=white,
  			inner sep=0pt,minimum size=0.2cm,node distance=0.75cm]
  
  \kern-40pt

  \node[main node] (A)              {$1$};
  \node[main node] (B) at (2,1.2) 	{$3$};
  \node[main node] (C) [right of=B] {$5$};
  \node[main node] (D) [below of=B] {$2$};
  \node[main node] (E) [below of=C] {$4$};
  \node[main node] (F) at (6.4,0)   {$6$};
    
  \path (A) edge   [loop above]   	node {$\frac{1}{2}$} (A)
  		(F) edge   [loop above]    	node {$b$} (F)
        (A) edge   []     		  	node {$\frac{1}{2}$ \kern-7pt} (B)
        (B) edge   []  				node {$\frac{1}{2}$} (C)
        (C) edge   []     			node {\kern-3pt $\frac{1}{2}$} (F)
        (F) edge   []     			node {\kern-3pt $a$} (E)
        (E) edge   []     			node {$\frac{1}{2}$} (D)
        (D) edge   []     			node {$\frac{1}{2}$ \kern-7pt} (A)
        (D) edge   [bend left]  	node {$\frac{1}{2}$} (B)
        (B) edge   [bend left]  	node {$\frac{1}{2}$} (D)
        (E) edge   [bend left]  	node {$\frac{1}{2}$} (C)
        (C) edge   [bend left]  	node {$\frac{1}{2}$} (E);
        
  \kern260pt 
  $\begin{aligned}
  & a,b\ge0, 
  \\ 
  & a+b=1,
  \end{aligned}$
  
\end{tikzpicture}
\end{center}
thus it is governed by the stochastic matrix
$$
 \Pi =
 \begin{pmatrix}
 	\frac{1}{2} & 0 & \frac{1}{2} & 0 & 0 & 0
	\\[3pt]
	\frac{1}{2} & 0 & \frac{1}{2} & 0 & 0 & 0
	\\[3pt]
	0 & \frac{1}{2} & 0 & 0 & \frac{1}{2} & 0
	\\[3pt]
	0 & \frac{1}{2} & 0 & 0 & \frac{1}{2} & 0
	\\[3pt]
	0 & 0 & 0 & \frac{1}{2} & 0 & \frac{1}{2}
	\\[3pt]
	0 & 0 & 0 & a & 0 & b  
 \end{pmatrix}.
$$

This matrix can be decomposed as
$$
\begin{aligned}
 \Pi & = 
 \left(
 \begin{array}{cccc|cc}
 	\frac{1}{2} & 0 & \frac{1}{2} & 0 & 0 & 0
	\\[2pt]
	\frac{1}{2} & 0 & \frac{1}{2} & 0 & 0 & 0
	\\[2pt]
	0 & \frac{1}{2} & 0 & \frac{1}{2} & 0 & 0
	\\[2pt]
	0 & \frac{1}{2} & 0 & \frac{1}{2} & 0 & 0
	\\[2pt] \hline
	& & & & & 
	\\[-12pt] 
	0 & 0 & 0 & 0 & 0 & 0
	\\[2pt]
	0 & 0 & 0 & 0 & 0 & 0  
 \end{array}
 \right) 
 +
 \left(
 \begin{array}{cc|cccc}
 	0 & 0 & 0 & 0 & 0 & 0
	\\[2pt]
	0 & 0 & 0 & 0 & 0 & 0
	\\ \hline
	& & & & & 
	\\[-12pt]
	0 & 0 & \frac{1}{2} & 0 & \frac{1}{2} & 0
	\\[2pt] 
	0 & 0 & \frac{1}{2} & 0 & \frac{1}{2} & 0
	\\[2pt]  
	0 & 0 & 0 & \frac{1}{2} & 0 & \frac{1}{2}
	\\[2pt]
	0 & 0 & 0 & a & 0 & b  
 \end{array} 
 \right)
 -
 \left(
 \begin{array}{cc|cc|cc}
 	0 & 0 & 0 & 0 & 0 & 0
	\\[2pt]
	0 & 0 & 0 & 0 & 0 & 0
	\\ \hline
	& & & & &
	\\[-12pt]  
	0 & 0 & \frac{1}{2} & \frac{1}{2} & 0 & 0
	\\[2pt] 
	0 & 0 & \frac{1}{2} & \frac{1}{2} & 0 & 0
	\\[2pt] \hline 
	& & & & & 
	\\[-12pt] 
	0 & 0 & 0 & 0 & 0 & 0
	\\[2pt]
	0 & 0 & 0 & 0 & 0 & 0  
 \end{array} 
 \right) 
 \\
 & =
 \left(
 \begin{array}{c|c} 
 	\\[-5pt]
	\kern7pt \Pi_L \kern9pt 
	\\[7pt] \hline
	&
	\\[-13pt] 
	& \kern1pt O_2 \kern-2pt
 \end{array}
 \right)
 +
 \left(
 \begin{array}{c|c} 
 	O_2 \kern-1pt
 	\\ \hline 
	\\[-5pt]
	& \kern9pt \Pi_R \kern7pt
	\\ [5pt]
 \end{array}
 \right)
 -
 \left(
 \begin{array}{c|c|c} 
 	O_2 \kern-1pt & & 
 	\\ \hline
	& &
	\\[-13pt]
	& \kern2pt \Pi_0 &
	\\ \hline
	& &
	\\[-13pt]
	& & \kern1pt O_2 \kern-2pt 
 \end{array}
 \right),
\end{aligned}
$$  
corresponding to an overlapping decomposition of the RW into left, right and center ones represented by
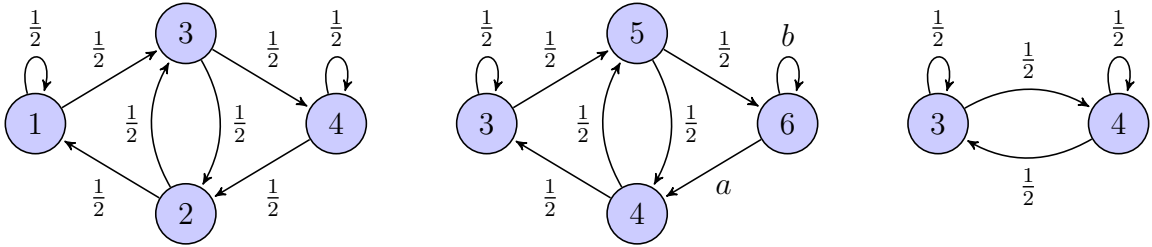
\begin{center}
\begin{tikzpicture}[->,>=stealth',shorten >=1pt,auto,node distance=2.4cm,
                    semithick]
  \tikzstyle{every state}=[fill=black,draw=none,text=white,
  			inner sep=0pt,minimum size=0.2cm,node distance=0.75cm]

  \node[main node] (A)              {$1$};
  \node[main node] (B) at (2,1.2) 	{$3$};
  \node[main node] (D) [below of=B] {$2$};
  \node[main node] (E) at (4,0) 	{$4$};
    
  \path (A) edge   [loop above]   	node {$\frac{1}{2}$} (A)
  		(E) edge   [loop above]   	node {$\frac{1}{2}$} (E)
        (A) edge   []     		  	node {$\frac{1}{2}$ \kern-7pt} (B)
        (B) edge   []  				node {\kern-3pt $\frac{1}{2}$} (E)
        (E) edge   []     			node {\kern-3pt $\frac{1}{2}$} (D)
        (D) edge   []     			node {$\frac{1}{2}$ \kern-7pt} (A)
        (D) edge   [bend left]  	node {$\frac{1}{2}$} (B)
        (B) edge   [bend left]  	node {$\frac{1}{2}$} (D);
        
  \node[main node] (C) at (6,0)     {$3$};
  \node[main node] (F) at (8,1.2) 	{$5$};
  \node[main node] (G) [below of=F] {$4$};
  \node[main node] (H) at (10,0) 	{$6$};
    
  \path (C) edge   [loop above]   	node {$\frac{1}{2}$} (C)
  		(H) edge   [loop above]   	node {$b$} (H)
        (C) edge   []     		  	node {$\frac{1}{2}$ \kern-7pt} (F)
        (F) edge   []  				node {\kern-3pt $\frac{1}{2}$} (H)
        (H) edge   []     			node {\kern-3pt $a$} (G)
        (G) edge   []     			node {$\frac{1}{2}$ \kern-7pt} (C)
        (G) edge   [bend left]  	node {$\frac{1}{2}$} (F)
        (F) edge   [bend left]  	node {$\frac{1}{2}$} (G);
        
  \node[main node] (I) at (12,0)    {$3$};
  \node[main node] (J) [right of=I] {$4$};
    
  \path (I) edge   [loop above]   	node {$\frac{1}{2}$} (I)
  		(J) edge   [loop above]   	node {$\frac{1}{2}$} (J)
        (I) edge   [bend left]     	node {$\frac{1}{2}$} (J)
        (J) edge   [bend left]  	node {$\frac{1}{2}$} (I);
        
\end{tikzpicture}
\end{center}
The overlap is on the states 3 and 4, a fact that, according to Theorems~\ref{thm:split}.{\it(i)} and \ref{thm:split2}.{\it(i)}, guarantees the decomposition $f=f_L+f_R-\Pi_0$ of the FR-function of the subset of states $\Omega=\{3,4\}$ for the original RW, 
\begin{equation} \label{eq:RW1-f-34}
 f(z) = P(\Pi(1-zQ\Pi)^{-1})P,
 \qquad
 P = 
 \left(
 \begin{smallmatrix}
 	\\ 0 \\ & 0 \\ & & 1 \\ & & & 1 \\ & & & & 0 \\ & & & & & 0 \\[2pt] 	
 \end{smallmatrix}
 \right),
 \qquad
 Q = I_6-P, 
\end{equation}
into similar FR-functions for the left and right RW, 
\begin{equation} \label{eq:RW1-fLR-34}
\begin{aligned}
 & f_L(z) = P_L(\Pi_L(1-zQ_L\Pi_L)^{-1})P_L,
 \qquad
 P_L = 
 \left(
 \begin{smallmatrix}
 	\\ 0 \\ & 0 \\ & & 1 \\ & & & 1 \\[1pt]  	
 \end{smallmatrix}
 \right),
 \qquad
 Q_L = I_4-P_L, 
 \\
 & f_R(z) = P_R(\Pi_R(1-zQ_R\Pi_R)^{-1})P_R,
 \qquad
 P_R = 
 \left(
 \begin{smallmatrix}
 	\\ 1 \\ & 1 \\ & & 0 \\ & & & 0 \\[1pt]  	
 \end{smallmatrix}
 \right),
 \qquad
 Q_R = I_4-P_R.
\end{aligned}
\end{equation}
The relation $f=f_L+f_R-\Pi_0$ can be checked by a direct computation of these FR-functions, which leads to the expressions
\begin{equation} \label{eq:RW1-fs-34}
 f(z) = 
 \begin{pmatrix}
 	-\frac{z}{2(z-2)} & \frac{z(2az-z+1)}{4(az-z+1)}
	\\[5pt] 
	-\frac{z}{2(z-2)} & \frac{z(2az-z+1)}{4(az-z+1)}
 \end{pmatrix},
 \quad
 f_L(z) =
 \begin{pmatrix}
 	-\frac{z}{2(z-2)} & \frac{1}{2}
	\\[5pt] 
	-\frac{z}{2(z-2)} & \frac{1}{2}
 \end{pmatrix},
 \quad
 f_R(z) =
 \begin{pmatrix}
 	\frac{1}{2} & \frac{z(2az-z+1)}{4(az-z+1)}
	\\[5pt] 
	\frac{1}{2} & \frac{z(2az-z+1)}{4(az-z+1)}
 \end{pmatrix}.
\end{equation}
Their limits at $z=1$, 
$$
\begin{aligned}
 & \lim_{x\uparrow1} f(x) = \lim_{x\uparrow1} f_L(x) = \lim_{x\uparrow1} f_R(x) 
 =
 \begin{pmatrix} 
 	\frac{1}{2} & \frac{1}{2} \\[3pt] \frac{1}{2} & \frac{1}{2} 
 \end{pmatrix}
 & \qquad & \text{if } a\ne0,
 \\
 & \lim_{x\uparrow1} f(x) = \lim_{x\uparrow1} f_R(x) =
 \begin{pmatrix} 
 	\frac{1}{2} & \frac{1}{4} \\[3pt] \frac{1}{2} & \frac{1}{4} 
 \end{pmatrix},
 \qquad
 \lim_{x\uparrow1} f_L(x) =
 \begin{pmatrix} 
 	\frac{1}{2} & \frac{1}{2} \\[3pt] \frac{1}{2} & \frac{1}{2} 
 \end{pmatrix},
 & & \text{if } a=0,
\end{aligned}
$$
give the probabilities of returning to $\Omega=\{3,4\}$ landing on a particular state, so that for $i,j\in\Omega$,
$$
 \pi(i\xrightarrow{\Omega}j) = \pi_L(i\xrightarrow{\Omega}j) = 
 \pi_R(i\xrightarrow{\Omega}j) = \frac{1}{2} 
 \quad \text{except for} \quad
 \pi(i\xrightarrow{\Omega}4) = \pi_R(i\xrightarrow{\Omega}4) = 
 \begin{cases} 
 	\frac{1}{2} & \text{if } a\ne0,
	\\ 
	\frac{1}{4} & \text{if } a=0.
 \end{cases}
$$
The sum of the rows of such limits provide the return probabilities to the set $\Omega=\{3,4\}$, given for any $i\in\Omega$ by
$$
 \pi_L(i\to\Omega) = 1,
 \qquad\quad
 \pi(i\to\Omega) = \pi_R(i\to\Omega) = 
 \begin{cases} 
 	1 & \text{if } a\ne0,
	\\ 
	\frac{3}{4} & \text{if } a=0.
 \end{cases}
$$
Therefore, the subset $\Omega$ is recurrent for the left RW and, if $a\ne0$, also for the original and right RW.    

In the recurrent cases, the expected return times are given by the sums of the rows of the following matrices
$$
 \lim_{x\uparrow1}(xf(x))' = 
 \begin{pmatrix}
 	\frac{3}{2} & 1+\frac{1}{4a}
	\\[3pt] 
	\frac{3}{2} & 1+\frac{1}{4a}
 \end{pmatrix},
 \qquad
 \lim_{x\uparrow1}(xf_L(x))' =
 \begin{pmatrix}
 	\frac{3}{2} & \frac{1}{2}
	\\[3pt] 
	\frac{3}{2} & \frac{1}{2}
 \end{pmatrix},
 \qquad
 \lim_{x\uparrow1}(xf_R(x))' =
 \begin{pmatrix}
 	\frac{1}{2} & 1+\frac{1}{4a}
	\\[3pt] 
	\frac{1}{2} & 1+\frac{1}{4a}
 \end{pmatrix},
$$ 
which yield for $i\in\Omega$
$$
 \tau(i\to\Omega) = \frac{5}{2}+\frac{1}{4a},
 \qquad
 \tau_L(i\to\Omega) = 2,
 \qquad
 \tau_R(i\to\Omega) = \frac{3}{2}+\frac{1}{4a}.
$$

The sum of each row of $f$ yields the FR-function for the returns to $\Omega$, which coincides for the states 3 and 4. The coefficients of its power expansion around the origin provide the probability of returning to $\Omega$ for the first time in $n$ steps, given for $i=3,4$ by   
$$
 \pi_n(i\to\Omega) = 
 \begin{cases}
 	0 & \text{ if } n=1,
	\\
	\frac{1}{2} & \text{ if } n=2,
	\\
	\frac{1}{2^n} + \frac{1}{4}a(1-a)^{n-3}  & \text{ if } n\ge3.
 \end{cases}
$$

These results can be inferred also from path counting, but the splitting techniques become an invaluable tool for the analysis of recurrence in more complex RW, like those described in the next examples.

The decomposition rules do not work for the return properties of the single state 4 because there is no decomposition of $\Pi$ with overlaps only at such state. However, the factorization rules are available for that purpose due to the overlapping factorization
\begin{equation} \label{eq:RW1-fact}
 \Pi =
 \left(
 \begin{array}{cccc|cc}
 	\frac{1}{2} & 0 & \frac{1}{2} & 0 & 0 & 0
	\\[2pt]
	\frac{1}{2} & 0 & \frac{1}{2} & 0 & 0 & 0
	\\[2pt]
	0 & \frac{1}{2} & 0 & \frac{1}{2} & 0 & 0
	\\[2pt]
	0 & \frac{1}{2} & 0 & \frac{1}{2} & 0 & 0
	\\[2pt] \hline
	& & & & & 
	\\[-12pt] 
	0 & 0 & 0 & 0 & 1 & 0
	\\[2pt]
	0 & 0 & 0 & 0 & 0 & 1  
 \end{array}
 \right)
 \left(
 \begin{array}{ccc|ccc}
 	1 & 0 & 0 & 0 & 0 & 0
	\\[2pt] 
	0 & 1 & 0 & 0 & 0 & 0
	\\[2pt] 
	0 & 0 & 1 & 0 & 0 & 0
	\\ \hline 
	& & & & & 
	\\[-12pt] 
	0 & 0 & 0 & 0 & 1 & 0
	\\[2pt] 
	0 & 0 & 0 & \frac{1}{2} & 0 & \frac{1}{2}
	\\[2pt] 
	0 & 0 & 0 & a & 0 & b
 \end{array}
 \right)
 = 
 \left(
 \begin{array}{c|c} 
 	\\[-5pt]
	\kern7pt \Pi_L \kern9pt 
	\\[7pt] \hline
	& 
	\\[-13pt] 
	& \kern1pt I_2 \kern-2pt
 \end{array}
 \right)
 \left(
 \begin{array}{c|c} 
 	I_3 
 	\\ \hline 
	\\[-5pt]
	& \kern9pt \Pi_R \kern7pt
	\\ [5pt]
 \end{array}
 \right)
\end{equation}
into the left and right RW diagrammatically represented by
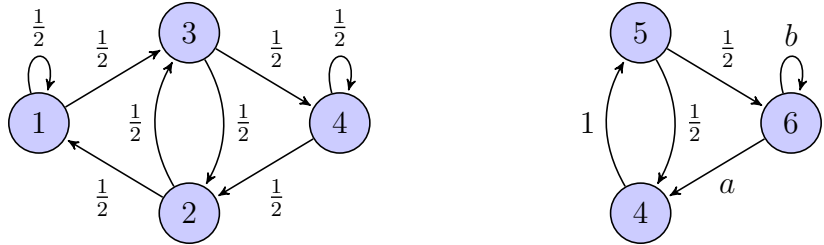
\begin{center}
\begin{tikzpicture}[->,>=stealth',shorten >=1pt,auto,node distance=2.4cm,
                    semithick]
  \tikzstyle{every state}=[fill=black,draw=none,text=white,
  			inner sep=0pt,minimum size=0.2cm,node distance=0.75cm]

  \node[main node] (A)              {$1$};
  \node[main node] (B) at (2,1.2) 	{$3$};
  \node[main node] (D) [below of=B] {$2$};
  \node[main node] (E) at (4,0) 	{$4$};
    
  \path (A) edge   [loop above]   	node {$\frac{1}{2}$} (A)
  		(E) edge   [loop above]   	node {$\frac{1}{2}$} (E)
        (A) edge   []     		  	node {$\frac{1}{2}$ \kern-7pt} (B)
        (B) edge   []  				node {\kern-3pt $\frac{1}{2}$} (E)
        (E) edge   []     			node {\kern-3pt $\frac{1}{2}$} (D)
        (D) edge   []     			node {$\frac{1}{2}$ \kern-7pt} (A)
        (D) edge   [bend left]  	node {$\frac{1}{2}$} (B)
        (B) edge   [bend left]  	node {$\frac{1}{2}$} (D);
        
  \node[main node] (F) at (8,1.2) 	{$5$};
  \node[main node] (G) [below of=F] {$4$};
  \node[main node] (H) at (10,0) 	{$6$};
    
  \path (H) edge   [loop above]   	node {$b$} (H)
        (F) edge   []  				node {\kern-3pt $\frac{1}{2}$} (H)
        (H) edge   []     			node {\kern-3pt $a$} (G)
        (G) edge   [bend left]  	node {$1$} (F)
        (F) edge   [bend left]  	node {$\frac{1}{2}$} (G);
        
\end{tikzpicture}
\end{center}
When applied to this overlapping factorization, Theorem~\ref{thm:split}.{\it(ii)} provides the factorization $f=f_Lf_R$ for the FR-functions of the state 4, where such FR-functions are as in \eqref{eq:RW1-f-34}, \eqref{eq:RW1-fLR-34}, but substituting $\Pi_R$ by the new right stochastic matrix and the projections $P$, $P_L$, $P_R$ by 
$$
 P = 
 \left(
 \begin{smallmatrix}
 	\\ 0 \\ & 0 \\ & & 0 \\ & & & 1 \\ & & & & 0 \\ & & & & & 0 \\[2pt] 	
 \end{smallmatrix}
 \right),
 \qquad
 P_L = 
 \left(
 \begin{smallmatrix}
 	\\ 0 \\ & 0 \\ & & 0 \\ & & & 1 \\[2pt]	
 \end{smallmatrix}
 \right),
 \qquad
 P_R = 
 \left(
 \begin{smallmatrix}
 	\\[-0.5pt] 1 \\ & 0 \\ & & 0 \\[2pt]	
 \end{smallmatrix}
 \right).
$$
The result 
$$
 f(z) = \frac{z(z-2)(2az-z+1)}{2(az-z+1)(z^2+2z-4)},
 \qquad
 f_L(z) = \frac{z-2}{z^2+2z-4},
 \qquad
 f_R(z) = \frac{z(2az-z+1)}{2(az-z+1)},
$$
makes evident the factorization $f=f_Lf_R$. From the FR-functions we find the return probabilities of the state 4 taking limits at $z=1$,
$$
 \pi_L(4\to4) = 1,
 \qquad\quad
 \pi(4\to4) = \pi_R(4\to4) = 
 \begin{cases}
 	1 & \text{ if } a\ne0,
 	\\
 	\frac{1}{2} & \text{ if } a=0.
 \end{cases}
$$
Hence, the sate 4 is recurrent for the left RW and, if $a\ne0$, also for the original and right RW. 

The limit at $z=1$ of the derivatives of the FR-functions provide the expected return times in the recurrent situations, 
$$
 \tau(4\to4) = 5+\frac{1}{2a},
 \qquad
 \tau_L(4\to4) = 4,
 \qquad
 \tau_R(4\to4) = 2+\frac{1}{2a}.
$$

The overlapping factorization \eqref{eq:RW1-fact} can be also used to factorize the FR-function of the subset of states $\Omega=\{3,4\}$. For this we must simply enlarge the right RW including the site 3, which amounts to reconsider the factorization \eqref{eq:RW1-fact} in the following way 
$$
 \Pi =
 \left(
 \begin{array}{cccc|cc}
 	\frac{1}{2} & 0 & \frac{1}{2} & 0 & 0 & 0
	\\[2pt]
	\frac{1}{2} & 0 & \frac{1}{2} & 0 & 0 & 0
	\\[2pt]
	0 & \frac{1}{2} & 0 & \frac{1}{2} & 0 & 0
	\\[2pt]
	0 & \frac{1}{2} & 0 & \frac{1}{2} & 0 & 0
	\\[2pt] \hline
	& & & & & 
	\\[-12pt] 
	0 & 0 & 0 & 0 & 1 & 0
	\\[2pt]
	0 & 0 & 0 & 0 & 0 & 1  
 \end{array}
 \right)
 \left(
 \begin{array}{cc|cccc}
 	1 & 0 & 0 & 0 & 0 & 0
	\\[2pt] 
	0 & 1 & 0 & 0 & 0 & 0
	\\ \hline
	& & & & & 
	\\[-12pt] 
	0 & 0 & 1 & 0 & 0 & 0
	\\[2pt] 
	0 & 0 & 0 & 0 & 1 & 0
	\\[2pt] 
	0 & 0 & 0 & \frac{1}{2} & 0 & \frac{1}{2}
	\\[2pt] 
	0 & 0 & 0 & a & 0 & b
 \end{array}
 \right)
 = 
 \left(
 \begin{array}{c|c} 
 	\\[-5pt]
	\kern7pt \Pi_L \kern9pt 
	\\[7pt] \hline
	& 
	\\[-13pt] 
	& \kern1pt I_2 \kern-2pt
 \end{array}
 \right)
 \left(
 \begin{array}{c|c} 
 	I_2  
 	\\ \hline 
	\\[-5pt]
	& \kern9pt \Pi_R \kern7pt
	\\ [5pt]
 \end{array}
 \right).
$$
Then, the FR-functions for the subset $\Omega=\{3,4\}$ are the same as in \eqref{eq:RW1-fs-34} except for that of the right RW,  
$$
 f(z) = 
 \begin{pmatrix}
 	-\frac{z}{2(z-2)} & \frac{z(2az-z+1)}{4(az-z+1)}
	\\[5pt] 
	-\frac{z}{2(z-2)} & \frac{z(2az-z+1)}{4(az-z+1)}
 \end{pmatrix},
 \qquad
 f_L(z) =
 \begin{pmatrix}
 	-\frac{z}{2(z-2)} & \frac{1}{2}
	\\[5pt] 
	-\frac{z}{2(z-2)} & \frac{1}{2}
 \end{pmatrix},
 \qquad
 f_R(z) =
 \begin{pmatrix}
 	1 & 0 
	\\ 
	0 & \frac{z(2az-z+1)}{2(az-z+1)}
 \end{pmatrix}.
$$
This makes explicit the factorization $f=f_Lf_R$.  

The return probabilities and expected return times for the subset $\Omega=\{3,4\}$ with respect to the new right RW follow from the above FR-function $f_R$ as in the previous cases, 
$$
\begin{aligned}
 & \pi_R(3\to\Omega) = 1, 
 & \qquad & \pi_R(4\to\Omega) =
 \begin{cases}
 	1 & \text{ if } a\ne0,
 	\\
 	\frac{1}{2} & \text{ if } a=0,
 \end{cases}
 \\
 & \tau_R(3\to\Omega) = 1, 
 & & \tau_R(4\to\Omega) = 2+\frac{1}{2a}.
\end{aligned}
$$

It is instructive to check in this simple example the general splitting rules given in Theorem~\ref{thm:split-RW}, which in more complicated situations become an useful tool to deal with recurrence properties of RW.
\end{ex}

\begin{ex} \label{ex:RW2}
We will use the birth-death processes to illustrate the use of Khrushchev formula for OPRL obtained in the previous section as an application of the splitting rules for FR-functions. This is possible because birth-death processes, represented by tridiagonal matrices, are always symmetrizable and thus lead to OPRL via Jacobi matrices. Actually, the fact that such Jacobi matrices define self-adjoint contractions on a Hilbert space implies that the related FR-functions are not only Schur functions, but also Nevanlinna functions. 

Consider the RW on the non-negative integers $\SS=\{0,1,2,\dots\}$ given by the  tridiagonal stochastic matrix
$$
 \Pi = 
 \begin{pmatrix}
 	b_0 & \kern-2pt q_0
	\\[2pt]
	p & \kern-2pt 0 & q
	\\[2pt]
	& \kern-2pt p & 0 & \kern3pt q
	\\[2pt]
	& & p & \kern3pt 0 & \kern3pt q
	\\
	& & & \kern-5pt \ddots & \ddots & \ddots
 \end{pmatrix},
 \qquad
 \begin{aligned}
 	& p,q,q_0>0, \quad b_0\ge0,
	\\[5pt]
 	& p+q=1, \quad b_0+q_0=1, 
 \end{aligned}
$$
which we can represent diagramatically as
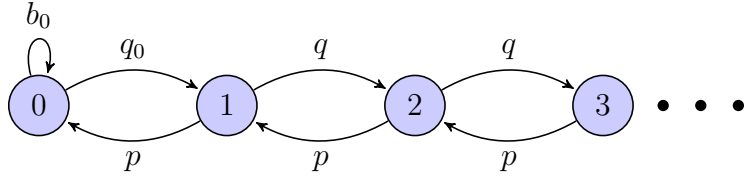
\begin{center}
\begin{tikzpicture}[->,>=stealth',shorten >=1pt,auto,node distance=2.5cm,
                    semithick]
  \tikzstyle{every state}=[fill=black,draw=none,text=white,
  			inner sep=0pt,minimum size=0.15cm,node distance=0.5cm]

  \node[main node] (A)              {$0$};
  \node[main node] (B) [right of=A] {$1$};
  \node[main node] (C) [right of=B] {$2$};
  \node[main node] (D) [right of=C] {$3$};
  
  \node[state] (E) at (8.3,0)  	{};
  \node[state] (F) [right of=E]	{};
  \node[state] (G) [right of=F]	{};
  
  \path (A) edge   [loop above]    node {$b_0$} (A)
        (A) edge   [bend left]     node {$q_0$} (B)
        (B) edge   [bend left]     node {$p$} (A)
        (B) edge   [bend left]     node {$q$} (C)
        (C) edge   [bend left]     node {$p$} (B)
        (C) edge   [bend left]     node {$q$} (D)
        (D) edge   [bend left]     node {$p$} (C);
\end{tikzpicture}
\end{center}  
We can symmetrize $\Pi$ by conjugation with a diagonal positive matrix $\Lambda$ so that $\mc{J}=\Lambda\Pi\Lambda^{-1}$ is a Jacobi matrix,
$$
 \mc{J} = 
 \begin{pmatrix}
 	b_0 & \kern-2pt a_0
	\\[2pt]
	a_0 & \kern-2pt 0 & a
	\\[2pt]
	& \kern-2pt a & 0 & \kern3pt a
	\\[2pt]
	& & a & \kern3pt 0 & \kern3pt a
	\\
	& & & \kern-5pt \ddots & \ddots & \ddots
 \end{pmatrix},
 \qquad
 \begin{aligned}
 	& a_0=\sqrt{pq_0},
	\\[3pt]
 	& a=\sqrt{pq}. 
 \end{aligned}
$$
The fact that $\Lambda$ is diagonal guarantees that $\Pi$ and $\mc{J}$ have the same FR-function $f(n;z)$ for each single state $n$,
$$
 f(n;z) = (\Pi(1-zQ^{(n)}\Pi)^{-1})_{n,n} 
 = (\mc{J}(1-zQ^{(n)}\mc{J})^{-1})_{n,n},
 \qquad
 Q^{(n)}_{i,j}=
 \begin{cases}
 	0 & i=j=n,
	\\
	1 & \text{otherwise}.
 \end{cases}
$$
Consequently, $f(n;z)$ is simultaneously a Schur and a Nevanlinna function. Bearing in mind the comments in Section~\ref{sec:OP}, $f(n;z)$ is the Nevanlinna function of $p_n^2\,d\mu$, where $p_n$ is the OPRL of degree $n$ generated by the three term recurrence relation associated with $\mc{J}$ and $\mu$ is the related orthogonality measure on the real line. The values of the FR-function $f(n;z)$ and its derivative at $z=1$ give the return probability and expected return time for the state $n$. 

The FR-function $f(z):=f(0;z)$ for the first state follows easily from the fact that it is a Nevanlinna function with Schur parameters $b_0,a_0,0,a,0,a,\dots$ with respect to the Schur algorithm \eqref{eq:Nalg} at the origin, so that its first iterate $f_1$ has constant Schur parameters $0,a,0,a,0,a,\dots$. From Example~\ref{ex:SA4} we find that 
\begin{equation} \label{eq:f1-BD}
 f_1(z) = \frac{1-\sqrt{1-4a^2z^2}}{2z},
\end{equation}
where the branch for the square root is that one taking the value 1 at $z=0$. Therefore,
\begin{equation} \label{eq:f-BD}
 f(z) = b_0 + \frac{a_0^2z}{1-zf_1(z)} = b_0 + \frac{a_0^2}{a^2} f_1(z).
\end{equation}
Using that $\sqrt{1-4a^2}=\sqrt{1-4p(1-p)}=|1-2p|=|p-q|$ we get
$$
 f_1(1) = \frac{1-|p-q|}{2} = \min\{p,q\},
 \qquad
 \lim_{x\uparrow1}(xf_1(x))' = \frac{2pq}{|p-q|}. 
$$
Hence, the return probability for the first state is given by 
$$
 \pi(0\to0) = f(1) = b_0 + \frac{q_0}{q}\min\{p,q\} = 
 \begin{cases}
 	1 & \text{ if } p\ge q, 
	\\
	1-q_0(1-\frac{p}{q}) & \text{ if } p<q,
 \end{cases}
$$ 
and, in the recurrent case $p\ge q$, the return to the first state happens in an expected return time
$$
 \tau(0\to0) = \lim_{x\uparrow1}(xf(x))' = 
 b_0 + \frac{2pq_0}{p-q} = 1 + \frac{q_0}{p-q}.
$$

However, the recurrence properties of an arbitrary state $n\ge1$ of such a semi-infinite chain are not so easy to obtain. Here, the splitting rules for FR-functions come to our rescue via Khrushchev formula for OPRL, which states that $f(n;z) = g_n(z) + f_n(z)$, where $f_n$ is the $n$-th iterate of the Nevanlinna function $f$ for the Schur algorithm \eqref{eq:Nalg} at the origin, while $g_n(z)=\kappa_{n-1}p_{n-1}(z^{-1})/\kappa_n p_n(z^{-1})$ and $\kappa_n=1/a_0a^{n-1}$ is the leading coefficient of $p_n$. According to the discussion of Section~\ref{sec:OP}, the splitting $f(n;z)=(g_n(z)+q)+(p+f_n(z))-1$ is associated with the overlapping decomposition $\Pi=(\Pi_L \oplus O) + (O_n \oplus \Pi_R) - (O_n \oplus 1 \oplus O)$, where $\Pi_L$ and $\Pi_R$ are stochastic matrices diagrammatically depicted as
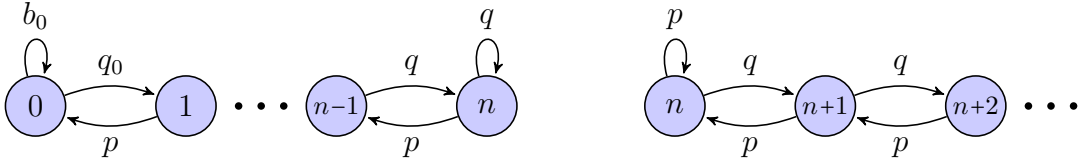
\begin{center}
\begin{tikzpicture}[->,>=stealth',shorten >=1pt,auto,node distance=2cm,
                    semithick]
  \tikzstyle{every state}=[fill=black,draw=none,text=white,
  			inner sep=0pt,minimum size=0.1cm,node distance=0.3cm]

  \node[main node] (BB) 			 	{$n$};
  \node[main node] (C) 	[right of=BB] 	{\footnotesize 
                   						$n \kern-1.5pt + \kern-2pt 1$};
  \node[main node] (D) 	[right of=C] 	{\footnotesize 
                   						$n \kern-1.5pt + \kern-2pt 2$};
  \node[main node] (Z)  at (-2.5,0) 	{$n$};
  \node[main node] (Y)  [left of=Z]      {\footnotesize 
                   						$n \kern-1.5pt - \kern-2pt 1$};
  \node[main node] (X)  [left of=Y] 		{$1$};
  \node[main node] (W)  [left of=X] 	{$0$};
  
  \node[state] (E) at (4.7,0)  	{};
  \node[state] (F) [right of=E]	{};
  \node[state] (G) [right of=F]	{};
  \node[state] (e) at (-5.2,0)  {};
  \node[state] (f) [left of=e]	{};
  \node[state] (g) [left of=f]	{};

  \path (W)  edge   [loop above]    	node {$b_0$} (W)
  		(W)  edge   [bend left=20]    	node {$q_0$} (X)
		(X)  edge   [bend left=20]    	node {$p$} (W)
		(Y)  edge   [bend left=20]    	node {$q$} (Z)
		(Z)  edge   [bend left=20]    	node {$p$} (Y)
		(Z)  edge   [loop above]    	node {$q$} (Z)
        (BB) edge   [loop above]    	node {$p$} (B)
        (BB) edge   [bend left=20]     	node {$q$} (C)
        (C)  edge   [bend left=20]     	node {$p$} (BB)
        (C)  edge   [bend left=20]     	node {$q$} (D)
        (D)  edge   [bend left=20] 		node {$p$} (C);
\end{tikzpicture}
\end{center}

The iterates $f_n$ coincide with $f_1$ because they have the same Schur parameters for \eqref{eq:Nalg}. As for the Nevanlinna function $g_n$, we need the OPRL $p_n$.

The OPRL $p_n$ satisfy for $n\ge1$ the same three term recurrence relation as the rescaled Chebyshev polynomials of the second kind $U_n(x/2a)$, where $U_n$ is given by the recurrence 
$$
 2xU_n(x) = U_{n+1}(x)+U_{n-1}(x), \qquad U_0=1, \qquad U_{-1}=0.
$$
This allows to express
$$
 p_n(x) = 
 \frac{x-b_0}{a_0} \, U_{n-1}({\textstyle\frac{x}{2a}}) 
 - \frac{a_0}{a} \, U_{n-2}({\textstyle\frac{x}{2a}}),
 \qquad n\ge1,
$$
so that 
\begin{equation} \label{eq:gn}
 g_n(z) = 
 \begin{cases}
 	\frac{a_0^2z}{1-b_0z} 
	& \text{ if } n=1,
	\\[3pt] 
	a \frac{a(1-b_0z)\,U_{n-2}(1/2az) - a_0^2z\,U_{n-3}(1/2az)}
	{a(1-b_0z)\,U_{n-1}(1/2az) - a_0^2z\,U_{n-2}(1/2az)} 
	& \text{ if } n\ge2.
 \end{cases}
\end{equation} 
Using the properties of the Chebyshev polynomials $U_n$ and their derivatives we find that
$$
 g_n(1) = p, 
 \qquad 
 \lim_{x\uparrow1}g_n'(x) = \frac{p}{q_0} w^{n-2} 
 \left(w^n + \frac{q_0}{q} \, U_{n-2}({\textstyle\frac{1}{2a}})\right),
 \qquad w = \sqrt{\frac{p}{q}},
 \qquad n\ge1.
$$
Therefore, for any state $n\ge1$, the return probability is given by
$$
 \pi(n\to n) = g_n(1) + f_n(1) = p + \min\{p,q\} =
 \begin{cases}
 	1 & \text{ if } p\ge q, 
	\\
	2p & \text{ if } p<q,
 \end{cases}
$$
and, in the recurrent case $p\ge q$, the corresponding expected return time has the expression
$$
\begin{aligned}
 \tau(n\to n) & = \lim_{x\uparrow1}(xg_n(x))' + \lim_{x\uparrow1}(xf_n(x))' =
 p + \frac{p}{q_0} w^{n-2} 
 \left(w^n + \frac{q_0}{q} \, U_{n-2}({\textstyle\frac{1}{2a}})\right)
 + \frac{2pq}{p-q}
 \\[2pt] 
 & = \frac{p}{q_0} \left[w^{n-2} 
 \left(w^n + \frac{q_0}{q} \, U_{n-2}({\textstyle\frac{1}{2a}})\right)
 + \frac{q_0}{p-q}\right].
\end{aligned}
$$

The splitting rules are also useful to obtain the FR-function of a subset of states in this example. For instance, the subset $\Omega=\{0,1\}$ naturally generates an overlapping splitting of the RW
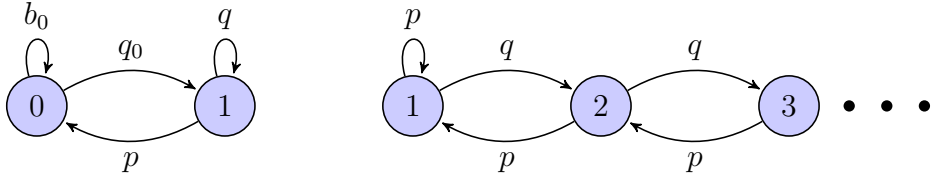
\begin{center}
\begin{tikzpicture}[->,>=stealth',shorten >=1pt,auto,node distance=2.5cm,
                    semithick]
  \tikzstyle{every state}=[fill=black,draw=none,text=white,
  			inner sep=0pt,minimum size=0.15cm,node distance=0.5cm]

  \node[main node] (A)               	{$0$};
  \node[main node] (B) 	[right of=A] 	{$1$};
  \node[main node] (BB) [right of=B] 	{$1$};
  \node[main node] (C) 	[right of=BB] 	{$2$};
  \node[main node] (D) 	[right of=C] 	{$3$};
  
  \node[state] (E) at (10.8,0)  {};
  \node[state] (F) [right of=E]	{};
  \node[state] (G) [right of=F]	{};
  
  \path (A) edge   [loop above]    node {$b_0$} (A)
        (A) edge   [bend left]     node {$q_0$} (B)
        (B) edge   [bend left]     node {$p$} (A)
        (B) edge   [loop above]    node {$q$} (B)
        (BB) edge  [loop above]    node {$p$} (B)
        (BB) edge  [bend left]     node {$q$} (C)
        (C) edge   [bend left]     node {$p$} (BB)
        (C) edge   [bend left]     node {$q$} (D)
        (D) edge   [bend left]     node {$p$} (C);
\end{tikzpicture}
\end{center}
which originates the following overlapping decomposition of the stochastic matrix $\Pi$, 
$$
\begin{aligned}
 \Pi & =
 \left(
 \begin{array}{cc|ccc}
 	b_0 & \kern-4pt q_0 & & &  
 	\\
	p & \kern-4pt q & & &  
	\\ \hline
	& & & &   
	\\[-13pt]
	& & 0 & &   
	\\
	& & & 0 &    
	\\
	& & & & \kern-3pt \ddots
 \end{array}
 \right)
 +
 \left(
 \begin{array}{c|ccccc} 
 	0 & & & & &
	\\ \hline
	& p & q & & & 
	\\
	& p & 0 & q & &
	\\
	& & p & 0 & q &
	\\
	& & & \kern-5pt \ddots & \kern-3pt \ddots & \kern-2pt \ddots 
 \end{array}
 \right)
 - \left(
 \begin{array}{c|c|ccc} 
 	0 & & & & 
	\\ \hline
	& 1 & & &  
	\\ \hline
	& & & &  
	\\[-13pt]
	& & 0 & & 
	\\
	& & & 0 & 
	\\
	& & & & \kern-3pt \ddots 
 \end{array}
 \right)
 \\
 & =
 \left(
 \begin{array}{c|c} 
 	\\[-9pt]
	\kern1pt \Pi_0 & 
	\\[3pt] \hline
	& 
	\\[-9pt] 
	& \kern5pt O \kern3pt
	\\[3pt]
 \end{array}
 \right)
 +
 \left(
 \begin{array}{c|c} 
 	\text{\footnotesize 0} &  
 	\\ \hline 
	\\[-5pt]
	& \kern10pt \Pi_R \kern6pt
	\\[6pt]
 \end{array}
 \right)
 -
 \left(
 \begin{array}{cc|c} 
 	\text{\footnotesize 0} & \kern-3pt \text{\footnotesize 0} &
	\\[-4pt]
	\text{\footnotesize 0} & \kern-3pt \text{\footnotesize 1} &
	\\[-0.5pt] \hline
	& &
	\\[-9pt]
	& & \kern5pt O \kern3pt
	\\[3pt] 
 \end{array}
 \right).
\end{aligned}
$$  
Hence, using Theorem~\ref{thm:split-RW}.{\it(i)} we find that the FR-function $h$ of the subset $\Omega=\{0,1\}$ with respect to $\Pi$ decomposes as $h=h_R+\Pi_0-\left(\begin{smallmatrix}0&0\\0&1\end{smallmatrix}\right)$, where $h_R$ is the FR-function of the same subset with respect to $0\oplus\Pi_R$. Obviously, $h_R=0\oplus f$, with $f$ the FR-function of the single state 1 with respect to $\Pi_R$, which is given by \eqref{eq:f-BD} with $b_0=p$ and $a_0=a$. Therefore, $f(z)=p+f_1(z)$ with $f_1$ as in \eqref{eq:f1-BD}, so we conclude that  
$$
 h(z) = 
 \begin{pmatrix}
 	b_0 & q_0
	\\
	p & f_1(z)
 \end{pmatrix}.
$$
The values
$$
 h(1) = 
 \begin{pmatrix}
 	b_0 & q_0
	\\[1pt]
	p & \min\{p,q\}
 \end{pmatrix},
 \qquad\quad 
 \lim_{x\uparrow1} (xh(x))' = 
 \begin{pmatrix}
 	b_0 & q_0
	\\[1pt]
	p & \frac{2pq}{|p-q|}
 \end{pmatrix},
$$ 
provide the following return probabilities and expected return times to $\Omega=\{0,1\}$,
$$
\begin{aligned}
 & \pi(0\to\Omega) = 1, 
 & \qquad & \pi(1\to\Omega) = 
 \begin{cases}
 	1 & \text{ if } p\ge q,
	\\
	2p & \text{ if } p<q,
 \end{cases}
 \\[1pt]
 & \tau(0\to\Omega) = 1,
 & & \tau(1\to\Omega) = \frac{p}{p-q} \kern6pt \text{ if } p\ge q.
\end{aligned}
$$

When $\Omega=\{n,n+1\}$ with $n\ge1$, a similar decomposition based on the splitting 
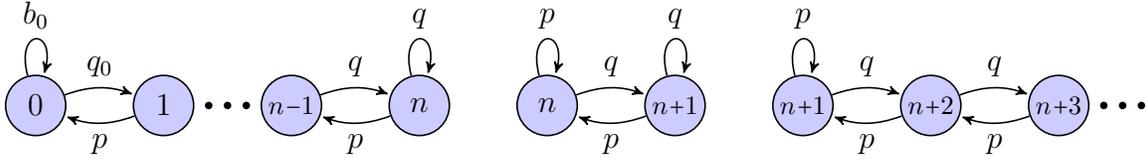
\begin{center}
\begin{tikzpicture}[->,>=stealth',shorten >=1pt,auto,node distance=1.7cm,
                    semithick]
  \tikzstyle{every state}=[fill=black,draw=none,text=white,
  			inner sep=0pt,minimum size=0.1cm,node distance=0.25cm]

  \node[main node] (A)               	{$n$};
  \node[main node] (B) 	[right of=A] 	{\footnotesize 
                   						$n \kern-1.5pt + \kern-2pt 1$};
  \node[main node] (BB) [right of=B] 	{\footnotesize 
                   						$n \kern-1.5pt + \kern-2pt 1$};
  \node[main node] (C) 	[right of=BB] 	{\footnotesize 
                   						$n \kern-1.5pt + \kern-2pt 2$};
  \node[main node] (D) 	[right of=C] 	{\footnotesize 
                   						$n \kern-1.5pt + \kern-2pt 3$};
  \node[main node] (Z)  [left of=A] 	{$n$};
  \node[main node] (Y)  [left of=Z]      {\footnotesize 
                   						$n \kern-1.5pt - \kern-2pt 1$};
  \node[main node] (X)  [left of=Y] 		{$1$};
  \node[main node] (W)  [left of=X] 	{$0$};
  
  \node[state] (E) at (7.4,0)  {};
  \node[state] (F) [right of=E]	{};
  \node[state] (G) [right of=F]	{};
  \node[state] (e) at (-4,0)  	{};
  \node[state] (f) [left of=e]	{};
  \node[state] (g) [left of=f]	{};

  \path (W)  edge   [loop above]    	node {$b_0$} (W)
  		(W)  edge   [bend left=20]    	node {$q_0$} (X)
		(X)  edge   [bend left=20]    	node {$p$} (W)
		(Y)  edge   [bend left=20]    	node {$q$} (Z)
		(Z)  edge   [bend left=20]    	node {$p$} (Y)
		(Z)  edge   [loop above]    	node {$q$} (Z)
        (A)  edge   [loop above]    	node {$p$} (A)
        (A)  edge   [bend left=20]     	node {$q$} (B)
        (B)  edge   [bend left=20]     	node {$p$} (A)
        (B)  edge   [loop above]    	node {$q$} (B)
        (BB) edge   [loop above]    	node {$p$} (B)
        (BB) edge   [bend left=20]     	node {$q$} (C)
        (C)  edge   [bend left=20]     	node {$p$} (BB)
        (C)  edge   [bend left=20]     	node {$q$} (D)
        (D)  edge   [bend left=20] 		node {$p$} (C);
\end{tikzpicture}
\end{center}
identifies the corresponding FR-function with respect to $\Pi$ as
$$
 h(n;z) = 
 \begin{pmatrix}
 g_n(z) & q \\ p & f_n(z)
 \end{pmatrix},
$$
where $f_n=f_1$ is the $n$-th iterate of $f$ for the Schur algorithm \eqref{eq:Nalg} and $g_n$ is given by \eqref{eq:gn}. This yields for $\Omega=\{n,n+1\}$ exactly the same return probabilities and expected return times as in the case of the subset $\{0,1\}$, except for the expected return time of the state $n$ to $\Omega=\{n,n+1\}$, given by
$$
 \tau(n\to\Omega) = \lim_{x\uparrow1}(xg_n(x))' + q = 1 + q +
 \frac{p}{q_0} w^{n-2} 
 \left(w^n + \frac{q_0}{q} \, U_{n-2}({\textstyle\frac{1}{2a}})\right),
 \qquad n\ge1.
$$
\end{ex}

\begin{ex} \label{ex:RW3} 
Suppose that a RW on a set $\SS=\Omega_1\cup\{0\}\cup\Omega_2$ has a single state 0 that separates two subsets of states, $\Omega_1$ and $\Omega_2$, which are not communicated by one-step transitions. In other words, the RW has a one-step diagram of the following type
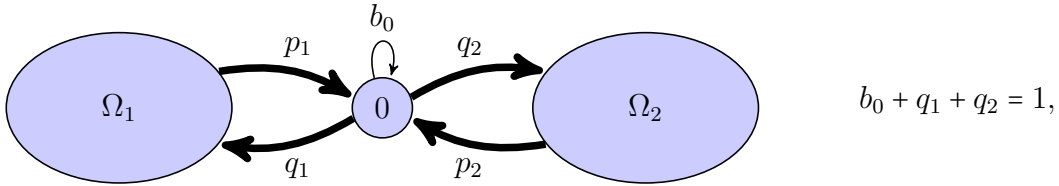
\begin{center}
\begin{tikzpicture}[->,>=stealth',shorten >=1pt,auto,node distance=1.7cm,
                    semithick]
  \tikzstyle{every state}=[ellipse,minimum height=2cm, 
    		minimum width=3cm,fill=blue!20,draw,inner sep=0pt,
			node distance=5cm]

  \kern-60pt
  
  \node[state] (L)  			{$\Omega_1$};
  \node[state] (R)  at (7,0) 	{$\Omega_2$};

  \node[main node] (0)  at (3.5,0)  	{0};
  
  \path
        (0)  edge   [loop above]				  node {$b_0$} (0)
        (0)  edge	[bend left=20,line width=1mm] node {\kern-7pt $q_1$} (L)
        (0)  edge   [bend left=20,line width=1mm] node {$q_2$ \kern-12pt} (R)
        (L)  edge   [bend left=20,line width=1mm] node {\kern-7pt $p_1$} (0)
        (R)  edge   [bend left=20,line width=1mm] node {$p_2$ \kern-12pt} (0);
        
  \kern280pt $b_0+q_1+q_2=1,$
        
\end{tikzpicture}
\end{center}
where the thick arrows summarize all the transitions between the subsets $\Omega_i$ and the state 0, and we denote by $p_i$, $q_i$ the sums of the probabilities of the corresponding transitions. This is equivalent to state that the RW is driven by a stochastic matrix with the shape
$$
\begin{aligned}
 \Pi & =
 \left(
 \begin{array}{c|c|c}
    & & 
    \\[-7pt] 
	\kern3pt A_1 \kern3pt & \kern-1pt B_1 \kern-2pt &  
	\\[5pt] \hline
	& &  
	\\[-13pt]
	C_1 & b_0 & C_2
	\\ \hline
	& & 
	\\[-9pt]
	& \kern-1pt B_2 \kern-2pt & \kern3pt A_2 \kern3pt
	\\[-9pt]
	& & 
 \end{array}
 \right)
 =
 \left(
 \begin{array}{c|c|c}
    & & 
    \\[-7pt] 
	\kern3pt A_1 \kern3pt & \kern-1pt B_1 \kern-2pt &  
	\\[5pt] \hline
	& &  
	\\[-13pt]
	C_1 & b_1 & 
	\\ \hline
	& & 
	\\[-9pt]
	& & 	
	\\[-9pt]
	& & \kern9pt
 \end{array}
 \right)
 +
 \left(
 \begin{array}{c|c|c}
    \kern9pt & & 
    \\[-7pt] 
	& &  
	\\[5pt] \hline
	& &  
	\\[-13pt]
	& b_2 & C_2
	\\ \hline
	& & 
	\\[-9pt]
	& \kern-1pt B_2 \kern-2pt & \kern3pt A_2 \kern3pt
	\\[-9pt]
	& & 
 \end{array}
 \right)
 -
 \left(
 \begin{array}{c|c|c}
    \kern9pt & & 
    \\[-7pt] 
	& &  
	\\[5pt] \hline
	& &  
	\\[-13pt]
	& 1 &
	\\ \hline
	& & 
	\\[-9pt]
	& &
	\\[-9pt]
	& & \kern9pt
 \end{array}
 \right)
 \\[2pt]
 & = \left(
 \begin{array}{c|c} 
 	\\[-5pt]
	\kern7pt \Pi^{[1]} \kern7pt 
	\\[7pt] \hline
	&
	\\[-13pt] 
	& \kern4pt
 \end{array}
 \right)
 +
 \left(
 \begin{array}{c|c} 
 	\kern4pt
 	\\ \hline 
	\\[-5pt]
	& \kern10pt \Pi^{[2]} \kern4pt
	\\ [5pt]
 \end{array}
 \right)
 -
 \left(
 \begin{array}{c|c|c} 
 	\kern4pt& & 
 	\\ \hline
	& &
	\\[-13pt]
	& 1 &
	\\ \hline
	& &
	\\[-13pt]
	& & \kern4pt 
 \end{array}
 \right),
 \qquad
 b_1=1-q_1, \qquad b_2=1-q_2. 
\end{aligned}
$$
As the previous relation shows, up to a trivial $-1$ term acting on the state 0, the evolution matrix $\Pi$ decomposes into a sum overlapping stochastic matrices $\Pi^{[i]}$ on $\SS_i=\Omega_i\cup\{0\}$, with 0 as overlapping sate. In other words, the RW essentially decomposes into a couple of smaller ones which diagrammatically look like
\begin{equation} \label{tz:Si-0}
\begin{tikzpicture}[->,>=stealth',shorten >=1pt,auto,node distance=1.7cm,
                    semithick,baseline=(current bounding box.center)]
  \tikzstyle{every state}=[ellipse,minimum height=2cm,
  			minimum width=3cm,fill=blue!20,draw,inner sep=0pt,
			node distance=5cm]

  \kern-40pt
  
  \node[state] (L)  					{$\Omega_i$};

  \node[main node] (0)  at (3.5,0)  	{0};
  
  \path
        (0)  edge   [loop above]				  node {$b_i$} (0)
        (0)  edge	[bend left=20,line width=1mm] node {\kern-7pt $q_i$} (L)
        (L)  edge   [bend left=20,line width=1mm] node {\kern-7pt $p_i$} (0);
        
  \kern170pt $b_i+q_i=1.$
        
\end{tikzpicture}
\end{equation}
Therefore, as a consequence of Theorems~\ref{thm:split}.{\it(i)} and \ref{thm:split2}.{\it(i)}, the FR-function $f$ of the state 0 with respect to $\Pi$ decomposes as $f=f^{[1]}+f^{[2]}-1$ in terms of similar FR-functions $f^{[i]}$ for $\Pi^{[i]}$. Hence, Theorem~\ref{thm:split-RW}.{\it(i)} implies that the return probabilities and expected return times to the state 0 are related by $\pi(0\to0)=\pi^{[1]}(0\to0)+\pi^{[2]}(0\to0)-1$ and $\tau(0\to0)=\tau^{[1]}(0\to0)+\tau^{[2]}(0\to0)-1$. 

The iteration of these relations leads to the decomposition $f=f^{[1]}+f^{[2]}+f^{[3]}-2$ for the FR-function of the state 0 when the RW has the structure
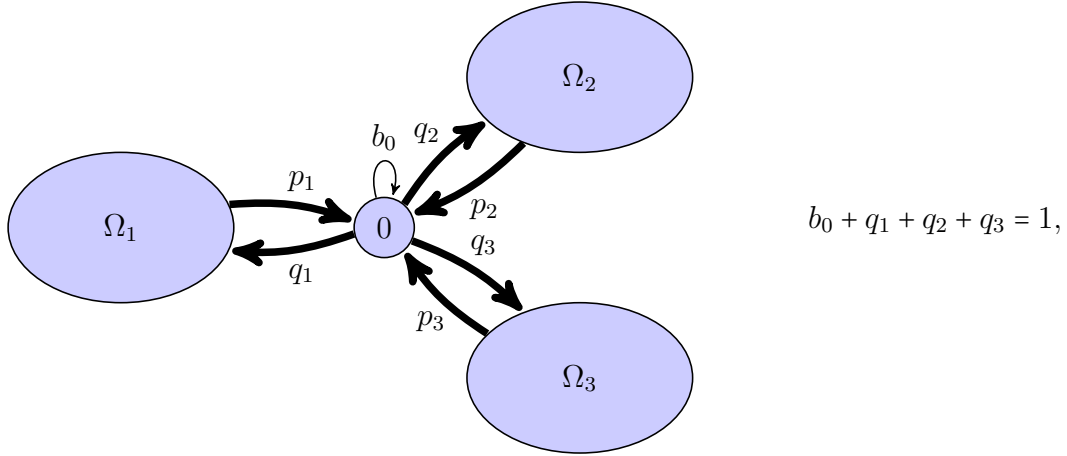
\begin{center}
\begin{tikzpicture}[->,>=stealth',shorten >=1pt,auto,node distance=1.7cm,
                    semithick]
  \tikzstyle{every state}=[ellipse,minimum height=2cm,
  			minimum width=3cm,fill=blue!20,draw,inner sep=0pt,
			node distance=5cm]

  \kern-60pt
  
  \node[state] (L)  			{$\Omega_1$};
  \node[state] (R)  at (6.1,2) 	{$\Omega_2$};
  \node[state] (D)  at (6.1,-2) 	{$\Omega_3$};

  \node[main node] (0)  at (3.5,0)  	{0};
  
  \path
        (0)  edge  	[loop above]				  node {$b_0$} (0)
        (0)  edge	[bend left=12,line width=1mm] node {\kern-7pt $q_1$} (L)
        (0)  edge   [bend left=12,line width=1mm] node {$q_2$ \kern-9pt} (R)
        (0)  edge  	[bend left=12,line width=1mm] node {\kern-6pt $q_3$} (D)
        (L)  edge   [bend left=12,line width=1mm] node {\kern-7pt $p_1$} (0)
        (R)  edge   [bend left=12,line width=1mm] node {\kern-6pt $p_2$} (0)
        (D)  edge   [bend left=12,line width=1mm] node {$p_3$ \kern-11pt} (0);
        
  \kern260pt $b_0+q_1+q_2+q_3=1,$
        
\end{tikzpicture}
\end{center}
and $f^{[i]}$ is the FR-function of the state 0 for the RW represented by \eqref{tz:Si-0}.

In general, if $\SS\setminus\{0\}=\Omega_1\cup\cdots\cup\Omega_n$ splits into $n$ subsets of states, $\Omega_1,\dots,\Omega_n$, without one-step transitions among them, the previous relation among FR-functions generalizes to $f=f^{[1]}+\cdots+f^{[n]}-(n-1)$, hence a similar relation holds for the probabilities and expected return times, namely, $\pi(0\to0)=\pi^{[1]}(0\to0)+\cdots+\pi^{[n]}(0\to0)-(n-1)$ and $\tau(0\to0)=\tau^{[1]}(0\to0)+\cdots+\tau^{[n]}(0\to0)-(n-1)$.

We will illustrate this multiple use of the splitting rules for FR-functions in the case of RW on spider graphs. A simple example of such a RW has a one-step diagram with the shape
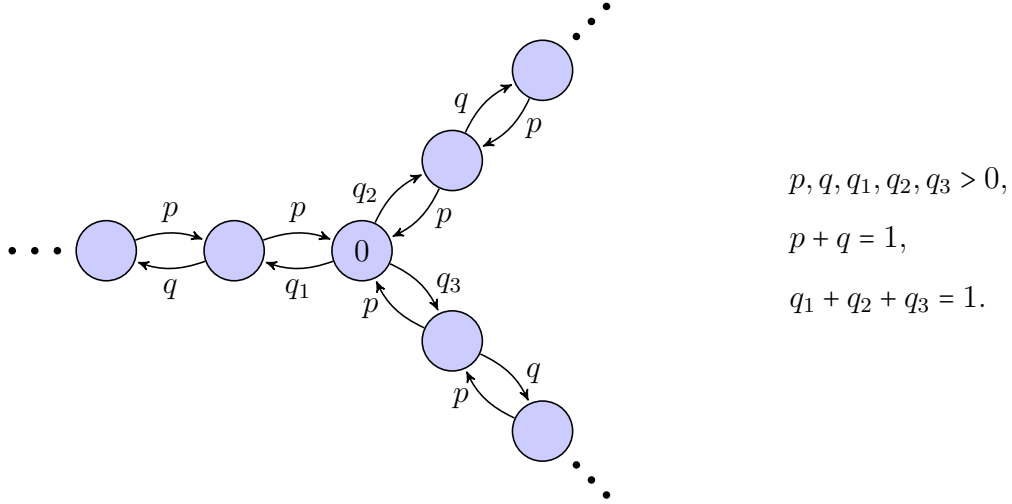
\begin{center}
\begin{tikzpicture}[->,>=stealth',shorten >=1pt,auto,node distance=1.7cm,
                    semithick]
  \tikzstyle{every state}=[fill=black,draw=none,text=white,
  			inner sep=0pt,minimum size=0.1cm,node distance=0.3cm]
  
  \kern-60pt

  \node[main node] (0)                	{0};
  \node[main node] (A1) [left of=0]  	{};
  \node[main node] (A2) [left of=A1] 	{};
  \node[main node] (B1) at (1.2,1.2) 	{};
  \node[main node] (B2) at (2.4,2.4) 	{};
  \node[main node] (C1) at (1.2,-1.2) 	{};
  \node[main node] (C2) at (2.4,-2.4) 	{};

  \node[state] (a1) at (-4.05,0)  		{};
  \node[state] (a2) [left of=a1]		{};
  \node[state] (a3) [left of=a2]		{};
  \node[state] (b1) at (2.85,2.85)  	{};
  \node[state] (b2) at (3.05,3.05) 		{};
  \node[state] (b3) at (3.25,3.25)		{};
  \node[state] (c1) at (2.85,-2.85)  	{};
  \node[state] (c2) at (3.05,-3.05) 	{};
  \node[state] (c3) at (3.25,-3.25)		{};
  
  \path (0)  edge   [bend left=20]     		node {$q_1$} (A1)
        (A1) edge   [bend left=20]     		node {$p$} (0)
        (A1) edge   [bend left=20]     		node {$q$} (A2)
        (A2) edge   [bend left=20]     		node {$p$} (A1)
        (0)  edge   [bend left=20,left]    	node {$q_2$ \kern-2pt} (B1)
        (B1) edge   [bend left=20,right]    node {\kern2pt $p$} (0)
        (B1) edge   [bend left=20,left]    	node {$q$ \kern-2pt} (B2)
        (B2) edge   [bend left=20,right]    node {\kern2pt $p$} (B1)
        (0)  edge   [bend left=20,right]    node {\kern2pt $q_3$} (C1)
        (C1) edge   [bend left=20,left]    	node {$p$ \kern-2pt} (0)
        (C1) edge   [bend left=20,right]    node {\kern2pt $q$} (C2)
        (C2) edge   [bend left=20,left]    	node {$p$ \kern-2pt} (C1);
        
  \kern160pt
   
  $\begin{aligned}
    & p,q,q_1,q_2,q_3>0, 
	\\[5pt] 
	& p+q=1,
	\\[5pt]  
	& q_1+q_2+q_3=1.
  \end{aligned}$
        
\end{tikzpicture}
\end{center}
As the general case previously discussed shows, the FR-function $f=f^{[1]}+f^{[2]}+f^{[3]}-2$ of the state 0 follows from those $f^{[i]}$ of the smaller RW given by
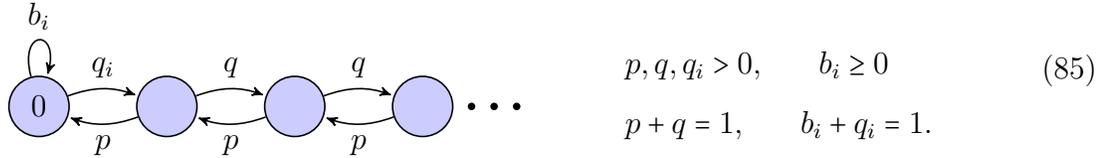
\begin{equation} \label{tz:spider-leg} 
\begin{tikzpicture}[->,>=stealth',shorten >=1pt,auto,node distance=1.7cm, 
              		semithick,baseline=(current bounding box.center)]
  \tikzstyle{every state}=[fill=black,draw=none,text=white,
  			inner sep=0pt,minimum size=0.1cm,node distance=0.3cm]

  \kern-85pt
   
  \node[main node] (0)                	{0};
  \node[main node] (A1) [right of=0]  	{};
  \node[main node] (A2) [right of=A1] 	{};
  \node[main node] (A3) [right of=A2] 	{};

  \node[state] (a1) at (5.75,0)  		{};
  \node[state] (a2) [right of=a1]		{};
  \node[state] (a3) [right of=a2]		{};
  
  \path (0)  edge	[loop above]			node {$b_i$} (0)
        (0)  edge   [bend left=20]     		node {$q_i$} (A1)
        (A1) edge   [bend left=20]     		node {$p$} (0)
        (A1) edge   [bend left=20]     		node {$q$} (A2)
        (A2) edge   [bend left=20]     		node {$p$} (A1)
        (A2) edge   [bend left=20]     		node {$q$} (A3)
        (A3) edge   [bend left=20]     		node {$p$} (A2);
        
  \kern220pt 
  $\begin{aligned}
    & p,q,q_i>0, \kern20pt b_i\ge0
    \\[5pt]  
	& p+q=1, \kern20pt b_i+q_i=1.
  \end{aligned}$
        
\end{tikzpicture}
\end{equation}

In the case of a RW on a similar spider graph with $n$ legs, we have the relation $f=f^{[1]}+\cdots+f^{[n]}-(n-1)$, where $f^{[i]}$ is the FR-function of the state 0 for the RW given by \eqref{tz:spider-leg}, which is exactly the kind of birth-death processes studied in the previous example. Therefore, we already know that
$$
 f^{[i]}(z) = b_i + \frac{q_i}{q} \frac{1-\sqrt{1-4pqz^2}}{2z}.
$$
We find that the FR-function of the central state for the spider RW is independent of the number of legs, as well as from the probabilities $q_i$ of the transitions starting at such a central state,
$$
 f(z) = \frac{1-\sqrt{1-4pqz^2}}{2qz}.
$$
So this independence also holds for the return probability and the expected return time, 
$$
 \pi(0\to0) = 
 \begin{cases} 
 	1 & \text{ if } p\ge q,
	\\
	p/q & \text{ if } p<q,
 \end{cases}
 \qquad\qquad
 \tau(0\to0) = 1+\frac{1}{p-q} \quad \text{ if } p\ge q.
$$ 
\end{ex}

\begin{ex} \label{ex:RW4}
Consider now a RW on a set of states $\SS=\Omega_1\cup\{0,1,2,3\}\cup\Omega_2$ with the following structure of one-step transitions, 
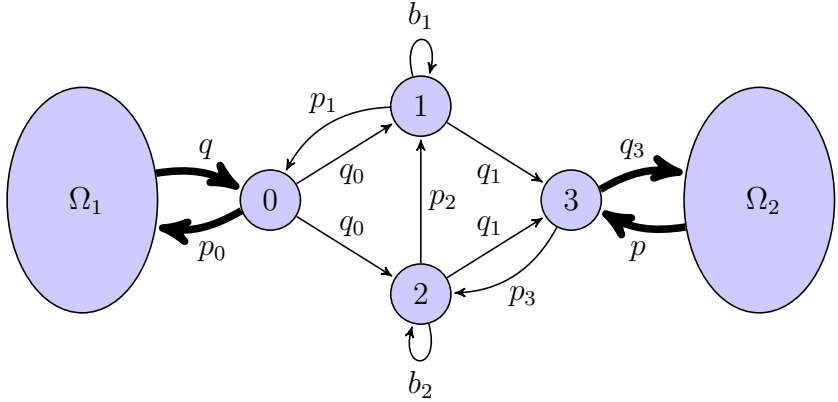
\begin{center}
\begin{tikzpicture}[->,>=stealth',shorten >=1pt,auto,node distance=1.7cm,
					semithick]
  \tikzstyle{every state}=[ellipse,minimum height=3cm,
  			minimum width=2cm,fill=blue!20,draw,inner sep=0pt,
			node distance=5cm]

  \kern-50pt
  
  \node[state] (L)  			{\kern3pt $\Omega_1$};
  \node[state] (R)  at (9,0) 	{\kern3pt $\Omega_2$};

  \node[main node] (0)  at (2.5,0)  	{0};
  \node[main node] (1)  at (4.5,1.25)  	{1};
  \node[main node] (2)  at (4.5,-1.25)  {2};
  \node[main node] (3)  at (6.5,0)  	{3};
  
  \path
        (0)  edge	[bend left=20,line width=1mm]  	node {\kern-6pt $p_0$} (L)
        (0)  edge   [below]  						node {\kern5pt $q_0$} (1)
        (0)  edge   [above]  						node {\kern5pt $q_0$} (2)
        (1)  edge   [loop above] 				   	node {$b_1$} (1)
        (1)  edge   [below]  						node {$q_1$ \kern1pt} (3)
        (1)  edge   [bend right,above]  			node {\kern-3pt $p_1$} (0)
        (2)  edge   [loop below] 				   	node {$b_2$} (2)
        (2)  edge   [right]  						node {\kern-1pt $p_2$} (1)
        (2)  edge   [above]  						node {$q_1$ \kern1pt} (3)
        (3)  edge   [bend left,below]  				node {\kern7pt $p_3$} (2)
        (3)  edge	[bend left=20,line width=1mm]  	node {$q_3$ \kern-11pt} (R)
        (L)  edge   [bend left=20,line width=1mm]  	node {\kern-6pt $q$} (0)
        (R)  edge   [bend left=20,line width=1mm]   node {$p$ \kern-11pt} (3);
        
  \kern315pt 
  $\begin{aligned}
  & p_0+2q_0=1,
  \\ 
  & b_1+p_1+q_1=1,
  \\
  & b_2+p_2+q_1=1,
  \\
  & p_3+q_3=1,
  \\[12pt]
  & p_2=b_1+b_2,
  \end{aligned}$
        
\end{tikzpicture}
\end{center}
where the thick arrows represent all the transitions between the subsets $\Omega_1,\Omega_2$ and the states $0,3$, while $p,q,p_0,q_3$ are the sums of the corresponding probabilities.

The condition $p_2=b_1+b_2$ guarantees that the related evolution stochastic matrix $\Pi$ factorizes into a couple of smaller stochastic matrices, $\Pi_L$ and $\Pi_R$, on the subsets $\SS_L=\Omega_1\cup\{0,1,2\}$ and $\SS_R=\{1,2,3\}\cup\Omega_2$ respectively,
$$
\begin{aligned}
 \Pi & =
 \left(
 \begin{array}{c|c|cc|c|c}
    & & & & &
    \\[-7pt] 
	\kern3pt A_1 \kern3pt & \kern-1pt B_1 \kern-2pt & & & &  
	\\[5pt] \hline
	& & & & & 
	\\[-13pt]
	C_1 & 0 & q_0 & q_0 & 0 & 
	\\ \hline
	& & & & & 
	\\[-13pt]
	& p_1 & b_1 & 0 & q_1 & 
	\\
	& 0 & p_2 & b_2 & q_2 &
	\\ \hline
	& & & & & 
	\\[-13pt]
	& 0 & 0 & p_3 & 0 & C_2
	\\ \hline
	& & & & & 
	\\[-9pt]
	& & & & \kern-1pt B_2 \kern-2pt & \kern3pt A_2 \kern3pt
	\\[-9pt]
	& & & & &
 \end{array}
 \right)
 =
 \\[2pt]
 &  = 
 \left(
 \begin{array}{c|c|cc|c|ccc}
    & & & & & & & 
    \\[-5pt] 
	\kern3pt A_1 \kern3pt & \kern-1pt B_1 \kern-2pt & & & & & &  
	\\[7pt] \hline
	& & & & & & & 
	\\[-13pt]
	C_1 & 0 & 2q_0 & \kern-5pt 0 & & & & 
	\\ \hline
	& & & & & & & 
	\\[-12pt]
	& p_1 & 0 & \kern-5pt b_1\kern-3pt+\!q_1 \kern-2pt & & & & 
	\\[3pt]
	& 0 & 2b_2 & \kern-5pt b_1\kern-3pt+\!q_1 \kern-2pt & & & & 
	\\[3pt] \hline
	& & & & & & & 
	\\[-13pt]
	& & & & 1 & & & 
	\\ \hline
	& & & & & & & 
	\\[-15pt]
	& & & & & \text{\footnotesize 1} & & 
	\\[-7pt]
	& & & & & & \kern-6pt \text{\footnotesize 1} & 
	\\[-6pt]
	& & & & & & & \kern-7pt \ddots
 \end{array}
 \right)
 \left(
 \begin{array}{ccc|c|cc|c|c}
    \ddots & & & & & & &
    \\[-6pt]
    & \kern-7pt \text{\footnotesize 1} & & & & & & 
    \\[-7pt] 
	& & \kern-7pt \text{\footnotesize 1} & & & & &  
	\\[-1pt] \hline
	& & & & & & & 
	\\[-13pt]
	& & & 1 & & & & 
	\\[-14pt]
    & & & & & & & 
	\\ \hline
	& & & & & & & 
	\\[-12pt]
	& & & & \frac{1}{2} & \frac{1}{2} & 0 & 
	\\[2pt]
	& & & & \frac{b_1}{b_1\kern-1pt+q_1} & 0 
	& \kern-4pt \frac{q_1}{b_1\kern-1pt+q_1} \kern-4pt &
	\\[3pt] \hline
	& & & & & & & 
	\\[-13pt]
	& & & & 0 & p_3 & 0 & C_2
	\\ \hline
	& & & & & & & 
	\\[-8pt]
	& & & & & & \kern-1pt B_2 \kern-2pt & \kern4pt A_2 \kern2pt
	\\[-8pt]
	& & & & & & & 
 \end{array}
 \right)
 \\
 & = \left(
 \begin{array}{c|c} 
 	\\[-5pt]
	\kern7pt \Pi_L \kern9pt 
	\\[7pt] \hline
	& 
	\\[-12pt] 
	& \kern3pt I
 \end{array}
 \right)
 \left(
 \begin{array}{c|c} 
 	\kern1pt I \kern2pt 
 	\\ \hline 
	\\[-5pt]
	& \kern9pt \Pi_R \kern7pt
	\\ [5pt]
 \end{array}
 \right),
\end{aligned}
$$
where $I$ denotes an identity matrix of appropriate size for each factor. The left and right RW given by $\Pi_L$ and $\Pi_R$ are represented diagrammatically by
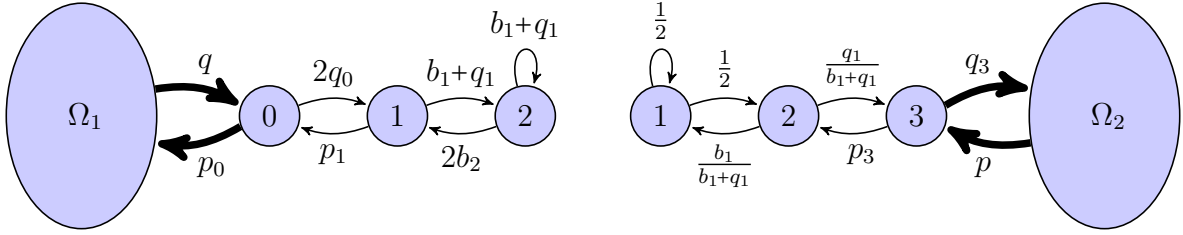
\begin{center}
\begin{tikzpicture}[->,>=stealth',shorten >=1pt,auto,node distance=1.7cm,
					semithick]
  \tikzstyle{every state}=[ellipse,minimum height=3cm,
  			minimum width=2cm,fill=blue!20,draw,inner sep=0pt,
			node distance=5cm]
  
  \node[state] (L)  				{\kern3pt $\Omega_1$};
  \node[state] (R)  at (13.6,0) 	{\kern3pt $\Omega_2$};

  \node[main node] (10)  at (2.5,0)  	{0};
  \node[main node] (11)  at (4.2,0)  	{1};
  \node[main node] (12)  at (5.9,0)  	{2};
  \node[main node] (21)  at (7.7,0)  	{1};
  \node[main node] (22)  at (9.4,0)  	{2};
  \node[main node] (23)  at (11.1,0)  	{3};
  
  \path
        (10)  edge	[bend left=20,line width=1mm]  	node {\kern-6pt $p_0$} (L)
        (10)  edge  [bend left=20]  				node {$2q_0$} (11)
        (11)  edge  [bend left=20] 				   	node {$p_1$} (10)
        (11)  edge  [bend left=20]  		   node {$b_1\kern-3pt+\!q_1$} (12)
        (12)  edge  [bend left=20]  				node {$2b_2$} (11)
        (12)  edge  [loop above] 			   node {$b_1\kern-3pt+\!q_1$} (12)
        (21)  edge  [loop above] 				   	node {$\frac{1}{2}$} (21)
        (21)  edge  [bend left=20]  				node {$\frac{1}{2}$} (22)
        (22)  edge  [bend left=20]   node {$\frac{b_1}{b_1\kern-1pt+q_1}$} (21)
        (22)  edge  [bend left=20]   node {$\frac{q_1}{b_1\kern-1pt+q_1}$} (23)
        (23)  edge  [bend left=20]  				node {\kern7pt $p_3$} (22)
        (23)  edge	[bend left=20,line width=1mm]  	node {$q_3$ \kern-11pt} (R)
        (L)  edge   [bend left=20,line width=1mm]  	node {\kern-6pt $q$} (10)
        (R)  edge   [bend left=20,line width=1mm]   node {$p$ \kern-11pt} (23);
        
\end{tikzpicture}
\end{center}
Since the left and right RW overlap on the set of states $\Omega=\{1,2\}$, the factorization of $\Pi$ translates into a factorization $f=f_Lf_R$ for the FR-function of $\Omega$ with respect to $\Pi$ in terms of similar FR-functions $f_L,f_R$ for $\Pi_L$, $\Pi_R$. 

Let us apply this result to a non-trivial perturbation of a birth-death process on the integers $\SS=\Z$ given by
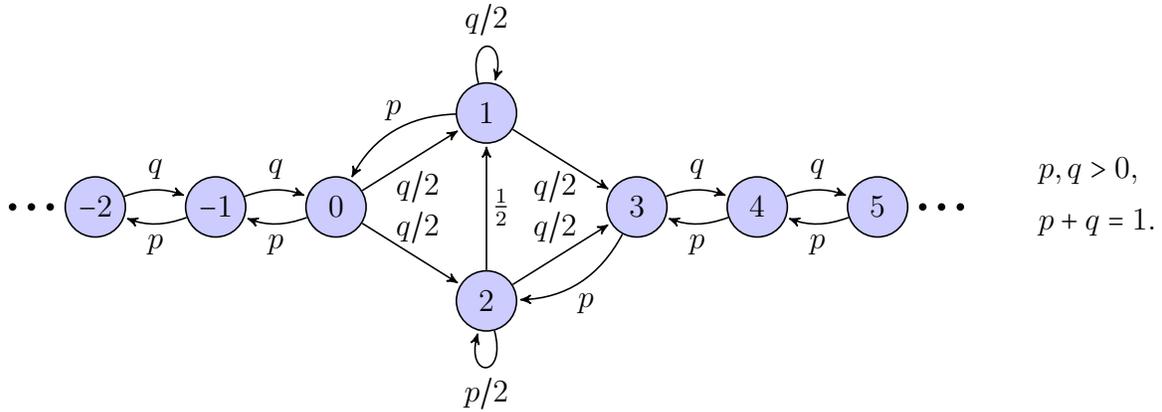
\begin{center}
\begin{tikzpicture}[->,>=stealth',shorten >=1pt,auto,node distance=1.7cm,
					semithick]
  \tikzstyle{every state}=[fill=black,draw=none,text=white,
  			inner sep=0pt,minimum size=0.1cm,node distance=0.25cm]

  \kern-35pt
  
  \node[main node] (-2) 	   			{$-2$};
  \node[main node] (-1)	at (1.6,0)  	{$-1$};
  \node[main node] (0)  at (3.2,0)  	{0};
  \node[main node] (1)  at (5.2,1.25)  	{1};
  \node[main node] (2)  at (5.2,-1.25)  {2};
  \node[main node] (3)  at (7.2,0)  	{3};
  \node[main node] (4)  at (8.8,0)  	{4};
  \node[main node] (5)  at (10.4,0)  	{5};
  
  \node[state] (a1) at (-0.6,0)  		{};
  \node[state] (a2) [left of=a1]		{};
  \node[state] (a3) [left of=a2]		{};
  \node[state] (b1) at (11,0)  			{};
  \node[state] (b2) [right of=b1]		{};
  \node[state] (b3) [right of=b2]		{};

  \path
        (-2) edge   [bend left=20]  		node {$q$} (-1)
        (-1) edge   [bend left=20]  		node {$p$} (-2)
        (-1) edge   [bend left=20]  		node {$q$} (0)
        (0)  edge   [bend left=20]  		node {$p$} (-1)
        (0)  edge   [below]  				node {\kern5pt $q/2$} (1)
        (0)  edge   [above]  				node {\kern5pt $q/2$} (2)
        (1)  edge   [loop above] 			node {$q/2$} (1)
        (1)  edge   [below]  				node {$q/2$ \kern1pt} (3)
        (1)  edge   [bend right,above]  	node {$p$} (0)
        (2)  edge   [loop below] 			node {$p/2$} (2)
        (2)  edge   [right]  				node {\kern-2pt $\frac{1}{2}$} (1)
        (2)  edge   [above]  				node {$q/2$ \kern1pt} (3)
        (3)  edge   [bend left,below]  		node {\kern5pt $p$} (2)
        (3)  edge   [bend left=20]  		node {$q$} (4)
        (4)  edge   [bend left=20]  		node {$p$} (3)
        (4)  edge   [bend left=20]  		node {$q$} (5)
        (5)  edge   [bend left=20]  		node {$p$} (4);
        
  \kern355pt 
  $\begin{aligned}
  & p,q>0,
  \\[2pt] 
  & p+q=1.
  \end{aligned}$
        
\end{tikzpicture}
\end{center}
According to the previous general result, this RW factorizes into the following left and right birth-death processes on a semi-infinite chain,
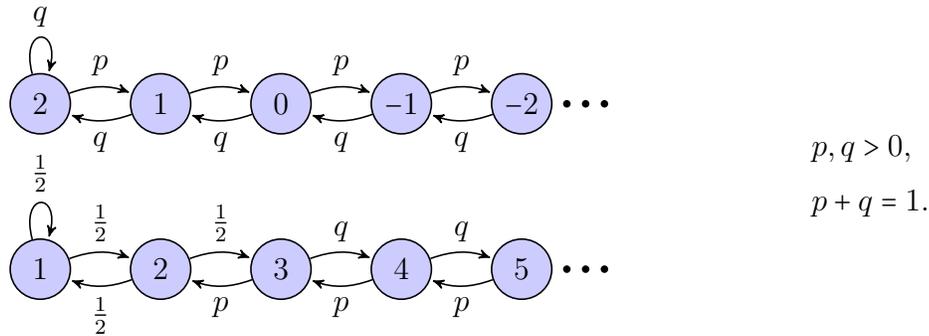
\begin{center} 
\begin{tikzpicture}[->,>=stealth',shorten >=1pt,auto,node distance=1.6cm,
					semithick,baseline=(current bounding box.center)]
  \tikzstyle{every state}=[fill=black,draw=none,text=white,
  			inner sep=0pt,minimum size=0.1cm,node distance=0.25cm]

  \kern-50pt
  
  \node[main node] (A0) at (0,1.1)      {2};
  \node[main node] (A1) [right of=A0]  	{1};
  \node[main node] (A2) [right of=A1] 	{0};
  \node[main node] (A3) [right of=A2] 	{$-1$};
  \node[main node] (A4) [right of=A3] 	{$-2$};
   
  \node[main node] (B0) at (0,-1.1)   	{1};
  \node[main node] (B1) [right of=B0]  	{2};
  \node[main node] (B2) [right of=B1] 	{3};
  \node[main node] (B3) [right of=B2] 	{4};
  \node[main node] (B4) [right of=B3] 	{5};

  \node[state] (a1) at (7,1.1)  		{};
  \node[state] (a2) [right of=a1]		{};
  \node[state] (a3) [right of=a2]		{};
  \node[state] (b1) at (7,-1.1)  		{};
  \node[state] (b2) [right of=b1]		{};
  \node[state] (b3) [right of=b2]		{};
  
  \path (A0) edge	[loop above]			node {$q$} (A0)
        (A0) edge   [bend left=20]     		node {$p$} (A1)
        (A1) edge   [bend left=20]     		node {$q$} (A0)
        (A1) edge   [bend left=20]     		node {$p$} (A2)
        (A2) edge   [bend left=20]     		node {$q$} (A1)
        (A2) edge   [bend left=20]     		node {$p$} (A3)
        (A3) edge   [bend left=20]     		node {$q$} (A2)
        (A3) edge   [bend left=20]     		node {$p$} (A4)
        (A4) edge   [bend left=20]     		node {$q$} (A3)
        (B0) edge	[loop above]			node {$\frac{1}{2}$} (B0)
        (B0) edge   [bend left=20]     		node {$\frac{1}{2}$} (B1)
        (B1) edge   [bend left=20]     		node {$\frac{1}{2}$} (B0)
        (B1) edge   [bend left=20]     		node {$\frac{1}{2}$} (B2)
        (B2) edge   [bend left=20]     		node {$p$} (B1)
        (B2) edge   [bend left=20]     		node {$q$} (B3)
        (B3) edge   [bend left=20]     		node {$p$} (B2)
        (B3) edge   [bend left=20]     		node {$q$} (B4)
        (B4) edge   [bend left=20]     		node {$p$} (B3);
        
  \kern290pt 
  $\begin{aligned}
    & p,q>0,
    \\[3pt]  
	& p+q=1.
  \end{aligned}$
        
\end{tikzpicture}
\end{center}
The FR-function of the subset $\Omega=\{1,2\}$ with respect to these birth-death processes follows from a slight generalization of the final results in Example~\ref{ex:RW2}, which yields
$$
 f_L(z) = 
 \begin{pmatrix}
 	q & p \\ q & g(z)
 \end{pmatrix},
 \qquad
 f_R(z) = \frac{1}{2}
 \begin{pmatrix}
 	1 & 1 \\ 1 & g(z)/q
 \end{pmatrix},
 \qquad
 g(z) = \frac{1-\sqrt{4pqz^2}}{2z}.
$$
The function $f_L$ is naturally given for the ordered set $\{2,1\}$. Changing the order of the states to $\{1,2\}$ we find that
$$
 f(z) = \frac{1}{2}
 \begin{pmatrix}
 	g(z) & q \\ p & q
 \end{pmatrix}
 \begin{pmatrix}
 	1 & 1 \\ 1 & g(z)/q
 \end{pmatrix} 
 = \frac{1}{2}
 \begin{pmatrix}
 	g(z)+q & 2g(z) \\[2pt] 1 & g(z)+p
 \end{pmatrix}.
$$
Summing the rows of 
$$
 f(1) = 
 \begin{pmatrix}
 	\min\{q,\frac{1}{2}\} & \min\{p,q\} 
	\\[3pt] 
	\frac{1}{2} & \min\{p,\frac{1}{2}\}
 \end{pmatrix},
 \qquad 
 \lim_{x\uparrow1}(xf(x))' = \frac{1}{2}
 \begin{pmatrix}
 	\frac{2pq}{|p-q|}+q & \frac{4pq}{|p-q|} 
	\\[5pt] 
	1 & \frac{2pq}{|p-q|}+p
 \end{pmatrix},
$$ 
we obtain the following return probabilities and expected return times to $\Omega=\{1,2\}$,
$$
\begin{gathered}
 \pi(1\to\Omega) =
 \begin{cases} 
 	2q,
	\\
	p+\frac{1}{2},
 \end{cases}
 \qquad\quad
 \pi(2\to\Omega) =
 \begin{cases} 
 	1 & \qquad\quad \text{ if } p\ge q,
	\\
	p+\frac{1}{2} & \qquad\quad \text{ if } p< q.
 \end{cases}
 \\[3pt]
 \tau(2\to\Omega) = \frac{1}{2} \left(1+\frac{p}{p-q}\right)
 \qquad\quad \text{ if } p\ge q.
\end{gathered}
$$
Hence, the states 1 and 2 are $\Omega$-recurrent for $p=q=1/2$ and $p\ge q$ respectively, but the return to $\Omega$ happens in a finite expected time only for the state 2 when $p>q$. 
\end{ex}

\section{Applications to quantum walks: QW recurrence}
\label{sec:QW}

We will move now to the more novel subject of QW. Born in the early ninetees as the quantum analogue of RW \cite{ADZ}, the interest in QW has been continuously increasing due to their role in quantum information processing. In particular, QW are central for the design of quantum algorithms of up to exponential speed ups over classical ones \cite{Ambainis,Childs,Kempe,Kendon}. This has triggered and intense research activity around QW from the experimental and theoretical points of view. 

Concerning return properties of QW, several different notions of recurrence have been recently introduced \cite{SJK,SJK2,SJK3,GVWW,BGVW,CGMV}. A key difference with respect to their classical analogue, the RW, is that the quantum collapse due to the measurement necessary to define first time returns alters the natural evolution of a QW, and therefore different schemes have been introduced to deal with such a situation. We will follow here a monitoring approach, first presented in \cite{GVWW} (see also \cite{BGVW,CGMV}), in which the collapse is taken as a natural ingredient of quantum recurrence, which is in the spirit of Quantum Mechanics and the prominent role of measurements in its postulates. Fortunately, despite the conceptual differences between recurrence for RW and QW, the quantum collapse due to a measurement is mathematically represented by the action of a projection, so that in the end both, classical and quantum recurrence, are controlled by a similar object, an FR-function, but defined in different contexts.

QW are discrete time quantum models. As in any quantum system, the states of a QW are the rays of a complex Hilbert space $\HH$, i.e. its one-dimensional subspaces $\spn\{\psi\}$, which for convenience are identified with any of their normalized generator vectors $\psi\in\HH$, $\|\psi\|=1$. The corresponding discrete time quantum evolution is given by a unitary operator $U\colon\HH\to\HH$, so that $|\<\phi|U^n\psi\>|^2$ is the $n$-th step probability transition $\psi\to\phi$. The complex quantity $\<\phi|U^n\psi\>$ is called the amplitude of such a transition. 

The probability of obtaining a state in a closed subspace $\HH_0\subset\HH$ after a measurement at a state $\psi\in\HH$ is controlled by the orthogonal projection $P$ of $\HH$ onto $\HH_0$. Such a probability is given by $\|P\psi\|^2$, and after a positive outcome of the measurement --i.e. finding the system in the subspace $\HH_0$-- the system collapses to the state $P\psi/\|P\psi\|$.

Suppose that we are interested in the probability of returning to $\HH_0$ when starting at $\psi\in\HH_0$. In the case of first time returns, a measurement must be performed after each step to know if the system has returned to $\HH_0$ or not. Then, the application of the projection $P$ after each step conditions on the event `return to $\HH_0$', while the application of the complementary projection $Q=1-P$ conditions on the event `no return to $\HH_0$'. 

The above remarks are key in the monitoring approach to recurrence because then the total return probability is obtained as a sum of first time return probabilities. Actually, the projections involved in the first time returns are central ingredients in the explicit expressions for each of the following precise notions about QW recurrence.

\begin{defn} 
Given a closed subspace $\HH_0\subset\HH$, we define the following notions for a QW on the Hilbert space $\HH$:

$\pi(\psi\to\HH_0)=$ probability of returning to $\HH_0$ when starting at the state $\psi\in\HH_0$.

$\pi(\psi\xrightarrow{\HH_0}\phi)=$ probability of landing on the state $\phi\in\HH_0$ when returning to $\HH_0$ starting \hspace*{79pt} at the state $\psi\in\HH_0$.

$\pi_n(\psi\to\HH_0)=$ probability of returning for the first time in $n$ steps to $\HH_0$ when starting \hspace*{85pt} at the state $\psi\in\HH_0$.

$\pi_n(\psi\xrightarrow{\HH_0}\phi)=$ probability of landing on the state $\phi\in\HH_0$ when returning for the first \hspace*{87pt} time in $n$ steps to $\HH_0$ starting at the state $\psi\in\HH_0$.

$\tau(\psi\to\HH_0)=$ expected return time to $\HH_0$ when starting at the state $\psi\in\HH_0$.

When $\HH_0=\spn\{\psi\}$ we write $\pi(\psi\to\psi)$, $\pi_n(\psi\to\psi)$ and $\tau(\psi\to\psi)$ for the corresponding quantities. 
\end{defn}

These quantities are closely related to the FR-function $f$ of the projection $P$ with respect to the unitary step $U$, given by 
$$
 f(z) = PU(1-zQU)^{-1}P = \sum_{n\ge1} z^{n-1} a_n,
 \qquad a_n=PU(QU)^{n-1}P,
 \qquad z\in\D,
$$
which is indeed a Schur function with values in operators on the Hilbert space $\HH_0$, as noticed in \cite{GVWW,BGVW}. According to the rules discussed above, we have the following explicit expressions, 
$$
\begin{aligned}
 & \pi_n(\psi\xrightarrow{\HH_0}\phi) = |\<\phi|U(QU)^{n-1}\psi\>|^2 = 
 |\<\phi|a_n\psi\>|^2,
 \\[3pt]
 & \pi_n(\psi\to\HH_0) = \|PU(QU)^{n-1}\psi\|^2 = \|a_n\psi\|^2,
 \\
 & \pi(\psi\xrightarrow{\HH_0}\phi) = 
 \sum_{n\ge1} \pi_n(\psi\xrightarrow{\HH_0}\phi) =
 \sum_{n\ge1}| \<\phi|a_n\psi\>|^2 =
 \int_0^{2\pi} |\<\phi|f(e^{i\theta})\psi\>|^2 \, \frac{d\theta}{2\pi},
 \\
 & \pi(\psi\to\HH_0) = \sum_{n\ge1} \pi_n(\psi\to\HH_0) =
 \sum_{n\ge1} \|a_n\psi\|^2 = 
 \int_0^{2\pi} \|f(e^{i\theta})\psi\|^2 \, \frac{d\theta}{2\pi},
 \\[3pt]
 & \tau(\psi\to\HH_0) = 
 \begin{cases}
 	\infty, 
	& \text{ if } \pi(\psi\to\HH_0)<1,
	\\
 	\displaystyle \sum_{n\ge1} n\,\pi_n(\psi\to\HH_0) 
 	= \sum_{n\ge1} n\|a_n\psi\|^2,
	& \text{ if } \pi(\psi\to\HH_0)=1,
 \end{cases}
 \\[3pt]
 & \kern45pt f(e^{i\theta}) := \lim_{r\uparrow1} f(re^{i\theta}),  
 \qquad r\in[0,1), \qquad \theta\in[0,2\pi).
\end{aligned}
$$
Just as in the case of RW, the inequality $\sum_{n\ge1}\|a_n\psi\|^2\le1$ requires a proof which can be found in \cite{BGVW}. 

The expected return time for the case $\pi(\psi\to\HH_0)=1$ also has the integral representation
$$
 \tau(\psi\to\HH_0) =  \lim_{r\uparrow1} \int_0^{2\pi} 
 \<a(re^{i\theta}) \psi \, | \, \partial_\theta a(re^{i\theta}) \psi\> 
 \, \frac{d\theta}{2\pi i},
 \qquad a(z)=zf(z)=\sum_{n\ge1}z^na_n,
$$
which in terms of the original Schur function $f$ reads as
$$
 \tau(\psi\to\HH_0) = 1 + \lim_{r\uparrow1} \int_0^{2\pi} 
 \<\psi \, | \, f(re^{i\theta})^\dag \partial_\theta f(re^{i\theta}) \psi\> 
 \, \frac{d\theta}{2\pi i}.
$$
Some readers will see here a connection with Berry's phase, see for instance \cite{BGVW}.
 
The radial boundary values $f(e^{i\theta})$ of a Schur function exist a.e. on the unit circle \cite[Chapter 11]{Rudin}, but this is not necessarily true for its derivative $\partial_\theta f(re^{i\theta})$. This is the reason for expressing the integral representation of the expected return time with the limit $r\uparrow1$ outside the integral.

While the FR-function $f$ is the generating function of the first return amplitudes $a_n$, the Stieltjes function 
$$
 s(z) = P(1-zU)^{-1}P = \sum_{n\ge0} z^n PU^nP,
 \qquad z\in\D,
$$ 
is the generating function of the return amplitudes without monitoring. The specialization of the generalized renewal equation \eqref{eq:sf} to unitary operators gives the renewal equation for QW, first obtained in \cite{GVWW,BGVW}.

Regarding the splitting rules for QW recurrence, since products preserve the unitarity, such splitting rules take the form of factorizations of Schur functions due to overlapping factorizations of the corresponding unitary operator $U$,
$$
 U = 
 \left(
 \begin{array}{c|c} 
 	\\[-5pt]
	\kern7pt U_L \kern9pt 
	\\[7pt] \hline 
	& 1_+ \kern-2pt
 \end{array}
 \right)
 \left(
 \begin{array}{c|c} 
 	1_- \kern-2pt 
 	\\ \hline 
	\\[-5pt]
	& \kern9pt U_R \kern7pt
	\\ [5pt]
 \end{array}
 \right).
$$ 
In contrast to the case of RW, the factorization of Schur functions does not give a simple relation among the return probabilities for the whole QW and the left/right ones. Nevertheless, some of the results of Theorem~\ref{thm:split-RW} and Corollary~\ref{cor:split-RW} have a quantum analogue.

\begin{thm}[\bf splitting rules for QW recurrence] \label{thm:split-QW}
Let $\HH=\HH_-\oplus\HH_0\oplus\HH_+$ be an orthogonal decomposition of a Hilbert space and $U_L$, $U_R$ unitary operators on $\HH_L=\HH_-\oplus\HH_0$, $\HH_R=\HH_0\oplus\HH_+$ respectively. Consider the QW given by $U = (U_L \oplus 1_+)(1_- \oplus U_R)$, where $1_\pm$ stands for the identity on $\HH_\pm$. Then, using the subscript $L,R$ to distinguish the quantities related to the QW given by $U_{L,R}$, for any state $\psi\in\HH_0$, 
$$
 \pi_L(\phi\to\HH_0)=\pi_R(\psi\to\HH_0)=1 \quad \forall \phi\in\HH_0
 \quad\Rightarrow\quad
 \pi(\psi\to\HH_0)=1 
 \quad\Rightarrow\quad 
 \pi_R(\psi\to\HH_0)=1.
$$

In particular, if $\HH_0=\spn\{\psi\}$,
$$
 \pi(\psi\to\psi)=1 
 \quad\Leftrightarrow\quad 
 \pi_L(\psi\to\psi)=\pi_R(\psi\to\psi)=1.
$$
\end{thm}

\begin{proof}
From Theorem~\ref{thm:split}.{\it (ii)} we know that the FR-function of $\HH_0$ with respect to $U$ is given by $f=f_Lf_R$, where $f_{L,R}$ is the FR-function of $\HH_0$ with respect to $U_{L,R}$. Hence, since $f_L$ is a Schur function,
$$
 \pi(\psi\to\HH_0) = 
 \int_0^{2\pi} \|f_L(e^{i\theta})f_R(e^{i\theta})\psi\|^2 
 \, \frac{d\theta}{2\pi}
 \le \int_0^{2\pi} \|f_R(e^{i\theta})\psi\|^2 \, \frac{d\theta}{2\pi}
 = \pi_R(\psi\to\HH_0).
$$
On the other hand, $\pi_L(\phi\to\HH_0)=1$ means that $\|f_L(e^{i\theta})\phi\|=1$ a.e. on the unit circle. Hence, $\HH_0$ is recurrent with respect to the left QW iff the boundary values of $f_L$ satisfy $f_L(e^{i\theta})^\dag f_L(e^{i\theta})=1_0$ a.e. on the unit circle. In this case,
$$
 \pi(\psi\to\HH_0) =
 \int_0^{2\pi} 
 \<f_R(e^{i\theta})\psi \,|\,
 f_L(e^{i\theta})^\dag f_L(e^{i\theta}) f_R(e^{i\theta})\psi\> 
 \, \frac{d\theta}{2\pi} 
 = \pi_R(\psi\to\HH_0),
$$
so that $\pi(\psi\to\HH_0)$ and $\pi_R(\psi\to\HH_0)$ are equal to one simultaneously. 

Besides, if $\HH_0=\spn\{\psi\}$, then
$$
 \pi(\psi\to\psi) = 
 \int_0^{2\pi} |f_L(e^{i\theta})f_R(e^{i\theta})|^2 \, \frac{d\theta}{2\pi}.
$$
Therefore, $\pi(\psi\to\psi)=1$ is equivalent to stating that $|f_L(e^{i\theta})|=|f_R(e^{i\theta})|=1$ a.e. on the unit circle, which means that $\pi_L(\psi\to\psi)=\pi_R(\psi\to\psi)=1$. 
\end{proof}

The following recurrence notions, similar to those already defined for RW, can be introduced for QW.

\begin{defn}
Given a state $\psi\in\HH$ and a closed subspace $\HH_0\subset\HH$, we define the following notions for a QW on $\HH$:

The state $\psi\in\HH$ is recurrent if $\pi(\psi\to\psi)=1$.

The state $\psi\in\HH_0$ is $\HH_0$-recurrent if $\pi(\psi\to\HH_0)=1$.

The subspace $\HH_0\subset\HH$ is recurrent if all its states are $\HH_0$-recurrent. 
\end{defn}

With this terminology, the same arguments as those given in the proof of Theorem~\ref{thm:split-QW} lead directly to the following results, which constitute the quantum version of Corollary~\ref{cor:indep-RW}. These results state that, given an overlapping factorization of a QW, its return probabilities to the overlapping subspace $\HH_0$ do not depend on the details of certain subsystems for which $\HH_0$ is recurrent. 

\begin{cor} \label{cor:indep-QW}
With the same notation as in Theorem~\ref{thm:split-QW}, the return probabilities to $\HH_0$ for the QW given by $U = (U_L \oplus 1_+)(1_- \oplus U_R)$ are independent of the left subsystem whenever the subspace $\HH_0$ is recurrent for such a subsystem. More precisely,
$$
 \pi_L(\psi\to\HH_0)=1 \quad \forall \psi\in\HH_0 
 \quad\Rightarrow\quad 
 \pi(\psi\to\HH_0)=\pi_R(\psi\to\HH_0) \quad \forall \psi\in\HH_0.
$$

If $\HH_0=\spn\{\psi\}$, a similar independence with respect to the right subsystem holds, i.e.
$$
 \pi_R(\psi\to\psi)=1 
 \quad\Rightarrow\quad 
 \pi(\psi\to\psi)=\pi_L(\psi\to\psi).
$$ 
Hence, the return probability of the overlapping state $\psi$ is independent of any of the left/right subsystems for which $\psi$ is recurrent.
\end{cor}

As it is proved in \cite{BGVW}, a finite-dimensional subspace is recurrent for a QW iff it is included in the singular subspace of the related unitary evolution operator. In particular, every subspace is recurrent for a QW on a finite-dimensional Hilbert space because this situation gives only pure point spectrum. Hence, Corollary~\ref{cor:indep-QW} has the following consequence. 

\begin{cor} \label{cor:indep-QW2}
With the same notation as in Theorem~\ref{thm:split-QW}, the return probabilities to $\HH_0$ for the QW given by $U = (U_L \oplus 1_+)(1_- \oplus U_R)$ are independent of the left subsystem whenever $\HH_L$ is finite-dimensional. More precisely,
$$
 \dim\HH_L<\infty 
 \quad\Rightarrow\quad 
 \pi(\psi\to\HH_0)=\pi_R(\psi\to\HH_0) \quad \forall \psi\in\HH_0.
$$

If $\HH_0=\spn\{\psi\}$, a similar independence with respect to the right subsystem holds, i.e.
$$
 \dim\HH_R<\infty 
 \quad\Rightarrow\quad 
 \pi(\psi\to\psi)=\pi_L(\psi\to\psi).
$$ 
Hence, the return probability of the overlapping state $\psi$ is independent of any of the left/right subsystems whose underlying Hilbert space is finite-dimensional.
\end{cor}

A key result of \cite{CGVWW} states that, if either $\dim\HH_L<\infty$ or $\dim\HH_R<\infty$, the existence of an overlapping factorization $U = (U_L \oplus 1_+)(1_- \oplus U_R)$ is characterized by the simple condition $P_+UP_-=0$, where $P_\pm$ is the orthogonal projection of $\HH$ onto $\HH_\pm$. This means that, if at least one the subspaces $\HH_L$, $\HH_R$ is finite-dimensional, the existence of splitting rules for the recurrence of a QW on $\HH=\HH_-\oplus\HH_0\oplus\HH_+$ is easily determined by checking the absence of one-step transition from $\HH_-$ to $\HH_+$, a situation represented pictorically by the following diagram of one-step transitions,
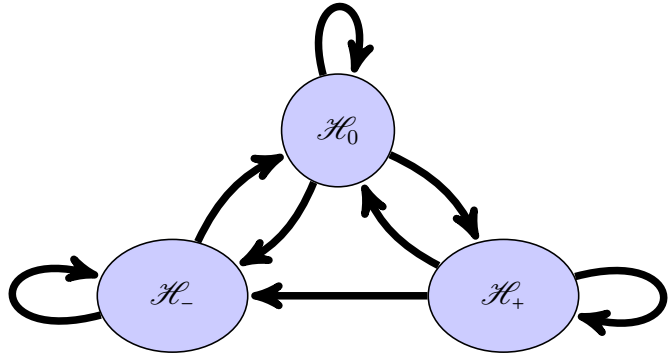
\begin{center}
\begin{tikzpicture}[->,>=stealth',shorten >=1pt,auto,node distance=1.7cm,
                    semithick]
  \tikzstyle{every state}=[ellipse,minimum height=1.5cm, 
    		minimum width=2cm,fill=blue!20,draw,inner sep=0pt,
			node distance=5cm]
  
  \node[state] (L)  			{$\HH_-$};
  \node[state] (R)  at (4.4,0) 	{$\HH_+$};
  \node[state,minimum height=1.5cm,minimum width=1.5cm] (0)  
  								at (2.2,2.2)  {$\HH_0$};
  
  \path
        (0)  edge   [loop above,line width=1mm]   	node {} (0)
        (L)  edge   [loop left,line width=1mm]   	node {} (L)
        (R)  edge   [loop right,line width=1mm]   	node {} (R)
        (0)  edge	[bend left=20,line width=1mm] 	node {} (L)
        (0)  edge   [bend left=20,line width=1mm] 	node {} (R)
        (L)  edge   [bend left=20,line width=1mm] 	node {} (0)
        (R)  edge   [bend left=20,line width=1mm] 	node {} (0)
        (R)  edge   [line width=1mm] 				node {} (L);
        
\end{tikzpicture}
\end{center}  

Explicit examples of the generating function approach to QW recurrence appear in \cite{GVWW,BGVW,CGMV2}. Additional aspects of the splitting techniques for unitary operators are discussed in \cite{CGVWW}, where one can find concrete examples of these splittings, including the case of CMV matrices, which provide the general form of a one-dimensional coined QW \cite{CGMV,CGMV2}. 

\section{Applications to open quantum walks: OQW recurrence}
\label{sec:OQW}

We have seen that the notion of FR-function has an important role to play in the study of return properties of RW and QW. To emphasize the wide applicability of the machinery developed in this paper, we indicate now very briefly that this notion is also useful in the context of a much newer topic, namely the open QW (OQW) introduced in \cite{APSS} as a kind of iterated quantum channels. A careful discussion of the use of FR-functions for the analysis of recurrence in OQW and general iterated quantum channels will be the purpose of a joint effort with Carlos Lardizabal and Albert Werner. For details on OQW the reader may consult the original reference \cite{APSS} as well as some of the contributions fostered by that one, see for instance \cite{APS,SP,SP2,SP3,SaPa,LS,LS2,CGL,CP,CP2,SKA,SKKA}. In particular, \cite{LS,CGL,SKKA} deal with recurrence problems in OQW and iterated quantum channels. Quick guides to general quantum channels are \cite{Wolf,BR} and \cite[Lectures 5 and 6]{Attal} for the finite- and infinite-dimensional settings respectively. 

OQW model QW in interaction with the enviroment. Every step of an OQW is a particular case of the so called quantum channels which, instead of the pure quantum states of a closed quantum system, represented by unit vectors of a Hilbert space $\HH$, deal with mixed quantum states, i.e. statistical ensembles of pure quantum states. A mixture of pure states $\psi_i\in\HH$ with probability distribution $p_i$ is represented by the so called density operator
$$
 \rho = \sum_i \, p_i \, |\psi_i\>\<\psi_i|,
 \qquad \psi_i\in\HH, \qquad p_i\ge0, \qquad \sum_i p_i=1,
$$
so that pure states $\psi$ correspond to rank one orthogonal projections $\rho=|\psi\>\<\psi|$. Density operators can be characterized as the non-negative definite operators on $\HH$ with trace one. The connection with open quantum systems comes from another characterization: the density operators coincide with those obtained from a pure state $\Psi$ in a dilated Hilbert space $\HH\otimes\HH_e$ by taking the partial trace of $|\Psi\>\<\Psi|$ over $\HH_e$ (see \cite{Wolf,BR} and \cite[Lecture 5]{Attal}). This allows us to treat an open system --with Hilbert space $\HH$-- coupled to an enviroment --with Hilbert space $\HH_e$-- using only objects built out of the intial open system.

The mixed quantum states are the non-negative operators of the unit sphere in the Banach space $\frak{B}_1(\HH)$ of trace class operators on $\HH$. A quantum channel must preserves such states, hence it is given by a trace preserving (TP) linear map $\mbb{T}\colon\frak{B}_1(\HH)\to\frak{B}_1(\HH)$ which is positive, i.e. mapping non-negative operators into non-negative operators. Actually, this should be the case even considering the system as a part of any dilated one $\HH\otimes\HH_e$, so $\mbb{T}\otimes 1_e$ must be positive for the identity on any $\HH_e$, which is abbreviated by saying that $\mbb{T}$ must be completeley positive (CP). CPTP maps can be equivalently described as the result of a unitary evolution on a dilated Hilbert space $\HH\otimes\HH_e$ when reduced to $\HH$ by taking the partial trace over $\HH_e$ (see \cite{Wolf,BR} and \cite[Lecture 6]{Attal}). This corresponds to the evolution of an open quantum system with Hilbert space $\HH$ under the influence of an enviroment with Hilbert space $\HH_e$. 
It can be shown that CP maps have the general form $\mbb{T}\rho=\sum_jK_j\rho K_j^\dag$ with $K_j$ bounded operators on $\HH$, and then the trace preserving condition means that $\sum_j\tr(K_j^\dag K_j)=1$ in the strong sense (see \cite{Wolf,BR} and \cite[Lecture 6]{Attal}). The additional condition $\mbb{T}(1)=\sum_j\tr(K_j K_j^\dag)=1$ defines the so called unital CPTP maps, an example of which are the transformations $\mbb{T}\rho=U\rho\,U^\dag$ induced by unitary operators $U$. 

Regarding the measurement problem in the density operator formalism, the probability of finding a state in a closed subspace $\HH_0\subset\HH$ after a measurement at a mixed quantum state represented by a density operator $\rho$ is given by $\tr(P\rho)$, where $P$ stands for the orthogonal projection of $\HH$ onto $\HH_0$. After a positive outcome of such a measurement the system collapses to the mixed state given by the new density operator $P\rho P/\tr(\rho P)$.

OQW appear by iterating certain kind of quantum channels \cite{APSS}, thus their dynamics is given by the powers of a CPTP map $\mbb{T}$. Actually, although for illustration we will use examples of OQW, the FR-function approach to recurrence in OQW presented below holds in general for iterated quantum channels. Bearing in mind the previous comments, the translation to OQW of the monitoring process to define return properties of a closed subspace $\HH_0$ in a QW amounts to the following substitutions:
$$
\begin{aligned}
 & 
 & \quad & \kern7pt \text{QW}
 & \quad & \text{OQW}
 \\
 & \text{Evolution} 
 & & \kern2pt \psi \to U\psi 
 & & \kern2pt \rho \to \mbb{T}\rho 
 \\
 & \text{Measurement} 
 \left\{ 
 \begin{aligned} 
 	& \text{Return} 
	\\[5pt] 
	& \text{No return} 
 \end{aligned}
 \right. 
 & & 
 \begin{aligned} 
 	& \psi \to P\psi 
	\\[5pt] 
	& \psi \to Q\psi 
 \end{aligned}
 & & 
 \begin{aligned} 
 	& \rho \to \mbb{P}\rho:=P\rho P 
	\\[5pt] 
	& \rho \to \mbb{Q}\rho:=Q\rho Q
 \end{aligned}
 & \quad & 
 \begin{aligned} 
 	& P = \parbox{110pt}{orthogonal projection of $\HH$ onto $\HH_0$} 
	\\[2pt] 
	& Q=1-P
	\\[-12pt]
	&
 \end{aligned}
\end{aligned}
$$
The operators $\mbb{P}T:=PTP$ and $\mbb{Q}T:=QTQ$ define projections of $\frak{B}_1(\HH)$ onto $\frak{B}_1(\HH_0)$ and $\frak{B}_1(\HH_0^\bot)$ respectively, understood as subspaces of $\frak{B}_1(\HH)$. Therefore, $\mbb{P}+\mbb{Q}\ne1$ because $\mbb{R}:=1-\mbb{P}-\mbb{Q}$ is given by
$$
 \mbb{R}T = \begin{pmatrix} 0 & PTQ \\ QTP & 0 \end{pmatrix}
$$
and defines a projection of $\frak{B}_1(\HH)$ onto the subspace of operators mapping $\HH_0$ on $\HH_0^\bot$ and viceversa. Hence, $\mbb{P}$ and $\mbb{Q}$ are not complementary projections, although $\mbb{P}\mbb{Q}=\mbb{Q}\mbb{P}=0$. Besides, for every $T\in\frak{B}_1(\HH)$ we have that $\|\mbb{P}T\|_1\le\|P\|^2\|T\|_1=\|T\|_1$, where $\|\cdot\|$ and $\|\cdot\|_1$ stand for the operator and trace norm respectively. Hence, the operator norm of these projections with respect to the trace norm in $\frak{B}_1(\HH)$ is $\|\mbb{P}\|=\|\mbb{Q}\|=1$. We will refer to $\mbb{P}$ as the projection on $\frak{B}_1(\HH)$ induced by the orthogonal projection $P$ on $\HH$, and analogously for $\mbb{Q}$.

The corresponding recurrence notions for OQW are related to FR-functions too, but for operators on the Banach space $\frak{B}_1(\HH)$ rather than on the Hilbert space $\HH$. The generating function which encodes the return properties of the closed subspace $\HH_0$ is the function with values in operators on $\frak{B}_1(\HH_0)$ given by 
\begin{equation} \label{eq:f-OQW}
 \mbb{F}(z) = \mbb{P}\mbb{T}(1-z\mbb{Q}\mbb{T})^{-1}\mbb{P}
 = \sum_{n\ge1} z^{n-1} \mbb{A}_n,
 \qquad \mbb{A}_n = \mbb{P}\mbb{T}(\mbb{Q}\mbb{T})^{n-1}\mbb{P}, 
 \qquad z\in\D,
\end{equation}
which is analytic and bounded for every $z\in\D$ because $\|\mbb{Q}\mbb{T}\|\le1$ since $\|\mbb{T}\|=1$ for every CPTP map $\mbb{T}$ on $\frak{B}_1(\HH)$ \cite{CP}. The expression \eqref{eq:f-OQW} does not define a strict FR-function because $\mbb{Q}\ne1-\mbb{P}$. Nevertheless, $\mbb{F}$ is the result of projecting a true FR-function, namely the FR-function $f$ of the projection $1-\mbb{Q}$ with respect to the operator $\mbb{T}$, 
\begin{equation} \label{eq:F-f-OQW}
\begin{gathered}
 \mbb{F}(z) = \mbb{P}f(z)\mbb{P}, \qquad z\in\D,
 \\
 f(z) = (1-\mbb{Q})\mbb{T}(1-z\mbb{Q}\mbb{T})^{-1}(1-\mbb{Q})
 = \sum_{n\ge1} z^{n-1} a_n,
 \quad a_n = (1-\mbb{Q})\mbb{T}(\mbb{Q}\mbb{T})^{n-1}(1-\mbb{Q}).
\end{gathered}
\end{equation}
By a slight abuse of language, we will also refer to $f$ as the FR-function of $\HH_0$ with respect to $\mbb{T}$, while $\mbb{F}$ will be called the corresponding reduced FR-function.

We will use an obvious adaptation of the notation and terminology for QW recurrence to the case of OQW. Then, the following relations summarize the connections between the reduced FR-function \eqref{eq:f-OQW} and the return properties of the closed subspace $\HH_0$ in an OQW driven by $\mbb{T}$. In these relations we assume that $\rho$ is a density operator in $\frak{B}_1(\HH_0)$ --understood as a subspace of $\frak{B}_1(\HH)$-- and $\psi$ is a pure state in $\HH_0$.  
$$
\begin{aligned}
 & \begin{aligned}
 	& \pi_n(\rho\xrightarrow{\HH_0}\psi) = 
	\tr(|\psi\>\<\psi|\mbb{T}(\mbb{Q}\mbb{T})^{n-1}\rho) = 
	\<\psi|\mbb{A}_n\rho\psi\>,
	\\[3pt]
	& \pi_n(\rho\to\HH_0) = \tr(\mbb{P}\mbb{T}(\mbb{Q}\mbb{T})^{n-1}\rho) =
	\tr(\mbb{A}_n\rho),
 \end{aligned}
 \qquad \mbb{A}_n=\mbb{P}\mbb{T}(\mbb{Q}\mbb{T})^{n-1}\mbb{P}, 
 \\
 & \pi(\rho\xrightarrow{\HH_0}\psi) = 
 \sum_{n\ge1} \pi_n(\rho\xrightarrow{\HH_0}\psi) =
 \sum_{n\ge1} \<\psi|\mbb{A}_n\rho\psi\> = 
 \lim_{x\uparrow1} \<\psi|\mbb{F}(x)\rho\psi\>,
 \\
 & \pi(\rho\to\HH_0) = \sum_{n\ge1} \pi_n(\rho\to\HH_0) =
 \sum_{n\ge1} \tr(\mbb{A}_n\rho) = \lim_{x\uparrow1} \tr(\mbb{F}(x)\rho),
 \\[5pt]
 & \tau(\rho\to\HH_0) = 
 \begin{cases}
 	\infty, 
	& \text{ if } \pi(\rho\to\HH_0)<1,
	\\ \displaystyle 
	\sum_{n\ge1} n\,\pi_n(\rho\to\HH_0) 
 	= \lim_{x\uparrow1} \frac{d}{dx} \, x\tr(\mbb{F}(x)\rho)
	\\ \displaystyle
	\kern88pt = 1 + \lim_{x\uparrow1} \frac{d}{dx} \tr(\mbb{F}(x)\rho),
	& \text{ if } \pi(\rho\to\HH_0)=1.
 \end{cases}
\end{aligned}
$$
The similarity among the above relations and the analogous ones for RW is remarkable. As it was shown in Sections~\ref{sec:RW} and \ref{sec:QW}, while the recurrence properties of RW are controled by the behaviour of the related Banach FR-functions around the single point 1, QW recurrence depends on the values of the corresponding Hilbert FR-functions on the whole unit circle. Nevertheless, as in the RW case, the recurrence properties of OQW are codified by the values of the related Banach FR-functions around 1. This does not mean that RW and OQW recurrence must be similar. Actually, unitary QW can be formulated as iterated quantum channels, hence the above Banach FR-function approach may be used to describe their recurrence properties, which are surprisingly different from those of RW, as shown in \cite{GVWW,BGVW}. Maybe the only useful conclusion to draw is that the density operator formalism places the generating function approach to quantum recurrence formally closer to that of classical recurrence.     

The FR-function $f$ given in \eqref{eq:F-f-OQW} is not only different from the reduced one $\mbb{F}$, but has values in a larger subspace of $\frak{B}_1(\HH)$, namely the range of the projection 
$$
 (1-\mbb{Q})T = \begin{pmatrix} PTP & PTQ \\ QTP & 0 \end{pmatrix},
$$
which is the subspace of operators mapping $\HH_0^\bot$ on $\HH_0$. Nevertheless, all the previous relations which express return probabilities and expected return times in terms of $\mbb{F}$ and its Taylor coefficients $\mbb{A}_n$ remain valid when substituting them by the FR-function $f$ and its Taylor coefficients $a_n$ respectively. The reason for this is that such relations always involve either inner products $\<\psi|\mbb{A}_n\rho\psi\>$ or traces $\tr(\mbb{A}_n\rho)$, where $\rho\in\frak{B}_1(\HH_0)$ and $\psi\in\HH_0$. Since $1-\mbb{Q}=\mbb{P}+\mbb{R}$ with $\mbb{R}\rho=P\rho Q+Q\rho P$ projecting onto operators mapping $\HH_0$ on $\HH_0^\bot$ and viceversa, we find that $\<\psi|a_n\rho\psi\>=\<\psi|\mbb{A}_n\rho\psi\>$ and $\tr(a_n\rho)=\tr(\mbb{A}_n\rho)$. 

The reduced FR-function $\mbb{F}$ is more convenient for practical calculations, like those shown in the examples below. However, the true FR-function $f$ becomes a key tool for theoretical discussions since it allows one to take advantage of the machinery developed for FR-functions. For instance, the generalized renewal equation \eqref{eq:sf} relates the FR-function $f$, generating function of the first return operators $a_n$, and the generating function of the return operators without monitoring given by the Stieltjes function 
\begin{equation} \label{eq:St-OQW}
 s(z) = (1-\mbb{Q})(1-z\mbb{T})^{-1}(1-\mbb{Q}) 
 = \sum_{n\ge0} z^n (1-\mbb{Q})\mbb{T}^n(1-\mbb{Q}),
 \qquad z\in\D.
\end{equation}
Such a relation can be considered as the renewal equation for OQW. An analogue of this renewal equation for the reduced FR-function $\mbb{F}$ should state that the Stieltjes function 
\begin{equation} \label{eq:SSt-OQW}
 \mbb{S}(z) = \mbb{P}(1-z\mbb{T})^{-1}\mbb{P}
\end{equation} 
and $\1_0-z\mbb{F}(z)$ are inverses of each other for $z\in\D$, where $\1_0$ stands for the identity on $\frak{B}_1(\HH_0)$. However, this relation is not true because $\mbb{F}$ is not an FR-function. Furthermore, the splitting rules of Theorem~\ref{thm:split} apply directly to $f$ since it is an FR-function. We will illustrate all this in the examples below, but a detailed discussion of the corresponding splitting rules for OQW and general iterated quantum channels will be explored in our joint collaboration with C.~Lardizabal and A.~H.~Werner.

Analogously to the case of RW and QW, we should prove that $\sum_{n\ge1}\tr(a_n\rho)\le1$ for every density operator $\rho$ in $\frak{B}_1(\HH_0)$. This is a direct consequence of the following result.

\begin{prop} \label{prop:schur-OQW}
Let $P$ be the orthogonal projection of a Hilbert space $\HH$ onto a closed subspace $\HH_0$, and $\mbb{Q}T=QTQ$ the projection on $\frak{B}_1(\HH)$ induced by $Q=1-P$. Then, the Taylor coefficients $a_n=(1-\mbb{Q})\mbb{T}(\mbb{Q}\mbb{T})^{n-1}(1-\mbb{Q})$ of the FR-function $f$ for the projection $1-\mbb{Q}$ with respect to a CPTP map $\mbb{T}$ on $\frak{B}_1(\HH)$ satisfy 
$$
 \sum_{n\ge1}\tr(a_nT)\le\tr T,
 \qquad T\in\frak{B}_1(\HH_0),
 \qquad T\ge0.
$$
\end{prop}

\begin{proof}
Using the decomposition 
$$
 \mbb{T}(\mbb{Q}\mbb{T})^{n-1} = a_n + (\mbb{Q}\mbb{T})^n,
$$
and the fact that $\mbb{T}$ is trace preserving we get
$$
 \tr(a_nT) = \tr((\mbb{Q}\mbb{T})^{n-1}) - \tr((\mbb{Q}\mbb{T})^n).
$$
Hence,
$$
 \sum_{n=1}^N \tr(a_nT) = \tr(T) - \tr((\mbb{Q}\mbb{T})^NT),
 \qquad T\in\frak{B}_1(\HH).
$$
If $T\ge0$, the above equality proves that $\sum_{n=1}^N \tr(a_nT) \le \tr(T)$ because $\tr((\mbb{Q}\mbb{T})^NT)\ge0$ since $\mbb{Q}$ and $\mbb{T}$ are positive maps. On the other hand, if $\mbb{P}$ is the projection on $\frak{B}_1(\HH)$ induced by $P$, then $\mbb{A}_n=\mbb{P}a_n\mbb{P}=\mbb{P}\mbb{T}(\mbb{Q}\mbb{T})^{n-1}\mbb{P}$ is also a positive map such that $\tr(a_nT)=\tr(\mbb{A}_nT)$ whenever $T\in\frak{B}_1(\HH_0)$, hence $\sum_{n\ge1} \tr(a_nT)$ is a series of non-negative terms for $T\ge0$, which therefore converges to a sum bounded by $\tr T$.
\end{proof}

\subsection{Examples of OQW recurrence}
\label{ssec:OQW-ex}

We illustrate the use of generating functions with a few examples of unital and non-unital OQW, kindly provided to us --together with the corresponding diagrams below-- by C.~Lardizabal. It is a pleasure to thank him for very useful guidance on the topic of OQW.

In all these examples the underlying Hilbert space $\HH=\C^2\otimes\spn\{|1\>,|2\>,\dots\}$ is constituted by several copies of $\C^2$ which describe the internal degrees of freedom at each site $|i\>$ of a network. According to \cite{APSS}, the corresponding OQW evolution of a general density operator $\rho$ on $\HH$ is given by the following CPTP map
$$
 \mbb{T}\rho = \sum_{i,j} M^i_j \rho {M^i_j}^\dag,
 \qquad
 M^i_j = B^i_j \otimes |i\>\<j|,
$$
where the $2\times2$ complex matrices $B^i_j$, which act on the $\C^2$ factor of $\HH$, satisfy the trace preserving condition $\sum_i {B^i_j}^\dag B^i_j=I_2$ for all $i$, with $I_2$ the $2\times2$ identity matrix. The unital examples satisfy the additional condition $\sum_j B^i_j {B^i_j}^\dag =I_2$ for all $j$. In any case, this kind of evolution makes the density operator $2\times2$-block diagonal after the first step because \cite{APSS}
$$
 \mbb{T}\rho = \sum_i (\sum_j B^i_j\rho_j{B^i_j}^\dag) \otimes |i\>\<i|, 
 \qquad \rho_j := \<j|\rho|j\>,
$$ 
so we can suppose without loss of generality that the mixed states have this block diagonal form $\rho=\sum_i\rho_i\otimes|i\>\<i|$. This means that we will consider $\mbb{T}$ as an operator on the trace class subspace of $\bigoplus_i\C^{2,2}\otimes|i\>\<i|$ rather than on the larger space $\frak{B}_1(\HH)$. Concerning the trace class condition, it will play no role in the following examples because they are all finite-dimensional.

\begin{ex} \label{ex:OQW1}

The first example is a unital OQW in a network with two sites. The one-step transitions considered are graphically represented in the following diagram,
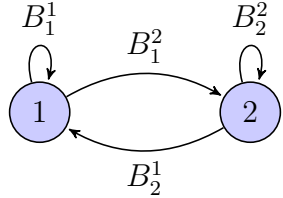
\begin{center}
\begin{tikzpicture}[->,>=stealth',shorten >=1pt,auto,node distance=2.8cm,
                    semithick]
  \tikzstyle{every state}=[fill=black,draw=none,text=white]

  \node[main node] (B)  {$1$};
  \node[main node] (E) [right of=B] {$2$};
  
  \path (B) edge   [loop above]     node {$B_1^1$} (B)
        (B) edge   [bend left]     	node {$B_1^2$} (E)
        (E) edge   [bend left]     	node {$B_2^1$} (B)
        (E) edge   [loop above]     node {$B_2^2$} (E);
        
\end{tikzpicture}
\end{center}
where the matrices $B_i^j$ are given as follows,
$$
 B_1^1=B_2^2=\frac{1}{\sqrt{3}}\begin{pmatrix} 1 & 1 \\ 0 & 1 \end{pmatrix},
 \quad
 B_1^2=B_2^1=\frac{1}{\sqrt{3}}\begin{pmatrix} 1 & 0 \\ -1 & 1 \end{pmatrix}.
$$
The reduced FR-function $\mbb{F}$ for the site $|1\>$ follows from \eqref{eq:f-OQW} by using the projection $P=I_2\otimes|1\>\<1|$. This leads to the expression  
$$
 \mbb{F}(z) = \frac{1}{(z-3)^3} 
 \begin{pmatrix}
 \scriptstyle -(z+3)(z^2-3z+3) & \scriptstyle z^3-3z^2+9z-9 
 & \scriptstyle z^3-3z^2+9z-9 & \scriptstyle -(4z^2-9z+9)
 \\
 \scriptstyle 2z(z^2-3z+3) & \scriptstyle -(z+3)(z^2-3z+3) 
 & \scriptstyle -z^2(z-1) & \scriptstyle z^3-3z^2+9z-9
 \\
 \scriptstyle 2z(z^2-3z+3) & \scriptstyle -z^2(z-1) 
 & \scriptstyle -(z+3)(z^2-3z+3) & \scriptstyle z^3-3z^2+9z-9
 \\
 \scriptstyle -z(3z^2-11z+12) & \scriptstyle 2z(z^2-3z+3) 
 & \scriptstyle 2z(z^2-3z+3) & \scriptstyle -(z+3)(z^2-3z+3)
 \end{pmatrix},
$$
which is understood as a function with values in operators $\rho=\rho_1\otimes|1\>\<1| \mapsto \sigma=\sigma_1\otimes|1\>\<1|$ acting as follows,
\begin{equation} \label{eq:Fvec}
 \mbb{F}(z) 
 \begin{pmatrix} 
 	\rho_{11} \\ \rho_{12} \\ \rho_{13} \\ \rho_{14} 
 \end{pmatrix} 
 = \begin{pmatrix} 
 	\sigma_{11}(z) \\ \sigma_{12}(z) \\ \sigma_{13}(z) \\ \sigma_{14}(z) 
 \end{pmatrix},
 \qquad
 \rho_1 = 
 \begin{pmatrix}
 	\rho_{11} & \rho_{12} \\ \rho_{13} & \rho_{14}
 \end{pmatrix},
 \qquad
 \sigma_1 = 
 \begin{pmatrix}
 	\sigma_{11} & \sigma_{12} \\ \sigma_{13} & \sigma_{14}
 \end{pmatrix}.
\end{equation} 
For instance, starting at the pure state $\psi=(1,0)^t\otimes|1\>$, which corresponds to the density operator $\rho=|\psi\>\<\psi|=\left(\begin{smallmatrix}1&0\\0&0\end{smallmatrix}\right)\otimes|1\>\<1|$, we have that
$$
 \rho_1 = 
 \begin{pmatrix} 
 	1 & 0 \\ 0 & 0 
 \end{pmatrix}
 \quad \Rightarrow \quad 
 \sigma_1 = 
 \begin{pmatrix}
 	\sigma_{11} & \sigma_{12} \\ \sigma_{13} & \sigma_{14}
 \end{pmatrix},
 \quad
 \begin{pmatrix} 
 	\sigma_{11}(z) \\ \sigma_{12}(z) \\ \sigma_{13}(z) \\ \sigma_{14}(z) 
 \end{pmatrix} 
 = \mbb{F}(z) 
 \begin{pmatrix} 
 	1 \\ 0 \\ 0 \\ 0 
 \end{pmatrix}
 = 
 \begin{pmatrix} 
 	\mbb{F}_{11}(z) \\ \mbb{F}_{21}(z) \\ \mbb{F}_{31}(z) \\ \mbb{F}_{41}(z) 
 \end{pmatrix}.
$$
Therefore, the probability of returning to site $|1\>$ starting at the state $\psi$ is given by the value at $z=1$ of
$$
 \tr\sigma_1(z) = \mbb{F}_{11}(z)+\mbb{F}_{41}(z) 
 = -\frac{4z^3-11z^2+6z+9}{(z-3)^3}.
$$
Hence $\pi(\rho\to|1\>)=1$, i.e. the state $\psi$ is $|1\>$-recurrent. The corresponding expected return time, obtained by taking the derivative of $z\tr\sigma_1(z)$ at $z=1$, turns out to be $\tau(\rho\to|1\>)=2$. Besides, the $n$-th coefficient of the Taylor expansion of $z\tr\sigma_1(z)$ around the origin is the probability of the first return to site $|1\>$ in $n$ steps. This leads to the following probabilities for the first values of $n$,

\smallskip

\begin{center}
\renewcommand{\arraystretch}{1.25}
\begin{tabular}{|c|c|c|c|c|c|c|c|}
\hline
	$n$   
  & 1 & 2 & 3 & 4 & 5 & 6 & 7
\\ \hline
 	$\pi_n(\rho\to|1\>)$ 
  & $1/3$ & 5/9 & 1/27 & 1/81 & 5/243 & 13/729 & 25/2187 
\\ \hline
\end{tabular}
\end{center} 

\medskip 

We can also ask about the probability of returning to site $|1\>$ landing on the initial state $\psi$. This is given by the value at $z=1$ of
$$
 \sigma_{11}(z)=\mbb{F}_{11}(z)=-\frac{(z+3)(z^2-3z+3)}{(z-3)^3},
$$
which yields $\pi(\rho\xrightarrow{|1\>}\psi)=1/2$, while its power expansion provides the $n$-step first return probabilities

\smallskip

\begin{center}
\renewcommand{\arraystretch}{1.25}
\begin{tabular}{|c|c|c|c|c|c|c|c|}
\hline
	$n$   
  & 1 & 2 & 3 & 4 & 5 & 6 & 7
\\ \hline
 	$\pi_n(\rho\xrightarrow{|1\>}\psi)$ 
  & $1/3$ & 1/9 & 0 & 1/81 & 4/243 & 1/81 & 16/2187 
\\ \hline
\end{tabular}
\end{center} 

\medskip

The above probabilities should be distinguished from those of just returning to the state $\psi$ --without any thought being given to returning to site $|1\>$--, which is calculated from a new reduced FR-function \eqref{eq:f-OQW} built out of the projection $P=|\psi\>\<\psi|=\left(\begin{smallmatrix}1&0\\0&0\end{smallmatrix}\right)\otimes|1\>\<1|$ instead of $P=I_2\otimes|1\>\<1|$. The result is the following generating function of first returns 
\begin{equation} \label{eq:F-ex1}
  \mbb{F}(z) = -\frac{z^4-2z^3+5z^2-9z+9}{z^4+z^3-15z^2+36z-27},
\end{equation}
whose value at $z=1$ gives $\pi(\rho\to\psi)=1$, thus the state $\psi$ is also recurrent. Evaluating the derivative of $z\mbb{F}(z)$ at $z=1$ we find the expected return time $\tau(\rho\to\psi)=4$. Also, the Taylor expansion of $\mbb{F}(z)$ yields

\smallskip

\begin{center}
\renewcommand{\arraystretch}{1.25}
\begin{tabular}{|c|c|c|c|c|c|c|c|}
\hline
	$n$   
  & 1 & 2 & 3 & 4 & 5 & 6 & 7
\\ \hline
 	$\pi_n(\rho\to\psi)$ 
  & $1/3$ & 1/9 & 4/27 & 2/27 & 17/243 & 5/81 & 113/2187 
\\ \hline
\end{tabular}
\end{center}

\medskip

To clarify the difference between the reduced FR-function $\mbb{F}$ and the true FR-function $f$ we will also present the last one for the state $\psi$, which is given by \eqref{eq:F-f-OQW} with $\mbb{Q}$ the projection on $(\C^{2,2}\otimes|1\>\<1|)\oplus(\C^{2,2}\otimes|2\>\<2|)$ defined by $\mbb{Q} \rho = (1-P) \rho (1-P)$, i.e.  
$$
 \mbb{Q} \rho = 
 \left(\left(\begin{smallmatrix}0&0\\0&1\end{smallmatrix}\right) \rho_1
 \left(\begin{smallmatrix}0&0\\0&1\end{smallmatrix}\right)
 \otimes|1\>\<1|\right) 
 \oplus (\rho_2\otimes|2\>\<2|),
 \qquad
 \rho = (\rho_1\otimes|1\>\<1|) \oplus (\rho_2\otimes|2\>\<2|).  
$$
Therefore, $1-\mbb{Q}$ is the rank 3 projection
$$
 (1-\mbb{Q})\rho = 
 \begin{pmatrix} 
 	\rho_{11} & \rho_{12} \\ \rho_{13} & 0
 \end{pmatrix}
 \otimes |1\>\<1|,
 \qquad
 \rho_1 =  
 \begin{pmatrix} 
 	\rho_{11} & \rho_{12} \\ \rho_{13} & \rho_{14}
 \end{pmatrix}, 
$$
and, according to \eqref{eq:F-f-OQW}, leads to the following $3\times3$ matrix FR-function
$$
 f(z) = \textstyle -\frac{1}{z^4+z^3-15z^2+36z-27}
 \begin{pmatrix}
 	\scriptstyle z^4-2z^3+5z^2-9z+9 
	& \scriptstyle -(z-1)(z^3-3z+9) 
	& \scriptstyle -(z-1)(z^3-3z+9)
	\\[3pt]
	\scriptstyle z(z-1)(z^2-6z+6) 
	& -\frac{(z^2-3z+3)(z^3-4z^2-3z+9)}{z-3}
	& -\frac{z^2(z^3-6z^2+13z-9)}{z-3}
	\\[3pt]
	\scriptstyle z(z-1)(z^2-6z+6) 
	& -\frac{z^2(z^3-6z^2+13z-9)}{z-3}
	& -\frac{(z^2-3z+3)(z^3-4z^2-3z+9)}{z-3}	
 \end{pmatrix}. 
$$
The value of $f(z)$ must be understood as the operator $\rho=\rho_1\otimes|1\>\<1| \mapsto \sigma=\sigma_1\otimes|1\>\<1|$ on $\RR(1-\mbb{Q})$ given by 
$$
 f(z)  
 \begin{pmatrix} 
 	\rho_{11} \\ \rho_{12} \\ \rho_{13} 
 \end{pmatrix} 
 = \begin{pmatrix} 
 	\sigma_{11}(z) \\ \sigma_{12}(z) \\ \sigma_{13}(z) 
 \end{pmatrix},
 \qquad
 \rho_1 = 
 \begin{pmatrix}
 	\rho_{11} & \rho_{12} \\ \rho_{13} & 0
 \end{pmatrix},
 \qquad
 \sigma_1 = 
 \begin{pmatrix}
 	\sigma_{11} & \sigma_{12} \\ \sigma_{13} & 0
 \end{pmatrix},
$$
so that $\mbb{F}=f_{1,1}$.
A direct calculation shows that the corresponding Stieltjes function \eqref{eq:St-OQW} has the expression 
$$
 s(z) = \textstyle \frac{1}{3(z^2-3)(2z^2-3z+3)}
 \begin{pmatrix}
 	\frac{2z^5-8z^4+3z^3+24z^2-45z+27}{z-1} 
	& \kern-9pt \scriptstyle -z(z^3-3z+9) 
	& \kern-9pt \scriptstyle -z(z^3-3z+9) 
	\\[3pt] 
	\scriptstyle z^2(z^2-6z+6) 
	& \kern-9pt \frac{2z^5-12z^4+3z^3+54z^2-108z+81}{2z-3} 
	& \kern-9pt -\frac{z^3(4z^2-15z+15)}{2z-3}
	\\[3pt] 
	\scriptstyle z^2(z^2-6z+6) 
	& \kern-9pt -\frac{z^3(4z^2-15z+15)}{2z-3} 
	& \kern-9pt \frac{2z^5-12z^4+3z^3+54z^2-108z+81}{2z-3} 
 \end{pmatrix},
$$
and is related to the FR-function $f$ by the renewal equation $s(z)=(I_3-zf(z))^{-1}$. However, the Stieltjes function \eqref{eq:SSt-OQW}, which reads in this case as
$$
 \mbb{S}(z) = s(z)_{1,1} 
 = \frac{2z^5-8z^4+3z^3+24z^2-45z+27}{3(z-1)(z^2-3)(2z^2-3z+3)},
$$
has no analogous relation with the reduced FR-function $\mbb{F}$ in \eqref{eq:F-ex1} since we have 
$$
 (1-z\mbb{F}(z))^{-1} = \frac{z^4+z^3-15z^2+36z-27}{(z-1)(z^4+6z^2-18z+27)}.
$$

\end{ex}

\begin{ex} \label{ex:OQW2}

Unital OQW are the closest ones to unitary QW. For instance, the previous example illustrates a general property of finite-dimensional unital OQW, actually proved in \cite{SKKA} for the general setting of iterated quantum channels and first uncovered for the unitary case in \cite{GVWW}: every state has an integer expected return time which coincides with the dimension of the subspace of the Hilbert space explored by the state, thus it can not be greater than the dimension of the whole Hilbert space. As the next example shows, this does not hold in the non-unital case. 

The non-unital example is again an OQW with two sites, but the one-step transition matrices are now as follows,
$$
 B_{11}=\frac{1}{\sqrt{2}}\begin{pmatrix} 1 & 1 \\ 0 & 0 \end{pmatrix},
 \quad 
 B_{12}=\frac{1}{\sqrt{2}}\begin{pmatrix} 0 & 0 \\ 1 & -1\end{pmatrix},
 \quad
 B_{21}=\frac{1}{\sqrt{3}}\begin{pmatrix} 1 & 1 \\ 0 & 1 \end{pmatrix},
 \quad
 B_{22}=\frac{1}{\sqrt{3}}\begin{pmatrix} 1 & 0 \\ -1 & 1 \end{pmatrix}.
$$
The corresponding reduced FR-function for the pure state $\psi=(1,0)^t\otimes|1\>$ is given by
$$
 \mbb{F}(z) = -\frac{3}{z^2+2z-6},
$$
and yields return probability 1 but with non-integer expected return time 7/3. The related $n$-step first return probabilities appear in the following table. 

\smallskip

\begin{center}
\renewcommand{\arraystretch}{1.25}
\begin{tabular}{|c|c|c|c|c|c|c|c|}
\hline
	$n$   
  & 1 & 2 & 3 & 4 & 5 & 6 & 7
\\ \hline
 	$\pi_n(\rho\to\psi)$ 
  & $1/2$ & 1/6 & 5/36 & 2/27 & 31/648 & 55/1944 & 203/11664 
\\ \hline
\end{tabular}
\end{center}

\end{ex}

\begin{ex} \label{ex:OQW3}

The return properties of OQW with two sites can be tackled with more direct methods since the path counting is not too difficult for them. However, this situation changes dramatically for OQW with three or more sites. As another example we will present an OQW with three sites which will show more clearly the advantages of the generating function approach. This example appeared originally in \cite{CP2} and its relevance was suggested to us by C.~Lardizabal. The corresponding one-step transitions are represented in the following diagram, 
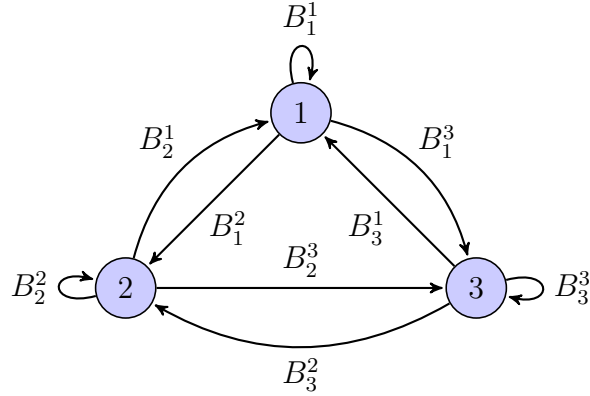
\begin{center}
\begin{tikzpicture}
[->,>=stealth',shorten >=1pt,auto,node distance=1.75cm,
                    semithick]
    \node[main node] (1) {$1$};
    \node[main node] (2) [below left = 1.75cm and 1.75cm of 1]  {$2$};
    \node[main node] (3) [below right = 1.75cm and 1.75cm of 1] {$3$};

    \path[draw,thick]
    (1) edge   [loop above]     node {$B_1^1$} (1)
    (1) edge      				node {\kern-6pt $B_1^2$} (2)
    (2) edge   [bend left]     	node {$B_2^1$ \kern-6pt} (1)
    (2) edge   [loop left]     	node {$B_2^2$} (2)
    (2) edge        			node {$B_2^3$} (3)
    (3) edge   [bend left]     	node {$B_3^2$} (2)
    (3) edge   [loop right]     node {$B_3^3$} (3)
    (3) edge        			node {$B_3^1$ \kern-7pt} (1)
    (1) edge   [ bend left]     node {\kern-5pt $B_1^3$} (3);
    
\end{tikzpicture}
\end{center}
where the matrices $B_i^j$ are chosen as
$$
 B_1^2=B_2^3=B_3^1=\frac{1}{2}\begin{pmatrix} 1 & 0 \\ -1 & 1\end{pmatrix},
 \quad 
 B_1^3=B_3^2=B_2^1=\frac{1}{2}\begin{pmatrix} 1 & 1 \\ 0 & 1\end{pmatrix},
 \quad
 B_1^1=B_2^2=B_3^3=\frac{1}{2}\begin{pmatrix} 1 & 0 \\ 0 & 1\end{pmatrix}.
$$
For this example we have computed again the reduced FR-function for the pure state $\begin{pmatrix}1,0\end{pmatrix}^t\otimes|1\>$, which is explicitly given by
$$
 \mbb{F}(z) = 
 -\frac{2z^7+7z^6-37z^5-10z^4+112z^3-512z^2+768z-512}
 {z^7-13z^6+10z^5+184z^4-1024z^3+2560z^2-3584z+2048}.
$$
This leads to a return probability equal to 1 and an integer expected return time equal to 6, in agreement with the general result in \cite{SKKA}. The corresponding $n$-step first return probabilities look like

\smallskip

\begin{center}
\renewcommand{\arraystretch}{1.25}
\begin{tabular}{|c|c|c|c|c|c|c|c|}
\hline
	$n$   
  & 1 & 2 & 3 & 4 & 5 & 6 & 7
\\ \hline
 	$\pi_n(\rho\to\psi)$ 
  & $1/4$ & 1/16 & 3/64 & 19/256 & 87/1024 & 371/4096 & 1361/16384 
\\ \hline
\end{tabular}
\end{center}

\end{ex}

\begin{ex} \label{ex:OQW4}

We now tackle the problem of OQW splitting and its impact on the reduced FR-function. We will illustrate this splitting by decomposing a three site OQW into a ``sum'' of a couple of two site OQW which overlap on a single site. Although CP maps are preserved by sums, this is not the case for trace preserving ones. Hence, the compatibility of trace preservation for the original OQW and the ``sumands'' will require an extra term for the overlapping site, very much in line with the case of the splitting rules for RW recurrence. It is worth noticing that the older result is a very good guide of how to proceed in this new situation.

The alluded three site OQW is an example with nearest neighbour interactions depicted in the following diagram
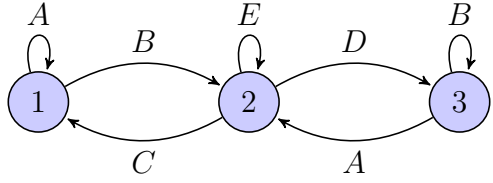
\begin{center}
\begin{tikzpicture}[->,>=stealth',shorten >=1pt,auto,node distance=2.8cm,
                    semithick]
  \tikzstyle{every state}=[fill=black,draw=none,text=white]

  \node[main node] (B)  {$1$};
  \node[main node] (D) [right of=B] {$2$};
  \node[main node] (E) [right of=D] {$3$};
  
  \path (B) edge   [loop above]     node {$A$} (B)
        (B) edge   [bend left]     	node {$B$} (D)
        (D) edge   [bend left]     	node {$C$} (B)
        (D) edge   [bend left]     	node {$D$} (E)
        (D) edge   [loop above]     node {$E$} (D)
        (E) edge   [bend left]     	node {$A$} (D)
        (E) edge   [loop above]     node {$B$} (E);
\end{tikzpicture}
\end{center}
where 
$$
 A = \frac{1}{\sqrt{2}} \begin{pmatrix} 1 & 1 \\ 0 & 0 \end{pmatrix},
 \quad 
 B = \frac{1}{\sqrt{2}} \begin{pmatrix} 0 & 0 \\1 & -1 \end{pmatrix},
 \quad
 C=\sqrt{1-\epsilon}B,
 \quad 
 D=\sqrt{1-\epsilon}A,
 \quad
 E=\sqrt{\epsilon} I_2,
$$
for some $\epsilon\in[0,1]$. The corresponding CP map 
$$
 \mbb{T}\rho = 
 (A\rho_1A^\dag+C\rho_2C^\dag)\otimes|1\>\<1| +
 (B\rho_1B^\dag+E\rho_2E^\dag+A\rho_3A^\dag)\otimes|2\>\<2| +
 (D\rho_2D^\dag+B\rho_3B^\dag)\otimes|3\>\<3|
$$
is trace preserving because $A^\dag A + B^\dag B = C^\dag C + D^\dag D +E^\dag E = I_2$, although it is unital only for $\epsilon=0$. In the example given here we will take $\epsilon=1/2$.

We proceed to construct two OQW with an overlap on the middle site, so that they respect most of the existing transitions in the larger OQW, in fact all of them except the self-transitions at the overlapping site. These QW are given by the diagrams
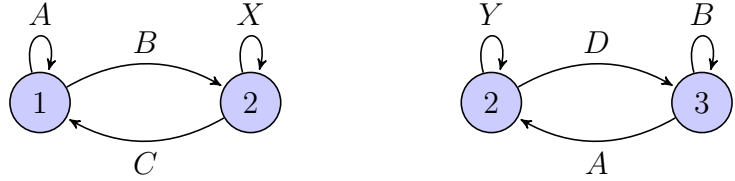
\begin{center}
\begin{tikzpicture}[->,>=stealth',shorten >=1pt,auto,node distance=2.8cm,
                    semithick]
  \tikzstyle{every state}=[fill=black,draw=none,text=white]

  \node[main node] (B)  {$1$};
  \node[main node] (D) [right of=B] {$2$};
  
  \path (B) edge   [loop above]     node {$A$} (B)
        (B) edge   [bend left]     	node {$B$} (D)
        (D) edge   [bend left]    	node {$C$} (B)
        (D) edge   [loop above]     node {$X$} (D);

  \node[main node] (C) at (6,0) {$2$};
  \node[main node] (E) [right of=C] {$3$};
 
  \path (C) edge   [bend left]     	node {$D$} (E)
        (C) edge   [loop above]     node {$Y$} (C)
        (E) edge   [bend left]     	node {$A$} (C)
        (E) edge   [loop above]     node {$B$} (E);
        
\end{tikzpicture}
\end{center}
and, in agreement with the previous terminology, we will refer to them as the left and right OQW respectively. The new transition matrices $X$ and $Y$ must be chosen such that the associated CP maps 
$$
\begin{aligned}
 & \mbb{T}_L\rho = (A\rho_1A^\dag+C\rho_2C^\dag)\otimes|1\>\<1| +
 (B\rho_1B^\dag+X\rho_2X^\dag)\otimes|2\>\<2|,
 \\ 
 & \mbb{T}_R\rho = (Y\rho_2Y^\dag+A\rho_3A^\dag)\otimes|2\>\<2| +
 (D\rho_2D^\dag+B\rho_3B^\dag)\otimes|3\>\<3|,
\end{aligned}
$$
are also trace preserving, i.e. $C^\dag C+X^\dag X = Y^\dag Y + D^\dag D = I_2$. Both $X$ and $Y$ are not unique and a possible choice is given by
$$
 X = -\frac{1}{2} 
 \begin{pmatrix} 
	\sqrt{2} & \sqrt{2}
	\\ 
	1 & -1 
 \end{pmatrix},
 \qquad
 Y = -\frac{1}{6}
 \begin{pmatrix}
 	\sqrt{2} & 3\sqrt{2}
 	\\
	5 & -3
 \end{pmatrix}.
$$

Obviously, the left and right OQW yield the following decomposition of the original one,
$$
 \mbb{T}\rho = (\mbb{T}_L\oplus0_3)\rho 
 + (0_1\oplus\mbb{T}_R)\rho 
 + (E\rho_2E^\dag-X\rho_2X^\dag-Y\rho_2Y^\dag)\otimes|2\>\<2|,
$$
where $0_i=O_2\otimes|i\>\<i|$ is the null operator on the site $|i\>$. Hence, Theorems~\ref{thm:split}.{\it(i)} and \ref{thm:split2}.{\it(i)} imply that the FR-function $f$ of the overlapping site $|2\>$ for the original OQW decomposes analogously in terms of the FR-functions $f_{L,R}$ of the same site with respect to the left/right OQW, i.e. 
\begin{equation} \label{eq:F-dec}
 f(z)\rho = f_L(z)\rho + f_R(z)\rho - 
 (X\rho_2X^\dag+Y\rho_2Y^\dag-E\rho_2E^\dag)\otimes|2\>\<2|.
\end{equation}

Bearing in mind \eqref{eq:F-f-OQW}, this decomposition can be equivalently rewritten using the corresponding reduced FR-functions $\mbb{F}$ and $\mbb{F}_{L,R}$. Following a convention similar to \eqref{eq:Fvec} for the representation of $\mbb{F}$ and $\mbb{F}_{L,R}$, one can check that these generating functions have the explicit form
$$
\begin{gathered}
 \mbb{F}(z) = \frac{1}{4}
 \begin{pmatrix}
 	\frac{z-4}{z-2} & -\frac{z}{z-2} & -\frac{z}{z-2} & -\frac{z}{z-2}
	\\[3pt]
	0 & 2 & 0 & 0
	\\
	0 & 0 & 2 & 0
	\\
	-\frac{z}{z-2} & \frac{z}{z-2} & \frac{z}{z-2} & \frac{z-4}{z-2} 
 \end{pmatrix},
 \\[3pt]
 \mbb{F}_L(z) = \frac{1}{2}
 \begin{pmatrix}
 	1 & 1 & 1 & 1
	\\
	\frac{1}{\sqrt{2}} & -\frac{1}{\sqrt{2}} & 
	\frac{1}{\sqrt{2}} & -\frac{1}{\sqrt{2}}
	\\[5pt]
	\frac{1}{\sqrt{2}} & \frac{1}{\sqrt{2}} & 
	-\frac{1}{\sqrt{2}} & -\frac{1}{\sqrt{2}}
	\\[5pt]
	-\frac{1}{z-2} & \frac{1}{z-2} & \frac{1}{z-2} & -\frac{1}{z-2} 
 \end{pmatrix},
 \kern20pt
 \mbb{F}_R(z) = \frac{1}{2}
 \begin{pmatrix}
 	-\frac{7z+4}{18(z-2)} & -\frac{z+4}{6(z-2)} & 
	-\frac{z+4}{6(z-2)} & \frac{z-4}{2(z-2)} 
	\\[5pt]
	\frac{5}{9\sqrt{2}} & -\frac{1}{3\sqrt{2}} & 
	\frac{5}{3\sqrt{2}} & -\frac{1}{\sqrt{2}}
	\\[7pt]
	\frac{5}{9\sqrt{2}} & \frac{5}{3\sqrt{2}} & 
	-\frac{1}{3\sqrt{2}} & -\frac{1}{\sqrt{2}}
	\\[7pt]
	\frac{25}{18} & -\frac{5}{6} & -\frac{5}{6} & \frac{1}{2} 
 \end{pmatrix}.
\end{gathered}
$$
Under this convention, the splitting \eqref{eq:F-dec} becomes
$$
 \mbb{F}(z) = \mbb{F}_L(z) + \mbb{F}_R(z) - 
 (X\otimes X+Y\otimes Y-E\otimes E),
$$
a relation that can be directly verified.

Consider the probability of returning to site $|2\>$ when starting at a mixed state $\rho=\rho_2\otimes|2\>\<2|$. The previous decomposition implies that such a probability $\pi(\rho\to|2\>)$ for the original OQW is determined by the corresponding ones $\pi_{L,R}(\rho\to|2\>)$ for the left/right OQW. The precise relation is given by 
$$
\begin{aligned}
 \pi(\rho\to|2\>) & = \lim_{x\uparrow1}\tr(f(x)\rho) 
 = \lim_{x\uparrow1}\tr(f_L(x)\rho) + \lim_{x\uparrow1}\tr(f_R(x)\rho) 
 - \tr((X^\dag X \kern-2pt + \kern-1pt Y^\dag Y \kern-2pt - \kern-1pt E^\dag E)\rho_2) 
 \\
 & = \pi_L(\rho\to|2\>) + \pi_R(\rho\to|2\>) - 1,
\end{aligned}
$$
where we have used that $X^\dag X+Y^\dag Y-E^\dag E=X^\dag X+Y^\dag Y+C^\dag C+D^\dag D-I_2=I_2$ and $\tr\rho_2=1$. The above relation is formally equal to the first splitting rule for RW recurrence given in Theorem~\ref{thm:split-RW}.{\it(i)}.

\end{ex}

\medskip

\noindent{\large \bf Acknowledgements}

\smallskip

The work of L. Vel\'azquez has been partially supported by the Spanish Government together with the European Regional Development Fund (ERDF) under grant MTM2014-53963-P from Ministerio de Econom\'{\i}a y Competitividad of Spain, and by Project E-64 of Diputaci\'on General de Arag\'on (Spain). 

We are thankful to C.~Lardizabal who made useful remarks on an earlier version of the paper, and to H.~Dym for help with some references.

\appendix

\section{Derivatives of Nevanlinna FR-functions}
\label{app:der}

The power expansion of FR-functions given in Proposition~\ref{prop:B-falt}.{\it (vii)} is valid in a neighbourhood of the origin for every bounded operator. The following result is a weak version of such an expansion for self-adjoint operators on Hilbert spaces which is valid even in the unbounded case.

\begin{lem} \label{lem:der}
If $f$ is the FR-function of a closed subspace $\HH_0$ with respect to a self-adjoint operator $T$, the following limits exist as operators in $\frak{B}(\HH_0)$, 
$$
\begin{aligned} 
 & f^{(0)}(0) = f(0) := \lim_{y\to0} f(iy) = PTP,
 & & \text{ if } \HH_0\subset\DD(T),
 \\ 
 & f^{(n)}(0) := \lim_{y\to0} \frac{f^{(n-1)}(iy)-f^{(n-1)}(0)}{iy} 
 = n!PTQ(QTQ)^{n-1}QTP,
 & & \text{ if } \HH_0\subset\DD(T^{n+1}), \kern12pt n\ge1,
\end{aligned}
$$
where $y\in\R$, $P$ is the orthogonal projection onto $\HH_0$ and $Q=1-P$. These normal limits must be understood as weak limits, thus they become limits in norm when $\dim\HH_0<\infty$. 
\end{lem}

\begin{proof}
Suppose $\HH_0\subset\DD(T)$. Then, from the proof of Theorem~\ref{thm:FR=N}.{\it (ii)} we know that $D$ is self-adjoint, so that $\C\setminus\R\subset\varrho(QTQ)$. Hence, using the notation \eqref{eq:T-block}, Proposition~\ref{prop:falt}.{\it (iii)} yields for $z\in\C\setminus\R$,
\begin{equation} \label{eq:f-PTP}
 f(z) = A + zB(1-zD)^{-1}C,
 \qquad A=PTP, \quad B=PTQ, \quad C=QTP, \quad D=QTQ, 
\end{equation}
where the relation $C \subset B^\dag$ becomes an equality because $\DD(C)=\HH_0$ since $\HH_0\subset\DD(T)$. 
Taking $z=iy$, $y\in\R$, in the above expression we get
$$
 |\<u|(f(iy)-A)v\>| = |y\<Cu|(i+yD)^{-1}Cv\>|, 
 \qquad \forall u,v\in\HH_0.
$$
The inequality $\|(z-T)^{-1}\| \le 1/|\im z|$, valid for every self-adjoint operator $T$, leads to $\|(i+yD)^{-1}\|\le1$. Thus, from the previous identity we find that
$$
 |\<u|(f(iy)-A)v\>| \le |y| \|Cu\|\|Cv\| \xrightarrow{y \to 0} 0,
 \qquad \forall u,v\in\HH_0.
$$
This proves that $f(0)=A$ as a weak limit.  

As for the rest of the limits, we will see first that the condition $\HH_0\subset\DD(T^{n+2})$ ensures that $\DD(D^nC)=\HH_0$ and $D^nC\HH_0\subset\sum_{k=0}^{n+1}T^k\HH_0\subset\DD(T)$, thus $\RR(D^nC)\subset\HH_0^\bot\cap\DD(T)=\DD(B)$ so that $BD^nC$ is everywhere defined on $\HH_0$. 

Let us argue by induction on $n$. 

For $n=0$, this follows by rewriting $\HH_0\subset\DD(T^2)$ as $\HH_0\subset\DD(T)$, $T\HH_0\subset\DD(T)$, which give $\DD(C)=\HH_0$ and
\begin{equation} \label{eq:rankC}
 \RR(C) = QT\HH_0 = (T-PT)\HH_0 \subset T\HH_0+\HH_0 \subset \DD(T),
 \qquad \text{ if } \HH_0\subset\DD(T^2).
\end{equation}

Assume now that $\DD(D^{n-1}C)=\HH_0$ and $D^{n-1}C\HH_0\subset\sum_{k=0}^nT^k\HH_0$ whenever $\HH_0\subset\DD(T^{n+1})$. Then, since the condition $\HH_0\subset\DD(T^{n+2})$ implies $\HH_0\subset\DD(T^{n+1})$ and $T^k\HH_0\subset\DD(T)$ for $0\le k\le n+1$, such a condition guarantees that $\RR(D^{n-1}C)\subset\HH_0^\bot\cap\DD(T)=\DD(D)$, hence $\DD(D^nC)=\HH_0$ and
\begin{equation} \label{eq:rankD^nC}
\begin{aligned}
 \RR(D^nC) = QTD^{n-1}C\HH_0 
 & \subset (T-PT) \sum_{k=0}^nT^k\HH_0
 \\
 & \subset \sum_{k=0}^{n+1}T^k\HH_0 \subset \DD(T),
 \quad \text{ if } \HH_0\subset\DD(T^{n+2}).
\end{aligned}
\end{equation}

Let us prove now the existence and give the explicit expression of $f^{(n)}(0)$ for $n\ge1$ by induction on $n$. If $\HH_0\subset\DD(T^2)$, then $\HH_0\subset\DD(T)$ so that $f(0)=A$ according to the previous result. Using \eqref{eq:f-PTP} we get for $z\in\C\setminus\R$,
$$
 \frac{f(z)-f(0)}{z} - BC = B((1-zD)^{-1}-1)C 
 = zB(1-zD)^{-1}DC,
$$
which makes sense on the whole subspace $\HH_0$ because, from \eqref{eq:rankC}, $\RR(C) \subset \HH_0^\bot\cap\DD(T)=\DD(D)$. Setting $z=iy$, $y\in\R$, we find that
$$
 |\<u|\textstyle(\frac{f(iy)-f(0)}{iy}-BC)v\>| 
 = |y\<Cu|(i+yD)^{-1}DCv\>|
 \le |y|\|Cu\|\|DCv\| \xrightarrow{y \to 0} 0,
 \qquad \forall u,v\in\HH_0, 
$$
concluding that $\displaystyle f'(0)=BC$ as a weak limit.

Suppose now that, for some $n\ge1$, $f^{(n)}(0)=n!BD^{n-1}C$ whenever $\HH_0\subset\DD(T^{n+1})$, and let us obtain $f^{(n+1)}(0)$ under the stronger condition $\HH_0\subset\DD(T^{n+2})$. From \eqref{eq:f-PTP} we get the following identity for $z\in\C\setminus\R$,
$$
\begin{aligned}
 \frac{f^{(n)}(z)-f^{(n)}(0)}{z} 
 & = n! z^{-1} B ((1-zD)^{-(n+1)}-1) D^{n-1}C
 \\
 & = n! z^{-1} B \sum_{k=0}^n(1-zD)^{-k} ((1-zD)^{-1}-1) D^{n-1}C
 \\
 & = n! B \sum_{k=1}^{n+1}(1-zD)^{-k} D^nC,
\end{aligned} 
$$
a relation valid on the whole subspace $\HH_0$ because we have proved that $D^nC$ is everywhere defined on $\HH_0$ for $\HH_0\subset\DD(T^{n+2})$. This allows us to rewrite
$$
\begin{aligned}
 \frac{f^{(n)}(z)-f^{(n)}(0)}{z} - (n+1)! B D^n C
 = n! B \sum_{k=1}^{n+1}((1-zD)^{-k}-1) D^nC.
\end{aligned} 
$$
Taking $z=iy$, $y\in\R$, and using the inequality $\|(i-yD)^{-k}\|\le\|(i-yD)^{-1}\|^k\le1$, we find that  
$$
\begin{aligned}
 |\<u|\textstyle(\frac{f^{(n)}(iy)-f^{(n)}(0)}{iy}-(n+1)!BD^nC)v\>| 
 & \le n! \sum_{k=1}^{n+1}|\<((1+iyD)^{-k}-1)Cu|D^nCv\>|
 \\
 & \kern-50pt  
 = n! \sum_{k=1}^{n+1}|\<(i-yD)^{-k}(1-(1+iyD)^k)Cu|D^nCv\>|
 \\
 & \kern-50pt 
 \le n! \sum_{k=1}^{n+1}\|(1-(1+iyD)^k)Cu\|\|D^nCv\| \xrightarrow{y \to 0} 0,
 \qquad \forall u,v\in\HH_0,
\end{aligned} 
$$
where we have used that, in view of \eqref{eq:rankD^nC}, the condition $\HH_0\subset\DD(T^{n+2})$ implies that $D^kC$ is everywhere defined on $\HH_0$ for $0\le k\le n+1$. Therefore, $f^{(n+1)}(0)=(n+1)!BD^nC$ as a weak limit.

Finally, we will see that $f^{(n)}(0)\in\frak{B}(\HH_0)$ if $\HH_0\subset\DD(T^{n+1})$. Note that $A,C\in\frak{B}(\HH_0)$ whenever $\HH_0\subset\DD(T)$ because then $TP$ is closed and everywhere defined on $\HH_0$. Since $f(0)=A$, this proves the result for $n=0$. Besides, for any operator $S\in\frak{B}(\HH_0)$ such that $\DD(BS)=\HH_0$, we have that $BS=PTQS\in\frak{B}(\HH_0)$ because $TQS$ is closed and everywhere defined on $\HH_0$. A similar argument shows that $DS\in\frak{B}(\HH_0)$ whenever $S\in\frak{B}(\HH_0)$ and $\DD(DS)=\HH_0$. This gives by induction on $n\ge1$ that $f^{(n)}(0)=n!BD^{n-1}C\in\frak{B}(\HH_0)$ if $\HH_0\subset\DD(T^{n+1})$ because we have proved that this condition ensures that $\DD(BD^{n-1}C)=\HH_0$.   
\end{proof}

\section{Characterizations of degenerate Nevanlinna functions}
\label{app:deg}

This appendix is devoted to the study of properties of Nevanlinna functions with emphasis on the distinction between degenerated and non-degenerated ones. Or aim is to prove Proposition~\ref{prop:deg-NFR}. For this purpose we will need an additional characterization of scalar degenerated Nevanlinna functions.

\begin{lem} \label{lem:deg}
A scalar Nevanlinna function $f$ is degenerate iff 
$$
 \lim_{y\to0} \frac{\im f(iy)}{y} = 0.
$$
\end{lem}

\begin{proof}
Degenerate scalar Nevanlinna functions are the real constant functions, so their imaginary part is null and trivally satisfy the asymptotic condition. 

To prove the converse we resort to the standard integral representation \eqref{eq:N-mu} of Nevanlinna functions, specialized to the scalar case, which states that $f$ is given by
$$
 f(z) = a + bz + \int \frac{1+zt}{t-z} \, d\nu(t),
 \qquad a\in\R, \qquad b\ge0,
$$
with $\nu$ a finite measure on the real line. Then,
$$
 \frac{\im f(z)}{\im z} = b + \int \frac{1+t^2}{|t-z|^2} \, d\nu(t).
$$
Setting $z=iy$ we obtain
$$
 \frac{\im f(iy)}{y} = b + \int \frac{1+t^2}{t^2+y^2} \, d\nu(t).
$$
The above expression increases when $|y|$ decreases, thus the only way to have a zero limit for $y\to 0$ is if both, $b$ and the measure $\nu$, vanish. Then, $f(z)=a$ is a real constant. 
\end{proof}

Now we can prove Proposition~\ref{prop:deg-NFR} for operator valued Nevanlinna functions. 

\begin{proof} [{\it Proof of Proposition~\ref{prop:deg-NFR}}]

It is enough to prove the first three equalities of this proposition: if $f\colon\C\setminus\R\to\frak{B}(\HH_0)$ is a Nevanlinna function and $x\in\R$, then 
$$
 \ker\im f(z) = \{v\in\HH_0:fv\text{ constant on }\C\setminus\R\} = 
 \ker(f(z)-f(x)) = \ker f'(x),
 \quad z\in\C\setminus\R, 
$$
where $f(x):=\lim_{y\to0}f(x+iy)$ and ${\displaystyle f'(x)}:=\lim_{y\to0}(f(x+iy)-f(0)))/iy$ are assumed to exist as weak limits everywhere defined on $\HH_0$. The characterizations {\it (i)} to {\it (vi)} of Proposition~\ref{prop:deg-NFR} are direct consequences of the above three equalitites.

By using the shifted Nevanlinna function $f(z-x)$ we can reduce the proof to the case $x=0$, bearing in mind that a real shift preserves Nevanlinna functions as well as their degenerate character. Hence, we will suppose without loss that $x=0$. 

Let $\ker\im f:=\ker \im f(z)$. This is a closed subspace of $\HH_0$ which we know is independent of $z\in\C\setminus\R$. The orthogonal projections $p$, $q$ of $\HH_0$ onto $\ker\im f$ and $\ker\im f^\bot$ generate the block representation
$$
 f(z) 
 = \begin{pmatrix}
 a(z) & b(z) \\ c(z) & d(z)
 \end{pmatrix} 
 = \begin{pmatrix}
 pf(z)p & pf(z)q \\ qf(z)p & qf(z)q
 \end{pmatrix},
 \qquad
 z\in\C\setminus\R.
$$ 
Since $(\im f(z))p=0$, we find that $a(z)^\dag=a(z)$ and $b(z)^\dag=c(z)$. Hence, $a(z)$, $c(z)$ and their adjoints are analytic on $\C\setminus\R$. This implies that $a(z)=a_\pm$ and $c(z)=c_\pm$ are constant functions for $z\in\C_\pm$. Also, $f(\overline{z})=f(z)^\dag$ gives $a(\overline{z})=a(z)$ and $c(\overline{z})=c(z)$, thus $a_+=a_-$ and $c_+=c_-$. That is, $a(z)=a$ and $c(z)=c$ are constant operator valued functions for $z\in\C\setminus\R$, so that   
$$
 f(z) 
 = \begin{pmatrix}
 a & c^\dag \\ c & d(z)
 \end{pmatrix},
 \qquad z\in\C\setminus\R.  
$$
Therefore, $f(z)v=(a+c)v$ is independent of $z\in\C\setminus\R$ for each $v\in\ker\im f$.

Conversely, if $fv$ is constant on $\C\setminus\R$ then $f(z)v=f(\overline{z})v=f(z)^\dag v$ for $z\in\C\setminus\R$, which implies that $v\in\ker\im f$. This finishes the proof of the first equality in the proposition, $\ker\im f=\{v\in\HH_0:fv\text{ is constant on }\C\setminus\R\}$.

Suppose now that $f(0)$ exists as an operator everywhere defined on $\HH_0$. Let $v\in\HH_0$ such that $fv$ is constant on $\C\setminus\R$. Then, given $z\in\C\setminus\R$, we have that $\<u|(f(z)-f(0))v\>=\lim_{y\to0}\<u|(f(z)-f(iy))v\>=0$ for every $u\in\HH_0$, thus $(f(z)-f(0))v=0$. Bearing in mind the previous result, this means that $\ker\im f\subset\ker(f(z)-f(0))$.  

To prove the opposite inclusion remember that $f(0)$ is self-adjoint by Proposition~\ref{prop:der-p}. Therefore, $v\in\ker(f(z)-f(0))$ for some $z\in\C\setminus\R$ not only yields $\<v|(f(z)-f(0))v\>=0$, but also $\<v|(f(z)^\dag-f(0))v\>=\<(f(z)-f(0))v|v\>=0$. Combining these two equalities we get $\<v|\im f(z)v\>=0$, which proves that $v\in\ker\im f(z)$ because $\im f(z)/\im z\ge0$.

Finally, assume the existence of $\displaystyle f'(0)$ everywhere defined on $\HH_0$. If $v\in\HH_0$ makes $fv$ constant on $\C\setminus\R$, then $(f(iy)-f(0))v=0$ for every $y\in\R$, so $\displaystyle f'(0)v=0$. In other words, $\ker\im f\subset\ker\displaystyle f'(0)$.

Conversely, suppose that $v\in\ker\displaystyle f'(0)$ and consider the scalar Nevanlinna function $g(z)=\<v|f(z)v\>$. Then, using that $f(iy)^\dag=f(-iy)$ for $y\in\R$, we get
$$
 \lim_{y\to0} \frac{\im g(iy)}{y} 
 = \frac{1}{2} \lim_{y\to0} 
 \left( \<v|\textstyle\frac{f(iy)-f(0)}{iy}v\> 
 + \<v|\textstyle\frac{f(-iy)-f(0)}{-iy}v\> \right)
 = \<v|f'(0)v\>=0. 
$$
By Lemma~\ref{lem:deg}, $g$ is a real constant function, so $\<v|\im f(z)v\>=\im g(z)=0$ for $z\in\C\setminus\R$, which implies that $v\in\ker\im f$.   
\end{proof}




\end{document}